\documentclass[11pt,reqno]{amsart}
\usepackage{amsmath,amssymb,amsthm,amsfonts,amsopn,cite,todonotes}
\usepackage{dsfont}
\usepackage{enumitem}
\usepackage[utf8]{inputenc}
\usepackage{mathrsfs}

\usepackage{color,url,setspace,wasysym}
\usepackage{lscape}
\usepackage[papersize={250mm, 297mm}, inner=36mm, outer=36mm, top=36mm, bottom=34mm, headsep=12mm, footskip=12mm]{geometry}
\usepackage[unicode,pdfencoding=unicode,pdfusetitle,
 bookmarks=true,bookmarksnumbered=false,bookmarksopen=false,
 breaklinks=false,pdfborder={0 0 1},backref=false,colorlinks=false]
 {hyperref}

\usepackage[ddmmyyyy,hhmmss]{datetime}
\newdateformat{daymonthyeardate}{%
  \THEDAY.\THEMONTH.\THEYEAR}
\usepackage{fancyhdr}
\usepackage{lastpage}
\usepackage{etoolbox}
\def\felixleftmark{}
\def\felixrightmark{}
\newcommand\updatemark{\markleft{\felixleftmark~\ifdefempty{\felixrightmark}{}{--- \felixrightmark}}}
\let\origsection\section
\renewcommand{\section}[1]{\origsection{#1}\sectionmark{#1}}
\renewcommand\sectionmark[1]{\def\temp{#1}\def\felixrightmark{}\edef\x{\noexpand\def\noexpand\felixleftmark{Section \thesection: \expandonce{\temp}}}\x\updatemark}
\let\origsubsection\subsection
\renewcommand{\subsection}[1]{\origsubsection{#1}\subsectionmark{#1}}
\renewcommand\subsectionmark[1]{\def\secondtemp{#1}\edef\y{\noexpand\def\noexpand\felixrightmark{Subsection \thesubsection: \expandonce{\secondtemp}}}\y\updatemark}

\theoremstyle{plain}
\newtheorem{maintheorem}{Theorem}
\newtheorem{theorem}{Theorem}[section]
\newtheorem{proposition}[theorem]{Proposition}
\newtheorem{lemma}[theorem]{Lemma}
\newtheorem{corollary}[theorem]{Corollary}
\theoremstyle{definition}
\newtheorem{definition}[theorem]{Definition}

\theoremstyle{remark}
\newtheorem{remark}[theorem]{Remark}
\theoremstyle{remark}
\newtheorem*{rem*}{Remark}

\numberwithin{equation}{section}

\newcommand{\bd}{\mathbf}

\newcommand{\LtRd}{\mathbf L^2(\mathbb{R}^d)}

\newcommand{\LiRd}{\mathbf L^1(\mathbb{R}^d)}
\newcommand{\Ltw}{\bd L^2_{\sqrt{w}}}

\newcommand{\RR}{\mathbb{R}}
\newcommand{\R}{\mathbb{R}}
\newcommand{\ZZ}{\mathbb{Z}}
\newcommand{\Z}{\mathbb{Z}}
\newcommand{\NN}{\mathbb{N}}
\newcommand{\N}{\mathbb{N}}
\newcommand{\CC}{\mathbb{C}}

\def\esssup{\mathop{\operatorname{ess~sup}}}

\def\supp{\mathop{\operatorname{supp}}}
\def\sgn{\mathop{\operatorname{sgn}}}
\newcommand{\inner}[2]{\langle #1,#2 \rangle}

\makeatletter
\newcommand{\LeftEqNo}{\let\veqno\@@leqno}

\DeclareFontFamily{U}{mathx}{\hyphenchar\font45}
\DeclareFontShape{U}{mathx}{m}{n}{
      <5> <6> <7> <8> <9> <10>
      <10.95> <12> <14.4> <17.28> <20.74> <24.88>
      mathx10
      }{}
\DeclareSymbolFont{mathx}{U}{mathx}{m}{n}
\DeclareFontSubstitution{U}{mathx}{m}{n}
\DeclareMathAccent{\widecheck}{0}{mathx}{"71}
\DeclareMathAccent{\wideparen}{0}{mathx}{"75}

\newcommand{\LtDF}{\bd L^{2,\mathcal F}(D)}

\newcommand{\Indicator}{{\mathds{1}}}
\newcommand{\modulation}{\bd{M}}
\newcommand{\translation}{\bd{T}}
\newcommand{\lebesgue}{\bd{L}}

\newcommand{\Fourier}{\mathcal{F}}
\newcommand{\DecompSp}{\mathcal{D}}

\newcommand{\CalQ}{\mathcal{Q}}
\newcommand{\CalS}{\mathcal{S}}
\newcommand{\CalP}{\mathcal{P}}

\newcommand{\CalO}{\mathcal{O}}
\newcommand{\CalB}{\mathcal{B}}

\newcommand{\CalC}{\mathcal{C}}
\newcommand{\GL}{\mathrm{GL}}
\newcommand{\DistributionSpace}{\mathcal{D}'}
\newcommand{\TestFunctionSpace}{\mathcal{D}}
\newcommand{\Schwartz}{\mathcal{S}}
\newcommand{\Co}{\operatorname{Co}}

\newcommand{\identity}{\mathrm{id}}

\newcommand{\smallVSpace}{\rule{0pt}{2.5mm}}
\newcommand{\eps}{\varepsilon}
\newcommand{\with}{\,:\,}
\newcommand{\PhVar}{\lambda}
\newcommand{\PhVarA}{\rho}

\newcommand{\PhVarC}{{\nu}}
\newcommand{\PhSpace}{\Lambda}
\newcommand{\GoodVectors}{\mathcal{H}_v^1}
\newcommand{\Reservoir}{(\GoodVectors)^\urcorner}

\newcommand{\vrho}{\varrho}
\newcommand{\vrhoinv}{\vrho_{\ast}}
\newcommand{\vsig}{\varsigma}
\newcommand{\vsiginv}{\vsig_{\ast}}

\usepackage{scalerel}
\renewcommand\upsilon{{\scaleobj{0.65}{\Upsilon}}}

\let\emptyset\varnothing

\begin{document}

\author{Nicki Holighaus\texorpdfstring{$^\dag$ $^\ddag$}{}}
 \address{$^\dag$ Acoustics Research Institute Austrian Academy of Sciences, Wohllebengasse 12--14, A-1040 Vienna, Austria}
 \thanks{This work was supported by the Austrian Science Fund (FWF): I\,3067--N30.
FV acknowledges support by the German Science Foundation (DFG)
in the context of the Emmy Noether junior research group VO 2594/1--1.
NH is grateful for the hospitality and support of the Katholische Universität Eichstätt-Ingolstadt during his visit.
FV would like to thank the Acoustics Research Institute for the hospitality during several visits, which were partially
supported by the Austrian Science Fund (FWF): 31225--N32.}
 \email{nicki.holighaus@oeaw.ac.at}
 \author{Felix Voigtlaender\texorpdfstring{$^*$ $^\ddag$}{}}
 \address{$^*$ Mathematical Institute for Machine Learning and Data Science, Katholische Universität Eichstätt-Ingolstadt, Auf der Schanz 49, 85049 Ingolstadt, Germany}
 \email{felix.voigtlaender@ku.de}

 \title[Smoothness spaces for warped time-frequency representations]
       {Smoothness spaces for warped time-frequency representations---Decomposition spaces and embedding relations}



 \begin{abstract}
In a recent paper,
 we have shown that warped time-frequency representations provide a 
rich framework for the construction and study of smoothness spaces matched to 
very general phase space geometries obtained by diffeomorphic deformations of $\RR^d$. 
Here, we study these spaces, obtained through the application of general coorbit 
theory, using the framework of decomposition spaces. This allows us to derive embedding 
relations between coorbit spaces associated to different warping functions, and relate them 
to established, important smothness spaces. In particular, we show that we obtain $\alpha$-modulation spaces 
and spaces of dominating mixed smoothness as special cases and, in contrast, that this is only 
possible for Besov spaces if $d=1$. 
 \end{abstract}

  \maketitle

  \renewcommand{\thefootnote}{\fnsymbol{footnote}} 
  \footnotetext[3]{Both authors contributed equally to this work.}
  \renewcommand*{\thefootnote}{\arabic{footnote}}

  \markright{}

  \section{Introduction}

    \emph{Warped time-frequency systems} were introduced in~\cite{howi14,bahowi15,VoHoPreprint} as a means for the construction of time-frequency representations adapted to rather general \emph{frequency scales}. In these works, it was demonstrated that such systems yield continuous, tight Parseval
    frames that admit the construction of general coorbit spaces under mild conditions. As translation-invariant spaces inducing a certain frequency scale, it seems natural to suspect a link of these \emph{warped coorbit spaces}  to decomposition spaces~\cite{fegr85,boni07}.    
    Whereas coorbit theory~\cite{fegr89,fegr89-1,gr91}, and its generalization to abstract continuous 
    frames~\cite{fora05,rauhut2011generalized,kempka2015general,bahowi15}, classifes a function $f$ 
    according to the size of its voice transform $V_\Psi f$ with respect to some frame $\Psi$, decomposition spaces $\mathcal{D}(\mathcal{Q}, \lebesgue^p, \ell_u^q)$ directly classify 
    functions with respect to their frequency localization, as determined by a rather general covering $\mathcal{Q}$ of the 
    frequency domain. The definition of these decomposition spaces generalizes classical Besov spaces, replacing the usual dyadic covering by a more general one: Given a suitable covering $\mathcal{Q}=\left(Q_{i}\right)_{i\in I}$
    and a suitable partition of unity $\left(\varphi_{i}\right)_{i\in I}$
    subordinate to $\mathcal{Q}$, the norm on the decomposition space
    $\mathcal{D}\left(\mathcal{Q},\lebesgue^{p},\ell_{u}^{q}\right)$
    is given by
    \[
        \left\Vert g\right\Vert _{\mathcal{D}\left(\mathcal{Q},\lebesgue^{p},\ell_{u}^{q}\right)}
        = \left\Vert
                \left(
                    u_{i}
                    \cdot \left\Vert
                            \mathcal{F}^{-1}\left(\varphi_{i}\cdot\smash{\widehat{g}}\right)
                        \right\Vert_{\lebesgue^{p}}
                \right)_{i\in I}
            \right\Vert_{\ell^{q}},
    \]
    where $u=\left(u_{i}\right)_{i\in I}$ is a suitable weight. Thus, the norm
    is obtained by a weighted $\ell^{q}$-norm of the $\lebesgue^{p}$-norms
    of the frequency-localized pieces $\mathcal{F}^{-1}\left(\varphi_{i}\cdot\widehat{g}\right)$
    of $g$, where the precise frequency localization is determined by
    the covering $\mathcal{Q}$.
    
    Relations between coorbit spaces associated to translation-invariant 
    time-frequency representations and decomposition spaces occur naturally, as previously observed for 
    ($\alpha$-)modulation spaces, see \cite{BorupNielsenAlphaModulationSpaces,fegr85,gr92-2,daforastte08,speckbacher2016alpha}, and 
    wavelet coorbit spaces, see \cite{fuhr2015wavelet,fuhr2021coorbit}. It is intuitively plausible that if the frame $\Psi$ has the same (or very similar) frequency localization as
    described by $\mathcal{Q}$, then the coorbit space $\mathrm{Co}(\Psi,\lebesgue_\kappa^p)$ should coincide with a suitable
    decomposition space $\mathcal{D}(\mathcal{Q}, \lebesgue^p, \ell_{\tilde{u}}^p)$.
    In fact, several recent works~\cite{boni07,CompactlySupportedFramesForDecompositionSpaces, ni14,StructuredBanachFrames,StructuredBanachFrames2} have considered certain sequence norms of frame coefficients, e.g., with respect to a suitable generalized shift-invariant system, for the characterization of decomposition spaces, in varying generality. 

    A rather powerful embedding theory~\cite{voigtlaender2015embedding,voigtlaender2016embeddings}
    for decomposition spaces was recently developed by one of the authors. Using this theory, once a given function space $X$
    is identified as a decomposition space $X = \mathcal{D}(\mathcal{Q}, \lebesgue^p, \ell_u^q)$,
    the existence of embeddings between $X$ and any other decomposition space
    $\mathcal{D}(\mathcal{P}, \lebesgue^{p_2}, \ell_v^{q_2})$ can in many cases
    be decided by comparing certain geometric and combinatorial properties of the coverings $\mathcal{Q}, \mathcal{P}$.
    Embeddings into classical Sobolev spaces have also been considered \cite{DecompositionIntoSobolev}.
    
    In the present paper, we specifically consider the coorbit spaces associated to \emph{warped time-frequency systems}, introduced in \cite{howi14,bahowi15,VoHoPreprint}. 
    Given a subset $D\subset \RR^d$, a diffeomorphism $\Phi\colon D\rightarrow \RR^d$, and a prototype function $\theta\colon \RR^d \rightarrow \CC$, 
    the warped time-frequency system $\mathcal G(\theta,\Phi) = \{g_{y,\omega}\}_{y\in\RR^d,\omega\in D}$ is defined by 
    \begin{equation}
      g_{y,\omega} := \translation_y \left( \mathcal F^{-1} g_\omega \right)\quad \text{with}\quad g_\omega = |\det{\mathrm{D}\Phi}(\omega)|^{1/2}\cdot (\translation_{\Phi(\omega)}\theta )\circ \Phi,\quad \text{for all } (y,\omega)\in\RR^d\times D.
    \end{equation}
    By virtue of composition with the diffeomorphism $\Phi$, referred to as \emph{warping function}, the family $\{g_\omega\}_{\omega\in D}$ covers the frequency domain $D$ uniformly 
    \emph{in warped coordinates},  thereby defining a frequency scale $\Phi^{-1}$ to which $\mathcal G(\theta,\Phi)$ is naturally adapted. 
    
    Based on results in \cite{VoHoModules}, the authors have provided in \cite{VoHoPreprint} straightforward conditions for $\mathcal G(\theta,\Phi)\subset \mathcal{T}$ to generate Banach coorbit spaces 
    \[\Co(\Phi,Y) = \{ f\in \mathcal{R} : \|(\langle f,g_{y,\omega}\rangle)_{(y,\omega)\in \RR^d\times D}\|_Y < \infty\}\]
    and to obtain atomic decompositions in $\Co(\Phi,Y)$ by sampling $\mathcal G(\theta,\Phi)$. Here, $\mathcal{T}$ is a suitable Banach space of test functions, the reservoir 
    $\mathcal{R}$ is the space of continuous, conjugate-linear functionals on $\mathcal{T}$, and $Y$ is a rich, solid Banach space. In particular, the coorbit spaces $\Co(\Phi,\lebesgue^{p,q}_\kappa)$, with respect to the weighted mixed-norm spaces $\lebesgue^{p,q}_v$, $p,q\geq 1$ are well-defined Banach spaces under certain conditions, see also Section \ref{subsec:warpedTF}. In the present work, we combine these results with the embedding theory for decomposition spaces~\cite{voigtlaender2016embeddings}.
   
\subsection{Main results}\label{subsec:mainresults}
   
   Our central result identifies the coorbit spaces $\Co(\Phi,\lebesgue^{p.q}_{\kappa})$ associated with warped time-frequency representations, hereafter also referred to by the conveniently shorter, if slightly ambiguous, term \emph{warped coorbit spaces}, with certain decomposition spaces. Crucially, we require a suitable covering $\CalQ$, induced by $\Phi$. In fact, by simply pulling back the uniform covering 
$\left(\delta \cdot B_{r} (\ell)\right)_{\ell \in \ZZ^d}$ (for suitable $r>0$)  of the ``warped frequency space'' $\RR^d$ using $\Phi^{-1}$, one obtains a the desired covering.

 \begin{definition}\label{def:frequency_space_covering}
     Let $\Phi : D \to \RR^d$ be a diffeomorphism.
     For $\delta, r > 0$, the family
     \[
       \CalQ_{\Phi}^{(\delta, r)}
       := \big(\, Q_{\Phi,k}^{(\delta, r)} \,\big)_{k \in \ZZ^d},\quad  \text{with } Q_{\Phi,k}^{(\delta, r)} := \Phi^{-1}(\delta \cdot B_r (k)).
     \]
     is called the \emph{$(\delta, r)$-fine frequency covering induced by $\Phi$}.
 \end{definition}

  The $(\delta, r)$-fine frequency covering $\CalQ_{\Phi}^{(\delta, r)}$ induced by $\Phi$ is studied more detail in Section \ref{sec:coverings}. With this family of coverings, we obtain the following identification result, proven in Section \ref{sec:CoorbitAsDecomposition}. 
   
   \begin{maintheorem}\label{thm:main1}
     Let $D\subset \RR^d$ open, $\Phi\colon D\rightarrow \RR^d$ be a $\mathcal C^\infty$-diffeomorphism and a $(d+1)$-admissible warping function with control weight $v_0\colon \RR^d \rightarrow [1,\infty)$, and let the weight function $\kappa: D\rightarrow \RR^+$ be such that $\kappa\circ\Phi^{-1}$ is $\tilde{\kappa}$-moderate, for some submultiplicative, radially increasing weight $\tilde{\kappa}\colon \RR^d\rightarrow \RR^+$. Finally, fix $\delta > 0$ and $r > \sqrt{d} / 2$ and consider the $(\delta, r)$-fine frequency covering $\CalQ_{\Phi}^{(\delta, r)}$ induced by $\Phi$.

    If we set $u = (u_k)_{k \in \ZZ^d}$ with $u_k = \kappa (\Phi^{-1}(\delta k)) \cdot [\det(\mathrm{D}\Phi^{-1}(\delta k))]^{\frac{1}{q} - \frac{1}{2}}$, then 
    \[
       \Co (\Phi, \lebesgue_\kappa^{p,q}) = \DecompSp (\CalQ_\Phi^{(\delta, r)}, \lebesgue^p, \ell_u^q).
    \]
 \end{maintheorem}
 
 To ensure that the warped coorbit space $\Co (\Phi, \lebesgue_\kappa^{p,q})$ is well-defined, the assumption that $\Phi$ is $(d+1)$-admissible (see Definition \ref{assume:DiffeomorphismAssumptions} for details) was required in \cite[e.g., Theorem 6.1]{VoHoPreprint}, as was the $\tilde{\kappa}$-moderateness of $\kappa\circ\Phi^{-1}$ (note that we restrict here to weights $\kappa$ that are independent of the time/space variable). Hence, the only additional requirement on $\Phi$ is that it is $\mathcal C^\infty$, which can always be achieved with arbitrarily small changes to $\Phi$, e.g., via convolution with a mollifier. 
 
 As consequences of Theorem \ref{thm:main1}, we further obtain embedding relations, or even equality, between warped coorbit spaces and important families of (semi-)classical smoothness spaces. In particular, we show that certain warped coorbit spaces coincide with $\alpha$-modulation spaces~spaces~\cite{gr92-2,BorupNielsenAlphaModulationSpaces} and Besov-type spaces of dominating mixed smoothness~\cite{nikol1962boundary,nikol1963boundary,vybiral2006function}. For inhomogeneous Besov spaces, we show that in dimension $d=1$, they coincide with certain warped coorbit spaces as well. For $d>1$, we show a certain closeness relation and that equality cannot be achieved. Specifically, we obtain the following results, proven in Sections \ref{sec:WarpedCoorbitAndBesov}--\ref{sec:DominatingMixedSmoothness} and using the notion of \emph{radial warping functions}. The radial warping function $\Phi_{\vrho}\colon\RR^d \rightarrow \RR^d$ induced by a smooth, antisymmetric diffeomorphism $\vrho\colon \RR\rightarrow \RR$ is given by $\Phi_{\vrho}(\xi) = \tfrac{\xi}{|\xi|}\cdot \vrho(|\xi|)$, see Section \ref{sec:radial} and \cite[Section 8]{VoHoPreprint} for details. The result uses a \emph{slow start construction} that linearizes $\vrho$ in a small neighborhood of the origin, but leaves it otherwise unchanged. This construction is introduced in Section \ref{sec:radial}, Theorem \ref{thm:SlowStartFullCriterion}, 
 
 \begin{maintheorem}\label{thm:main2}
     For $\alpha\in (-\infty,1)$, let $\vrho_\alpha$ be an arbitrary \emph{slow start version} of $\vsig_\alpha\colon [0,\infty)\rightarrow [0,\infty),\ \xi\mapsto (1+|\xi|)^{1-\alpha}-1$ and $\vrho_1$ a slow start version of $\vsig_1\colon [0,\infty)\rightarrow [0,\infty),\ \xi\mapsto \ln(1+|\xi|)$. Then, the following hold for all $1\leq p,q \leq \infty$:
     \begin{enumerate}
      \item For $\alpha\in (-\infty,1)$ and $s\in\RR$, we have
      \[
        \Co(\Phi_{\vrho_{\alpha}},\lebesgue^{p,q}_{\kappa}) = M_{p,q}^{s,\alpha}(\RR^d) \, ,\quad \text{with}\quad \kappa := \kappa^{(\alpha,s,d,q)}\,\colon\, \RR^d\rightarrow [0,\infty),\ \xi \mapsto (1 + |\xi|)^{s-d\alpha(q^{-1}-2^{-1})}, 
     \]
     up to canonical identifications. Here, $M_{p,q}^{s,\alpha}(\RR^d)$ is the $\alpha$-modulation space with exponents $p,q$ and weight parameter $s$.
     \item If $\Phi_{\vrho_{1}}\colon \RR \rightarrow \RR$ is the $1$-dimensional radial warping function associated with $\vrho_{1}$ and 
     \[
        \Phi: \RR^d \rightarrow \RR^d,\ \xi\mapsto (\Phi_{\vrho_{1}}(\xi_1),\ldots,\Phi_{\vrho_{1}}(\xi_d)),
     \]
     then we have, for all $s\in\RR$, 
    \[
        S^s_{p,q} B (\RR^d)
        = \Co (\Phi, \lebesgue^{p,q}_\kappa)\, ,\quad \text{with}\quad \kappa \colon \RR^d \to \RR^+, \xi = (\xi_1,\dots,\xi_d) \mapsto \prod_{\ell=1}^d (1 + |\xi_\ell|)^{s_\ell + 2^{-1} - q^{-1}} \, ,
    \]  
    up to canonical identifications. Here, $S^s_{p,q} B (\RR^d)$ is the Besov-type space of dominating mixed smoothness with exponents $p,q$ and weight parameter $s$.
    \item Let $s\in\RR$, $d\in\NN$, and $\Phi_{\vrho_{1}}\colon \RR^d \rightarrow \RR^d$ be the $d$-dimensional radial warping function associated with $\vrho_{1}$. We have, for all $\eps>0$, and with $\kappa^{(s,d,q)} \colon \RR^d \rightarrow \RR^+,\ \xi \mapsto (1+|\xi|)^{s+d(2^{-1}-q^{-1})}$, 
    \[
      \Co (\Phi_{\vrho_{1}}, \lebesgue^{p,q}_{\kappa^{(s+\eps,d,q)}})
    \hookrightarrow
    B_s^{p,q} (\RR^d)
    \hookrightarrow
    \Co (\Phi_{\vrho_{1}}, \lebesgue^{p,q}_{\kappa^{(s-\eps,d,q)}}),
    \]
    up to canonical identifications. Here, $B_s^{p,q} (\RR^d)$ is the inhomogeneous Besov space with exponents $p,q$ and weight parameter $s$. If and only if $d=1$, or $p=q=2$, we have the equality
    \[
      B_s^{p,q} (\RR^d) = \Co ( \Phi_{\vrho_{1}}, \lebesgue^{p,q}_{\kappa^{(s,q)}}).
    \]   
    \end{enumerate}
    \end{maintheorem}
    
    The formal definition of the spaces $M_{p,q}^{s,\alpha}(\RR^d)$, $S^s_{p,q} B (\RR^d)$ and $B_s^{p,q} (\RR^d)$ is given in Sections \ref{sec:WarpedCoorbitAndBesov}, \ref{sec:alpha} and \ref{sec:DominatingMixedSmoothness}. Finally, we obtain conditions for equality of warped coorbit spaces and embeddings between warped coorbit spaces with respect to radial warping functions, in Section \ref{sec:WarpedEmbeddings}. 
    
    \begin{maintheorem}\label{thm:main3}
    Let $\Phi_1,\Phi_2 \colon D\rightarrow \RR^d$ satisfy the conditions of Theorem \ref{thm:main1} and let $\kappa : \RR^d \rightarrow \RR^+$ satisfy the conditions of Theorem \ref{thm:main1} with $\Phi = \Phi_2$. Then the following hold:
      \begin{enumerate}
        \item Suppose that there is a constant $C>0$, such that 
       \[
            \| \mathrm{D}\Phi_{3-j} (\Phi_{j}^{-1} (\xi)) \cdot \mathrm{D}\Phi_j^{-1}(\xi) \| \leq C,\quad \text{for all } \xi\in \RR^d \text{ and } j\in\{1,2\}.
       \]
      Then we have, for all $p,q\in [1,\infty]$, the equality $\Co(\Phi_{1},\lebesgue_{\kappa}^{p,q}) = \Co(\Phi_{2},\lebesgue_\kappa^{p,q})$.
      \item Suppose that $\Phi_1= \Phi_{\vrho_1} ,\Phi_2 = \Phi_{\vrho_2}$ are radial, and that there are a constant $C>0$ and a function $C_{\vrho_1} : [1,\infty)\rightarrow [1,\infty)$, such that
      \[
            \vrho_2'(\xi) \leq C \vrho_1'(\xi),\quad \text{and}\quad C_{\vrho_1}(\alpha)^{-1}\vrho_1'(\xi)\leq \vrho_1'(\alpha\xi) \leq C_{\vrho_1}(\alpha)\vrho_1'(\xi),\quad \text{for all } \alpha\in [1,\infty),\ \xi\in\RR^+,
      \]
     then we have, for all $1\leq p,q\leq \infty$, the embeddings
     \[
        \Co(\Phi_{\vrho_1},\lebesgue_{\kappa_{\vrho_1,\vrho_2,-\tilde{t}}}^{p,q})\hookrightarrow \Co(\Phi_{\vrho_2},\lebesgue_\kappa^{p,q}) \hookrightarrow \Co(\Phi_{\vrho_1},\lebesgue_{\kappa_{\vrho_1,\vrho_2,t}}^{p,q}),
     \]
    with $\tilde{t} = \max(0,\max(p^{-1},1-p^{-1})-q^{-1})\leq 1$, and $t = \max(0,q^{-1}-\min(p^{-1},1-p^{-1}))\leq 1$. Here, for $t\in\RR$, $\kappa_{\vrho_1,\vrho_2,t} = \kappa \cdot \left[\frac{w_2\circ\Phi_{\vrho_2}}{w_1\circ \Phi_{\vrho_1}}\right]^{\frac{1}{q}-\frac{1}{2}-t} \quad\text{and}\quad w_j = \det(\mathrm{D}\Phi^{-1}_{\vrho_j}),\quad j\in\{1,2\}$.
     \end{enumerate}
    \end{maintheorem}
    
    We now proceed to introduce fundamental notation used in the paper and, in Section \ref{sec:prelims}, recall essential results on warped time-frequency systems, warping functions, and their associated coorbit spaces, as well as fundamental notions and results from the theory of decomposition spaces. Note that some results of potential interest, but not crucial to our main contributions, are deferred to the appendices. 
           
    \subsection{Notation and fundamental definitions}
\label{subsec:notation}
We write $\R^+ = (0,\infty)$ for the set of positive real numbers. 
For the composition of functions $f$ and $g$ we use the notation $f \circ g = f(g(\bullet))$.
For a subset $M \subset X$ of a fixed base set $X$
(which is usually understood from the context), we use the indicator function
$\Indicator_{M}$ of the set $M$, where $\Indicator_M (x) = 1$ if $x \in M$
and $\Indicator_M (x) = 0$ otherwise.

The (topological) dual space of a (complex) topological vector space $X$
(i.e., the space of all continuous linear functionals $\varphi : X \to \CC$)
is denoted by $X'$, while the (topological) \emph{anti}-dual of a Banach space $X$
(i.e., the space of all \emph{anti-linear} continuous functionals on $X$)
is denoted by $X^\urcorner$.

We use the convenient short-hand notations $\lesssim$ and $\asymp$,
where $A \lesssim B$ means $A \leq C \cdot B$, for some constant $C > 0$ that depends
on quantities that are either explicitly mentioned or clear from the context.
$A \asymp B$ means $A \lesssim B$ and $B \lesssim A$.

\subsubsection{Norms and related notation}

We write $|x|$ for the Euclidean norm of a vector $x \in \RR^d$,
and we denote the operator norm of a linear operator $T : X \to Y$
by $\| T \|_{X \to Y}$, or by $\| T \|$, if $X,Y$ are clear from the context.
In the expression $\| A \|$, a matrix $A \in \RR^{n \times d}$ is interpreted
as a linear map $(\RR^d,|\bullet|) \to (\RR^n,|\bullet|)$.
The open (Euclidean) ball around $x \in \R^d$ of radius $r > 0$ is denoted by $B_r (x)$. 

\subsubsection{Fourier-analytic notation and function spaces}

The Lebesgue measure of a (measurable) subset $M \subset \RR^d$ is denoted by $\mu(M)$.
The Fourier transform is given by
$\widehat{f}(\xi) = \Fourier f(\xi) = \int_{\RR^d} f(x) \, e^{-2 \pi i \langle x, \xi\rangle} \, dx$,
for all $f \in \LiRd$.
It is well-known that $\Fourier$ extends to a unitary automorphism of $\LtRd$.
The inverse Fourier transform is denoted by $\widecheck{f} := \Fourier^{-1}f$.
We write $\LtDF := \mathcal F^{-1}(\lebesgue^2(D))$ for the space of
square-integrable functions whose Fourier transform vanishes (a.e.) outside of $D \subset \R^d$.
In addition to the Fourier transform, the \emph{modulation} and 
\emph{translation operators} 
$\bd{M}_\omega f = f\cdot e^{2\pi i\langle\omega, \bullet\rangle}$ and 
$\bd T_y f = f(\bullet - y)$,
will be used frequently.

We write $\TestFunctionSpace(D) = \mathcal C_c^\infty (D)$ for the space of test functions on an open set $D\subset\RR^d$, i.e., the space of infinitely
differentiable functions with compact support in $D$. We denote by $\TestFunctionSpace'(D)$ the space of distributions on $D$,
i.e., the topological dual space of $\TestFunctionSpace(D)$. For details regarding the topology on $\TestFunctionSpace(D)$, we refer to
\cite[Chapter 6]{RudinFA}. We also consider the space $\mathcal{S}(\RR^d)$ of Schwartz functions and its topological dual $\mathcal{S}'(\RR^d)$, the space of tempered distributions. As is well-known (see e.g. \cite[Corollary 8.28]{FollandRA}),
the Fourier transform extends to a linear homeomorphism $\Fourier : \mathcal{S} (\RR^d) \to \mathcal{S} (\RR^d)$,
and by duality also to a homeomorphism $\Fourier : \mathcal{S}' (\RR^d) \to \mathcal{S} ' (\RR^d)$, where by definition
$\langle \Fourier \varphi, f \rangle_{\mathcal{S}', \mathcal{S}}= \langle \varphi, \Fourier f \rangle_{\mathcal{S}', \mathcal{S}}$, cf.~\cite[Page 295]{FollandRA}.

We consider a fixed open set $D\subset \RR^d$ and consistently write $\PhSpace = \RR^d\times D$. The weighted,
\emph{mixed} Lebesgue spaces $\bd L^{p,q}_\kappa(\PhSpace)$,
for $1\leq p,q\leq \infty$, consist of all measurable functions $F: \PhSpace \to \CC$
for which
\begin{equation}\label{eq:mixedLebesgue}
  \|F \|_{\bd L_\kappa^{p,q}}
  := \left\|
          \vphantom{\sum}
          \PhVar_2 \mapsto \left\|
                              (\kappa\cdot F)(\bullet, \PhVar_2)
                            \right\|_{\bd L^p (\RR^d)}
     \right\|_{\bd L^q (D)} <  \infty.
\end{equation}
Here, $\kappa : \PhSpace \to \RR^+$ is a (measurable) weight function. Likewise, for a discrete set $I$ and a weight $u=(u_i)_{i\in I}$, $\ell^q(I)$ denotes the space of all sequences $(c_i)_{i\in I}$, such that the norm $\|(c_i)_{i\in I}\|_{\ell^q_u}:=\|(u_i\cdot c_i)_{i\in I}\|_{\ell^q}$ is finite.

\subsubsection{Matrix notation}

For matrix-valued functions $A : U \to \RR^{d \times d}$, the notation
$A(x)\langle y\rangle := A(x)\cdot y$ denotes the multiplication of the matrix $A(x)$, $x \in U$,
with the vector $y \in \R^d$ in the usual sense.
Likewise, for a set $M\subset \RR^d$, we write
\[
  A(x)\langle M\rangle := \big\{ A(x)\langle y\rangle~:~ y\in M \big\}.
\]
Here, as in the remainder of the paper, the notation $A^T$
denotes the transpose of a matrix $A$. Further, for matrices $A_i \in \RR^{d_i\times d_i}$ with $d_i\in\NN$, for 
$i = 1,\ldots,n$, we denote by 
\[
 \mathrm{diag}(A_1,\ldots,A_n) \in \RR^{d \times d},\quad \text{with}\quad d = d_1 + \cdots + d_n,
\]
the block-diagonal matrix with blocks $A_1,\ldots,A_n$, in the indicated order. 

\subsubsection{Convention for variables}

Throughout this article, $x,y,z \in \R^d$ will be used to denote variables in time/position space,
$\xi,\omega,\eta \in D$ in frequency space, $\PhVar,\PhVarA,\PhVarC \in \PhSpace = \R^d \times D\subset \RR^{2d}$ in phase space, 
and finally $\sigma, \tau,\upsilon \in \R^d$ denote variables in warped frequency space.
Unless otherwise stated, this also holds for subscript-indexed variants;
precisely, subscript indices (i.e., $x_i$) are used to denote the $i$-th element
of a vector $x \in \RR^d$.
In some cases, we also use subscripts to enumerate multiple vectors,
e.g., $x_1,\dots,x_n \in \R^d$.
In this case, we denote the components of $x_i$ by $(x_i)_j$.

  \section{Preliminaries}
  \label{sec:prelims}
  
  In this section we collect definitions and results from the literature that will be subsequently used. We begin by recalling the cornerstones of the theory of warped time-frequency systems in \cite{VoHoPreprint}.
  
  \subsection{Warped time-frequency systems}\label{subsec:warpedTF}
  
  \begin{definition}\label{def:warpfundamental}
  Let $D \subset \RR^d$ be open.
  A $C^1$-diffeomorphism $\Phi: D \rightarrow \RR^d$ is called
  a \emph{warping function}, if $\det(\mathrm{D}\Phi^{-1}(\tau))>0$
  for all $\tau \in \RR^d$ and if and the associated weight function
  \begin{equation}
    w : \quad
    \RR^d \to \RR^+, \quad
    w(\tau) = \det(\mathrm{D}\Phi^{-1}(\tau)),
    \label{eq:wDefinition}
  \end{equation}
  is $w_0$-moderate for some submultiplicative weight $w_0 : \RR^d \to \R^+$, i.e., $w(\tau+\upsilon) 
  \leq w_0(\tau) \cdot w(\upsilon)$ and $w_0(\tau+\upsilon) 
  \leq w_0(\tau) \cdot w_0(\upsilon)$, for all $\tau,\upsilon\in\RR^d$.
  
  If $0\neq \theta \in \TestFunctionSpace(\RR^d)$, then the
   \emph{(continuous) warped time-frequency system} generated by $\theta$
  and $\Phi$ is the collection of functions
  $\mathcal G(\theta,\Phi):=(g_{y,\omega})_{(y,\omega)\in\Lambda}$, where
  \begin{equation}
    g_{y,\omega} := \bd T_{y} \widecheck{g_\omega},
    \quad \text{ with } \quad
    g_\omega := w(\Phi(\omega))^{-1/2} \cdot (\bd T_{\Phi(\omega)} \theta)\circ \Phi
    \text{ for all } y\in \RR^d,\ \omega\in D.
    \label{eq:WarpedSystemDefinition}
  \end{equation}
  Here, the function $g_\omega : D \to \CC$ is extended by zero to a
  function on all of $\RR^d$, so that $\widecheck{g_\omega}$ is well-defined. The associated 
  \emph{voice transform} is defined by 
   \begin{equation}\label{eq:warpedtransform}
    V_{\theta, \Phi} f: \quad
    \RR^d  \times D \to \CC, \quad
    (y,\omega) \mapsto \inner{f}{g_{y,\omega}}, \text{ for all } f\in\LtDF.
 \end{equation}
\end{definition}

It was established in \cite[Theorem 3.5]{VoHoPreprint} that $\mathcal G(\theta,\Phi)$ is a 
continuous tight frame, with the frame bound given by $\|\theta\|_{\lebesgue^2}^2<\infty$ and that the map $(y,\omega) \mapsto g_{y,\omega}$ is continuous. Note that we introduce the restriction to $\theta \in \TestFunctionSpace(\RR^d)$ for technical reasons, while in \cite{VoHoPreprint}, warped time-frequency systems were defined under the weaker assumption $\theta \in \Ltw(\RR^d)\supset \TestFunctionSpace(\RR^d)$. Finally, it was noted in \cite[Section 3]{VoHoPreprint}, that
\begin{equation}\label{eq:PhiJacobianDetExpressedInW}
   w(\Phi(\xi)) = [\det(D\Phi(\xi)]^{-1},\quad \text{for all } \xi\in D.
\end{equation}

 For warping functions that satisfy a mild regularity condition, it was shown that $\Phi$ induces a family of coorbit spaces that is independent of the particular choice of $\theta\in\TestFunctionSpace(\RR^d)\subset \Ltw(\RR^d)$. We briefly recall these results in the detail required for this work.

\begin{definition}\label{assume:DiffeomorphismAssumptions}
  Let $\emptyset \neq D \subset \RR^d$ be an open set and fix an integer $k\in\NN_0$.
  A map $\Phi:D\to\mathbb{R}^{d}$ is a \emph{$k$-admissible warping function}
  with \emph{control weight} $v_0:\RR^d\rightarrow [1,\infty)$,
  if $v_0$ is continuous, submultiplicative and radially increasing
  and $\Phi$ satisfies the following assumptions:
  \begin{itemize}
   \item $\Phi$ is a $C^{k+1}$-diffeomorphism and $A := \mathrm{D}\Phi^{-1}$ has positive determinant.
   \item With
          \begin{equation}
              \phi_{\tau}\left(\upsilon\right)
              :=\left(A^{-1}(\tau) A(\upsilon+\tau)\right)^T
              = A^T (\upsilon+\tau) \cdot A^{-T}(\tau),
              \label{eq:PhiDefinition}
          \end{equation}
          we have
          \begin{equation}\label{eq:PhiHigherDerivativeEstimate}
            \left\Vert
              \partial^{\alpha}\phi_{\tau}\left(\upsilon\right)
            \right\Vert
            \leq v_0 (\upsilon) 
            \qquad \text{ for all } \tau, \upsilon \in \mathbb{R}^{d}
                   \text{ and all multiindices }
                   \alpha \in \mathbb{N}_{0}^{d}, ~ \left|\alpha\right|\leq k.  
          \end{equation}
  \end{itemize}
\end{definition}

 \begin{remark}\label{rem:SimpleWeightRemark}
 We say that a
 function $v:\RR^d\rightarrow \RR$ is \emph{radially increasing}, if
     \[
       |x|\leq |y|\quad \text{ implies }\quad v(x)\leq v(y), \quad \text{ for all } x,y\in\Lambda.
     \]
 In particular, $v$ is radial, i.e., $v(x) = v(y)$ if $|x|=|y|$. 
 The requirement that the control weight $v_0$ be radially increasing 
 leads to significant simplifications and simple computations show that this is no restriction. 
 \end{remark}
 
 Here, we are specifically interested in coorbit spaces with respect to weighted, mixed Lebesgue spaces, i.e., $\Co(\mathcal G(\theta,\Phi),\lebesgue^{p,q}_\kappa)$. Unless the weight $\kappa\colon\PhSpace \to \RR^+$ is of no consequence for the desired result, we will from now on restrict to warping functions $\Phi\colon D\rightarrow \RR^d$ and weights $\kappa$ that are compatible in the sense of the following definition. In particular $\kappa(\PhVar)$ and $v_{\Phi,\kappa}(\PhVar)$ defined below are independent of the first component of $\PhVar = (x,\xi)\in\PhSpace$, and we will commit the convenient abuse of notation of interpreting $\kappa$ and $v_{\Phi,\kappa}$ as
functions on $D$, i.e., $\kappa : D \to \RR^+$. 
 
 \begin{definition}\label{def:StandingDecompositionAssumptions}
     We say that the tuple 
     $(\Phi,\kappa)$, with $\Phi\colon D \to \RR^d$ and $\kappa\colon D \to \RR^+$ forms a \emph{compatible pair}, or that $\Phi$ and $\kappa$ are \emph{compatible}, if the following hold:
     \begin{itemize}
         \item The set $D\subset\RR^d$ is open and $\Phi\colon D \to \RR^d$ is a $\mathcal C^\infty$-diffeomorphism and a $(d+1)$-admissible warping function, with control weight $v_0\colon\RR^d\rightarrow [1,\infty)$.
         \item The \emph{transplanted weight} $\kappa_{\Phi} := \kappa \circ \Phi^{-1}$ is $\tilde{\kappa}$-moderate for some continuous, submultiplicative and radially increasing weight $\tilde{\kappa} : \RR^d \to [1,\infty)$.
     \end{itemize}
     If $\Phi$ and $\kappa$ are compatible, and $v_0$ is a control weight for $\Phi$, then we define the weight $v_{\Phi,\kappa}\colon \PhSpace \rightarrow [1,\infty)$ by
     \begin{equation}\label{eq:SpecialGoodVectorsChoice}
        v_{\Phi,\kappa}(x,\xi) = \tilde{v}(\Phi(\xi)),\quad \text{for all } (x,\xi)\in \PhSpace,\quad \text{and with } \tilde{v} := \tilde{\kappa}\cdot v_0^d.
\end{equation}
 \end{definition}

The following theorem summarizes \cite[Theorem 2.13, Lemma 4.9, Theorem 7.1]{VoHoPreprint} in the form required here.

\begin{theorem}[\cite{VoHoPreprint}]\label{cor:warped_disc_frames}
Assume that $\Phi\colon D \to \RR^d$ and $\kappa\colon D \to \RR^+$ are compatible and let $v_0\colon\RR^d\rightarrow [1,\infty)$ be a control weight for $\Phi$ and $v := v_{\Phi,\kappa}\colon \PhSpace \rightarrow [1,\infty)$ given by \eqref{eq:SpecialGoodVectorsChoice}. For arbitrary  $0\neq\theta\in\TestFunctionSpace(\RR^d)$, the following hold:
\begin{enumerate}
\item  The moderating weight $w_0$ in Definition \ref{def:warpfundamental} can be chosen as $w_0 = v_0^d$.
\item The space
\begin{equation}\label{eq:DefOfGoodVectors}
    \GoodVectors := \{ f\in \LtDF \colon \|f\|_{\GoodVectors} < \infty \}, \text{ with } \|\bullet\|_{\GoodVectors} \colon f \mapsto \|V_{\theta,\Phi} f\|_{\lebesgue^1(\Lambda)}
\end{equation}
is a well-defined Banach space satisfying $\mathcal G(\theta,\Phi)\subset \GoodVectors$. 
It is independent of $\theta$ in the sense that different choices yield equivalent norms.
\item The space  
\begin{equation}\label{eq:DefOfCoorbits}
    \Co(\mathcal G(\theta,\Phi),\lebesgue^{p,q}_\kappa) := \{ f\in \Reservoir \colon \|f\|_{\Co(\Phi,\lebesgue^{p,q}_\kappa)} < \infty \}, \text{ with } \|\bullet\|_{\Co(\Phi,\lebesgue^{p,q}_\kappa)} \colon f\mapsto \|V_{\theta,\Phi} f\|_{\lebesgue^{p,q}_\kappa(\Lambda)}
    \end{equation}
    is a well-defined Banach function space for all $1\leq p,q\leq \infty$. For any two nonzero $\theta_1,\theta_2\in\TestFunctionSpace(\RR^d)$, we have $\Co(\mathcal G(\theta_1,\Phi),\lebesgue^{p,q}_\kappa) = \Co(\mathcal G(\theta_2,\Phi),\lebesgue^{p,q}_\kappa)$, justifying the notation $\Co(\Phi,\lebesgue^{p,q}_\kappa):=\Co(\mathcal G(\theta,\Phi),\lebesgue^{p,q}_\kappa)$.
\end{enumerate}
\end{theorem}

Note that Item (1) holds for any $(d+1)$-admissible warping function, and \cite[Theorem 2.13]{VoHoPreprint} can be applied since all its assumptions are covered by the conditions of Theorem \ref{cor:warped_disc_frames}.

\begin{remark}\label{rem:SpecialGoodVectorsWeight}
  A closer inspection of the proof of \cite[Theorem 7.1]{VoHoPreprint} shows that any weight $v \gtrsim v_{\Phi,\kappa}$ can be used in place of $v_{\Phi,\kappa}$ in Theorem \ref{cor:warped_disc_frames}, so long as $\sup_{x\in\RR^d} v(x,\Phi^{-1}(\bullet))$ is moderate with respect to some submultiplicative weight. This is clearly the case for $v_{\Phi,\kappa}$ itself, since $\tilde{v}$ in \eqref{eq:SpecialGoodVectorsChoice} is submultiplicative. In practice, $v$ is usually chosen such that $\Reservoir$ is large enough to cover all coorbit spaces of interest. As long as that is the case, coorbit spaces $\Co(\Phi,\lebesgue^{p,q}_\kappa)$ are independent of the specific weight $v$ used to define the reservoir $\Reservoir$, see \cite[Lemma 2.26]{kempka2015general}.
\end{remark}
%

 \subsection{Decomposition spaces and coverings}\label{sub:DecompositionSpaces}

 In a nutshell, decomposition spaces are a generalization of Besov and modulation spaces.
 The main difference is that instead of the usual dyadic covering used for Besov spaces, or the uniform covering used for modulation spaces, one can pick any covering that satisfies mild regularity conditions. 
 Such coverings will be called \emph{decomposition coverings}, see Definition \ref{def:SemiStructuredCoverings}. In particular, every decomposition 
 covering $\CalQ = (Q_i)_{i \in I}$ of $\CalO \subset \RR^d$ can be equipped with 
 a \emph{bounded admissible partition of unity (BAPU)} $\Phi = (\varphi_i)_{i \in I}$ 
 subordinate to the covering. For in-depth discussions of decomposition spaces, 
 we recommend the papers \cite{fegr85, boni07, voigtlaender2016embeddings, DecompositionIntoSobolev}.
 
\begin{definition}\label{def:SemiStructuredCoverings}(cf.\@ \cite[Definition 2.5]{voigtlaender2016embeddings}, 
  inspired by \cite[Definition 2.2]{fe82} and \cite[Definitions 2 and 7]{boni07})

   Let $\emptyset \neq \CalO \subset \RR^d$ be an open set.
   \begin{enumerate}
     \item\label{item:admissibility} A family $\CalQ = (Q_i)_{i \in I}$ of subsets of $\CalO$
    is called an \emph{admissible covering} of $\CalO$,
    if we have $\CalO = \bigcup_{i \in I} Q_i$ and if $\sup_{i \in I} |i^\ast| < \infty$,
    where the set $i^\ast$ of \emph{neighbors} of $i \in I$ is given by
    $i^\ast := \{j \in I \,:\, Q_j \cap Q_i \neq \emptyset \}$ for $i \in I$.
     \item An admissible covering $\CalQ = (Q_i)_{i \in I}$ of $\CalO$ is called
           \emph{semi-structured} if there is, for each $i \in I$, some invertible
           linear map $T_i \in \GL(\RR^d)$, a translation $b_i \in \RR^d$ and an
           open subset $Q_i ' \subset \RR^d$ such that
           the following properties hold:
          \begin{itemize}
            \item We have $Q_i = T_i Q_i ' + b_i$ for all $i \in I$,
            \item There is some $R>0$ satisfying
                  $Q_i ' \subset \overline{B_R (0)}$ for all $i \in I$.
            \item There is some $C>0$ satisfying
                  \[
                      \| T_i^{-1} T_\ell \| \leq C
                      \qquad \forall i \in I \text{ and all } \ell \in i^\ast.
                  \]
          \end{itemize}
%

      \item An admissible covering $\CalQ = (Q_i)_{i \in I}$ of $\CalO$
            is called a \emph{decomposition covering}
            if there exists a 
            \emph{bounded admissible partition of unity for $\CalQ$ ($\CalQ$-BAPU)}.
            A family $(\varphi_i)_{i \in I} \subset \TestFunctionSpace(\CalO)$ is called a $\CalQ$-BAPU,
            if it satisfies the following properties:
            \begin{itemize}
              \item $\supp \varphi_i \subset Q_i$ for all $i \in I$,

              \item $\sum_{i \in I} \varphi_i \equiv 1$ on $\CalO$, and 

              \item $\exists\ C>0 \colon \| \Fourier^{-1} \varphi_i \|_{\lebesgue^1}
                      \leq C 
                      \qquad \forall \, i\in I$.
            \end{itemize}

     \item A semi-structured covering
           $\CalQ = (Q_i)_{i \in I} = (T_i Q_i ' + b_i)_{i \in I}$ is called
           \emph{tight}, if there is some $\eps > 0$ and, for each $i \in I$, some
           $c_i \in \RR^d$ satisfying $B_{\eps} (c_i) \subset Q_i '$.

     \item A weight $u = (u_i)_{i \in I}$ on the index set $I$ of
           $\CalQ = (Q_i)_{i \in I}$ is called \emph{$\CalQ$-moderate},
           if there is a constant $C>0$ satisfying
           \begin{equation}\label{eq:Qmoderateness}
               u_i \leq C \cdot u_\ell
               \qquad \forall i \in I \text{ and all } \ell \in i^\ast.
           \end{equation}
   \end{enumerate}
 \end{definition}
%
 
 \begin{remark}\label{rem:higher_cluster_sets}
  For later use, we also introduce the following generalization
  of the sets $i^\ast$:
  For $L \subset I$, we set $L^\ast := \bigcup_{i \in L} i^\ast$.
  Then, we inductively define the sets $L^{n\ast}$ by $L^{1\ast} := L^\ast$
  and $L^{(n+1)\ast} := (L^{n\ast})^{\ast}$, in particular, we let $i^{n\ast} := \{i\}^{n\ast}$.
  Furthermore, we also set $Q_{L} := \bigcup_{\ell \in L} Q_\ell$, for any sequence of sets $(Q_i)_{i \in I}$ in $\CalO$.
\end{remark}
 
 The definition of decomposition spaces relies on a special reservoir space
 $Z'(\CalO)$, the topological dual of the space of functions whose Fourier transforms 
 are test functions. 
 
 \begin{definition}\label{def:SpecialReservoir}
   For an open subset $\CalO \subset \RR^d$, let
   \[
     Z(\CalO)
     := \Fourier (\TestFunctionSpace (\CalO))
     = \left\{
         \widehat{f} \,:\, f \in \TestFunctionSpace(\CalO)
       \right\}
     \subset \Schwartz(\RR^d)
   \]
   and equip this space with the unique topology which makes
   $\Fourier : \TestFunctionSpace(\CalO) \to Z(\CalO)$ a homeomorphism.
   We denote the topological dual space of $Z(\CalO)$ by $Z'(\CalO)$,
   and equip it with the weak-$\ast$-topology,
   that is, with the topology of pointwise convergence on $Z(\CalO)$.

   Finally, we extend the Fourier transform by duality to $Z'(\CalO)$, that is,
   we define
   \[
     \Fourier :
     Z'(\CalO) \to \DistributionSpace(\CalO), f \mapsto f \circ \Fourier
   \]
   and write as usual $\widehat{f} := \Fourier f$ for $f \in Z'(\CalO)$.
   It is clear $\Fourier$ is an isomorphism
   $\Fourier : Z'(\CalO) \to \DistributionSpace(\CalO)$, such that
   $Z'(\CalO) = \Fourier^{-1}(\DistributionSpace(\CalO))$.
 \end{definition}
 
 \begin{definition}\label{def:DecompositionSpaces}
   Given a decomposition covering $\CalQ = (Q_i)_{i\in I}$ of the open set $\CalO \subset \RR^d$,
   a $\CalQ$-moderate weight ${u = (u_i)_{i \in I}}$, and integrability exponents
   $p,q \in [1,\infty]$, we define the decomposition space
   \[
     \DecompSp(\CalQ, \lebesgue^p, \ell_u^q)
     := \left\{
          f \in Z'(\CalO)
          \,:\,
           \| f \|_{\DecompSp(\CalQ, \lebesgue^p, \ell_u^q)} < \infty
        \right\},
   \]
   with 
    \[
       \|f\|_{\DecompSp(\CalQ, \lebesgue^p, \ell_u^q)}
       := \left\|
              \left(
                  \| \Fourier^{-1}(\varphi_i \cdot \hat{f})\|_{\lebesgue^p}
              \right)_{i \in I}
          \right\|_{\ell_u^q}
          \in [0,\infty]
       \qquad \text{ for } f \in Z'(\CalO).
     \]
 \end{definition}

Consider that $\hat{f}\in\DistributionSpace(\CalO)$, such that $\varphi_i \cdot \hat{f}$ is a distribution with compact support. Hence, 
 $\Fourier^{-1}(\varphi_i \cdot \hat{f})$ is a smooth function by the Paley Wiener theorem, 
 see \cite[Theorem 7.23]{RudinFA}, and the
 expression $\|\Fourier^{-1} (\varphi_i \cdot \hat{f}) \|_{\lebesgue^p}\in [0,\infty]$
 makes sense. One can show (see for instance
 \cite[Corollary 3.18, Lemma 4.13 and Theorem 3.21]{voigtlaender2016embeddings}) 
 that the resulting space is independent of the chosen $\CalQ$-BAPU and
 complete, i.e., a Banach space.

 We further require criteria for comparing coverings. For this purpose, Feichtinger and
 Gr\"obner \cite[Definition 3.3]{fegr85} introduced the notions of \emph{subordinate}, \emph{equivalent}, and \emph{moderate} coverings; see also Proposition 3.5 in the same
 paper for equivalent reformulations of the notions introduced here.
 
 \begin{definition}\label{def:subordinateness}
   Let $\CalQ = (Q_i)_{i \in I}$ and $\CalP = (P_j)_{j \in J}$ be two admissible coverings of a set
   $\CalO \subset \RR^d$.
   \begin{itemize}
     \item We say that $\CalQ$ is \emph{almost subordinate} to $\CalP$,
           if there is some $N \in \NN$ and, for each $i \in I$, some
           $j = j_i \in J$ with $Q_i \subset P_{{j}^{N\ast}}$, where the set
           $P_{{j}^{N\ast}}$ is as in Remark \ref{rem:higher_cluster_sets}.

     \item We say that $\CalQ$ is \emph{weakly subordinate} to $\CalP$, if we have
           \[
             \sup_{i \in I}
               |\{ j \in J \,:\, P_j \cap Q_i \neq \emptyset\}|
             < \infty.
           \]

     \item We say that $\CalQ$ and $\CalP$ are \emph{equivalent} if $\CalQ$
           is almost subordinate to $\CalP$ and if also $\CalP$ is almost
           subordinate to $\CalQ$. Finally, $\CalQ$ and $\CalP$ are called
           \emph{weakly equivalent} if $\CalQ$ is weakly subordinate to $\CalP$
           and also $\CalP$ is weakly subordinate to $\CalQ$.
   \end{itemize}
 \end{definition}
 \begin{remark}\label{rem:AlmostSubImpliesWeakSub}
   Since $\CalP$ and $\CalQ$ are admissible, weak subordinateness is clearly implied by almost subordinateness. 
   Furthermore, on the set of admissible coverings, the relation of being almost subordinate
   is transitive, and the same is true of weak subordinateness,
   cf.~\cite[Proposition 3.5]{fegr85}.
   In particular, the relation defined by (weak) equivalence of admissible
   coverings is an equivalence relation on the set of admissible coverings of
   each fixed set $\CalO$.
 \end{remark}
 
 The above notion of weak equivalence yields a condition for equality of decomposition spaces.
 This result is essentially taken from \cite[Theorem 3.7]{fegr85};
 a different proof is given in \cite[Lemma 6.11]{voigtlaender2016embeddings}.
 It turns out that it is in fact quite sharp,
 see \cite[Theorem 6.9]{voigtlaender2016embeddings}.
 \begin{theorem}\label{thm:decomposition_space_coincidence}
   Let $\CalQ = (Q_i)_{i \in I}$ and $\CalP = (P_j)_{j \in J}$ be two
   decomposition coverings of the open set $\CalO \subset \RR^d$,
   and let $u = (u_i)_{i \in I}$ and $v = (v_j)_{j \in J}$ be $\CalQ$-moderate
   and $\CalP$-moderate, respectively.

   If $\CalQ$ and $\CalP$ are weakly equivalent and if there is $C>0$
   with $C^{-1} \cdot u_i \leq v_j \leq C \cdot u_i$ for all $i \in I$ and
   $j \in J$ with $Q_i \cap P_j \neq \emptyset$, then we have
   $\DecompSp(\CalQ, \lebesgue^p, \ell_u^q)
    = \DecompSp(\CalP, \lebesgue^p, \ell_v^q)$
   with equivalent norms, for all $p,q \in [1,\infty]$.
 \end{theorem}
 
 It is often difficult to verify equivalence and almost subordinateness of two
 coverings directly, even when their \emph{weak} variants are easily shown. 
 Therefore, the following result, originally \cite[Proposition 3.6]{fegr85}, is rather useful.

 \begin{lemma}\label{lem:almost_subordinateness_connected}
   Let $\CalQ = (Q_i)_{i \in I}$ and $\CalP = (P_j)_{j \in J}$ be two admissible
   coverings of a connected subset $\CalO \subset \RR^d$ such that each $Q_i$
   is path-connected and such that each $P_j$ is open.
   Then $\CalQ$ is almost subordinate to $\CalP$ if and only if $\CalQ$ is
   weakly subordinate to $\CalP$.
 \end{lemma}
 \begin{rem*}
   In \cite[Proposition 3.6]{fegr85}, it is claimed (under the same assumptions)
   that $\CalP$ is almost subordinate to $\CalQ$ if and only if $\CalP$ is
   weakly subordinate to $\CalQ$. The proof, however, establishes the claim with
   interchanged roles of $\CalQ,\CalP$, that is, as stated above.
 \end{rem*}

  \section{The frequency-space coverings induced by a warping function \texorpdfstring{$\Phi$}{Φ}}\label{sec:coverings}
In this section, we study the $(\delta, r)$-fine frequency covering induced by $\Phi$. Recall that 
\[
  \CalQ_{\Phi}^{(\delta, r)} = \big(\, Q_{\Phi,k}^{(\delta, r)} \,\big)_{k \in \ZZ^d},\quad \text{with}\quad Q_{\Phi,k}^{(\delta, r)} := \Phi^{-1}(\delta \cdot B_r (k)),
\]
for $\delta, r > 0$. It is immediate that the parameter $\delta > 0$ determines the ``sampling density'' (in frequency) of the covering, while $r$ only influences the size of the individual elements of the covering and thus is a kind of ``redundancy'' parameter. We will now show that $\CalQ_{\Phi}^{(\delta, r)}$ indeed covers $D$, if $r$ is large enough. 
 
 \begin{lemma}\label{lem:WhenIsFrequencyCoveringCovering}
   $\CalQ_{\Phi}^{(\delta, r)}$ is a covering of $D$ if and only if $r > \sqrt{d}/2$.
 \end{lemma}

 \begin{proof}
   Since $\Phi : D \to \RR^d$ is bijective, the claim is equivalent to showing
   that $\RR^d = \bigcup_{k \in \ZZ^d} \left(\delta \cdot B_r (k)\right)$ if and only if
   $r > \sqrt{d} / 2$.
   Since $\RR^d \to \RR^d, \tau \mapsto \delta \tau$ is bijective,
   we can furthermore assume $\delta = 1$.
   The proof is completed by observing that
   $\max_{\tau\in\RR^d}\min_{k\in\ZZ^d} |\tau-k| = \sqrt{d}/2$.
   Indeed, the maximum is attained for $\tau = \frac{1}{2} \cdot (1,\dots,1)$.
 \end{proof}
 
 The next lemma shows that $\CalQ_{\Phi}^{(\delta, r)}$, $\delta>0$, $r > \sqrt{d}/2$, is a  semi-structured covering and all $(\delta, r)$-fine frequency coverings $\CalQ_{\Phi}^{(\delta, r)}$ are equivalent, provided that $\Phi$ is at least $0$-admissible.
 
 \begin{lemma}\label{lem:DecompInducedCoveringIsNice}
    Let $\Phi : D \to \RR^d$ be a $0$-admissible warping function, and let
   $\delta > 0$ and $r > \sqrt{d} / 2$ be arbitrary.
   For $k \in \ZZ^d$, we have
   \begin{equation}\label{eq:covering_semi_linearization}
     Q^{(\delta,r)}_{\Phi,k} = T_k Q_k ' + b_k,
     \quad \text{with } \quad
     T_k := A(\delta k),
     \quad b_k := \Phi^{-1}(\delta k)
     \quad \text{ and } \quad
     Q_k ' := T_k^{-1} \langle Q^{(\delta,r)}_{\Phi,k} - b_k \rangle.
   \end{equation}
   The family
   $\CalQ_{\Phi}^{(\delta, r)} = (Q^{(\delta,r)}_{\Phi,k})_{k \in \ZZ^d}$
   is a semi-structured admissible covering of $D \subset \RR^d$, in particular, it is admissible. Specifically, $Q_k' \subset \overline{B_R(0)}$, for all $k\in\ZZ^d$, with $R = \delta r\cdot v_0(\delta r\cdot e_1)$. 
   Furthermore, for arbitrary $r,\rho > \sqrt{d} / 2$ and $\eps, \delta > 0$,
   the two coverings $\CalQ_{\Phi}^{(\delta, r)}$ and
   $\CalQ_{\Phi}^{(\eps, \rho)}$ are equivalent in the sense of
   Definition \ref{def:subordinateness}.
   
    If $v':\RR^d \to \RR^+$ is a $v$-moderate weight, for a submultiplicative weight $v : \RR^d \to \RR^+$ then the weight $\tilde{v} := v'\circ \Phi$ has the following properties:
   \begin{enumerate}
    \item There are constants $C_0, C_1>0$, such that $C_0 \cdot \tilde{v}(\xi)\leq \tilde{v}(\omega) \leq C_1 \cdot \tilde{v}(\xi)$
          for all $\xi \in Q^{(\delta,r)}_{\Phi,k}, \omega\in Q^{(\delta,r)}_{\Phi,\ell}$, for all $\ell \in k^\ast$ and all $k\in \ZZ^d$.
    \item The sequence $(\tilde{v}_k)_{k\in \ZZ^d}$, $\tilde{v}_k:=\tilde{v}(y_k)$, for some arbitrary $y_k\in Q^{(\delta,r)}_{\Phi,k}$, is $\CalQ^{(\delta,r)}_{\Phi}$-moderate.
    \end{enumerate}
 \end{lemma}
 \begin{proof}
    By Lemma \ref{lem:WhenIsFrequencyCoveringCovering}, $\CalQ_{\Phi}^{(\delta, r)}$ is a covering of $D$ for all $r > \sqrt{d} / 2$.
    Note that $Q_{\Phi,k}^{(\delta, r)}$ is path-connected and open for all $\delta,r>0$, such that, by Lemma \ref{lem:almost_subordinateness_connected}, weak equivalence of of $\CalQ_{\Phi}^{(\delta, r)}$ and $\CalQ_{\Phi}^{(\eps, \rho)}$ implies equivalence. Moreover, admissibility of $\CalQ_{\Phi}^{(\delta, r)}$ is the same as weak equivalence to itself.  By symmetry, it suffices to show weak subordinateness of $\CalQ_{\Phi}^{(\delta, r)}$ to $\CalQ_{\Phi}^{(\eps, \rho)}$.
   
   Let $k \in \ZZ^d$ be arbitrary. Then, each $\ell \in \ZZ^d$ with
   $Q_{\Phi,\ell}^{(\eps, \rho)} \cap Q_{\Phi,k}^{(\delta, r)} \neq \emptyset$
   satisfies
   $\emptyset \neq [\eps \cdot B_\rho (\ell)] \cap [\delta \cdot B_r (k)]$,
   and thus
   \[
     \emptyset
     \neq B_\rho (\ell) \cap \left(\eps^{-1}\delta \cdot B_r (k)\right)
     =    B_\rho (\ell) \cap B_{\delta r / \eps} (\delta k / \eps),
   \]
   so that $|\ell - \delta k / \eps| < \rho + \frac{\delta r}{\eps} =: \tilde{r}$.
   Thus, if we fix some $\ell_0=\ell_0(k) \in \ZZ^d$ with
   $Q_{\Phi,\ell_0}^{(\eps, \rho)} \cap Q_{\Phi,k}^{(\delta, r)} \neq \emptyset$,
   then every $\ell$ as above satisfies $|\ell - \ell_0| < 2\tilde{r}$,
   and thus $\ell \in \ell_0 + \{-R, \dots, R\}^d$ for $R := \lceil 2\tilde{r} \rceil$.
   We have thus shown
   \[
     \left|
       \left\{
         \ell \in \ZZ^d
         \,:\,
         Q_{\Phi,\ell}^{(\eps, \rho)} \cap Q_{\Phi,k}^{(\delta, r)}
         \neq \emptyset
       \right\}
     \right|
     \leq |
     \left\{ -R, \dots, R \right\}^d |
     = (1 + 2R)^d
   \]
   for all $k \in \ZZ^d$. Hence, $\CalQ_{\Phi}^{(\delta, r)}$
   is weakly subordinate to $\CalQ_{\Phi}^{(\eps, \rho)}$.
   
   The representation of $Q^{(\delta,r)}_{\Phi,k}$ as per Equation \eqref{eq:covering_semi_linearization} is obvious. For $\CalQ^{(\delta,r)}_\Phi$ to be semi-structured, it remains to show that $\|A(\delta k)^{-1}A(\delta \ell)\|\leq C$, for all $k\in\ZZ^d$, $\ell \in k^\ast$ and some $C>0$, as well as $Q_k' \subset \overline{B_R(0)}$, for all $k\in\ZZ^d$ and some $R>0$.
   
   Since $\Phi$ is $0$-admissible, Equation \eqref{eq:PhiHigherDerivativeEstimate} implies $\|A(\delta k)^{-1}A(\delta \ell)\|\leq v_0(\delta(k-\ell))\leq v_0(2\delta r\cdot e_1)$, for all $k,\ell\in\ZZ^d$. Here we used $|\ell-k|<2r$ for all $\ell\in k^{\ast}$ similar to above. Finally, if we express the directional derivative of $\Phi^{-1}$ through its Jacobian, we obtain, for all $\tau,\upsilon\in \RR^d$
   \begin{equation}\label{eq:basicIntegralEstimate}
      \begin{split}
      \left|[(D\Phi^{-1})(\upsilon)]^{-1}\langle \Phi^{-1}(\tau)-\Phi^{-1}(\upsilon) \rangle \right|
       & = \left|[(D\Phi^{-1})(\upsilon)]^{-1} \int_{0}^{1} (D\Phi^{-1})(\upsilon+s(\tau-\upsilon))\left\langle \tau-\upsilon \right\rangle~ds\right|\\
       & \leq |\tau-\upsilon| \cdot \max_{s\in[0,1]} \left\|[(D\Phi^{-1})(\upsilon)]^{-1}(D\Phi^{-1})(\upsilon+s(\tau-\upsilon))\right\|\\
       & = |\tau-\upsilon|\cdot \max_{s\in[0,1]} \left\| \phi_\upsilon(\upsilon+s(\tau-\upsilon))\right\|.
       \end{split}
   \end{equation}
   If $\upsilon=\delta k$ and $\tau\in \delta\cdot B_r(k)$, then we conclude that 
   \[
        |[(D\Phi^{-1})(\upsilon)]^{-1}\langle \Phi^{-1}(\tau)-\Phi^{-1}(\upsilon) \rangle|
        < \delta r \cdot v_0(\delta r\cdot e_1).
   \]
   Hence, $Q_k' = A(\delta k)^{-1} \langle \Phi^{-1}(\delta\cdot B_r(k)) - \Phi^{-1}(\delta k) \rangle \subset B_R(0)  \subset \overline{B_R(0)}$, with $R = \delta r \cdot v_0(\delta r\cdot e_1)$, independent of $k\in\ZZ^d$, i.e., $\CalQ_{\Phi}^{(\delta, r)}$ is semi-structured.
   
   Finally, we prove the claims regarding the weight $\tilde{v}$: To prove (1), observe that 
   $\tau\in Q^{(\delta,r)}_{\Phi,k}$ and $\upsilon\in Q^{(\delta,r)}_{\Phi,\ell}$ with $\ell\in k^\ast$
   implies $|\Phi(\tau)-\Phi(\upsilon)| < 4\delta r$,
   so that $v$-moderateness of $v'$ yields
   \[
     \tilde{v}(\upsilon)
     = v'(\Phi(\upsilon))
     \leq v'(\Phi(\tau))v(\Phi(\upsilon)-\Phi(\tau))
     \leq \left[ \sup_{\tau_0\in{B_{4\delta r}(0)}} v(\tau_0) \right] \cdot \tilde{v}(\tau) = \left[ \max_{\tau_0\in\overline{B_{4\delta r}(0)}} v(\tau_0) \right] \cdot \tilde{v}(\tau).
   \]

   Interchange the roles of $\tau$ and $\upsilon$ and set $C_0^{-1} := C_1 := \max_{\tau_0\in\overline{B_{4\delta r}(0)}} v(\tau_0)$ to complete the proof of Item (1). Item (2) is a direct consequence of (1); more specifically,
   the constant $C$ in \eqref{eq:Qmoderateness} can be chosen as $C_1$.
 \end{proof}

 One might expect that the existence of $r>0$ such that $\Phi^{-1}({B_r(x_k)}) \subset Q^{(\delta,r)}_{\Phi,k}$, for all $k\in\ZZ^d$, implies tightness of $\CalQ^{(\delta,r)}_{\Phi}$. This is indeed the case, but we require some auxiliary results to prove the statement. In particular, we need a fundamental relation between the sizes of $\Phi^{-1}(\Phi(\xi) + B_\vartheta (0))$ and 
 $D\Phi^{-1}(\Phi(\xi)) \left\langle B_\vartheta (0)\right\rangle$.
 
\begin{lemma}\label{lem:PhiDerivativeComparison}
   Let $\Phi : D \to \RR^d$ be a $1$-admissible warping function with
   control weight $v_0$.
   Then we have
   \[
     \|\identity_{\RR^d} - A^{-1}(\tau_0) \, A(\tau)\|
     = \|\identity_{\RR^d} - A^T (\tau) \, A^{-T}(\tau_0)\|
     \leq d
          \cdot v_0 (|\tau - \tau_0|\cdot e_1)
          \cdot \|\tau - \tau_0\|_{\infty} \, ,
   \]
  for all $\tau,\tau_0 \in \RR^d$.
 \end{lemma}
 \begin{proof}
   The first identity simply follows from $\|B^T\| = \|B\|$ for
   $B \in \R^{d \times d}$.
   Next, by definition of $\phi_{\tau_0}$ (see Equation~\eqref{eq:PhiDefinition}),
   we have
   $\phi_{\tau_0}(\tau - \tau_0)
    = A^{T} \big( (\tau - \tau_0) + \tau_0\big) \, A^{-T}(\tau_0)
    = A^{T} (\tau) \, A^{-T}(\tau_0)$,
   and $\phi_{\tau_0} (0) = \identity_{\R^d}$.
   Therefore,
   \begin{align*}
     \|\identity_{\R^d} - A^{T} (\tau) \, A^{-T} (\tau_0)\|
     = \|\phi_{\tau_0} (\tau - \tau_0) - \phi_{\tau_0} (0)\|
     & = \left\|
           \int_0^1
             \frac{d}{ds} \Big|_{s=t}
             \Big[
               \phi_{\tau_0} \big( s (\tau - \tau_0) \big)
             \Big]
           \, dt
         \right\| \\
     & = \left\|
           \int_0^1
             \sum_{\ell=1}^d
               (\partial_\ell \, \phi_{\tau_0})
                 \big( t (\tau - \tau_0) \big)
               \cdot (\tau - \tau_0)_{\ell}
           \,\, dt
         \right\| \\
     & \leq \sum_{\ell=1}^d
            \left[
              |(\tau - \tau_0)_{\ell}|
              \cdot \int_0^1
                      \big\|
                        (\partial_\ell \, \phi_{\tau_0})
                          \big(t (\tau - \tau_0)\big)
                      \big\|
                    \, d t
            \right] \\
     & \leq d
            \cdot v_0 (|\tau - \tau_0|\cdot e_1)
            \cdot \|\tau - \tau_0\|_{\infty} \, .
   \end{align*}
   Here, we used in the last step that $v_0$ is radially increasing, so that
   $v_0 \big( t (\tau - \tau_0) \big) \leq v_0 (|\tau - \tau_0|\cdot e_1)$,
   and furthermore that $\|\partial_\ell \, \phi_{\tau_0} (\upsilon) \|
   \leq v_0 (t (\tau - \tau_0))$ for all $t \in [0,1]$,
   as a consequence of \eqref{eq:PhiHigherDerivativeEstimate}.
 \end{proof} 

 We are now ready to compare $\Phi^{-1}(B_{\vartheta} (\Phi(\xi)))$ to its 
 ``linearization''
 $\xi + D\Phi^{-1}(\Phi(\xi)) \langle B_{\vartheta} (\Phi(\xi)) - \Phi(\xi) \rangle$.

 \begin{lemma}\label{lem:CoveringLinearization}
   Let $\Phi : D \to \RR^d$ be a $1$-admissible warping function with control weight $v_0$.
   Then, for every $0 < \vartheta \leq \vartheta_0:= (2d\cdot v_0(1))^{-1}$ and every $\xi \in D$, we have
   \begin{equation}
     \xi + D\Phi^{-1}(\Phi(\xi)) \langle \overline{B_{\vartheta / 4}(0)}\rangle
     \subset \Phi^{-1}\big(\Phi(\xi) + B_\vartheta (0)\big).
     \label{eq:covering_linearization}
   \end{equation}
   In particular,
   $\xi + D\Phi^{-1}(\Phi(\xi)) \left\langle B_{\vartheta / 4} (0) \right\rangle
    \subset D$
   for all $\xi \in D$. 
 \end{lemma}
 \begin{proof}

     First of all, we define for fixed, but arbitrary $\xi \in D$ the map
     \[
       g\colon 
       \RR^d \to \RR^d ,\ 
       \upsilon \mapsto \upsilon - A^{-1}(\Phi(\xi))
                                    \left\langle
                                      \Phi^{-1}(\Phi(\xi) + \upsilon)
                                    \right\rangle.
     \]
     The Jacobian of $g$ is given by
     \begin{equation}
         Dg (\upsilon)
       = \mathrm{id} - A^{-1}(\Phi(\xi)) \cdot D\Phi^{-1}(\Phi(\xi) + \upsilon)
       = \mathrm{id} - A^{-1}(\Phi(\xi)) \cdot A(\Phi(\xi) + \upsilon) \, .
       \label{eq:covering_linearization_g_derivative}
     \end{equation}
     In particular, $Dg(0) = 0$.
     For the moment, we assume that there is some $\vartheta_0 > 0$, such that
     \begin{equation}\label{eq:Dg_estimate}
       \sup_{\upsilon\in \overline{B_{\vartheta_0} (0)}}
         \|Dg (\upsilon)\|\leq 1/2.
     \end{equation}
     With this estimate, proving Equation~\eqref{eq:covering_linearization}
     essentially amounts to repeating the usual proof of the inverse mapping
     theorem
     (see for instance \cite[Chapter XIV, Theorem 1.2]{LangRealFunctional}).

     Indeed, let us fix $0 < \vartheta \leq \vartheta_0$.
     We define
     \[
       h_\tau\colon \overline{B_\vartheta (0)} \to     \RR^d,
                \upsilon                   \mapsto \tau - g(0) + g(\upsilon),\ 
       \qquad \text{for arbitrary} \qquad
       \tau \in \overline{B_{\vartheta/2}(0)}.
     \]
     Note that $h_\tau (0) = \tau$ and
     $\| D h_\tau (\upsilon) \| = \| D g (\upsilon) \| \leq 1/2$, by
     Equation~\eqref{eq:Dg_estimate}.
     The latter yields, using the fundamental theorem of calculus, that
     $| h_\tau (\upsilon) - h_\tau (\upsilon') |
      \leq \frac{1}{2} | \upsilon - \upsilon' |$
     for all $\upsilon, \upsilon' \in \overline{B_\vartheta (0)}$, that is,
     $h_\tau$ is a contraction.
     Moreover,
     \[
       |h_\tau (\upsilon)| \leq |h_\tau (\upsilon) - h_\tau (0)| + |h_\tau (0)|
       \leq \frac{1}{2} |\upsilon - 0| + |\tau| \leq \vartheta,
       \quad \text{ for all } \upsilon \in \overline{B_\vartheta(0)},
     \]
     so that $h_\tau : \overline{B_\vartheta (0)} \to\overline{B_\vartheta (0)}$
     is a self-map.

     As an application of Banach's fixed point theorem
     (see for instance \cite[Chapter XIV, Lemma 1.1]{LangRealFunctional}),
     there is thus a (unique)
     $\upsilon_\tau \in \overline{B_\vartheta}(0)$ satisfying
     $h_\tau (\upsilon_\tau) = \upsilon_\tau$.

     Noting that $(A^{-1}(\Phi(\xi)))^{-1} = D\Phi^{-1}(\Phi(\xi))$, we obtain  
     \begin{alignat*}{3}
       &&   \upsilon_\tau
          = h_\tau (\upsilon_\tau)
          = \tau - g(0) + g(\upsilon_\tau)
       &  = \tau + \upsilon_\tau - \left(
                                A^{-1}(\Phi(\xi))
                                  \langle
                                      \Phi^{-1}(\Phi(\xi) + \upsilon_\tau)
                                  \rangle - A^{-1}(\Phi(\xi))
                                  \langle
                                      \xi
                                  \rangle
                              \right) \\
            \Longleftrightarrow
       &&   A^{-1}(\Phi(\xi)) \langle \Phi^{-1}(\Phi(\xi) + \upsilon_\tau) \rangle
       &  = \tau 
            + A^{-1}(\Phi(\xi)) \langle \xi \rangle \\
            \Longleftrightarrow
       &&   \Phi^{-1}(\Phi(\xi) + \upsilon_\tau)
       &  = \xi + D\Phi^{-1}(\Phi(\xi)) \left\langle \tau \right\rangle.
     \end{alignat*}
     Consequently, since $\tau \in \overline{B_{\vartheta/2}(0)}$ was arbitrary 
     and $\upsilon_\tau \in \overline{B_{\vartheta} (0)}$,
     we see
     \[
               \xi + D\Phi^{-1}(\Phi(\xi))
                       \langle \overline{B_{\vartheta/2}(0)} \rangle
       \subset \Phi^{-1} \big(\, \Phi(\xi) + \overline{B_\vartheta(0)} \, \big),
     \]
     for all $0 < \vartheta \leq \vartheta_0$.
     Applying the above with $\vartheta / 2 < \vartheta \leq \vartheta_0$ instead of
     $\vartheta$ itself yields the desired inclusion in Equation~\eqref{eq:covering_linearization}:
     \[
       \xi + D\Phi^{-1}(\Phi(\xi))
               \left\langle \overline{B_{\vartheta / 4} (0)}\right\rangle
       \subset \Phi^{-1}
                 \big( \, \Phi(\xi) + \overline{B_{\vartheta / 2}(0)} \, \big)
       \subset \Phi^{-1} \big(\Phi(\xi) + B_{\vartheta}(0) \big).
     \]
     
     \medskip{}

     It remains to show that \eqref{eq:Dg_estimate} indeed holds for $\vartheta_0 := (2d \cdot v_0 (1))^{-1}$. But Lemma \ref{lem:PhiDerivativeComparison} implies that 
     \[
       \|\identity_{\R^d} - A^{-1}(\tau_0) \, A(\tau)\|
       \leq d \cdot v_0 (|\tau - \tau_0|\cdot e_1) \cdot |\tau - \tau_0|
       \quad \text{for all}\quad \tau, \tau_0 \in \R^d \, .
     \]
     Further note that $v_0$ is radially increasing and thus $\vartheta_0 \leq 1$, since 
     $1 \leq v_0(0) \leq v_0(1)$. Recalling the formula \eqref{eq:covering_linearization_g_derivative}
     for the derivative of $g$, we see for $\upsilon \in \overline{B_{\vartheta_0}} (0)$ that
     \[
       \|D g (\upsilon)\|
       = \|\identity
       - A^{-1}(\Phi(\xi)) \, A(\Phi(\xi) + \upsilon)\|
       \leq d \cdot v_0 (|\upsilon|\cdot e_1) \cdot |\upsilon|
       \leq d \cdot v_0 (e_1) \cdot |\upsilon|
       \leq \frac{1}{2} \, ,
     \]
     as desired. Here, we used again that $v_0$ is radially increasing, implying  
     $v_0 (|\upsilon|)\leq v_0(1)$.\qedhere
 \end{proof}
 
 As the first consequence of Lemma \ref{lem:CoveringLinearization},
 we can now show that the family $\CalQ_{\Phi}^{(\delta, r)}$
 is indeed a tight, semi-structured admissible covering of $D$,
 if $r > \sqrt{d}/2$.

 \begin{lemma}\label{lem:DecompInducedCoveringIsTight}
   Let $\Phi : D \to \RR^d$ be a $1$-admissible warping function, and let
   $\delta > 0$ and $r > \sqrt{d} / 2$ be arbitrary. The semi-structured covering 
   $\CalQ_{\Phi}^{(\delta, r)} = (Q^{(\delta,r)}_{\Phi,k})_{k \in \ZZ^d}$ is tight.
 \end{lemma}
 \begin{proof}
   Let $\vartheta_0 > 0$ as in
   Lemma~\ref{lem:CoveringLinearization}, and set
   $\vartheta := \min \{ \delta r , \vartheta_0 \} > 0$.
   Then, we apply Equation \eqref{eq:covering_linearization}
   (with $\xi = \Phi^{-1}(\delta k) = b_k$) to derive
   \begin{align*}
              Q_{\Phi,k}^{(\delta, r)}
      =       \Phi^{-1}(\delta \cdot B_r (k))
     &\supset \Phi^{-1}( \Phi(\xi) + B_{\vartheta} (0) ) \\
     ({\scriptstyle{\text{Equation } \eqref{eq:covering_linearization}}})
     &\supset \xi + D\Phi^{-1}(\Phi(\xi))
                      \left\langle B_{\vartheta / 4}(0)\right\rangle \\
     &=       b_k + T_k \left\langle B_{\vartheta / 4}(0)\right\rangle.
   \end{align*}
   But this implies $Q_{k} ' \supset B_{\vartheta / 4}(0)$.
   Since this holds for all $k \in \ZZ^d$, we see that
   $\CalQ_{\Phi}^{(\delta,r)}$ is tight.
 \end{proof}
 
 \begin{remark}\label{rem:StructuredCovering}
   Although not required in this work, it may sometimes be more convenient to consider a structured covering, in the sense of \cite[Definition 2.5]{voigtlaender2016embeddings}. In Appendix \ref{appendix:StructuredCovering}, we show that this is indeed possible, and there is a family of structured coverings equivalent to $\CalQ^{(\delta,r)}_\Phi$, albeit for restricted ranges of the parameters $\delta,r$. 
 \end{remark}
 
 To ensure that the decomposition spaces $\DecompSp(\CalQ^{(\delta,r)}_\Phi,\lebesgue^p,\ell^q_u)$ are properly defined, we need to construct a $\CalQ^{(\delta,r)}_\Phi$-BAPU, such that $\CalQ^{(\delta,r)}_\Phi$ is a decomposition covering. Such a construction is rather straightforward, but we now construct a
 special BAPU that will be useful for proving equivalence of the
 coorbit spaces generated by $\mathcal G(\theta,\Phi)$ with the decomposition
 spaces $\DecompSp(\CalQ^{(\delta,r)}_\Phi,\lebesgue^p,\ell^q_u)$.

 \begin{definition}\label{def:BAPUgenerators}
   Fix some $\delta > 0$, $r > \sqrt{d} / 2$ and some
   $0 < \vartheta < r - \sqrt{d} / 2$.
   Let $\zeta\in \TestFunctionSpace(\RR^d)$ be such that
   \begin{equation}
     \zeta \geq 0,
     \quad \int_{\RR^d} \zeta (\tau) d\tau = 1,
     \quad \text{ and } \quad
     K_0 := \supp \zeta \subset \delta \cdot B_{\vartheta} (0).
     \label{eq:DecompAnalyzingVectorAssumptions}
   \end{equation}
   For all $\xi\in D$, define the functions
   \begin{equation}
       \zeta^{(\xi)}
       := [w(\Phi(\xi))]^{-1/2} \cdot (\translation_{\Phi(\xi)} \zeta) \circ \Phi
       \quad \text{ and } \quad
       \varphi^{(\xi)} := [w(\Phi(\xi))]^{-1/2} \cdot \zeta^{(\xi)}.
       \label{eq:BAPUBuildingBlocksDefinition}
   \end{equation}
   Furthermore, with
   $M_k := \Phi^{-1} \left(
                        \delta
                        \cdot \left(k + \smash{ \left[ -\frac{1}{2}, \frac{1}{2} \right)^d}
                                \vphantom{ \left[ \frac{1}{2} \right)}\,
                              \right)
                     \right)$,
   $k\in\ZZ^d$, define
   \begin{equation}
     \varphi_k : D    \to     [0,\infty),
                 \eta \mapsto \int_{M_k} \varphi^{(\xi)}(\eta) \, d \xi.
     \label{eq:DecompSpecialBAPUDefinition}
   \end{equation}
 \end{definition}

 The collection $(\varphi_k)_{k \in \ZZ^d}$ will be shown to be the desired BAPU in Proposition \ref{prop:DecompSpecialBAPU}.
 Note that $\varphi^{(\xi)} \in \TestFunctionSpace(D) \subset \TestFunctionSpace(\RR^d)$
 for each $\xi \in D$, provided that $\Phi\in \mathcal{C}^\infty(D)$. 
 The next lemma is a crucial ingredient for proving the estimate
 $\| \Fourier^{-1} \varphi_k \|_{\lebesgue^1} \leq C$
 for our $\CalQ_{\Phi}^{(\delta, r)}$-BAPU.

 \begin{lemma}\label{lem:BAPUBuildingBlocksEstimate}
   Let $\Phi\in \mathcal C^\infty(D)$ be a $(d+1)$-admissible warping function and $\zeta\in \TestFunctionSpace(\RR^d)$ as in Definition \ref{def:BAPUgenerators}. There is a constant $C > 0$, depending on $\zeta$ and the chosen control weight $v_0$ for $\Phi$, satisfying
   \[
     \| \Fourier^{-1} \varphi^{(\xi)} \|_{\lebesgue^1}
     \leq C \cdot[w(\Phi(\xi))]^{-1}\quad \text{for all}\quad \xi \in D.
   \]
 \end{lemma}

 \begin{proof}
   Because of $\zeta \in \TestFunctionSpace(\RR^d)$,
   there is some $\widetilde{\zeta} \in \TestFunctionSpace(\RR^d)$ with
   $\widetilde{\zeta} \geq 0$ and $\widetilde{\zeta} \equiv 1$ on $\supp \zeta$.
   Now, we set $\tau_0 := \Phi(\xi) \in \RR^d$ and calculate
   \begin{align*}
   \left\Vert \Fourier^{-1}\varphi^{\left(\xi\right)}\right\Vert _{\lebesgue^{1}}
   & =\int_{\RR^{d}}
          \left|
              \int_{D}\:
                  \frac{1}{w\left(\Phi\left(\xi\right)\right)}
                  \cdot\zeta\left(
                              \Phi\left(\eta\right)-\Phi\left(\xi\right)
                            \right)
                  \cdot e^{2\pi i\left\langle y,\eta\right\rangle }
              \,d\eta
          \right|\,dy\\
   \left({\scriptstyle \upsilon=\Phi\left(\eta\right)-\tau_{0}}\right)
   & =\int_{\RR^{d}}
          \left|
              \int_{\RR^{d}}
                  \frac{w\left(\upsilon+\tau_{0}\right)}{w\left(\tau_{0}\right)}
                  \cdot\zeta\left(\upsilon\right)
                  \cdot e^{2\pi i\left\langle y,\Phi^{-1}\left(\upsilon+\tau_{0}\right)\right\rangle }
              \,d\upsilon
          \right|
      \,dy\\
   \left({\scriptstyle x=-A^{T}\left(\tau_{0}\right)\langle y \rangle
          \text{ and }
          \left|\det A^{T}\left(\tau_{0}\right)\right|=w\left(\tau_{0}\right)}\right)
   & =\frac{1}{w\left(\tau_{0}\right)}
      \cdot\int_{\RR^{d}}
              \left|
                  \int_{\RR^{d}}
                      \frac{w\left(\upsilon+\tau_{0}\right)}{w\left(\tau_{0}\right)}
                      \cdot \zeta\left(\upsilon\right)
                      \cdot e^{-2\pi i\left\langle A^{-T}\left(\tau_{0}\right) \langle x \rangle,
                                                   \Phi^{-1}\left(\upsilon+\tau_{0}\right)
                                      \right\rangle }
                  \,d\upsilon\right|
           \,dx\\
   & =: \frac{1}{w(\tau_0)} \cdot \int_{\RR^d} F_{\tau_0}(x) \, dx .
   \end{align*}

   Since $\zeta, \widetilde{\zeta} \in \TestFunctionSpace(\RR^d)$, we can invoke \cite[Theorem 4.8]{VoHoPreprint} (with $\tau = 0$) to obtain the existence of a finite constant $C = C(\zeta,v_0) > 0$, such that $F_{\tau_0}(x) \leq C \cdot (1 + |x|)^{-(d+1)}$ for all $x \in \RR^d$, uniformly over $\tau_0\in\RR^d$. We get
   \[
       \| \Fourier^{-1} \varphi^{(\xi)} \|_{\lebesgue^1}
       \leq C \cdot \int_{\RR^d} (1+|x|)^{-(d+1)} \, dx \cdot \frac{1}{w(\tau_0)}
       = C' \cdot \frac{1}{w(\Phi(\xi))},
   \]
   as claimed.
 \end{proof}

 \begin{proposition}
     If $\Phi\in \mathcal C^\infty$ is a $(d+1)$-admissible warping function, then the collection of functions $(\varphi_k)_{k \in \ZZ^d}$ as given in Definition \ref{def:BAPUgenerators} is a BAPU for $\CalQ_{\Phi}^{(\delta, r)}$.
     \label{prop:DecompSpecialBAPU}
 \end{proposition}

 \begin{proof}
    We leave to the reader the straightforward proof that $\varphi_k \in C^\infty$ and begin by noting 
     for each measurable $M \subset D$ and arbitrary $\eta \in D$, because of
     $[w(\Phi(\xi))]^{-1} = \det D\Phi(\xi)$ for $\xi \in D$,  (cf.~Equation~\eqref{eq:PhiJacobianDetExpressedInW}) that
     \begin{equation}
         \begin{alignedat}{3}
            && \varphi_{M} (\eta) &:= \int_{M} \varphi^{(\xi)}(\eta) \, d \xi 
                                  &=  \int_{M} [w(\Phi(\xi))]^{-1} \cdot \zeta (\Phi(\eta) - \Phi(\xi)) \, d \xi \\
     &&({\scriptstyle{\tau = \Phi(\eta) - \Phi(\xi)}})
                                  &=  \int_{\Phi(\eta) - \Phi(M)} \zeta (\tau) \, d\tau.
         \end{alignedat}
         \label{eq:SpecialBAPUTransformed}
     \end{equation}
     In particular, since $\Phi(\eta) - \Phi(D) = \Phi(\eta) - \RR^d = \RR^d$ and because of
     $\int_{\RR^d} \zeta(\tau) d\tau = 1$ (see \eqref{eq:DecompAnalyzingVectorAssumptions}),
     we get $\varphi_{D} (\eta) = 1$ for all $\eta \in D$.
     But it is not hard to see $D = \biguplus_{k \in \ZZ^d} M_k$, so that the monotone convergence theorem yields
     \[
         \sum_{k \in \ZZ^d} \varphi_k (\eta)
         = \sum_{k \in \ZZ^d} \varphi_{M_k} (\eta)
         = \sum_{k \in \ZZ^d} \int_{\Phi(\eta) - \Phi(M_k)} \zeta(\tau) \, d\tau
         = \int_{\Phi(\eta) - \Phi(D)} \zeta (\tau) \, d\tau
         = \varphi_D (\eta) = 1
     \]
     for all $\eta \in D$.


     Next, if $0 \neq \varphi_k (\eta) = \varphi_{M_k} (\eta)$, then Equation \eqref{eq:SpecialBAPUTransformed}
     shows that there is some $\tau \in \Phi(\eta) - \Phi(M_k)$ satisfying $\tau \in \supp \zeta = K_0$.
     Hence, $\Phi(\eta) \in \tau + \Phi(M_k) \subset K_0 + \Phi(M_k)$, so that $\eta \in \Phi^{-1}(K_0 + \Phi(M_k))$.
     Since the set on the right-hand side is compact, and since $\vartheta < r - \sqrt{d} / 2$, so that
     $[-\frac{1}{2}, \frac{1}{2})^d + B_\vartheta (0) \subset \overline{B_{\sqrt{d}/2}(0)} + B_\vartheta (0) \subset B_r (0)$,
     we conclude in view of \eqref{eq:DecompAnalyzingVectorAssumptions} that
     \[
         \supp \varphi_k
         \subset \Phi^{-1}(K_0 + \Phi(M_k))
         \subset \Phi^{-1} \left( \delta \cdot B_\vartheta (0)
                                  + \delta \left(
                                                  k + \smash{\left[ - \textstyle{\frac{1}{2}}, \frac{1}{2}\right)^d} \,
                                                  \vphantom{ \left[ \textstyle{\frac{1}{2}} \right]}
                                           \right)
                           \right)
         \subset \Phi^{-1}\left(\delta \cdot (k + B_r (0))\right)
         = Q_k^{(\delta, r)},
     \]
     as desired.

    It remains to estimate $\| \Fourier^{-1} \varphi_k \|_{\lebesgue^1}$. But Fubini's theorem, the application of which will be justified post-hoc, shows
%
     \begin{align*}
         (\Fourier^{-1} \varphi_k)(x)
         &= \int_{\RR^d}
                \int_{M_k}
                    \varphi^{(\xi)} (\eta) \cdot e^{2\pi i \left\langle \eta, x\right\rangle}
                d\xi \,
            d\eta \\
         &= \int_{M_k}
                \int_{\RR^d}
                    \varphi^{(\xi)}(\eta) \cdot e^{2\pi i \left\langle \eta, x\right\rangle}
                d\eta \,
            d \xi \\
         &= \int_{M_k} (\Fourier^{-1} \varphi^{(\xi)}) (x) \, d \xi,
     \end{align*}
     so that Minkowski's inequality for integrals (cf. \cite[Theorem 6.19]{FollandRA})
     and Lemma \ref{lem:BAPUBuildingBlocksEstimate} yield
     \begin{alignat*}{3}
         &&\| \Fourier^{-1} \varphi_k \|_{\lebesgue^1}
         \leq \int_{M_k} \| \Fourier^{-1} \varphi^{(\xi)} \|_{\lebesgue^1} d \xi
         & \leq C \cdot \int_{M_k} \frac{1}{w(\Phi(\xi))} d \xi \\
         &&({\scriptstyle{\text{Equation} \eqref{eq:PhiJacobianDetExpressedInW}}})
         & =    C \cdot \mu(\Phi(M_k))
         = C \cdot \mu \left(\delta (k + [\textstyle{-\frac{1}{2},\frac{1}{2}})^d)\right)
         = C \cdot \delta^{d},
     \end{alignat*}
     where the constant $C>0$ is provided by Lemma \ref{lem:BAPUBuildingBlocksEstimate}.
     Since this estimate holds uniformly for all $k \in \ZZ^d$, we are done, once we justify the application
     of Fubini's theorem above.
     But if $\xi \in M_k$ and $\eta \in \RR^d$ with
     $0 \neq \varphi^{(\xi)}(\eta) = [w(\Phi(\xi))]^{-1} \cdot \zeta(\Phi(\eta) - \Phi(\xi))$, then
     $\Phi(\eta) - \Phi(\xi) \in \supp \zeta = K_0$,
     so that ${\Phi(\eta) \in \Phi(\xi) + K_0 \subset \overline{M_k} + K_0}$,
     and hence $\eta \in \Phi^{-1} (\overline{M_k} + K_0) =: K_k$,
     with $K_k \subset D$ compact.
     Furthermore, using
     $|\varphi^{(\xi)}(\eta)|
     \leq \|\Fourier \Fourier^{-1} \varphi^{(\xi)}\|_{\lebesgue^\infty}
     \leq \| \Fourier^{-1} \varphi^{(\xi)}\|_{\lebesgue^1}
     \leq C \cdot [w(\Phi(\xi))]^{-1}$ (cf.~Lemma \ref{lem:BAPUBuildingBlocksEstimate}), we get
     that
     \[
         \int_{M_k} \int_{\RR^d} |\varphi^{(\xi)}(\eta)| d \eta \, d\xi
         \leq \int_{M_k} \mu(K_k) \cdot \| \Fourier^{-1} \varphi^{(\xi)}\|_{\lebesgue^1} d \xi
         \leq C \cdot \mu(K_k) \cdot \int_{M_k}  [w(\Phi(\xi))]^{-1} d\xi  < \infty , 
     \]
     thereby justifying the application of Fubini's theorem above.
 \end{proof}

  \section{Warped coorbit spaces as decomposition spaces}\label{sec:CoorbitAsDecomposition}

 Recall that we always assume that the warping function $\Phi$ and the weight $\kappa$ are compatible, as per Definition \ref{def:StandingDecompositionAssumptions}. In this setting, we will show 
 \begin{equation}
   \Co(\Phi, \lebesgue_{\kappa}^{p,q})
   = \DecompSp(\CalQ_{\Phi}^{(\delta,r)}, \lebesgue^p, \ell^q_{u}),
   \label{eq:CoorbitDecompositionIdentification}
 \end{equation}
 up to trivial identifications. The weight $u = u^{(q)} = (u^{(q)}_k)_{k \in \ZZ^d}$ will be given by
 \begin{equation}
     u_k
     := u^{(q)}_k
     := \kappa(\Phi^{-1}(\delta k)) \cdot [w(\delta k)]^{\frac{1}{q} - \frac{1}{2}}
     \quad \text{for all}\quad k \in \ZZ^d,
     \label{eq:DecompositionSpaceWeight}
 \end{equation}
 with $w = \det (D\Phi^{-1})$ as usual, so that $u^{(q)}$ solely depends on $\Phi$, $\kappa$ and the exponent $q$, as well as on the ``sampling density'' $\delta > 0$ used for the covering $\CalQ_{\Phi}^{(\delta,r)}$.
 
%
 
 \begin{rem*}
        In principle, it would suffice for $\Phi$ to be a $\mathcal C^{d+2}(D)$-diffeomorphism. However, in order to rely on standard theory, a $\CalQ_{\Phi}^{(\delta,r)}$-BAPU contained in $\TestFunctionSpace(D)$ is required. This is the sole reason for requiring $\Phi \in \mathcal C^\infty(D)$ in Definition \ref{def:StandingDecompositionAssumptions}.         
 \end{rem*}
 
 We can apply Theorem \ref{cor:warped_disc_frames} with $\tilde{\kappa}$ as in Definition \ref{def:StandingDecompositionAssumptions}, to see that $\Co(\Phi, \lebesgue_{\kappa}^{p,q})$ is a well-defined Banach space, for any $1\leq p,q\leq \infty$. In the following, the weight $v:\PhSpace\rightarrow [1,\infty)$ in the definition of the space $\GoodVectors$ will always be chosen as $v = v_{\Phi,\kappa} = (\tilde{\kappa} \cdot v_0^d) \circ \Phi$, as per Definition \ref{def:StandingDecompositionAssumptions}. 

  \subsection{\texorpdfstring{The embedding $\Reservoir \hookrightarrow Z'(D)$}
                            {The embedding of the ``coorbit reservoir'' into the ``decomposition space reservoir''}}
 \label{sub:DecompSpReservoirEmbedding}
 
 To demonstrate that the coorbit spaces
 $\Co(\Phi,\lebesgue_\kappa^{p,q})$, 
 which are subspaces of $\Reservoir$, and the decomposition spaces
 $\DecompSp(\CalQ_{\Phi}^{(\delta,r)}, \lebesgue^p, \ell_u^q)$, which are subspaces of $Z'(D)$, coincide, we will establish an isomorphism between these spaces. The first step towards this goal is to show that 
 $\Reservoir \hookrightarrow Z'(D)$. Once this embedding is established, it will be shown that it restricts to an embedding of the coorbit space into the decomposition space. 
 
 \begin{remark}
   In the following, we will always assume that the weight $v$ used to define the space $\GoodVectors$ is chosen according to Remark \ref{rem:SpecialGoodVectorsWeight}, i.e., sufficiently large such that $v \gtrsim v_{\Phi,\kappa}$, with $v_{\Phi,\kappa}$ as in Equation \eqref{eq:SpecialGoodVectorsChoice}, for all $\kappa$ of interest that are compatible with $\Phi$ and such that $\sup_{x\in\RR^d} v(x,\Phi^{-1}(\bullet))$ is moderate with respect to some submultiplicative weight.
 \end{remark}

  To establish $\Reservoir \hookrightarrow Z'(D)$, we first consider the following ``dual'' statement:
 \begin{lemma}\label{lem:DecompSpDualReservoirEmbedding}
     Assume that $\Phi:D\rightarrow \RR^d$ and $\kappa:D\rightarrow \RR^+$ are compatible as per Definition \ref{def:StandingDecompositionAssumptions}, and consider the space $\GoodVectors$ defined with respect to $\mathcal G(\theta,\Phi)$, with $0\neq \theta\in\TestFunctionSpace(\RR^d)$. Let $Z(D)$ be as in Definition \ref{def:SpecialReservoir}, with $\CalO=D$.
     The map
     \[
         \iota\ \ :\ \ Z(D) \to \GoodVectors,\ \ g \mapsto \overline{g}
     \]
     is well-defined and continuous.
     Precisely, for each compact $K \subset D$, there is
     $C_1 = C_1 (d, \theta, {{v}}, K, \Phi) > 0$ satisfying
     \[
       \| \overline{g} \|_{\GoodVectors} = \| \iota g \|_{\GoodVectors}
       \leq C_1 \cdot \max_{|\alpha| \leq d+1} \| \partial^{\alpha} [\Fourier^{-1} g] \|_{\lebesgue_\infty}
     \]
     for all $g \in Z(D)$
     with $\supp(\Fourier^{-1} g) \subset K$.

     Likewise, for each $p \in \NN_0$ and each compact set $K \subset D$,
     there is $C_2 = C_2 (d, p, \theta, K, \Phi) > 0$ satisfying
     \[
         |(V_{\theta,\Phi} \overline{g})(y,\omega)|
         \leq C_2 \cdot \left(\max_{|\alpha|\leq d+p+1} \|\partial^{\alpha} [\Fourier^{-1}g] \|_{\lebesgue_\infty}\right)
              \cdot (1 + |y|)^{-(d+p+1)} \cdot \Indicator_{L_K} (\omega)
         \quad \text{for all}\quad (y,\omega) \in \RR^d \times D,
     \]
     for some compact set $L_K \subset D$ (which only depends on $K,\theta,\Phi$) and
     all $g \in Z(D)$ with $\supp(\Fourier^{-1} g) \subset K$.
 \end{lemma}
 \begin{proof}
     Let $K \subset D$ be compact and $g \in Z(D)$ with $\supp(\Fourier^{-1} g) \subset K$.
     We first note the elementary identity $\Fourier \overline{g} = \overline{\Fourier^{-1} g}$.
     This identity shows that $\widehat{\overline{g}}$ vanishes on $\RR^d\setminus D$, so that $g \in \LtDF$.
     Furthermore, we get for arbitrary $(y,\omega) \in \RR^d \times D$, and with $g_{y,\omega}$ and $g_\omega$ as in
     \eqref{eq:WarpedSystemDefinition}, that
     \begin{equation}
         \begin{split}
              (V_{\theta,\Phi} \overline{g})(y, \omega)
          &= \left\langle \overline{g}, \, g_{y, \omega}\right\rangle_{\lebesgue^2} \\
          ({\scriptstyle{\text{Plancherel and } \mathcal F\overline{g} = \overline{\mathcal F^{-1}g}}}) &= \left\langle
                \overline{\Fourier^{-1} g} ,
                \, \modulation_{-y} \, g_\omega
             \right\rangle_{\lebesgue^2} \\
          ({\scriptstyle{\supp(\Fourier^{-1} g) \subset D}})&= \int_{D}
                \overline{\Fourier^{-1} g(\xi)}
                \cdot \frac{1}{\sqrt{w(\Phi(\omega))}}
                \cdot \overline{\theta(\Phi(\xi) - \Phi(\omega))}
                \cdot e^{2\pi i \left\langle y, \xi\right\rangle}
             \, d\xi.
         \end{split}
         \label{eq:DecompSpVoiceTransformRewritten}
     \end{equation}
     Thus, if we define
     \[
         F \colon D \times D \to \CC,\
             (\xi, \omega) \mapsto \frac{1}{\sqrt{w(\Phi(\omega))}} \cdot \overline{\theta(\Phi(\xi) - \Phi(\omega))},
     \]
     then $F \in \mathcal C^{\infty}(D\times D)$ and we have shown
     \[
         (V_{\theta,\Phi} \overline{g})(y, \omega)
         = \left(\Fourier^{-1} \left[ \overline{\Fourier^{-1} g} \cdot F(\cdot, \omega) \right]\right) (y)
         =: (\Fourier^{-1} F_{g,\omega})(y)
         \quad \text{for all}\quad (y,\omega) \in \RR^d \times D.
     \]
     Note that $F_{g, \omega} \in \TestFunctionSpace(D) \subset \TestFunctionSpace(\RR^d)$, since $\theta\in\TestFunctionSpace(\RR^d)$.

     Now, note that if $(V_{\theta,\Phi} \overline{g})(y,\omega) \neq 0$, then $F_{g, \omega} \not \equiv 0$,
     so that there is some $\xi \in \supp(\Fourier^{-1} g) \subset K$ satisfying $0 \neq F(\xi, \omega)$ and hence
     $\theta (\Phi(\xi) - \Phi(\omega)) \neq 0$. Thus, $\Phi(\xi) - \Phi(\omega) \in \supp( \theta) =: K_0$, which finally implies
     \[
         \Phi(\omega) \in \Phi(\xi) - K_0 \subset \Phi(K) - K_0
         \qquad \text{ and hence } \qquad
         \omega \in \Phi^{-1} (\Phi(K) - K_0) =: L_{K},
     \]
     where $L_{K} \subset D$ is compact and depends only on $K,\theta,\Phi$.

     Next, by standard properties of the Fourier transform (see, e.g.,~\cite[Theorem 8.22]{FollandRA}), we also have for
     arbitrary $k \in \NN_0$ and $h \in \Schwartz(\RR^d)$ that
     \begin{equation}
         \begin{split}
             (1 + |y|)^{k} \cdot | (\Fourier^{-1} h)(y) |
             &\leq C_k \cdot \left(1 + \sum_{j=1}^d |y_j|^k\right) \cdot  | (\Fourier^{-1} h)(y)| \\
             &\leq C_k \cdot \left(\| h \|_{\lebesgue^1} + \sum_{j=1}^d \| \partial_{j}^k h \|_{\lebesgue^1} \right)
             \leq C_{k}(d+1) \cdot \max_{|\alpha| \leq k} \| \partial^{\alpha} h \|_{\lebesgue^1} \quad ,
         \end{split}
         \label{eq:SmoothnessYieldsFourierDecay}
     \end{equation}
     for arbitrary $y \in \RR^d$. 

     Now, since $F \in \mathcal C^{\infty}(D \times D)$, the constant
     \[
         C_{\theta,d,p,K}
         := \max_{|\alpha| \leq d+p+1}
                \,\, \sup_{\omega \in L_{K}, \xi \in K}
                    | \partial_{\xi}^{\alpha} F(\xi, \omega)|
     \]
     is finite. Thus, Leibniz's rule yields 
     \[
         |\partial^{\alpha} F_{g, \omega} (\xi) |
         = \left|\partial^{\alpha}_\xi \left( \overline{(\Fourier^{-1} g)(\xi)} \cdot F(\xi, \omega) \right)\right|
         \lesssim C_{\theta,d,p,K} \cdot \max_{|\beta| \leq d+p+1} \| \partial^{\beta} [ \Fourier^{-1} g] \|_{\lebesgue_\infty}
     \]
     for all $\xi \in K$, $\omega \in L_K$ and $|\alpha|\leq d+p+1$, where th implied constant only depends on $d, p\in \NN_0$.
     But by what we showed above and since $\supp(\Fourier^{-1} g) \subset K$,
     the left-hand side vanishes for all $(\xi, \omega) \in (D \times D) \setminus (K \times L_K)$,
     so that the estimate is indeed valid for all $\xi, \omega \in D$. Thus, we finally conclude
     \begin{align*}
         |(V_{\theta,\Phi} \overline{g})(y,\omega)|
         &=    |(\Fourier^{-1} F_{g,\omega})(y)| \\
         ({\scriptstyle{\text{by Eq. } \eqref{eq:SmoothnessYieldsFourierDecay}}})
         &\lesssim \max_{|\alpha| \leq d+p+1}
                        \| \partial^{\alpha} F_{g,\omega} \|_{\lebesgue^1}
               \cdot (1 + |y|)^{-(d+p+1)} \\
         ({\scriptstyle{\text{since } \supp( F_{g, \omega}) \subset \supp(\Fourier^{-1} g) \subset K}})
         &\lesssim 
               \max_{|\alpha| \leq d+p+1}
                         \|\partial^{\alpha} [\Fourier^{-1} g]\|_{\lebesgue_\infty}
               \cdot \mu(K)
               \cdot (1 + |y|)^{-(d+p+1)} \\
               &\lesssim
               \max_{|\alpha| \leq d+p+1} \| \partial^{\alpha} [\Fourier^{-1} g] \|_{\lebesgue_\infty} \cdot (1+ |y|)^{-(d+p+1)}.
     \end{align*}
     Noting that the implied constant depends precisely un $d,p\in\NN_0$, $\theta$, $\Phi$ and $K\subset D$ and that $(V_{\theta,\Phi} \overline{g})(y,\omega) = 0$ unless $\omega \in L_K$,
     we have established the second part of the lemma.

     For the first part, simply apply the preceding estimate (with $p=0$) to obtain
     \begin{alignat*}{3}
         \| \overline{g} \|_{\GoodVectors}
         &= \| V_{\theta,\Phi} \overline{g} \|_{\lebesgue_{v}^1}
         = \int_D \| {v}(\omega) \cdot (V_{\theta,\Phi} \overline{g})(\cdot, \omega) \|_{\lebesgue^1} \, d \omega \\
         &\lesssim 
               \| (1 + |\cdot|)^{-(d+1)} \|_{\lebesgue^1}
               \cdot \int_{L_K} {v}(\omega) \, d\omega <\infty 
     \end{alignat*}
     where the last step used that $L_K \subset D$ is compact and that ${v}$ is continuous,
     so that $\int_{L_K} {v}(\omega) \, d\omega < \infty$.
 \end{proof}

 Using the preceding lemma, we now get the embedding $\Reservoir \hookrightarrow Z'(D)$ by duality. Some care is required, however, since the space $\Reservoir$ consists of \emph{antilinear} functionals,
 while $Z'(D)$ consists of linear functionals. The following corollary also takes a first step towards establishing a connection
 between the voice transform of $f \in \Reservoir$ and an expression related to the decomposition space norm of certain functionals in $Z'(D)$. 
 
 \begin{corollary}\label{cor:DecompSpReservoirEmbedding}
     Assume that $\Phi:D\rightarrow \RR^d$ and $\kappa:D\rightarrow \RR^+$ are compatible and let $\iota\ \ :\ \ Z(D) \to \GoodVectors,\ \ g \mapsto \overline{g}$ be as in Lemma \ref{lem:DecompSpDualReservoirEmbedding}.
     For $f \in \Reservoir$, we define $\psi_f = f\circ \iota$. Then,
     \[
         \psi_f : Z(D) \to \CC,
                  g \mapsto f(\overline{g}) = \langle f, \overline{g} \rangle_{\Reservoir, \GoodVectors}.
     \]
     is a well-defined continuous \emph{linear} functional, i.e., an element of $Z'(D)$. Furthermore, the \emph{linear} map
     \[
         \Psi : \Reservoir \to Z'(D), f \mapsto \psi_f
     \]
     is continuous (with respect to the weak-$\ast$-topology on $\Reservoir$), and injective.

     Finally, with $g_\omega = [w(\Phi(\omega))]^{-1/2} \cdot (\translation_{\Phi(\omega)} \theta) \circ \Phi$ as usual,
     we have
     \begin{equation}
         (V_{\theta,\Phi} f)(y, \omega)
         = \left(\Fourier^{-1} [\overline{g_\omega} \cdot \Fourier \psi_f] \right)(y)
         \quad \text{for all}\quad (y, \omega) \in \RR^d \times D \text{ and } f \in \Reservoir.
         \label{eq:DecompSpVoiceTransformFourierLocalization}
     \end{equation}
 \end{corollary}
 \begin{remark}\label{rem:RemarkOnThatCorollary}
     The preceding corollary shows that $\Psi : \Reservoir \to Z'(D)$ is continuous
     and injective. In order to fully justify the interpretation of $\Psi$ as an \emph{embedding},
     let us verify that $\psi_f = f$ for $f \in \LtDF$.

     On the left-hand side, $\psi_f \in Z'(D)$ is a functional on
     $Z(D) = \Fourier (\TestFunctionSpace (D)) \subset \Schwartz (\RR^d)$, and on the right-hand side,
     we interpret the function $f \in \LtDF \subset \bd L^2 (\RR^d)$ as a tempered distribution, by virtue of
     $\langle f, \phi \rangle_{\Schwartz' , \Schwartz} = \int_{\RR^d} f(x) \phi(x) dx$, as usual.
     In particular, $f$ is thus a functional on $Z(D) \hookrightarrow \Schwartz(\RR^d)$.
     Taking into account all these ``trivial'' identifications, we get for arbitrary
     $g \in Z(D) \subset \Schwartz (\RR^d)$ that
     \begin{align*}
         \left\langle \psi_f, g \right\rangle_{Z', Z}
         &= \left\langle f, \overline{g} \right\rangle_{\Reservoir, \GoodVectors}
         =  \left\langle f, \overline{g}\right\rangle_{\lebesgue^2}
         = \left\langle f, g\right\rangle_{\Schwartz', \Schwartz}
         = \left\langle f, g\right\rangle_{Z', Z}.
     \end{align*}
     Hence, $f\in \LtDF$ satisfies $\psi_f = f$ as elements of $Z'$. Finally, note that the inclusion $\LtDF \hookrightarrow Z'(D)$
     is injective, since if $0 = \left\langle f, g\right\rangle_{Z', Z}$ for all $g \in Z(D)$, then
     \[
         0 = \left\langle f, g\right\rangle_{Z', Z} 
         = \left\langle f, \overline{g} \right\rangle_{\lebesgue^2}
         = \langle \widehat{f}, \widehat{\overline{g}} \rangle_{\lebesgue^2 (\RR^d)}
         = \langle  \widehat{f}, \overline{\Fourier^{-1} g} \rangle_{\lebesgue^2 (D)},
     \]
     where the last step used $\widehat{f} = 0$ almost everywhere on $\RR^d\setminus D$.
     Since the above holds for all $g\in Z(D)$ and $Z(D) = \Fourier (\TestFunctionSpace(D))$,
     we conclude that $\widehat{f} = 0$ almost everywhere on $D$, and thus $f = 0$.
     Altogether, these considerations justify the interpretation of $\Psi$
     as a \emph{canonical embedding}
     $\Psi : \Reservoir \hookrightarrow Z'(D)$.
 \end{remark}
 \begin{proof}[Proof of Corollary \ref{cor:DecompSpReservoirEmbedding}]
     In view of Lemma \ref{lem:DecompSpDualReservoirEmbedding},
     the first assertion follows, i.e., $\psi_f\in Z'(D)$.
     Thus, we only have to prove continuity and injectivity of $\Psi$,
     as well as validity of \eqref{eq:DecompSpVoiceTransformFourierLocalization}.

     For continuity of $\Psi$, assume that the net $(f_\alpha)_\alpha$ in $\Reservoir$ satisfies
     $f_\alpha \to f$ in the weak-$\ast$-sense. Then $f_\alpha (h) \to f(h)$ for all $h \in \GoodVectors$.
     Since $\overline{g} \in \GoodVectors$ for all $g \in Z(D) = \Fourier (\TestFunctionSpace (D))$, this implies
     for arbitrary $g \in Z(D)$ that $\psi_{f_\alpha} (g) = f_\alpha (\overline{g}) \to f(\overline{g}) = \psi_f (g)$,
     and hence $\psi_{f_\alpha} \to \psi_f$ in $Z'(D)$, since this space is equipped with the weak-$\ast$-topology.

     Next, we verify Equation \eqref{eq:DecompSpVoiceTransformFourierLocalization}.
     To this end, first note that $\theta \in \TestFunctionSpace(\RR^d)$ implies by definition of $g_\omega$ (Equation  \eqref{eq:WarpedSystemDefinition})
     that $g_\omega \in \TestFunctionSpace(D)$. Now, note
        $\overline{g_{y, \omega}}
         = \Fourier(\modulation_{y}\overline{\smallVSpace g_{\omega}}) \in Z(D)$,
         such that we obtain
     \begin{align*}
         (V_{\theta,\Phi} f)(y, \omega)
         &= \left\langle f, g_{y, \omega}\right\rangle_{\Reservoir, \GoodVectors} 
         = \psi_f (\overline{g_{y,\omega}}) 
         = \langle \psi_f, \Fourier (\modulation_y \overline{g_\omega}) \rangle_{Z' (D), Z(D)}
         = \left\langle
                \Fourier \psi_f, \modulation_y \overline{g_\omega}
            \right\rangle_{\mathcal{D}'(D), \TestFunctionSpace(D)}\,\,.
     \end{align*}
     Now, recall from the definition of the Fourier transform on $Z'(D)$ (Definition \ref{def:SpecialReservoir})
     that $\Fourier \psi_f \in \DistributionSpace(D)$ and note $\overline{g_{\omega}} \in \TestFunctionSpace(D)$,
     so that $\overline{g_\omega} \cdot \Fourier \psi_f$ is a compactly supported (and hence tempered) distribution on all of
     $\RR^d$, which is even a well-defined functional on $\mathcal C^{\infty}(\RR^d)$.
     Then, the Paley-Wiener theorem (cf. \cite[Theorem 7.23]{RudinFA}) shows that the inverse Fourier transform
     $\Fourier^{-1} (\overline{g_\omega} \cdot \Fourier \psi_f)$ is given by (integration against)
     a smooth function, which is pointwise defined by
     \[
         (\Fourier^{-1} [\overline{g_{\omega}} \cdot \Fourier \psi_f])(y)
         = \left\langle
               \overline{\smallVSpace g_{\omega}} \cdot \Fourier \psi_f,
                e^{2\pi i \left\langle y, \cdot\right\rangle}
           \right\rangle_{(\mathcal C^{\infty})', \mathcal C^{\infty}}
         = \langle \Fourier \psi_f , \, \modulation_y \overline{g_\omega} \rangle_{\mathcal{D}'(D), \TestFunctionSpace(D)},
     \]
     proving Equation \eqref{eq:DecompSpVoiceTransformFourierLocalization}.

     Finally, to show injectivity of $\Psi$, assume $\psi_f \equiv 0$ for some $f \in \Reservoir$.
     By Equation \eqref{eq:DecompSpVoiceTransformFourierLocalization}, this implies $V_{\theta,\Phi} f \equiv 0$ and hence $f = 0$,
     since $V_{\theta,\Phi} : \Reservoir \to \lebesgue_{1/{v}}^\infty (\Lambda)$ is injective, see ,e.g.,  \cite[Corollary 2.19]{kempka2015general}.
 \end{proof}

 \subsection{\texorpdfstring{The embedding
                             $\Psi :
                              \Co(\Phi,\lebesgue_{\kappa}^{p,q}) \to
                              \DecompSp(\CalQ_{\Phi}^{(\delta,r)},
                                        \lebesgue^p,
                                        \ell_u^q)$}
                            {The embedding of the coorbit space into the
                             decomposition space}}
 \label{sub:CoorbitEmbedsIntoDecomposition}

 In Theorem \ref{cor:warped_disc_frames}, we have noted that $\Co(\Phi,\lebesgue_{\kappa}^{p,q})$ is independent of the choice of $\theta\in\TestFunctionSpace(\RR^d)$ in $\mathcal G(\theta,\Phi)$. Nonetheless, to prove that the embedding $\Psi$ given in Corollary \ref{cor:DecompSpReservoirEmbedding} restricts to an embedding of $\Co(\Phi,\lebesgue_{\kappa}^{p,q})$ into $\DecompSp(\CalQ_{\Phi}^{(\delta,r)},\lebesgue^p,\ell_u^q)$, it will be helpful to choose $\theta$ appropriately, see Lemma \ref{lem:CoorbitDecompositionConnection} below. Before proceeding, we  verify that the decomposition space $\DecompSp (\CalQ_{\Phi}^{(\delta,r)}, \lebesgue^p, \ell_u^q)$ is indeed
 well-defined. In fact, Proposition \ref{prop:DecompSpecialBAPU} has established the existence of a $\CalQ_{\Phi}^{(\delta,r)}$-BAPU, such that $\CalQ_{\Phi}^{(\delta,r)}$ is a semi-structured, admissible decomposition covering of $D$ by Lemma
 \ref{lem:DecompInducedCoveringIsNice}. It only remains to show is that the weight $u = (u_k)_{k \in \ZZ^d}$ is 
 $\CalQ_{\Phi}^{(\delta,r)}$-moderate, cf.~Definition \ref{def:SemiStructuredCoverings}.
 
 \begin{lemma}\label{lem:DecompositionWeightIsModerate}
    Let $1\leq q \leq \infty$ and let $\Phi: D\rightarrow \RR^d$ and $\kappa: D\rightarrow \RR^+$ be compatible.
    Then, with $\kappa_\Phi = \kappa\circ \Phi^{-1}$, the weight $u = u^{(q)}$, given by
    \[
     u^{(q)} \colon \ZZ^d \rightarrow \RR^+,\ k\mapsto u_k^{(q)} := \kappa_\Phi(\delta k)\cdot [w(\delta k)]^{\frac{1}{q} - \frac{1}{2}},
    \]
    as in Equation  \eqref{eq:DecompositionSpaceWeight}, is $\CalQ_{\Phi}^{(\delta,r)}$-moderate.
 \end{lemma}
 \begin{proof}
     This is a direct application of Lemma \ref{lem:DecompInducedCoveringIsNice}, Item (2), with $y_k = \Phi^{-1}(\delta k)$, $v' = \kappa_\Phi \cdot w^{\frac{1}{q} - \frac{1}{2}}$ and $v = \tilde{\kappa}\cdot v_0^{\frac{d}{q} + \frac{d}{2}}$. To see this, note that $v_0\geq 1$ and that any nonnegative power of a submultiplicative weight is itself submultiplicative and $w^{\frac{1}{q} - \frac{1}{2}}$ is $v_0^{|\frac{1}{q} - \frac{1}{2}|}$-moderate.
 \end{proof}

 While the decomposition space norm is computed using the Fourier-localized
 pieces $\Fourier^{-1} [\varphi_k \cdot \Fourier \psi_f]$,
 the coorbit space norm on $\Co (\Phi,\lebesgue_\kappa^{p,q})$ relies on
 the voice transform $V_{\theta,\Phi} f$.
 The following lemma establishes a first connection between the two quantities
 which will prove essential for showing
 $\| \psi_f \|_{\DecompSp (\mathcal{Q}_\Phi^{(\delta,r)}, \lebesgue^p, \ell_u^q)}
  \lesssim \| f \|_{\Co (\Phi,\lebesgue_\kappa^{p,q})}$.
 For genuine $\lebesgue^2$-functions $f$, the lemma is proven by a simple
 but lengthy computation; the main technical challenge is the extension to all of
 $\Reservoir$.
 
 \begin{lemma}\label{lem:CoorbitDecompositionConnection}
     Let $\Phi: D\rightarrow \RR^d$ and $\kappa: D\rightarrow \RR^+$ be compatible. Fix $\delta > 0$, $r > \sqrt{d} / 2$, and $0 < \vartheta < r - \sqrt{d}/2$. Finally, choose  $0\neq \theta \in \TestFunctionSpace(\RR^d)$ with 
     \begin{equation}\label{eq:CanonicalThetaChoice}
       \theta \geq 0\,, \quad \|\theta\|_1 = 1
       , \quad \text{ and } \quad K_0 := \supp( \theta) \subset \delta \cdot B_\vartheta (0).
     \end{equation}
     Define $\varphi^{(\omega)} := [w(\Phi(\omega))]^{-1/2} g_\omega$, with $g_\omega$ 
     as usual, and 
     $M_k := \Phi^{-1} \left( \delta \cdot
                               \left(
                                   k + \smash{ \left[ \textstyle{-\frac{1}{2}, \frac{1}{2}} \right)^d}
                                   \vphantom{ \left[ \frac{1}{2} \right]} \,
                               \right)
                           \right)$, for all $k\in\ZZ^d$.
    Then $(\varphi_k)_{k\in\ZZ^d}$, with 
    \[
         \varphi_k : D \to [0,\infty), \eta \mapsto \int_{M_k} \varphi^{(\xi)} (\eta) \, d\xi,
    \]
    is a BAPU for the covering  $\CalQ_{\Phi}^{(\delta,r)}$.
    
    \medskip{}
    
    Furthermore, for $f \in \Reservoir$, we have
     \begin{equation}\label{eq:FouLocVsVoice}
         [\Fourier^{-1} (\varphi_k \cdot \Fourier \psi_f)](y)
         =  \int_{M_k} [w(\Phi(\omega))]^{-1/2} \cdot (V_{\theta,\Phi}f)(y, \omega) \, d\omega
     \end{equation}
     for arbitrary $y \in \RR^d$ and $k \in \ZZ^d$.     
 \end{lemma}

 \begin{proof}
    That $(\varphi_k)_{k\in\ZZ^d}$ is a BAPU for $\CalQ_{\Phi}^{(\delta,r)}$ is a direct consequence of Proposition \ref{prop:DecompSpecialBAPU}: Simply set $\zeta = \theta$ in Definition \ref{def:BAPUgenerators}. We proceed to prove the second assertion. For arbitrary $f \in \Reservoir$ and $y \in \RR^d$, $k \in \ZZ^d$, we showed, in the proof of Corollary \ref{cor:DecompSpReservoirEmbedding}, that $\varphi_k \cdot \Fourier \psi_f$ is a (tempered) distribution with compact support.
    Hence,
     \begin{equation}\label{eq:FourierVsInner}
         [\Fourier^{-1}(\varphi_k \cdot \Fourier \psi_f)](y)
        = \langle f, \overline{\Fourier [\modulation_y \varphi_k]} \rangle_{\Reservoir, \GoodVectors} = (\ast) \, \, ,
     \end{equation}
     where Lemma \ref{lem:DecompSpDualReservoirEmbedding} shows that $\overline{\Fourier[\modulation_y \varphi_k]} \in \GoodVectors$, since $\varphi_k \in \TestFunctionSpace(D)$. In particular, this calculation shows that the left-hand side of
     the Equation  \ref{eq:FouLocVsVoice} depends weak-$\ast$-continuously on $f \in \Reservoir$.

     Next, if $f \in \LtDF \subset \Reservoir$, 
     we can continue the preceding calculation as follows:
     \begin{equation}
         \begin{split}
             (\ast) &= \left\langle
                    f, \,
                    \overline{\Fourier[\modulation_y \varphi_k]}
                \right\rangle_{\lebesgue^2}
                 = \left\langle
                    f, \,
                    \Fourier^{-1}[ \, \overline{\modulation_y \varphi_k} \, ]
                \right\rangle_{\lebesgue^2}
             = \left\langle
                    \widehat{f}, \,
                    \overline{\modulation_y \varphi_k}
                \right\rangle_{\lebesgue^2} \\
             &= \int_{D}
                    \widehat{f} (\xi) \cdot
                    \varphi_k (\xi)\cdot e^{2\pi i \left\langle y, \xi\right\rangle}
                \, d \xi \\
             ({\scriptstyle{\text{Def. of }\varphi_k \text{ and Fubini}}})
             &= \int_{M_k}
                    \int_{D}
                        \varphi^{(\omega)} (\xi) \cdot
                        \widehat{f}(\xi) \cdot
                        e^{2\pi i \left\langle y, \xi\right\rangle}
                    \, d \xi
                \, d \omega \\
             ({\scriptstyle{g_{\omega} \text{ real}}})
             &= \int_{M_k}
                    [w(\Phi(\omega))]^{-1/2} \cdot
                    \int_{D}
                        \overline{g_{\omega} (\xi)} \cdot
                        \widehat{f}(\xi) \cdot
                        e^{2\pi i \left\langle y, \xi\right\rangle}
                    \, d \xi
                \, d \omega \\
        ({\scriptstyle{\text{Plancherel}}})
             &= \int_{M_k}
                    [w(\Phi(\omega))]^{-1/2} \cdot
                    \left\langle
                        f, \,
                        \translation_y [\Fourier^{-1} \, g_{\omega}]
                    \right\rangle_{\lebesgue^2 (\RR^d)}
                \, d \omega \\
             &=  \int_{M_k}
                    [w(\Phi(\omega))]^{-1/2} \cdot
                    (V_{\theta,\Phi} f)(y, \omega)
                \, d \omega.
         \end{split}
         \label{eq:CoorbitDecompositionConnectionCalculation}
     \end{equation}
     This shows Equation  \eqref{eq:FouLocVsVoice}, for $f\in \LtDF\subset \Reservoir$. To justify the application of Fubini's theorem, note the following:
     Straightforward calculations using $v_0^d$-moderateness of $w$ (see Theorem \ref{cor:warped_disc_frames}) show that $[w(\Phi(\omega))]^{1/2}\|\varphi^{(\omega)}\|_{\lebesgue^2(D)}=\|g_\omega\|_{\lebesgue^2(D)} \leq \| \theta \|_{\lebesgue^2_{v_0^{d/2}}(\RR^d)}<\infty$. Hence,
     \begin{align*}
         \int_{M_k} \int_{D} |\varphi^{(\omega)}(\xi) \cdot \widehat{f}(\xi)| \, d\xi \, d \omega
         &\leq \int_{M_k} \|\varphi^{(\omega)}\|_{\lebesgue^2(D)} \cdot \|\widehat{f}\|_{\lebesgue^2(D)} \, d \omega\\
         &\leq  \|f\|_{\lebesgue^2} \cdot \| \theta \|_{\lebesgue^2_{v_0^{d/2}}}
               \cdot \int_{M_k}
                        [w(\Phi(\omega))]^{-1/2}
                     \, d \omega
         < \infty.
     \end{align*}
     Here, we used that the integrand on the right-hand side is continuous and $\overline{M_k}$ is compact.
    
     \medskip{}
     
     It remains to lift the identity
     \eqref{eq:CoorbitDecompositionConnectionCalculation} from $f \in \LtDF$
     to arbitrary $f \in \Reservoir$.
     To this end, fix $k \in \ZZ^d$ and note that the
     set $M_k$ satisfies
     $\overline{M_k} \subset \Phi^{-1} (\delta \cdot (k + [-1,1]^d)) =: N_k$,
     where the right-hand side is a compact subset of $D$. Furthermore, $(\Phi,\kappa)$ is a compatible pair and, in particular, $\Phi$ is a $(d+1)$-admissible warping function. Hence, by Theorem \ref{cor:warped_disc_frames}, the space $\GoodVectors$ is well-defined and we can apply \cite[Lemma 2.13]{kempka2015general} to obtain a constant $C>0$, such that 
     \begin{equation}\label{eq:FrameElementGoodVectorNormEstimate}
      \| g_{y,\omega} \|_{\GoodVectors} \leq C v(y,\omega),\quad \text{for all}\quad (y,\omega)\in \PhSpace. 
     \end{equation}
     In particular, for fixed $y \in \RR^d$ and $\omega\in N_k \supset M_k$, we obtain $\| g_{y,\omega} \|_{\GoodVectors} \leq C_2(y,k) =: C_2$, for some constant $C_2(y,k)>0$ depending only on $y\in\RR^d$ and $k\in\ZZ^d$.
     
     It was shown in \cite[Lemma 2]{fora05} that
     $\LtDF$ is sequentially weak-$\ast$-dense in
     $\Reservoir$. Hence, there is a sequence $(f_n)_{n \in \NN}$ in $\LtDF \subset \Reservoir$
     with
     \[
         \langle f_n, g \rangle_{\Reservoir, \GoodVectors}
         \xrightarrow[n\to\infty]{} \langle f, g \rangle_{\Reservoir, \GoodVectors}
         \quad \text{for all}\quad g \in \GoodVectors.
     \]
     In particular, this means
     $V_{\theta,\Phi} f_n (y,\omega)
     = \langle f_n, g_{y,\omega} \rangle_{ \Reservoir, \GoodVectors}
     \to \langle f, g_{y,\omega} \rangle_{ \Reservoir, \GoodVectors} = V_{\theta,\Phi} f(y,\omega)$
     for all $\omega \in M_k$.
     Furthermore, by the uniform boundedness principle, there is a finite constant $C$, depending on the full sequence $(f_n)_n$, such that $\| f_n \|_{\Reservoir} \leq C < \infty$
     for all $n \in \NN$, and hence
     \[
         |V_{\theta,\Phi} f_n (y,\omega)|
         = |\langle f_n, g_{y,\omega} \rangle_{\Reservoir, \GoodVectors}|
         \leq C\| g_{y,\omega} \|_{\GoodVectors}
         \leq C \cdot C_2
     \]
     for all $\omega \in M_k$.
     Finally, since $M_k \subset N_k$ with $N_k \subset D$ compact, we see that $\omega \mapsto [w(\Phi(\omega))]^{-1/2}$
     is bounded on $M_k$, and that $\mu(M_k) < \infty$, so that the dominated convergence theorem yields
     \[
         \int_{M_k} [w(\Phi(\omega))]^{-1/2} \cdot V_{\theta,\Phi} f_n (y,\omega) d\omega
         \xrightarrow[n\to\infty]{}  \int_{M_k} [w(\Phi(\omega))]^{-1/2} \cdot V_{\theta,\Phi} f (y,\omega) d\omega.
     \]
     As we observed above, the left-hand side of
     \eqref{eq:FourierVsInner} 
     is also
     weak-$\ast$-continuous on $\Reservoir$, so that \eqref{eq:FouLocVsVoice}
     transfers from $f \in \LtDF$ to arbitrary $f \in \Reservoir$, as desired.
 \end{proof}
 
 \begin{remark}
   Note that, strictly speaking, \cite[Lemma 2.13]{kempka2015general} only yields \eqref{eq:FrameElementGoodVectorNormEstimate} on a dense subset of $\PhSpace$. It is, however, easy to see that the estimate extends to all of $\PhSpace$ if the map $(y,\omega)\mapsto g_{y,\omega}$ is (weakly) continuous, which is clearly the case here, see also \cite[Proposition 3.4]{VoHoPreprint}. 
 \end{remark}

  We are now ready to prove the desired embedding of $\Co(\mathcal G(\theta,\Phi),\lebesgue_{\kappa}^{p,q})$ into $\DecompSp(\CalQ_{\Phi}^{(\delta,r)}, \lebesgue^p, \ell_u^q)$.
 
 \begin{theorem}\label{thm:CoorbitIntoDecomposition}
     Let $\Phi: D\rightarrow \RR^d$ and $\kappa: D\rightarrow \RR^+$ be compatible. Fix $\delta > 0$, $r > \sqrt{d} / 2$, and $0 < \vartheta < r - \sqrt{d}/2$. Finally, choose  $0\neq \theta \in \TestFunctionSpace(\RR^d)$ with  $\|\theta\|_1 = 1$ and $\supp( \theta) \subset \delta \cdot B_\vartheta (0)$. 
     
     Then, the map $\Psi : \Reservoir \to Z'(D)$ from Corollary \ref{cor:DecompSpReservoirEmbedding} restricts
     for arbitrary $p,q \in [1,\infty]$ to a bounded injective linear map
     \[
         \Psi \colon \Co(\mathcal G(\theta,\Phi),\lebesgue_{\kappa}^{p,q}) \to \DecompSp(\CalQ_{\Phi}^{(\delta,r)}, \lebesgue^p, \ell_u^q)
     \]
     with
     \[
         u = (u_k)_{k \in \ZZ^d}
         \qquad \text{ where } \qquad
         u_k := \kappa_\Phi(\delta k) \cdot [w(\delta k)]^{\frac{1}{q} - \frac{1}{2}}.
     \]
 \end{theorem}
 \begin{proof}
     Let $f \in \Co(\mathcal G(\theta,\Phi),\lebesgue_{\kappa}^{p,q}) \subset \Reservoir$ be arbitrary.
     The assumptions on $\theta$ imply that Equation \eqref{eq:CanonicalThetaChoice} is satisfied. Hence, Lemma \ref{lem:CoorbitDecompositionConnection} and Minkowski's inequality for integrals
     (see e.g.~\cite[Theorem 6.19]{FollandRA}) yield for arbitrary $k \in \ZZ^d$ the estimate
     \[
         d_k :=   \| \Fourier^{-1} [\varphi_k \cdot \Fourier \psi_f] \|_{\lebesgue^p}
         \leq  \int_{M_k} [w(\Phi(\omega))]^{-1/2} \cdot \| (V_{\theta,\Phi} f)(\cdot, \omega) \|_{\lebesgue^p} \, d \omega.
     \]
     Now, let us first consider the case $q < \infty$. In this case, and since we only consider the case $q \geq 1$,
     Jensen's inequality~\cite[Theorem 3.3]{rudin1987real} yields
     \begin{equation}
      \begin{split}
         (u_k \cdot d_k)^q
         &\leq \left[
                    u_k \cdot
                   \mu(M_k) \cdot
                   \int_{M_k}
                        \| (V_{\theta,\Phi} f)(\cdot, \omega)\|_{\lebesgue^p}
                        \cdot [w(\Phi(\omega))]^{-1/2}
                   \frac{d\omega}{\mu(M_k)}
               \right]^q \\
         &\leq u_k^q \cdot
               [\mu(M_k)]^q \cdot
               \int_{M_k}
                    \| (V_{\theta,\Phi} f)(\cdot, \omega)\|_{\lebesgue^p}^q
                    \cdot [w(\Phi(\omega))]^{-q/2}
               \frac{d\omega}{\mu(M_k)},
       \end{split}\label{eq:WeightedDecompComponentEstimate}
     \end{equation}
    where $\mu$ denotes the usual $d$-dimensional Lebesgue measure.

     Recall that $v_0$ is radially increasing and that $w$ is $v_0^d$-moderate, by Theorem \ref{cor:warped_disc_frames}. Further, note for
     $\omega \in M_k = \Phi^{-1} \left(
                                     \delta \cdot
                                     \left(
                                         k + \smash{ \left[ \textstyle{-\frac{1}{2}, \frac{1}{2}} \right)^d}
                                         \vphantom{ \left[ \frac{1}{2} \right]} \,
                                     \right)
                                 \right)$
     that $\Phi(\omega) \in \delta k + [-\delta,\delta]^d \subset \delta k + \overline{B_{d\delta}} (0)$, so that
     \[
         w(\delta k)
         = w\left( \Phi(\omega) + [\delta k - \Phi(\omega)] \right)
         \leq w(\Phi(\omega)) \cdot v_0^d (\delta k - \Phi(\omega))
         \leq w(\Phi(\omega)) \cdot v_0^d (d\delta\cdot e_1)
         =: C_{v_0, \delta} \cdot w(\Phi(\omega)).
     \]
     Likewise, we have
     \begin{align*}
         \mu(M_k)
         = \mu \left(
                  \Phi^{-1} \left(
                                \delta \cdot \left( k + [{\textstyle{-\frac{1}{2}, \frac{1}{2})^d}} \right)
                            \right)
               \right)
         = \int_{\delta \cdot \left( k + [{\textstyle{-\frac{1}{2}, \frac{1}{2})^d}} \right)}
               |\det (D\Phi^{-1})(\upsilon)|
           \,d \upsilon
         = \int_{\delta \cdot \left( k + [{\textstyle{-\frac{1}{2}, \frac{1}{2})^d}} \right)}
                w(\upsilon)
           \,d \upsilon,
     \end{align*}
     where
     \[
         w(\upsilon)
         = w(\delta k + (\upsilon - \delta k))
         \leq w(\delta k) \cdot v_0^d (\upsilon - \delta k)
         \leq w(\delta k) \cdot v_0^d (d\delta\cdot e_1)
         \leq C_{v_0, \delta} \cdot w(\delta k)
     \]
     for arbitrary $\upsilon \in \delta \cdot \left[ k + [{\textstyle{-\frac{1}{2}, \frac{1}{2})^d}} \right]
     \subset \delta k + \overline{B_{d \delta} (0)}$.
     Thus, $\mu(M_k) \leq \delta^d \cdot C_{v_0,\delta} \cdot w(\delta k)$.

     Finally, since $\tilde{\kappa}$ is radially increasing, we have for $\omega \in M_k$ 
     \[
         \frac{\kappa (\Phi^{-1}(\delta k))}{\kappa (\omega)}
         = \frac{\kappa_{\Phi} (\delta k)}{\kappa_{\Phi} (\Phi(\omega))}
         \leq \tilde{\kappa}(\delta k - \Phi(\omega))
         \leq \tilde{\kappa}(d\delta\cdot e_1) =: C_{\tilde{\kappa}, \delta}.
     \]

     Altogether, these considerations show
     \begin{align*}
         (u_k \cdot d_k)^q
         &\leq u_k^q
               \cdot [\mu(M_k)]^q
               \cdot \int_{M_k}
                        \| (V_{\theta,\Phi} f)(\cdot, \omega)\|_{\lebesgue^p}^q
                        \cdot [w(\Phi(\omega))]^{-q/2}
                     \frac{d\omega}{\mu(M_k)} \\
         &\leq \delta^{d(q-1)}\cdot C_{v_0,\delta}^{3q/2-1}
               \cdot [w(\delta k)]^{\frac{q}{2} - 1}
               \cdot u_k^q
               \cdot \int_{M_k}
                        \| (V_{\theta,\Phi} f)(\cdot, \omega) \|_{\lebesgue^p}^q
                     \, d \omega \\
         ({\scriptstyle{\text{by def. of } u_k}})
         &\leq \delta^{d(q-1)}\cdot C_{v_0,\delta}^{3q/2-1} 
               \cdot \int_{M_k}
                        [\kappa(\Phi^{-1}(\delta k))]^q
                        \cdot \| (V_{\theta,\Phi} f)(\cdot, \omega) \|_{\lebesgue^p}^q
                     \, d\omega \\
         ({\scriptstyle{\text{with suitable constant } C_{v_0,\tilde{\kappa},\delta,q}}})
         &\leq C_{v_0, \kappa, \delta, q} \cdot
               \int_{M_k}
                    [\kappa(\omega)]^q\cdot \| (V_{\theta,\Phi} f)(\cdot, \omega) \|_{\lebesgue^p}^q
               \, d\omega.
     \end{align*}

     Now, since $(\varphi_k)_{k \in \ZZ^d}$ is a BAPU for $\CalQ_\Phi^{(\delta,r)}$ by Proposition \ref{prop:DecompSpecialBAPU},
     we can sum the preceding estimate over $k \in \ZZ^d$, to obtain because of $D = \biguplus_{k \in \ZZ^d} M_k$
     (which follows directly from the definition of $M_k$ and bijectivity of $\Phi$) that
     \begin{align*}
         \| \psi_f \|_{\DecompSp(\CalQ^{(\delta,r)}, \lebesgue^p, \ell_u^q)}^q
         = \sum_{k \in \ZZ^d} (u_k \cdot d_k)^q
         &\leq C_{v_0, \kappa, \delta, q}
               \cdot \sum_{k \in \ZZ^d}
                        \int_{M_k}
                            [\kappa(\omega)]^q\cdot \| (V_{\theta,\Phi} f)(\cdot, \omega) \|_{\lebesgue^p}^q\,
                        d\omega  \\
         &=    C_{v_0, \kappa, \delta, q}
               \cdot \int_{D} \,\,
                       [\kappa(\omega)]^q\cdot \| (V_{\theta,\Phi} f)(\cdot, \omega) \|_{\lebesgue^p}^q\,
                     d\omega  \\
         &=    C_{v_0, \kappa, \delta, q}
               \cdot \| f \|_{\Co(\mathcal G(\theta,\Phi),\lebesgue_{\kappa}^{p,q})}^q
          <    \infty,
     \end{align*}
     which completes the proof in case of $q  < \infty$.

     \medskip{}

     It remains to consider the case $q=\infty$. Here, we simply note for arbitrary $k \in \ZZ^d$ as a consequence of the estimate \eqref{eq:WeightedDecompComponentEstimate} (setting $q=1$ in that inequality only),
     \begin{align*}
               u_k \cdot d_k
         &\leq u_k
               \cdot \int_{M_k}
                        [w(\Phi(\omega))]^{-1/2}
                        \cdot \| (V_{\theta,\Phi} f)(\cdot, \omega) \|_{\lebesgue^p}
                     \, d \omega \\
         &\leq C_{v_0,\delta}^{1/2}\cdot  [w(\delta k)]^{-1}
               \cdot \mu(M_k)
               \cdot \kappa_\Phi(\delta k)
               \cdot \esssup_{\omega \in M_k}
                        \| (V_{\theta,\Phi} f)(\cdot, \omega)\|_{\lebesgue^p} \\
         &\leq \delta^{d} C_{v_0,\delta}^{3/2}
               \cdot  \kappa_\Phi(\delta k)
               \cdot \esssup_{\omega \in M_k}
                        \| (V_{\theta,\Phi} f)(\cdot, \omega)\|_{\lebesgue^p} \\
         &\leq  C_{v_0, \kappa, \delta}
               \cdot \|f \|_{\Co(\mathcal G(\theta,\Phi),\lebesgue_{\kappa}^{p,\infty})}
         <     \infty.
     \end{align*}
     Because of $\| \psi_f \|_{\DecompSp(\CalQ_{\Phi}^{(\delta,r)}, \lebesgue^p, \ell_u^\infty)}
     = \sup_{k \in \ZZ^d} (u_k \cdot d_k)$, this completes the proof also for $q = \infty$.
 \end{proof}
 

 \subsection{\texorpdfstring{The embedding
                             $\DecompSp(\CalQ_{\Phi}^{(\delta,r)},
                                        \lebesgue^p,
                                        \ell_u^q)
                              \hookrightarrow \Co(\Phi,\lebesgue_{\kappa}^{p,q})$}
                            {The embedding of the decomposition space into the
                             coorbit space}}
 \label{sub:DecompositionEmbedsIntoCoorbit}

 The main technical obstruction to establishing the embedding of $\DecompSp(\CalQ_{\Phi}^{(\delta,r)},\lebesgue^p,\ell_u^q)$ into $\Co(\Phi,\lebesgue_{\kappa}^{p,q})$ is that the action of $f \in \DecompSp(\CalQ_{\Phi}^{(\delta,r)}, \lebesgue^p, \ell_u^q) \subset Z'(D)$ is only defined on the small space $Z(D)\hookrightarrow \GoodVectors$.
 Hence, some kind of extension procedure is required to obtain functionals defined on all of $\GoodVectors$. We will establish this extension specifically for elements of the decomposition space $\DecompSp(\CalQ_{\Phi}^{(\delta,r)}, \lebesgue^p, \ell_u^q)$ instead of general elements of $Z'(D)$. Nonetheless, our first step is the extension of the voice transform $V_{\theta,\Phi}$ to $Z'(D)$.

 \begin{definition}\label{def:VoiceTransformOnLargerReservoir}
     Let $\Phi: D\rightarrow \RR^d$ and $\kappa: D\rightarrow \RR^+$ be compatible, and $0\neq \theta\in \DistributionSpace(\RR^d)$. For $f \in Z'(D)$, define
     \[
         V_{\theta,\Phi} f : \RR^d \times D \to \CC,
                          (y,\omega) \mapsto \left\langle
                                                f,
                                                \smash{\overline{\rule{0pt}{2.5mm} g_{y,\omega}}} \,
                                             \right\rangle_{Z'(D), Z(D)}.
     \]
 \end{definition}
  \begin{remark}\label{rem:ElaborationOfVoiceTransformDef}
     By an argument similar to the one in the proof of Corollary  \ref{cor:DecompSpReservoirEmbedding}, we can see that with this definition $V_{\theta,\Phi} f(y,\omega) = \langle \Fourier f, 
     \modulation_y \overline{g_\omega} \rangle_{\DistributionSpace,\TestFunctionSpace} = \Fourier^{-1}[\overline{g_\omega}\cdot \Fourier f](y)$. In particular, $V_{\theta,\Phi} f$ is well-defined. Furthermore, by that corollary, the above definition coincides with the definition on $\Reservoir$, for all $f\in\Reservoir \hookrightarrow Z'(D)$, where the embedding $\Reservoir \hookrightarrow Z'(D)$ is given by the map $\Psi$ defined in Corollary \ref{cor:DecompSpReservoirEmbedding}.
 \end{remark}

 We proceed to show that the extension of $V_{\theta,\Phi}$ to $Z'(D)$ inherits important properties, such as continuity and the reproducing property. To this end, we will use the following technical lemma.

 \begin{lemma}
     Let $\emptyset \neq U \subset \RR^\ell$ and $\emptyset \neq V \subset \RR^d$ be open and let
     $\Gamma \in \mathcal C^{\infty} (U \times V)$ with the following property:
     \[
         \text{For each compact set $K \subset U$, there is a compact set $L_K \subset V$ satisfying }
         \forall\ u \in K: \, \supp (\Gamma(u, \bullet)) \subset L_K.
     \]
     Then the map
     \[
         \Theta\colon U \to \TestFunctionSpace(V), u \mapsto \Gamma(u, \bullet)
     \]
     is well-defined and continuous.
     \label{lem:OneArgumentFreeze}
 \end{lemma}

 \begin{proof}
     By assumption, if we set $K = \{u\}$, for any $u \in U$,
     we see that $\supp( \Gamma(u, \bullet)) \subset L_{\{u\}} \subset V$ is compact,
     so that $\Gamma(u, \bullet) \in \TestFunctionSpace(V)$. Hence, $\Theta$ is well-defined.

     To prove continuity, let $(u_n)_{n \in \NN}$ be a sequence in $U$ satisfying $u_n 
     \rightarrow u_0$
     for some $u_0 \in U$. It is not hard to see (using e.g.~the definition of compactness) that
     $K := \{u_n \,:\, n\in \NN_0\} \subset U$ is compact. By assumption, there is thus a compact set $L_K \subset V$
     satisfying $\supp( \Gamma(u_n, \cdot)) \subset L_K$ for all $n \in \NN_0$.

     Now, let $\alpha \in \NN_0^d$ be arbitrary and let $\eps > 0$.
     Since $\partial_v^{\alpha} \Gamma$ is uniformly continuous on the compact set $K \times L_K \subset U\times V$,
     there is some $\delta > 0$ satisfying
     \[
         |(\partial^{\alpha}_v \Gamma)(u,v) - (\partial^{\alpha}_v \Gamma)(\tilde{u}, \tilde{v})| < \eps
     \]
     for all $u, \tilde{u} \in K$ and $v, \tilde{v} \in L_K$ with $|u-\tilde{u}| < \delta$ and $|v - \tilde{v}| < \delta$.
     Since $u_n \to u_0$, there is some $n_0 \in \NN$ such that $|u_n - u_0| < \delta$ for all $n \geq n_0$.

     Now, for $n \geq n_0$ and $v \in V$, there are two cases:

     \textbf{Case 1}: $v \in L_K$. In this case, we have $u_0, u_n \in K$ and $v \in L_K$ with $|u_n - u_0| < \delta$ and
                      $|v-v| = 0 < \delta$. Hence,
                      \[
                          |\partial^\alpha [\Theta(u_n)](v) - \partial^\alpha [\Theta(u_0)](v)|
                          = |(\partial_v^\alpha \Gamma)(u_n, v) - (\partial_v^{\alpha} \Gamma)(u_0, v)| < \eps.
                      \]

     \textbf{Case 2}: $v \notin L_K$. Since $L_K$ is compact and hence closed, there is some $r > 0$ with
                      $B_r (v) \subset \RR^d\setminus L_K$. Because of $\supp( \Theta(u_\ell)) \subset L_K$ for all $\ell \in \NN_0$,
                      this implies $\Theta(u_\ell) \equiv 0$ on a neighborhood of $v$, for each $\ell \in \NN_0$. Hence,
                      \[
                          |\partial^{\alpha} [\Theta(u_n)](v) - \partial^\alpha [\Theta(u_0)] (v)| = |0 - 0| = 0 < \eps.
                      \]

     Thus, we have shown that $\partial^\alpha [\Theta(u_n)] \to \partial^\alpha [\Theta(u_0)]$ uniformly
     on $V$ for arbitrary $\alpha \in \NN_0^d$, as well as $\supp(\Theta(u_n)) \subset L_K$ for all $n \in \NN_0$.
     In view of \cite[Theorem 6.5 and the ensuing remark]{RudinFA}, this implies $\Theta(u_n) \to \Theta(u_0)$ with
     convergence in $\TestFunctionSpace(V)$, as desired. Note that it indeed suffices to consider sequential continuity,
     since the domain $U \subset \RR^\ell$ of $\Theta$ is first countable.
 \end{proof}
 
 Now we are ready to collect the most important properties of the extended voice transform.
 \begin{lemma}\label{lem:VoiceTransformOnLargerReservoirProperties}
     Let $\Phi: D\rightarrow \RR^d$ and $\kappa: D\rightarrow \RR^+$ be compatible, and $0\neq \theta\in \TestFunctionSpace(\RR^d)$ with $\|\theta\|_{\lebesgue^2(\RR^d)}=1$. Then,
     $g_{y,\omega} \in \Fourier^{-1}(\TestFunctionSpace(D))$, for all $(y,\omega)\in \PhSpace$.
     For $f \in Z'(D)$, the following hold:
     \begin{enumerate}
         \item \label{enu:GeneralVoiceTrafoContinuous} The voice transform $V_{\theta,\Phi} f$, defined in  Definition \ref{def:VoiceTransformOnLargerReservoir}, is continuous.
         \item \label{enu:VoiceTrafoOrthRelations} For any $h\in \Fourier^{-1}(\DistributionSpace(D))$, we have $\overline{h}\in Z(D)$ and
             \[
                 \langle f,\overline{h} \rangle_{Z'(D),Z(D)} = \int_{\PhSpace} V_{\theta,\Phi}f(y,\omega)\cdot \overline{V_{\theta,\Phi}h(y,\omega)}~d(y,\omega),\quad \text{for all } f\in Z'(D).
             \]
         \item \label{enu:GeneralVoiceTrafoReproducing} If $K_{\theta,\Phi}\colon \PhSpace\times\PhSpace \rightarrow \CC$ is given by
         \[
          K_{\theta,\Phi}( (z,\eta), (y,\omega) )
                 := \left\langle g_{y,\omega} , \, g_{z,\eta}\right\rangle_{\lebesgue^2}
                 =  \overline{V_{\theta,\Phi} \, g_{z,\eta}(y,\omega)},\quad \text{for all } (y,\omega), (z,\eta)\in\PhSpace,
         \]
         we have
             \[
                 V_{\theta,\Phi} f = K_{\theta,\Phi}(V_{\theta,\Phi} f) := \int_\PhSpace K_{\theta,\Phi}( \bullet ,(y,\omega) ) \cdot V_{\theta,\Phi} f (y,\omega) ~d(y,\omega),\quad \text{for all } f\in Z'(D).
             \]
         \item \label{enu:GeneralVoiceTrafoInjective} The map
                 \[
                     V_{\theta,\Phi}\colon Z'(D) \to \mathcal{C}(\PhSpace), f \mapsto V_{\theta,\Phi} f
                 \]
                 is injective.
     \end{enumerate}
 \end{lemma}
 \begin{proof}
  To prove Item (1), it suffices by Remark \ref{rem:ElaborationOfVoiceTransformDef} to show that the map
  \[
   \Theta\colon \PhSpace \rightarrow \TestFunctionSpace(D),\ (y,\omega)\mapsto \modulation_y \overline{g_\omega}
  \]
  is continuous. Note that the map
  \[
   \Gamma \colon \PhSpace\times D \rightarrow \CC,\ ((y,\omega),\xi)\mapsto \modulation_y\overline{g_\omega}(\xi) = [w(\Phi(\omega))]^{-1/2}\cdot \overline{\theta(\Phi(\xi)-\Phi(\omega))} \cdot e^{2\pi i \langle y,\xi\rangle}
  \]
  is smooth. Here, we used that $\Phi$ is a $\mathcal C^{\infty}$-diffeomorphism, such that $w = \det(D\Phi^{-1}(\bullet))$ is smooth. If $K\subset \PhSpace$ is compact, then there are $K_0\subset \RR^d$ and some compact $K_1\subset D$ with $K \subset K_0\times K_1$. This implies, for $(y,\omega)\in K$, that
  \begin{equation}\label{eq:SuppGammaEstimate}
   \begin{split}
    \supp(\Gamma((y,\omega),\bullet)) & = \supp(\modulation_y \overline{g_\omega}) = \supp(g_\omega)\\
    & \subset \Phi^{-1}(\Phi(\omega)+\supp(\theta)) \subset \Phi^{-1}(\Phi(K_1)+\supp(\theta)) =: L_{K_1},
    \end{split}
  \end{equation}
  where $L_{K_1}\subset D$ is compact. Hence, we can apply Lemma \ref{lem:OneArgumentFreeze} to prove continuity of $\Theta$ and thus of $V_{\theta,\Phi} f$.

  \medskip{}

  To prove Item (3), note that $\theta\in \TestFunctionSpace(\RR^d)$. Hence, $\Fourier g_{y,\omega} = \Fourier (\translation_y \widecheck{g_\omega}) = \modulation_{-y}g_{\omega} \in \TestFunctionSpace(D)$, for all $(y,\omega)\in\PhSpace$, implying $g_{y,\omega}\in \Fourier^{-1}(\TestFunctionSpace(D))$. If we assume for the moment that Item (2) holds, then Item (3) follows by setting $h = g_{z,\eta}$ in Item (2).

   \medskip{}

  Furthermore, if Item (2) holds, and $f\in Z'(D)$ is such that $V_{\theta,\Phi} f\equiv 0$, then
  \[
   \langle f, h \rangle_{Z'(D),Z(D)} = \left\langle f, \overline{\overline{h}} \right\rangle_{Z'(D),Z(D)} = \int_{\PhSpace} V_{\theta,\Phi}f(y,\omega)\cdot \overline{V_{\theta,\Phi}\overline{h}(y,\omega)}~d(y,\omega) = 0,\quad \text{for all } h\in Z(D).
  \]
  Hence $f \equiv 0$, showing that $f\mapsto V_{\theta,\Phi} f$ is injective, proving Item (4).

  \medskip{}

  It remains to prove that Item (2) indeed holds. We will instead prove the equivalent identity
  \begin{equation}\label{eq:AltVoiceTrafoOrthRelations}
    \langle F, \Fourier^{-1}(\overline{h})\rangle_{\DistributionSpace(D),\TestFunctionSpace(D)} = \int_{\PhSpace} V_{\theta,\Phi}(\Fourier^{-1} F)(y,\omega)\cdot \overline{V_{\theta,\Phi} h (y,\omega)}~d(y,\omega),
  \end{equation}
  for all $F\in\DistributionSpace(D),\ h\in\Fourier^{-1}(\TestFunctionSpace(D))$: By elementary properties of the Fourier transform, $h\in \Fourier^{-1}(\TestFunctionSpace(D))$ implies $\overline{h}\in Z(D)$, such that 
  $\langle F,\Fourier^{-1} \overline{h}\rangle_{\DistributionSpace(D),\TestFunctionSpace(D)} = \langle   \Fourier^{-1} F,\overline{h}\rangle_{Z'(D),Z(D)}$, for all $F\in\DistributionSpace(D),\ h\in\Fourier^{-1}(\TestFunctionSpace(D))$.

  To prove Equation \eqref{eq:AltVoiceTrafoOrthRelations}, note that $\Fourier h\in \TestFunctionSpace(D)$, such that
  \begin{equation}\label{eq:KIsCompact}
    K:= \Phi^{-1}(\Phi(\supp(\Fourier h))-\supp(\theta))\subset D\quad \text{is compact.}
  \end{equation}
  Now, if $\omega\in D$ and $y\in\RR^d$ with
  \[
     0 \neq V_{\theta,\Phi} h(y,\omega) \overset{(\textrm{Rem.} \ref{rem:ElaborationOfVoiceTransformDef})}{=} \langle \Fourier f,
     \modulation_y \overline{g_\omega} \rangle_{\DistributionSpace,\TestFunctionSpace},
  \]
  then
  \[
   \emptyset\neq \supp(\Fourier h) \cap \supp(g_\omega) =
   \supp(\Fourier h) \cap \Phi^{-1}(\Phi(\omega) + \supp(\theta)) =
   \Phi^{-1}( \Phi(\supp(\Fourier h)) \cap (\Phi(\omega) + \supp(\theta))),
  \]
  which implies $\Phi(\omega) \in \Phi(\supp(\Fourier h))-\supp(\theta)$ and hence $\omega\in K$. Overall, with $L_K$ as in Equation \eqref{eq:SuppGammaEstimate}, we have  $L_K = \Phi^{-1}(\Phi(K)+\supp(\theta)) = \Phi^{-1}(\Phi(\supp(\Fourier h))-\supp(\theta)+\supp(\theta))$. We see that the effective domain of integration in Equation \eqref{eq:AltVoiceTrafoOrthRelations} is $\RR^d\times K$, and on this set, the support of $\modulation_y \overline{g_\omega}$ is contained in $L_K$.

  Now, since $F\in \DistributionSpace(D)$, we can apply \cite[Theorem 6.8]{RudinFA} to prove the existence of $C=C(F,h)>0$ and $N=N(F,h)\in \NN$ such that
  \[
   |\langle F,g\rangle_{\DistributionSpace,\TestFunctionSpace}| \leq C \|g\|_{\mathcal C^N_b(D)},\quad \text{for all } g\in\TestFunctionSpace(D), \text{ with } \supp(g)\subset \supp(\Fourier h)\cup K\cup L_K.
  \]
  Here, $\mathcal C^N_b(D)$ is the vector space of all $\mathcal C^N$-functions $g\colon D\rightarrow \CC$, such that $\|g\|_{\mathcal C^N_b(D)} := \max_{\alpha\in\NN_0^d,|\alpha|\leq N} \|g\|_{\lebesgue^{\infty}(D)} < \infty$. The Hahn-Banach theorem then yields a bounded, linear functional $\widetilde{F}\in (\mathcal C^N_b(D))'$ that agrees with $F\in\DistributionSpace(D)$ on all $g\in\TestFunctionSpace(D)$ with $\supp(g)\subset \supp(\Fourier h)\cup K\cup L_K$. Since the application of bounded linear functionals can be interchanged with Bochner integrals (see \cite[Chapter VI, Theorem 4.1]{LangRealFunctional}) and since
  \[
   V_{\theta,\Phi}(\Fourier^{-1}F)(y,\omega) = \langle F,\modulation_y \overline{g_\omega}\rangle_{\DistributionSpace,\TestFunctionSpace} = \langle \widetilde{F},\modulation_y \overline{g_\omega}\rangle_{(\mathcal C^N_b(D))',\mathcal C^N_b(D)},\quad \text{for all } (y,\omega)\in \RR^d\times K,
  \]
  we can rewrite the right-hand side of Equation \eqref{eq:AltVoiceTrafoOrthRelations} as
  \begin{equation}\label{eq:AltVoiceTrafoOrthRelations_SecondForm}
    \int_{\PhSpace} V_{\theta,\Phi}(\Fourier^{-1} F)(y,\omega)\cdot \overline{V_{\theta,\Phi} h (y,\omega)}~d(y,\omega) = \left\langle \widetilde{F}, \int_{\PhSpace} \overline{V_{\theta,\Phi} h (y,\omega)} \cdot \modulation_y\overline{g_\omega}~d(y,\omega)\right\rangle_{(\mathcal C^N_b(D))',\mathcal C^N_b(D)},
  \end{equation}
  provided that the (Bochner) integral
  \begin{equation}\label{eq:AlmostInversionFormula}
    \chi_h := \int_{\PhSpace} \overline{V_{\theta,\Phi} h (y,\omega)} \cdot \modulation_y\overline{g_\omega}~d(y,\omega) \in \mathcal C^N_b(D)
  \end{equation}
  exists. Assume for now that this is the case. Then it only remains to show that $\chi_h$ coincides with $\mathcal F^{-1}\overline{h}$. 
  Let $\gamma\in \mathcal C_c(D)$ be arbitrary and note that $h \mapsto \langle h,\gamma\rangle$ is a bounded, linear functional on $\mathcal C^N_b(D)$. Since $\|\theta\|_{\lebesgue^2(\RR^d)} = 1$ and $h\in\Fourier^{-1}(\TestFunctionSpace(D))\subset \LtDF$, we obtain by \cite[Theorem 3.5]{VoHoPreprint},
  \[
   \begin{split}
   \langle \Fourier^{-1}\overline{h}, \gamma\rangle_{\lebesgue^2} & = \langle \overline{\Fourier\gamma},h \rangle_{\lebesgue^2}\\
   & \overset{(\ast)}{=} \int_{\PhSpace} V_{\theta,\Phi}(\overline{\Fourier \gamma})(y,\omega)\overline{V_{\theta,\Phi} h(y,\omega)}~d(y,\omega) \\
   & \overset{(\triangle)}{=} \int_{\PhSpace} \langle \modulation_y\overline{g_\omega},\gamma\rangle_{\lebesgue^2} \cdot \overline{V_{\theta,\Phi} h(y,\omega)}~d(y,\omega) \\
   &\overset{(\square)}{=} \left\langle \int_{\PhSpace} \overline{V_{\theta,\Phi} h(y,\omega)}\modulation_y\overline{g_\omega}~d(y,\omega),\gamma \right\rangle_{\lebesgue^2} \\
   & = \langle \chi_h,\gamma\rangle_{\lebesgue^2}.
   \end{split}
  \]
  Here, the step marked with $(\ast)$ is justified by noting $\overline{\Fourier \gamma} = \Fourier^{-1}\overline{\gamma}\in \Fourier^{-1}(\mathcal C_c(D)) \subset \LtDF$, the step marked with $(\triangle)$ by $V_{\theta,\Phi}(\overline{\Fourier \gamma})(y,\omega) = \langle \overline{\Fourier \gamma},\translation_y \widecheck{g_\omega}\rangle_{\lebesgue^2} = \langle \overline{\gamma},\modulation_{-y} g_\omega\rangle_{\lebesgue^2} = \langle \modulation_{y} \overline{g_\omega},\gamma\rangle_{\lebesgue^2}$. Finally, The step marked with $(\square)$ uses the assumption that the integral defining $\chi_h$ exists and that the application of bounded linear functionals can be interchanged with Bochner integrals (see \cite[Chapter VI, Theorem 4.1]{LangRealFunctional}).

  Altogether, we have shown that $\chi_h\in \mathcal C^N_b(D)\subset \mathcal C(D)$ and $\Fourier^{-1}\overline{h} = \overline{\Fourier h}\in \TestFunctionSpace(D)$ are two continuous functions that satisfy $\langle \Fourier^{-1}\overline{h}, \gamma\rangle = \langle \chi_h,\gamma\rangle$ for all $\gamma\in \mathcal C_c(D)$. Hence, $\chi_h = \Fourier^{-1}\overline{h}$. We get
  \[
  \begin{split}
   \langle F,\Fourier^{-1}\overline{h}\rangle_{\DistributionSpace,\TestFunctionSpace} & =
   \langle \widetilde{F},\Fourier^{-1}\overline{h}\rangle_{(\mathcal C^N_b(D))',\mathcal C^N_b(D)} = \langle \widetilde{F},\chi_h\rangle_{(\mathcal C^N_b(D))',\mathcal C^N_b(D)}\\
   & = \int_{\PhSpace} \langle \widetilde{F},\modulation_y\overline{g_\omega}\rangle_{(\mathcal C^N_b(D))',\mathcal C^N_b(D)} \cdot \overline{V_{\theta,\Phi} h(y,\omega)}~d(y,\omega) \\
   & = \int_{\PhSpace} V_{\theta,\Phi}(\Fourier^{-1}F)(y,\omega)\cdot \overline{V_{\theta,\Phi} h(y,\omega)}~d(y,\omega),
   \end{split}
  \]
  which is the desired equality \eqref{eq:AltVoiceTrafoOrthRelations}.

  \medskip{}

  Finally, we complete the proof by showing that the Bochner integral \eqref{eq:AlmostInversionFormula} does indeed exist. It is well-known that if $X$ is a Banach space, $(\Omega,\mu_\Omega)$ a measure space, and $f\colon \Omega\rightarrow X$ is Bochner measurable with $\int_\Omega \|f\|_X~d\mu_\Omega < \infty$, then the Bochner integral $\int_\Omega f~d\mu_\Omega\in X$ exists, see, e.g., \cite[Chapter VI, Corollary 5.9]{LangRealFunctional}. We first show Bochner measurability of the integrand in \eqref{eq:AlmostInversionFormula}: Recall from the proof of Item (1) that
  $\Theta\colon D\rightarrow \TestFunctionSpace(D) \hookrightarrow \mathcal C^N_b(D),\ (y,\omega)\mapsto \modulation_y\overline{g_\omega}$ is continuous. Since $V_{\theta,\Phi} h\colon \PhSpace \rightarrow \CC$ is continuous as well, see Section \ref{subsec:warpedTF}, the integrand in \eqref{eq:AlmostInversionFormula} is continuous as a map into $\mathcal C^N_b(D)$. Since the domain $\PhSpace = \RR^d\times D\subset \RR^{2d}$ is separable, this implies that the integrand is separable-valued and Borel measurable (by continuity). By \cite[Chapter VI, M11 (Page 124)]{LangRealFunctional}, this implies that the integrand is Bochner measurable.

  Now, we proceed to show that $\int_{\PhSpace} \| \overline{V_{\theta,\Phi} h (y,\omega)} \cdot \modulation_y\overline{g_\omega}\|_{\mathcal C^N_b(D)}~d(y,\omega) < \infty$. Since $\Theta$ is continuous, see above, the map $D\rightarrow \TestFunctionSpace(D),\ \omega\mapsto \overline{g_\omega} = \Theta(0,\omega)$ is continuous as well. Since $\|\bullet\|_{\mathcal C^N_b(D)}$ is a continuous semi-norm on $\TestFunctionSpace(D)$, we conclude that 
  $\max_{\omega\in K} \|\overline{g_\omega}\|_{\mathcal C^N_b(D)}$ is finite. Therefore, Leibniz' rule yields, for $(y,\omega)\in \RR^d\times K$, $\xi\in D$, and $\alpha\in \NN^d_0$ with $|\alpha|\leq N$, that
  \[
    |\partial^\alpha [\modulation_y\overline{g_\omega}](\xi)| \leq \sum_{\beta\leq \alpha} \binom{\alpha}{\beta} |\partial^\beta_\xi e^{2\pi i\langle y,\xi\rangle}|\cdot |(\partial^{\alpha-\beta} \overline{g_\omega})(\xi)| \lesssim 
    \sum_{\beta\leq \alpha} \binom{\alpha}{\beta} |2\pi y|^\beta \lesssim 
    (1+|y|)^N,
  \]
  where the final implied constant depends on $N$, $K$, $\Phi$ and $\theta$. 
  In short, $\|\modulation_y \overline{g_\omega}\|_{\mathcal C^N_b(D)} \lesssim 
  (1+|y|)^N$, for all $y\in\RR^d$ and $\omega\in K$.

  Since $\overline{h}\in Z(D)$ with $\supp(\Fourier^{-1}\overline{h}) = \supp(\Fourier h) =: \tilde{K}$ compact, Lemma \ref{lem:DecompSpDualReservoirEmbedding} yields $C_2 = C_2(d,N,\theta,\tilde{K},\Phi)>0$, such that 
  \[
   |V_{\theta,\Phi}h(y,\omega)|\leq C_2\cdot (1+|y|)^{-(d+N+1)} \Indicator_K(\omega),\quad \text{for all } (y,\omega)\in \PhSpace.
  \]
  Here, we implicitly used that $V_{\theta,\Phi}h(y,\omega) = 0$ whenever $\omega\in D\setminus K$, with $K$ as in Equation \eqref{eq:KIsCompact}, which we already noted in the discussion after said equation. Overall, this implies
  \[
   \int_{\PhSpace} \| \overline{V_{\theta,\Phi} h (y,\omega)} \cdot \modulation_y\overline{g_\omega}\|_{\mathcal C^N_b(D)}~d(y,\omega) \lesssim \cdot \int_{\PhSpace} (1+|y|)^N\cdot (1+|y|)^{-(d+N+1)} \cdot \Indicator_K(\omega)~d(y,\omega) < \infty,
  \]
  since $K\subset D$ is compact. The final implied constant here has the same dependencies as $C_2$ above. We conclude that the Bochner integral \eqref{eq:AlmostInversionFormula} exists and the proof is finished.
 \end{proof}

 \begin{remark}\label{rem:WhatIfThetaIsNotNormalized}
   In Lemma \ref{lem:VoiceTransformOnLargerReservoirProperties}, we assume that $0\neq \theta\in\TestFunctionSpace(\RR^d)$ is such that $\|\theta\|_{\lebesgue^2(\RR^d)}=1$. If that is not the case, it suffices to note that $V_{\theta,\Phi} f$ is conjugate linear in the choice of $\theta$ to see that Lemma \ref{lem:VoiceTransformOnLargerReservoirProperties} remains valid, up to modification by a constant factor in Items \ref{enu:VoiceTrafoOrthRelations} and \ref{enu:GeneralVoiceTrafoReproducing}. Specifically, we have
   \[
        \langle f,\overline{h} \rangle_{Z'(D),Z(D)} = \|\theta\|^{-2}_{\lebesgue^2(\RR^d)}\cdot \int_{\PhSpace} V_{\theta,\Phi}f(y,\omega)\cdot \overline{V_{\theta,\Phi}h(y,\omega)}~d(y,\omega),
   \]
   and
   \[
        V_{\theta,\Phi} f = \|\theta\|^{-2}_{\lebesgue^2(\RR^d)}\cdot K_{\theta,\Phi}(V_{\theta,\Phi} f),
   \]
   for all $0\neq \theta\in\TestFunctionSpace(\RR^d)$ and $f\in Z'(D)$.
 \end{remark}

The next result establishes the main part of the embedding
 $\DecompSp(\CalQ_{\Phi}^{(\delta,r)}, \lebesgue^p, \ell_u^q) \hookrightarrow \Co (\lebesgue_{\kappa}^{p,q})$,
 by showing $\| V_{\theta,\Phi} f\|_{\lebesgue_\kappa^{p,q}}
 \lesssim \| f \|_{\DecompSp(\CalQ_{\Phi}^{(\delta,r)}, \lebesgue^p, \ell_u^q)}$.

 \begin{lemma}\label{lem:DecompositionSpaceGivesVoiceTransformDecay}
     Let $\Phi: D\rightarrow \RR^d$ and $\kappa: D\rightarrow \RR^+$ be compatible 
     and fix $\delta > 0$, $r > \sqrt{d} / 2$, and $0 < t < r - \sqrt{d}/2$. Finally, choose  $0\neq \theta \in \TestFunctionSpace(\RR^d)$ such that Equation \eqref{eq:CanonicalThetaChoice} is satisfied. 
     
     For $p,q \in [1,\infty]$ and
     \[
         u = (u_k)_{k \in \ZZ^d}
         \qquad \text{ with } u_k := \kappa_\Phi(\delta k) \cdot [w(\delta k)]^{\frac{1}{q} - \frac{1}{2}},
     \]
     there is a constant $C>0$ such that each
     $f \in \DecompSp(\CalQ_{\Phi}^{(\delta,r)}, \lebesgue^p, \ell_u^q)$ satisfies
     \[
         \| V_{\theta,\Phi} f \|_{\lebesgue_{\kappa}^{p,q}}
         \leq C \cdot \| f \|_{\DecompSp(\CalQ_{\Phi}^{(\delta,r)}, \lebesgue^p, \ell_u^q)} < \infty.
     \]
 \end{lemma}
 \begin{proof}
     Let $f \in \DecompSp(\CalQ_{\Phi}^{(\delta,r)}, \lebesgue^p, \ell_u^q) \subset Z'(D)$ be fixed.
     As seen in Remark \ref{rem:ElaborationOfVoiceTransformDef}, we have
     \[
         (V_{\theta,\Phi} f)(y,\omega)
         = [\Fourier^{-1}(\overline{\smallVSpace g_{\omega}} \cdot \Fourier f)](y)
         = [\Fourier^{-1}(g_{\omega} \cdot \Fourier f)](y)
         \quad \text{for all}\quad (y,\omega) \in \RR^d \times D,
     \]
     since $\theta$, and thus also $g_\omega$, is real. Assume for now that
     $\omega \in M_k: = \Phi^{-1} \left( \delta \cdot \left( k + \smash{ \left[ -\frac{1}{2}, \frac{1}{2} \right)^d}
         \vphantom{ \left[ \frac{1}{2} \right]} \, \right) \right)$.
     Since $\supp( \theta) \subset \delta \cdot B_t (0)$, with $0<t<\sqrt{r}/2$, this implies, by \eqref{eq:DecompAnalyzingVectorAssumptions},
     \begin{align*}
         \supp( g_{\omega})
         =       \Phi^{-1} \left( \Phi(\omega) + \supp( \theta) \right)
         & \subset \Phi^{-1} \left(
                              \delta \cdot \left[
                                  k + \smash{ \left[ \textstyle{-\frac{1}{2}, \frac{1}{2}} \right)^d}
                                                     \vphantom{ \left[ \textstyle{\frac{1}{2}} \right]}
                                           \, \right]
                              + \delta \cdot B_t (0)
                           \right)\\
         &\subset \Phi^{-1} \left( \delta k + \delta \cdot B_{\sqrt{d} / 2}(0) + \delta \cdot B_t (0)\right) \\
         ({\scriptstyle{\text{since } t + \sqrt{d} / 2 < r}})
         &\subset \Phi^{-1} \left( \delta [k + B_r (0)] \right) = Q_{\Phi,k}^{(\delta,r)}.
     \end{align*}
     But since $(\varphi_k)_{k \in \ZZ^d}$ is a $\CalQ_{\Phi}^{(\delta,r)}$-BAPU, we have for
     $\varphi_k^{\ast} := \sum_{\ell \in k^{\ast}} \varphi_{\ell}$, with $k^\ast =\{\ell\in\ZZ^d~:~ Q_{\Phi,\ell}^{(\delta,r)} \cap Q_{\Phi,k}^{(\delta,r)} \neq \emptyset\}$, that
     $\varphi_k^{\ast} \equiv 1$ on $Q_k^{(\delta,r)} \supset \supp (g_{\omega})$.
     Hence, $g_{\omega} = \varphi_k^{\ast} \cdot g_{\omega}$ and thus
     \[
         (V_{\theta,\Phi} f)(y,\omega)
         = [\Fourier^{-1}(g_{\omega} \cdot \Fourier f)](y)
         = [\Fourier^{-1}(g_{\omega} \cdot \varphi_k^{\ast} \cdot \Fourier f)](y)
         \quad \text{for all}\quad \omega \in M_k.
     \]
     Using Young's inequality and Lemma \ref{lem:BAPUBuildingBlocksEstimate}
     (and recalling $g_{\omega} = \sqrt{w(\Phi(\omega))} \cdot \varphi^{(\omega)}$), we thus obtain for all $\omega \in M_k$:
     \begin{align*}
         \| (V_{\theta,\Phi} f)(\cdot, \omega) \|_{\lebesgue^p}
         &\leq \|\Fourier^{-1} g_{\omega}\|_{\lebesgue^1}
               \cdot \|\Fourier^{-1}(\varphi_k^{\ast} \cdot \Fourier f)\|_{\lebesgue^p} \\
         ({\scriptstyle{\text{Lem. } \ref{lem:BAPUBuildingBlocksEstimate}}})
         &\lesssim 
         [w(\Phi(\omega))]^{-1/2}
               \cdot \sum_{\ell \in k^\ast}
                        \|\Fourier^{-1}(\varphi_\ell \cdot \Fourier f)\|_{\lebesgue^p} \\
         &= 
         [w(\Phi(\omega))]^{-1/2} \cdot \sum_{\ell \in k^{\ast}} d_\ell,\quad \text{ with } d_\ell:= \|\Fourier^{-1}(\varphi_\ell \cdot \Fourier f)\|_{\lebesgue^p},\quad \text{for all}\quad \ell\in \ZZ^d.
     \end{align*}
     Here, the implied constant 
     is provided by Lemma \ref{lem:BAPUBuildingBlocksEstimate}, with $\zeta = \theta$.

     \smallskip{}

     Note that $M_k \subset Q_{\Phi,k}^{(\delta,r)}$, since $r>\sqrt{d}/2$ and define $v' := \kappa_\Phi \cdot w^{-1/2}$. The weight $v'$ is $v$-moderate with $v:= \tilde{\kappa}\cdot v_0^{d/2}$, where $v_0$ is a control weight for $\Phi$, see also
     \cite[Lemma 4.9]{VoHoPreprint}. Thus, an application of Lemma \ref{lem:DecompInducedCoveringIsNice}, Item (1), with $\tilde{v} = v'\circ \Phi = \kappa \cdot (w\circ \Phi)^{-1/2}$ shows for all $\omega \in M_k\subset Q^{(\delta,r)}_k$,
     \begin{equation}
         \begin{split}
             F(\omega) &:= \kappa(\omega) \cdot \| (V_{\theta,\Phi} f)(\cdot, \omega) \|_{\lebesgue^p} \\
             & \lesssim 
             \kappa(\omega)[w(\Phi(\omega))]^{-1/2}
                        \cdot \sum_{\ell \in k^{\ast}} d_\ell\\
                  ({\scriptstyle |\delta k-\Phi(\omega)| < \delta r })
                  &\lesssim
                  \kappa_{\Phi}(\delta k)
                             \cdot [w(\delta k)]^{-1/2}
                             \cdot \sum_{\ell \in k^{\ast}} d_{\ell} \\
                       &=  
                       \frac{u_k}{[w(\delta k)]^{1/q}}
                             \cdot \sum_{\ell \in k^{\ast}} d_{\ell}.
         \end{split}
         \label{eq:DecompositionIntoCoorbitMainEstimate}
     \end{equation}
    Note that since the weight $u$ is $\CalQ_{\Phi}^{(\delta,r)}$-moderate
     (cf.~Lemma \ref{lem:DecompositionWeightIsModerate}), the so-called $\CalQ_{\Phi}^{(\delta, r)}$-clustering map
     \[
     \Gamma \colon \ell_u^q (\ZZ^d) \to \ell_u^q (\ZZ^d),\quad (c_k)_{k \in \ZZ^d} \mapsto (c_k^{\ast})_{k \in \ZZ^d},\quad \text{with}\quad  c_k^{\ast} := \sum_{\ell \in k^{\ast}}c_\ell
     \]
      is well-defined and bounded; see
     e.g.~\cite[Definition 2.5 and Lemma 3.2]{fegr85}.
     This implies $
     (d_k^{\ast})_{k \in \ZZ^d} \in \ell_u^q$ with
     $
     \| (d_k^{\ast})_{k \in \ZZ^d} \|_{\ell_u^q} \leq \|\Gamma\| \cdot \|(d_k)_{k \in \ZZ^d}\|_{\ell_u^q}
     = \| \Gamma \| \cdot \|f \|_{\DecompSp(\CalQ_{\Phi}^{(\delta,r)}, \lebesgue^p, \ell_u^q)}$.

     Now, we distinguish the two cases $q < \infty$ and $q = \infty$.
     For $q < \infty$, we take the $q$-th power of \eqref{eq:DecompositionIntoCoorbitMainEstimate} and integrate over
     $\omega \in M_k$ to obtain
     \[
         \int_{M_k} [F(\omega)]^q \, d \omega
         \lesssim 
         \frac{\mu(M_k)}{w(\delta k)} \cdot \left( u_k \cdot \sum_{\ell \in k^{\ast}} d_{\ell} \right)^q.
     \]
     Next, note that since $w$ is $v_0^d$-moderate and $v_0$ radially increasing, we have (with $\Phi(M_k) = \delta k + [-\delta/2,\delta/2]$)
     \[
     \begin{split}
         \mu(M_k)
         &=  \mu \left( \Phi^{-1} \left( \delta \cdot \left[ k + \smash{ \left[ \textstyle{-\frac{1}{2}, \frac{1}{2}} \right)^d}
                                                                          \vphantom{ \left[ \textstyle{\frac{1}{2}} \right]} \,
                                                \right]
                                  \right)
                 \right)\\
         &= \int_{\Phi(M_k)} w(\tau) \, d \tau
         \leq \int_{\Phi(M_k)} w(\delta k) \cdot v_0^d (\tau - \delta k) \, d \tau \\
         &\leq \delta^d \cdot v_0^d(\delta \sqrt{d}) \cdot w(\delta k).
     \end{split}
    \]
     Thus, with an appropriately large constant $C>0$, depending on $\Phi$, $\theta$ and the control weight $v_0$, as well as $d$ and $\delta$, 
     we have $\int_{M_k} [F(\omega)]^q \, d\omega \leq C \cdot (u_k \cdot c_k)^q$ for all
     $k \in \ZZ^d$. But since $D = \biguplus_{k \in \ZZ^d} M_k$, we can sum this estimate over $k \in \ZZ^d$ to conclude
         \begin{align*}
             \| V_{\theta,\Phi} f\|_{\lebesgue_{\kappa}^{p,q}}^q
             &= \int_{D} [F(\omega)]^q \, d \omega
             = \sum_{k \in \ZZ^d} \int_{M_k} [F(\omega)]^q \, d \omega \\
             &\leq C \cdot \sum_{k \in \ZZ^d} (u_k \cdot c_k)^q
             =    C \cdot \| c \|_{\ell_u^q}^q \\
             &\leq C \cdot \|\Gamma\|^q \cdot \|f\|_{\DecompSp(\CalQ_{\Phi}^{(\delta,r)}, \lebesgue^p, \ell_u^q)}^{q}
             < \infty.
         \end{align*}
     This proves the claim in case of $q < \infty$.

     For $q=\infty$, the proof is slightly simpler: From Equation \eqref{eq:DecompositionIntoCoorbitMainEstimate}, and denoting by $\tilde{C}>0$ the implied constant there, we get
     \begin{align*}
         F(\omega)
         &\leq \tilde{C} \cdot u_k \cdot \sum_{\ell \in k^\ast} d_{\ell} 
         = \tilde{C} \cdot u_k \cdot d_k^{\ast} \\
         &\leq \tilde{C} \cdot \|(d_k^{\ast})_{k \in \ZZ^d}\|_{\ell_u^\infty} 
         \leq \tilde{C} \cdot \|\Gamma\| \cdot \|d\|_{\ell_u^\infty} \\
         &\leq \tilde{C} \cdot \|\Gamma\| \cdot \|f\|_{\DecompSp(\CalQ_{\Phi}^{(\delta,r)}, \lebesgue^p, \ell_u^q)}
         < \infty,
     \end{align*}
     which proves the claim, since $q=\infty$ implies
     $\|V_{\theta,\Phi} f\|_{\lebesgue_{\kappa}^{p,q}} = \esssup_{\omega \in D} F(\omega)$.
 \end{proof}

 Now, we can finally show our claim that the coorbit space $\Co (\Phi,\lebesgue_{\kappa}^{p,q})$ is canonically isomorphic to
 the decomposition space $\DecompSp(\CalQ_{\Phi}^{(\delta,r)}, \lebesgue^p, \ell_u^q)$.
 \begin{corollary}\label{cor:CoorbitIsomorphicToDecomposition}
     Let $\Phi: D\rightarrow \RR^d$ and $\kappa: D\rightarrow \RR^+$ be compatible, see Definition \ref{def:StandingDecompositionAssumptions}, and fix $\delta > 0$, $r > \sqrt{d} / 2$, and $0 < t < r - \sqrt{d}/2$. Finally, choose  $0\neq \theta \in \TestFunctionSpace(\RR^d)$ such that Equation \eqref{eq:CanonicalThetaChoice} is satisfied. 
     
     For $p,q \in [1,\infty]$, the bounded linear map
     \[
      \Psi\colon \Co (\Phi,\lebesgue_{\kappa}^{p,q}) \to \DecompSp(\CalQ_{\Phi}^{(\delta,r)}, \lebesgue^p, \ell_u^q),\quad f\mapsto \psi_f\quad \text{with}\quad
      \psi_f : Z(D) \to \CC,\quad
      g \mapsto f(\overline{g}) = \langle f, \overline{g} \rangle_{\Reservoir, \GoodVectors}.
     \]
     is an isomorphism of Banach spaces. Here, $u = (u_k)_{k \in \ZZ^d}$, $u_k := \kappa_\Phi(\delta k) \cdot [w(\delta k)]^{\frac{1}{q} - \frac{1}{2}}$ for $k\in\ZZ^d$, as before.
 \end{corollary}
 \begin{proof}
     By Theorem \ref{thm:CoorbitIntoDecomposition} and the bounded inverse theorem, it suffices to show that $\Psi$ is surjective.
     Thus, let $f \in \DecompSp(\CalQ_{\Phi}^{(\delta,r)}, \lebesgue^p, \ell_u^q) \subset Z'(D)$ be arbitrary,
     and let $F := V_{\theta,\Phi} f$.
     By Lemma \ref{lem:DecompositionSpaceGivesVoiceTransformDecay}, we have $F \in \lebesgue_{\kappa}^{p,q}$.
     Furthermore, Lemma \ref{lem:VoiceTransformOnLargerReservoirProperties} and Remark \ref{rem:WhatIfThetaIsNotNormalized} show that $F = \|\theta\|^{-2}_{\lebesgue^2(\RR^d)}\cdot K_{\theta,\Phi} (F)$, where
     $K_{\theta,\Phi}$ is as in the lemma.
     
     Apply \cite[Lemma 2.30]{kempka2015general} (with the reproducing kernel $R = \|\theta\|^{-2}_{\lebesgue^2(\RR^d)}\cdot K_{\theta,\Phi}$) to see that $F = V_{\theta,\Phi} f_0$ for some $f_0 \in \Co(\Phi,\lebesgue_{\kappa}^{p,q})$.
     With $f := \Psi f_0 = \psi_{f_0}$, Theorem \ref{thm:CoorbitIntoDecomposition} shows
     $f \in \DecompSp (\CalQ_{\Phi}^{(\delta,r)}, \lebesgue^p, \ell_u^q) \subset Z'(D)$.
     Furthermore, we have
     \begin{alignat*}{3}
         (V_{\theta,\Phi} f)(y,\omega)
         &= \left\langle f, \overline{g_{y,\omega}}\right\rangle_{Z'(D), Z(D)}
         &&= \left\langle \psi_{f_0}, \overline{g_{y,\omega}}\right\rangle_{Z'(D), Z(D)} \\
         ({\scriptstyle{\text{Def. of } \psi_{f_0}}})
         &= \left\langle f_0, g_{y,\omega}\right\rangle_{\Reservoir, \GoodVectors}
         &&= (V_{\theta,\Phi} f_0)(y,\omega) = G(y,\omega) = (V_{\theta,\Phi} f)(y,\omega)
     \end{alignat*}
     for all $(y,\omega) \in \RR^d \times D$.

     But as shown in Lemma \ref{lem:VoiceTransformOnLargerReservoirProperties},
     this implies $f = \Psi f_0 \in \mathrm{image}(\Psi)$, so that $\Psi$ is indeed surjective and the proof is complete.
 \end{proof}

 At last, we drop our special assumptions regarding $\theta$.

\begin{theorem}\label{thm:CoorbitDecompositionIsomorphism}
     Let $\Phi: D\rightarrow \RR^d$ and $\kappa: D\rightarrow \RR^+$ be compatible, see Definition \ref{def:StandingDecompositionAssumptions}, 
     and let $0\neq\theta\in\TestFunctionSpace(\RR^d)$. 
     Further, fix $\delta > 0$ and $r > \sqrt{d} / 2$ and consider the $(\delta, r)$-fine frequency covering $\CalQ_{\Phi}^{(\delta, r)}$ induced by $\Phi$, see Definition \ref{def:frequency_space_covering}.

     The map $\Psi$ defined in Corollary \ref{cor:DecompSpReservoirEmbedding} restricts to an isomorphism of Banach spaces
     \[
         \Psi : \Co (\mathcal G(\theta,\Phi), \lebesgue_\kappa^{p,q})
                \to \DecompSp (\CalQ_\Phi^{(\delta, r)}, \lebesgue^p, \ell_u^q)
         \quad \text{ with } \quad u = (u_k)_{k \in \ZZ^d}
         \text{ and } u_k = \kappa (\Phi^{-1}(\delta k)) \cdot [w(\delta k)]^{\frac{1}{q} - \frac{1}{2}} \,\, .
     \]
 \end{theorem}
 \begin{proof}
   By Theorem \ref{cor:warped_disc_frames}, the coorbit space
   $\Co (\Phi, \lebesgue_\kappa^{p,q}) = \Co (\mathcal{G}(\theta, \Phi), \lebesgue_\kappa^{p,q})$
   is independent of the choice of $0\neq \theta \in \TestFunctionSpace (\RR^d)$. Hence, the result follows immediately from
   Corollary \ref{cor:CoorbitIsomorphicToDecomposition}.
 \end{proof}

  \section{Fundamentals of radial warping}\label{sec:radial}

The frequency coverings associated with well-known instances of decomposition spaces, say Besov or $\alpha$-modulation spaces, exhibit symmetries with respect to the modulus of the frequency position. The same can be said about related constructions of frame decompositions. This observation has motivated the study of \emph{radial warping functions} in \cite{VoHoPreprint}. Since such radially warped time-frequency systems feature prominently in our embedding results presented in Sections \ref{sec:WarpedEmbeddings}--\ref{sec:DominatingMixedSmoothness}. Finally, we obtain conditions for equality of warped coorbit spaces and embeddings between warped coorbit spaces with respect to radial warping functions, in Section , we shortly recall the most important results on radial warping presented in \cite{VoHoPreprint}.

\begin{definition} [{\cite[Definition 8.3]{VoHoPreprint}}]\label{def:RadialWarpingFunction}
  For a diffeomorphism $\vrho : \RR \to \RR$, with $\vrho(\xi) = c \xi$ for all
  $\xi \in (-\eps , \eps)$ and some $\eps,c > 0$, the
  \textbf{associated radial warping function} is given by
  \begin{equation}
    \Phi_\vrho
    : \RR^d \to \RR^d,
    \xi \mapsto \widetilde{\vrho}(|\xi|) \cdot \xi,
    \quad \text{with} \quad
    \widetilde{\vrho}(t) := \vrho(t)/t
    \quad \text{for } t \in \RR \setminus \{0\} \, ,
    \quad \text{and} \quad \widetilde{\vrho}(0) := c \,.
    \label{eq:RadialWarping}
  \end{equation}
\end{definition}

For a more concise notation and to stay consistent with \cite{VoHoPreprint} we denote by $\vrhoinv:=\vrho^{-1}$ the inverse of a bijection $\vrho : \RR \to \RR$. If $\vrho$ satisfies some more structured assumptions, specified in the following definition, then $\Phi_\vrho$ is a $k$-admissible warping function in the sense of Definition \ref{assume:DiffeomorphismAssumptions}, see below.

\begin{definition}[\hspace{1sp}{\cite[Definition 8.1]{VoHoPreprint}}]\label{def:AdmissibleRho}
  Let $k \in \NN_0$. A function $\vrho : \RR \to \RR$ is called a
  \textbf{$k$-admissible radial component with control weight
  $v : \RR \to (0, \infty)$},
  if the following hold:
  \begin{enumerate}[leftmargin=0.7cm]
    \item $\vrho$ is a strictly increasing $\mathcal C^{k+1}$-diffeomorphism with inverse
          $\vrhoinv = \vrho^{-1}$.
    \item $\vrho$ is antisymmetric, that is, $\vrho(-\xi) = -\vrho(\xi)$ for all
          $\xi \in \RR$.  In particular, $\vrho(0)=0$.

    \item There are $\eps > 0$ and $c > 0$ with $\vrho(\xi) = c \xi$
          for all $\xi \in (-\eps, \eps)$.

    \item The weight $v$ is continuous, submultiplicative,
          and radially increasing. Additionally, $\vrhoinv'$ and
          \begin{equation}
            \widetilde{\vrhoinv} : \RR \to (0,\infty),
            \qquad \text{defined by} \qquad
            \widetilde{\vrhoinv}(\xi) := \vrhoinv(\xi)/\xi
            \quad \text{for} \quad \xi \neq 0,
            \quad \text{and} \quad \widetilde{\vrhoinv}(0) := c^{-1}.
            \label{eq:PsiTildeDefinition}
          \end{equation}
          are $v$-moderate.
    \item There are constants $C_0, C_1 > 0$ with
          \begin{equation}
            C_0 \cdot \widetilde{\vrhoinv}(\xi)
            \leq \vrhoinv'(\xi)
            \leq C_1 \cdot (1+\xi) \cdot \widetilde{\vrhoinv}(\xi)
            \qquad \forall \, \xi \in (0,\infty).
            \label{eq:AdmissibilityFirstDerivative}
          \end{equation}

    \item We have $|\vrhoinv^{(\ell)}| \lesssim \vrhoinv'$, for all $\ell \in \underline{k + 1}$.
    \end{enumerate}
\end{definition}

\begin{remark}
 Note that here, we use a simpler, equivalent form of Item (6) already mentioned in \cite{VoHoPreprint},  after Definition 8.1 therein. An analogous simplification applies in the case of weakly admissible radial components, see Definition \ref{def:WeaklyAdmissibleRadialComponent} below. 
\end{remark}

\begin{theorem}[\hspace{1sp}{\cite[Proposition 8.7, Corollary 8.8]{VoHoPreprint}}]\label{cor:RadialWarpingIsWarping}
  Let $\vrho : \RR \to \RR$ be a $k$-admissible radial component, for some
  $k \in \NN$ with $k \geq d+1$, with control weight $v : \RR \!\to\! (0,\infty)$. 
  Then the associated radial warping function
  $\Phi_\vrho : \RR^d \to \RR^d$ is a $k$-admissible warping function. Furthermore, there is a constant $C>0$, such that the control weight $v_0$ for $\Phi_\vrho$ can be chosen as
  \begin{equation} \label{eq:RadialWarpingControlWeightExplicit}
    v_0 : \RR^d \to (0, \infty),
          \tau \mapsto C \cdot (1+|\tau|) \cdot v(|\tau|),
  \end{equation}
  and the weight $w = \det(\mathrm{D}\Phi_\vrho^{-1})$ is given by
  \begin{equation}
    w(\tau)
    = \vrhoinv ' (|\tau|) \cdot [\widetilde{\vrhoinv} (|\tau|)]^{d-1}
    \, .
    \label{eq:RadialWarpingWeightExplicit}
  \end{equation}
\end{theorem}
 
The technical requirement that $\vrho$ be linear in a neighborhood of the origin (Condition (3) in Definition \ref{def:AdmissibleRho}) is often unnatural, but for every function $\vsig$ that satisfies the following weaker assumption, an \emph{equivalent} $k$-admissible radial component $\vrho$ can be constructed, see also Remark \ref{rem:SlowStartEquivalence} for a discussion of said equivalence.  

\begin{definition}[\hspace{1sp}{\cite[Definition 8.12]{VoHoPreprint}}]\label{def:WeaklyAdmissibleRadialComponent}
  Let $k \in \NN_0$.
  A continuous function $\vsig : [0,\infty) \to [0,\infty)$
  is called a \textbf{weakly $k$-admissible radial component with control weight
  $u : [0,\infty) \to (0,\infty)$}, if it satisfies the following conditions:
  \begin{enumerate}
\item $\vsig$ is $\mathcal C^{k+1}$ on $(0,\infty)$, with $\vsig'(\xi) > 0$
          for all $\xi \in (0,\infty)$. In particular, $\vsig_\ast = \vsig^{-1}$ exists and is $\mathcal C^{k+1}$ on its domain.

    \item $\vsig (0) = 0$ and $\vsig (\xi) \to \infty$ as $\xi \to \infty$.

    \item The control weight $u$ is continuous and increasing with
          $u(\xi+\eta) \leq u(\xi)\cdot u(\eta)$ for all $\xi,\eta\in [0,\infty)$.
          Furthermore, there are $\delta > 0$ and $C_0, C_1 > 0$, such that the following properties hold for all $\xi,\eta\notin B_{\delta}(0)$:
          \begin{align}
            C_0 \cdot \widetilde{\vsiginv}(\xi)
            \leq \vsiginv ' (\xi)
            \leq C_1 \cdot \vsiginv (\xi),
            &
            \label{eq:SlowStartDerivativeLarge} \\
            \widetilde{\vsiginv}(\xi) \leq \widetilde{\vsiginv}(\eta) \cdot u(| \xi-\eta |),
            &
            \label{eq:SlowStartPsiTildeModerate} \\
            |\vsiginv^{(\ell)} (\xi)| \lesssim \vsiginv'(\xi) 
            & \qquad \forall \, 
                                \ell \in \underline{k + 1} \, ,
            \label{eq:SlowStartHigherDerivativesBound}
          \end{align}
          where $\widetilde{\vsiginv}(\eta) = \vsiginv(\eta)/\eta$, for all $\eta\in(0,\infty)$, analogous to $\widetilde{\vrho_\ast}$ in \eqref{eq:PsiTildeDefinition}.
  \end{enumerate}
\end{definition}

\begin{theorem}[{\cite[Proposition 8.14]{VoHoPreprint}}]\label{thm:SlowStartFullCriterion}
  Let $k \in \NN_0$, and let $\vsig : [0,\infty) \to [0,\infty)$ be a
  weakly $k$-admissible radial component with control weight
  $u : [0,\infty) \to (0,\infty)$.
  
  Fix some $\eps > 0$, some 
   $c \in \big(0, \vsig (\eps) / (2\eps)\big)$, 
  and an even function $\Omega \in C_c^\infty(\R)$ that satisfies
   $\Omega(\xi) = 1$ for $x\in B_\eps(0)$,  $\Omega(\xi) = 0$ for $x\not\in B_{2\eps}(0)$,
   and $\Omega'(\xi) \leq 0$ for $\xi\in[0,\infty)$. 
 Then the function
 \begin{equation}
  \vrho\colon \RR \to \RR,\
  \xi \mapsto
  \begin{cases}
    c \xi \cdot \Omega(\xi)
    + \sgn (\xi) \cdot (1 - \Omega(\xi)) \cdot \vsig (|\xi|),
    & \text{if } \xi \neq 0 \, , \\
    0, & \text{if } \xi = 0 \, ,
  \end{cases}
  \label{eq:SlowStartDefinition}
\end{equation}
is a $k$-admissible radial component. There is a a constant $C \geq 1$, such that
  \[
    v \colon \RR \to (0,\infty),\ \xi \mapsto C \cdot u(|\xi|) \,
  \]
can be chosen as control weight for $\vrho$.
\end{theorem}

\begin{remark}
In \cite{VoHoPreprint}, the radial component $\vrho$ constructed in Theorem \ref{thm:SlowStartFullCriterion} is referred to as a \textbf{slow start version} of $\vsig$. We will adopt this terminology in the following sections. Although $\Phi_\vsig : \xi \mapsto \xi/|\xi| \cdot \vsig(|\xi|)$ is not a warping function in the sense of Definition \ref{def:warpfundamental}, it is continuous and invertible, such that we can define the induced $(\delta,r)$-fine frequency coverings $\CalQ_{\Phi_\vsig}^{(\delta, r)} := \big[\, Q_{\Phi_\vsig,k}^{(\delta, r)} \,\big]_{k \in \ZZ^d}$ by $Q_{\Phi_\vsig,k}^{(\delta, r)} := \Phi_\vsig^{-1}(\delta \cdot B_r (k))$, for all $k\in\ZZ^d$, as in Definition \ref{def:frequency_space_covering}. The coverings $\CalQ^{(\delta,r)}_{\Phi_\vsig}$ and $\CalQ^{(\delta,r)}_{\Phi_\vrho}$ coincide up to finitely many covering elements which have nontrivial intersection with the ball $B_{2\eps}(0)$. It is thus easy to show that they are weakly equivalent. Since all the elements of either covering are path-connected, they are equivalent by Lemma \ref{lem:almost_subordinateness_connected} and, by Theorem \ref{thm:decomposition_space_coincidence}, generate the same decomposition spaces (in the sense of equivalent norms). Strictly speaking, it must be verified that $\CalQ^{(\delta,r)}_{\Phi_\vsig}$ is indeed a decomposition covering. This, however, is a straightforward task, which only requires the modification of the BAPU in Definition \ref{def:BAPUgenerators} for the finite number of elements in $\CalQ^{(\delta,r)}_{\Phi_\vsig}$ that differ from those in $\CalQ^{(\delta,r)}_{\Phi_\vrho}$.
\end{remark}

  \section{Embedding relation between warped coorbit spaces}
\label{sec:WarpedEmbeddings}

We begin by recalling some results from \cite{voigtlaender2016embeddings}, which will be useful for the verification of embedding relations between pairs of warped coorbit spaces in this section and between Besov spaces and warped coorbit spaces in Section \ref{sec:WarpedCoorbitAndBesov}. In Sections \ref{sec:alpha} and \ref{sec:DominatingMixedSmoothness} we will further show that $\alpha$-modulation spaces and certain spaces of dominating mixed smoothness coincide with certain warped coorbit spaces. This will be accomplished by verifying the conditions of Theorem \ref{thm:decomposition_space_coincidence}. 

\begin{definition}\label{def:relativeModerateness}
Let $\emptyset\neq D\subset \RR^d$ be open and let $\CalQ = (Q_i)_{i\in I} = (T_iQ'_i + b_i)_{i\in I}$ and $\mathcal P = (P_j)_{j\in J} = (S_j P'_j + c_j)_{j\in J}$ be two open, tight, semi-structured coverings of $D$, and $v = (v_j)_{j\in J}$ a $\CalP$-moderate weight. We say that $v$ is \emph{relatively $\CalQ$-moderate}, if
\[
 \sup_{i\in I} \sup_{k,\ell \in \{ j\in J \with P_j\cap Q_i\neq \emptyset\}} v_k/v_\ell < \infty,
\]
and $\CalP$ is \emph{relatively $\CalQ$-moderate}, if
\[
 \sup_{i\in I} \sup_{k,\ell \in \{ j\in J \with P_j\cap Q_i\neq \emptyset\}} |\det(S_k)|/|\det(S_\ell)| < \infty.
\]
\end{definition}

In the following theorem, we use the convention that the conjugate exponent $p'$ is defined for $p\in (0,1]$ to equal $p' = \infty$.

\begin{theorem}\label{thm:DecompEmbeddings}
  Let $\emptyset\neq D\subset \RR^d$ be open and let $\CalQ = (Q_i)_{i\in I} = (T_iQ'_i + b_i)_{i\in I}$ and $\mathcal P = (P_j)_{j\in J} = (S_j P'_j + c_j)_{j\in J}$ be two open, tight, semi-structured coverings of $D$. Furthermore, let $u = (u_i)_{i\in I}$ be $\CalQ$-moderate and $v = (v_j)_{j\in J}$ be $\mathcal P$-moderate, and let $p_1,p_2,q_1,q_2\in [1,\infty]$. Assume that $\CalP$ is almost subordinate to $\CalQ$ and choose for each $i\in I$ some $j_i\in J$ with $Q_i \cap P_{j_i} \neq \emptyset$. If $\mathcal P$ and $v$ are relatively $\CalQ$-moderate. Then the following hold:
  \begin{enumerate}
    \item \emph{\cite[Theorem 7.2 (4)]{voigtlaender2016embeddings}} $\DecompSp(\CalQ,\lebesgue^{p_1},\ell^{q_1}_u)\hookrightarrow \DecompSp(\CalP,\lebesgue^{p_2},\ell^{q_2}_v)$ if and only if
    \[
      p_1 \leq p_2\quad \text{and}\quad K_t:= \left\|\left(\frac{v_{j_i}}{u_i}\cdot |\det T_i|^{t} \cdot |\det S_{j_i}|^{p_1^{-1} - p_2^{-1} - t}\right)_{i\in I}\right\|_{\ell^{q_2\cdot (q_1/q_2)'}}< \infty,
    \]
    with $t =  \max(0,q_2^{-1} - \min(p_1^{-1},1-p_1^{-1}))$.
     \item \emph{\cite[Theorem 7.4 (4)]{voigtlaender2016embeddings}} $\DecompSp(\CalP,\lebesgue^{p_1},\ell^{q_1}_v)\hookrightarrow \DecompSp(\CalQ,\lebesgue^{p_2},\ell^{q_2}_u)$ if and only if
    \[
      p_1 \leq p_2\quad \text{and}\quad K_{\tilde{t}}:= \left\|\left(\frac{u_{i}}{v_{j_i}}\cdot |\det T_{i}|^{\tilde{t}} \cdot |\det S_{j_i}|^{p_1^{-1} - p_2^{-1} - \tilde{t}}
      \right)_{i\in I}\right\|_{\ell^{q_2\cdot (q_1/q_2)'}}< \infty,
    \]
    with $\tilde{t} = \max(0,\max(p_2^{-1},1-p_2^{-1})-q_1^{-1})$.
    \item \emph{\cite[Corollary 7.3]{voigtlaender2016embeddings}} If $\CalQ = \mathcal P$, then $\DecompSp(\CalQ,\lebesgue^{p_1},\ell^{q_1}_u)\hookrightarrow \DecompSp(\CalQ,\lebesgue^{p_2},\ell^{q_2}_v)$ if and only if
    \[
      p_1 \leq p_2\quad \text{and}\quad K:= \left\|\left(\frac{v_i}{u_i}\cdot |\det T_i|^{p_1^{-1} - p_2^{-1}}\right)_{i\in I}\right\|_{\ell^{q_2\cdot (q_1/q_2)'}}< \infty.
    \]
  \end{enumerate}
\end{theorem}

Interpreting warped coorbit spaces as decomposition spaces provides access to Theorem \ref{thm:DecompEmbeddings} for proving embedding relations between spaces $\Co(\Phi_1,\bd L^{p_1,q_1}_{\kappa_1})$ and $\Co(\Phi_2,\bd L^{p_2,q_2}_{\kappa_2})$. In general, for an embedding  $\Co(\Phi_1,\bd L^{p_1,q_1}_{\kappa_1}) \hookrightarrow \Co(\Phi_2,\bd L^{p_2,q_2}_{\kappa_2})$, Theorem \ref{thm:DecompEmbeddings} suggests that $p_1\leq p_2$ is necessary, such that we can assume $p_\delta = p_1^{-1}-p_2^{-1} \geq 0$. 

\subsection{Embeddings between warped coorbit spaces generated by the same warping function}\label{sub:WarpedEmbeddingsSame}

For now, assume that $\Phi = \Phi_1 = \Phi_2$ and $(\Phi,\kappa_1)$ and $(\Phi,\kappa_2)$ form compatible pairs as per Definition \ref{def:StandingDecompositionAssumptions}. In particular, $(\kappa_j)_\Phi = \kappa_j\circ \Phi^{-1}$, $j\in\{1,2\}$, is $\tilde{\kappa_j}$-moderate, where $\tilde{\kappa_j}$ is continuous, submultiplicative and radially increasing. Recall from Lemma~\ref{lem:DecompInducedCoveringIsNice} and
Proposition~\ref{prop:DecompSpecialBAPU} that the induced $(\delta,r)$-fine frequency covering $\CalQ^{(\delta,r)}_\Phi$, with $\delta>0$ and $r > \sqrt{d}/2$ is tight and semi-structured. Under these assumptions, Theorem~\ref{thm:CoorbitDecompositionIsomorphism} yields
\[\Co(\Phi, \lebesgue^{p_1,q_1}_{\kappa_1})
 = \DecompSp(\CalQ_\Phi^{(\delta,r)}, \lebesgue^{p_1}, \ell_{u^{(q_1)}}^{q_1})\quad \text{and}\quad \Co(\Phi, \lebesgue^{p_2,q_2}_{\kappa_2})
 = \DecompSp(\CalQ_\Phi^{(\delta,r)}, \lebesgue^{p_2}, \ell_{v^{(q_2)}}^{q_2})\] up to canonical identifications, with the sequence elements of $u^{(q_1)}, v^{(q_2)}$ given by 
 \[u^{(q_1)}_k = (\kappa_1)_\Phi(\delta k)\cdot [w(\delta k)]^{\frac{1}{q_1}-\frac{1}{2}}\quad \text{and}\quad v^{(q_2)}_k = (\kappa_2)_\Phi(\delta k)\cdot [w(\delta k)]^{\frac{1}{q_2}-\frac{1}{2}}, \quad k\in\ZZ^d.\] 
 Finally, the linear maps $T_k$, $k\in\ZZ^d$, in Definition \ref{def:SemiStructuredCoverings} (2) take the form $T_k = A(\delta k) = D\Phi^{-1}(\delta k)$, such that $|\det T_k| = w(\delta k)$, see Lemma \ref{lem:DecompInducedCoveringIsNice}. 
 
 Overall, we can apply Theorem \ref{thm:DecompEmbeddings} (3) to show that the relation $\Co(\Phi,\bd L^{p_1,q_1}_{\kappa_1})
 \hookrightarrow \Co(\Phi,\bd L^{p_2,q_2}_{\kappa_2})$ 
  holds if and only if 
 \begin{equation}\label{eq:SimplifiedWarpedEmbeddingCond}
  p_1\leq p_2\quad \text{and}\quad K := \left\|
         \Big(
           [w(\delta k)]^{p_\delta-q_\delta}
           \cdot \frac{(\kappa_2)_\Phi(\delta k)}{(\kappa_1)_\Phi(\delta k)}
         \Big)_{k \in \ZZ^d}
       \right\|_{\ell^{q_2 \cdot (q_1 / q_2)'}}
    < \infty \quad \text{with } q_\delta = q_1^{-1}-q_2^{-2}.
\end{equation}
If $q_1\leq q_2$, then $(q_1/q_2)' = \infty$ such that the above simplifies to $(\kappa_2)_\Phi  \lesssim  w^{q_\delta-p_\delta} \cdot (\kappa_1)_\Phi$, which is equivalent to 
\[
  \kappa_2 \lesssim  [w\circ \Phi]^{q_\delta-p_\delta} \cdot \kappa_1.
\]
If even $\kappa_1\asymp \kappa_2$, we simply obtain $[w(\delta k)]^{p_\delta-q_\delta} \lesssim 1$. If on the other hand $q_1>q_2$, then \eqref{eq:SimplifiedWarpedEmbeddingCond} depends heavily on the  integrability of $w^{p_\delta-q_\delta}$ or the compensation of said term's non-integrability by means of a weight $\kappa_1$ with sufficient growth.
For example, if $\kappa_2 \equiv 1$, then it is easy to verify that
\[
  \begin{split}
  \text{\eqref{eq:SimplifiedWarpedEmbeddingCond} fails if}\quad \kappa_1 & \lesssim [w\circ \Phi]^{p_\delta-q_\delta}\cdot [1+|\Phi(\bullet)|]^{\frac{d}{q_2 \cdot (q_1 / q_2)'}},\\
  \text{but \eqref{eq:SimplifiedWarpedEmbeddingCond} holds if}\quad \kappa_1 & \gtrsim [w\circ \Phi]^{p_\delta-q_\delta}\cdot [1+|\Phi(\bullet)|]^{\frac{d+\eps}{q_2 \cdot (q_1 / q_2)'}}, \text{ for some } \eps > 0.
  \end{split}
\]

\subsection{Embeddings between warped coorbit spaces generated by different warping functions}\label{sub:WarpedEmbeddingsDifferent} 
Here, we only consider the case $p=p_1=p_2$ and $q=q_1=q_2$. Specifically, Theorem \ref{thm:EqualityOfWarpedCoorbitSpaces} introduces a condition that yields equality of warped coorbit spaces, whereas Theorem \ref{thm:EmbeddingForRadialComp} studies embeddings between warped coorbit spaces in the important special case of radial warping functions. As preparation, the next lemma is concerned with  subordinateness relations between the $(\delta,r)$-coverings induced by $\Phi_1$ and $\Phi_2$.

\begin{lemma}\label{lem:AbstractWarpingSuffCond}
  Let $\Phi_1,\Phi_2 : D \rightarrow \RR^d$ be two $0$-admissible warping functions. If there exists a constant $C>0$, such that
    \begin{equation}\label{eq:InducedCoverSufficientSubordinateCond}
    \| D\Phi_2 (\Phi_1^{-1} (\tau)) \cdot D\Phi_1^{-1}(\tau) \| \leq C,\quad \text{for all } \tau\in \RR^d,
  \end{equation}
  then $\CalQ^{(\delta_1,r_1)}_{\Phi_1}$ is almost subordinate to $\CalQ^{(\delta_2,r_2)}_{\Phi_2}$, for all $\delta_1,\delta_2>0$ and $r_1,r_2>\sqrt{d}/2$.

  Given a weight $\kappa: D\rightarrow \RR^+$ such that $\kappa\circ \Phi_2^{-1}$ is $\tilde{\kappa}$-moderate for some continuous, submultiplicative, and radially increasing weight $\tilde{\kappa}\colon \RR^d\rightarrow [1,\infty)$, we additionally have
  \begin{equation}\label{eq:CompatibleWeightModerateness}
    \kappa(\Phi_1^{-1}(\delta_1 k)) \asymp \kappa(\Phi_2^{-1}(\delta_2 \ell)),\quad \text{for all } k,\ell\in\ZZ^d, \text{ with } Q_{\Phi_1,k}^{(\delta_1,r_1)}\cap Q_{\Phi_2,\ell}^{(\delta_2,r_2)}\neq \emptyset.
  \end{equation}
  In particular, if $\Phi_1,\Phi_2$ are $\CalC^\infty$-diffeomorphisms and $(d+1)$-admissible warping functions, then compatibility of $\Phi_2$ and $\kappa$ in the sense of Definition \ref{def:StandingDecompositionAssumptions} implies that $\Phi_1$ and $\kappa$ are compatible as well.
\end{lemma}
\begin{proof}
   Our goal is to show that $\CalQ^{(\delta_1,r_1)}_{\Phi_1}$ is almost subordinate to $\CalQ^{(\delta_2,r_2)}_{\Phi_2}$, for all $\delta_1,\delta_2>0$ and $r_1,r_2>\sqrt{d}/2$. To this end, we begin by showing that, for a given $\delta>0$ and $r>\sqrt{d}/{2}$, and provided that \eqref{eq:InducedCoverSufficientSubordinateCond} holds, there exists $R>\sqrt{d}/2$, such that $\CalQ_{\Phi_1}^{(\delta,r)}$ is subordinate, and thereby almost subordinate, to $\CalQ_{\Phi_2}^{(\delta,R)}$. Since almost subordinateness is a transitive relation, and equivalent covers are almost subordinate to one another, Lemma \ref{lem:DecompInducedCoveringIsNice}, which showed the equivalence of all $\Phi$-induced $(\delta,r)$-fine coverings, can be applied to $\Phi_1$ and to $\Phi_2$ in order to obtain 
   the result for arbitrary $\delta_1,\delta_2>0$ and $r_1,r_2>\sqrt{d}/2$.

   Fix any $\delta>0$, $r>\sqrt{d}/2$. Since $\CalQ_{\Phi_2}^{(\delta,r)}$ is a covering of $D$ by Lemma \ref{lem:DecompInducedCoveringIsNice}, there exists, for every $k\in\ZZ^d$, a $j = j(k)\in\ZZ^d$, such that $
   \Phi^{-1}_1(\delta k) \in \Phi^{-1}_2(\delta \cdot B_{r}(j))$, or equivalently
   \begin{equation}\label{eq:SubordinatenessOfWarpingsOne}
     \Phi_2(\Phi^{-1}_1(\delta k)) \in  B_{\delta r}(\delta j).
   \end{equation}
   By the chain rule, \eqref{eq:InducedCoverSufficientSubordinateCond} states that
   \[
    \| D(\Phi_2\circ \Phi_1^{-1}) (\tau) \| = \| D\Phi_2 (\Phi_1^{-1} (\tau)) \cdot D\Phi_1^{-1}(\tau) \| \leq C,
   \]
   for all $\tau\in \RR^d$, i.e., $\Phi_2\circ\Phi_1^{-1}\colon \RR^d\rightarrow \RR^d$ is Lipschitz-continuous with Lipschitz constant no larger than $C$. If $\xi\in Q^{(\delta,r)}_{\Phi_1,k}$, then $\tau:= \Phi_1(\xi)\in B_{\delta r}(\delta k)$, and thus
   \[
    \left|\Phi_2(\Phi^{-1}_1(\tau)) - \Phi_2(\Phi^{-1}_1(\delta k))\right| < C\delta r.
   \]
   Together with \eqref{eq:SubordinatenessOfWarpingsOne} and $R :=  r+C\delta r/\delta = (C+1)r$, we obtain $\Phi_2(\xi) = \Phi_2(\Phi^{-1}_1(\tau)) \in  \delta B_{R}(j)$ and thus $\tau\in Q^{(\delta,R)}_{\Phi_2,j}$. Since the choice of $\xi \in Q^{(\delta,r)}_{\Phi_1,k}$ was arbitrary, this implies
   \[
    Q^{(\delta,r)}_{\Phi_1,k} \subset Q^{(\delta,R)}_{\Phi_2,j}.
   \]
   Since further $R = (C+1)r$ is independent of the choice of $k\in\ZZ^d$, $\CalQ_{\Phi_1}^{(\delta,r)}$ is subordinate to $\CalQ_{\Phi_2}^{(\delta,R)}$, proving the first part of the lemma, as discussed in the first paragraph of the proof.

   \medskip{}

   To prove the second part, assume that $k,\ell\in\ZZ^d$ are such that $Q_{\Phi_1,k}^{(\delta_1,r_1)}\cap Q_{\Phi_2,\ell}^{(\delta_2,r_2)}\neq \emptyset$ and let $\xi\in Q_{\Phi_1,k}^{(\delta_1,r_1)}\cap Q_{\Phi_2,\ell}^{(\delta_2,r_2)}$. This implies
   \[
    |\Phi_1(\xi)-\delta_1 k| < \delta_1 r_1\quad \text{and}\quad |\Phi_2(\xi)-\delta_2 \ell| < \delta_2 r_2.
   \]
   Using the Lipschitz-continuity of $\Phi_2\circ\Phi_1^{-1}$, see above, and the triangle inequality, we conclude that
   \[
    |\Phi_2(\xi)- \Phi_2(\Phi_1^{-1}(\delta_1 k))| < C\delta_1 r_1\quad \text{and}\quad |\delta_2  \ell - \Phi_2(\Phi_1^{-1}(\delta_1 k))| < \delta_2 r_2 + C\delta_1 r_1.
   \]
   With this estimate and the $\tilde{\kappa}$-moderateness of $\kappa\circ \Phi_2$, it follows that
   \[
    \frac{\kappa(\Phi^{-1}_1(\delta_1 k))}{\kappa(\Phi^{-1}_2(\delta_2 \ell))} = \frac{(\kappa\circ \Phi^{-1}_2)((\Phi_2\circ\Phi^{-1}_1)(\delta_1 k))}{(\kappa\circ \Phi^{-1}_2)(\delta_2  \ell)} \leq \tilde{\kappa}((\delta_2 r_2 + C\delta_1 r_1)\cdot e_1),
   \]
   and the lower bound $\tfrac{\kappa(\Phi^{-1}_1(\delta_1 k))}{\kappa(\Phi^{-1}_2(\delta_2 k))} \geq [\tilde{\kappa}((\delta_2 r_2 + C\delta_1 r_1)\cdot e_1]^{-1}$ is obtained analogously. Here, we used that $\tilde{\kappa}$ is radially increasing. 
   
   The statement about compatibility of $\Phi_1,\ \Phi_2$ and $\kappa$ now follows by noting that the assumptions on $\kappa$ are exactly those in Definition \ref{def:StandingDecompositionAssumptions}, whereas the calculation 
   \[
    \frac{\kappa(\Phi^{-1}_1(\tau))}{\kappa(\Phi^{-1}_1(\sigma))} = \frac{\kappa(\Phi^{-1}_2(\Phi_2\circ \Phi^{-1}_1(\tau)))}{\kappa(\Phi^{-1}_2(\Phi_2\circ \Phi^{-1}_1(\sigma)))} \leq \tilde{\kappa}(|\Phi_2\circ \Phi^{-1}_1(\tau)-\Phi_2\circ \Phi^{-1}_1(\sigma)|\cdot e_1) \leq \tilde{\kappa}(C\cdot |\tau-\sigma|\cdot e_1) = \tilde{\kappa}(C\cdot (\tau-\sigma))
   \]
   shows that $\kappa\circ \Phi^{-1}_1$ is $(\tilde{\kappa}(C\bullet))$-moderate, i.e., $\Phi_1$ and $\kappa$ satisfy the conditions of Definition \ref{def:StandingDecompositionAssumptions} as well. 
\end{proof}

If \eqref{eq:InducedCoverSufficientSubordinateCond} additionally holds with interchanged roles of $\Phi_1$ and $\Phi_2$, then we obtain equality of warped coorbit spaces.

\begin{theorem}\label{thm:EqualityOfWarpedCoorbitSpaces}
  Let $\Phi_1,\Phi_2 \colon D\rightarrow \RR^d$ and $\kappa : D \rightarrow \RR^+$ be such that $(\Phi_{1},\kappa)$ and $(\Phi_{2},\kappa)$ form compatible pairs as per Definition \ref{def:StandingDecompositionAssumptions}, and suppose that there is a constant $C>0$, such that
  \begin{equation}\label{eq:InducedCoverSufficientEquivalenceCond}
    \| D\Phi_{3-j} (\Phi_{j}^{-1} (\tau)) \cdot D\Phi_j^{-1}(\tau) \| \leq C,\quad \text{for all } \tau\in \RR^d \text{ and } j\in\{1,2\}.
  \end{equation}
  Then we have, for all $p,q\in [1,\infty]$, the equality
  \begin{equation}\label{eq:EqualityOfWarpedCoorbitSpaces}
    \Co(\Phi_{1},\lebesgue_{\kappa}^{p,q}) = \Co(\Phi_{2},\lebesgue_\kappa^{p,q}).
  \end{equation}
\end{theorem}
\begin{proof}
  By Theorem \ref{thm:CoorbitDecompositionIsomorphism} we have, for arbitrary $\delta>0$, $r>\sqrt{d}/2$ and $j\in\{1,2\}$, the equalities
   \[
     \Co(\Phi_{j},\lebesgue_\kappa^{p,q}) = \DecompSp(\CalQ^{(\delta,r)}_{\Phi_{j}},\lebesgue^p,\ell^q_{u_j}),\quad \text{with } (u_j)_\ell = \kappa(\Phi_{j}^{-1}(\delta \ell)) \cdot [w_{j}(\delta \ell)]^{1/q-1/2},\ \ell\in\ZZ^d.
   \]
   Since \eqref{eq:InducedCoverSufficientEquivalenceCond} holds, we can apply Lemma \ref{lem:AbstractWarpingSuffCond} as is and with interchanged roles of $\Phi_1$ and $\Phi_2$, to obtain equivalence of the coverings $\CalQ^{(\delta,r)}_{\Phi_{1}}$ and $\CalQ^{(\delta,r)}_{\Phi_{2}}$. By Theorem \ref{thm:decomposition_space_coincidence}, we obtain the desired equality \eqref{eq:EqualityOfWarpedCoorbitSpaces}, provided that $u_1\asymp u_2$. Lemma \ref{lem:AbstractWarpingSuffCond} further implies that
   \[
    \kappa(\Phi_{1}^{-1}(\delta k)) \asymp \kappa(\Phi_{2}^{-1}(\delta \ell)),\quad \text{for all } k,\ell\in\ZZ^d,\ \text{with } Q^{(\delta,r)}_{\Phi_{1},k}\cap Q^{(\delta,r)}_{\Phi_{2},\ell}\neq \emptyset.
   \]
   It remains to show that, likewise,
   \[
   w_1(\delta k) \asymp w_2(\delta \ell),\quad \text{for all } k,\ell\in\ZZ^d,\ \text{with } Q^{(\delta,r)}_{\Phi_{1},k}\cap Q^{(\delta,r)}_{\Phi_{2},\ell}\neq \emptyset.
   \]
   To do so, first note that, for any $\delta>0$, $r>\sqrt{d}/2$ and $j\in\{1,2\}$, we can estimate the Lebesgue measure $\mu$ of $Q^{(\delta,r)}_{\Phi_{j},k}$ a follows:
   \begin{equation}\label{eq:MeasureOfCoveringElements}
     \mu(Q^{(\delta,r)}_{\Phi_{j},k}) = \int_{D} \Indicator_{Q^{(\delta,r)}_{\Phi_{j},k}}(\xi)~d\xi = \int_{\RR^d} |\det(D\Phi_{j}^{-1})(\tau)|\cdot \Indicator_{Q^{(\delta,r)}_{\Phi_{j},k}}(\Phi_{j}^{-1}(\tau))~d\tau = \int_{\delta\cdot B_{r}(k)} w_j(\tau)~d\tau \asymp w_j(\delta k),
   \end{equation}
   for all $k\in\ZZ^d$. Here, we used that $w_j$ is $v_j$-moderate, with respect to the control weight $v_j$ for the $(d+1)$-admissible warping function $\Phi_{j}$, $j\in\{1,2\}$, such that the implied constants in \eqref{eq:MeasureOfCoveringElements} depend on $\delta$, $r$, $d$, and $v_j$. Now let $R = (C + 1)r$ be as in the proof of Lemma \ref{lem:AbstractWarpingSuffCond} and $R_2 = (C+1) R$, with $C>0$ as in \eqref{eq:InducedCoverSufficientEquivalenceCond}. Then, as in the proof of said Lemma, $Q^{(\delta,r)}_{\Phi_{1},k}\cap Q^{(\delta,r)}_{\Phi_{2},\ell}\neq \emptyset$ implies $Q^{(\delta,r)}_{\Phi_{1},k}\subset Q^{(\delta,R)}_{\Phi_{2},\ell}$ and, likewise, $Q^{(\delta,R)}_{\Phi_{2},\ell}\subset Q^{(\delta,R_2)}_{\Phi_{1},k}$. By an application of \eqref{eq:MeasureOfCoveringElements} with $j=1$, we thus find that 
   \[
    w_1(\delta k) \lesssim \mu(Q^{(\delta,r)}_{\Phi_{1},k}) \leq \mu(Q^{(\delta,R)}_{\Phi_{2},\ell}) \leq \mu(Q^{(\delta,R_2)}_{\Phi_{1},k}) \lesssim w_1(\delta k).
   \]
   Applying \eqref{eq:MeasureOfCoveringElements} again with $j=2$, we obtain $w_1(\delta k) \asymp w_2(\delta \ell)$ as desired. Therefore, we can apply Theorem \ref{thm:CoorbitDecompositionIsomorphism} to show the desired equality \eqref{eq:EqualityOfWarpedCoorbitSpaces} of warped coorbit spaces.
\end{proof}

As we now show, the assumptions of both results above simplify greatly in the case of radial warping functions $\Phi_1 = \Phi_{\vrho_1}$, $\Phi_2 = \Phi_{\vrho_2}$ generated from the the radial components $\vrho_1,\vrho_2$.

\begin{corollary}\label{cor:SubordinatenessForBoundingDerivativeRadialComp}
  Let $\vrho_1,\vrho_2 : \RR \to \RR$ be $0$-admissible radial components, such that there is a constant $C>0$ with
  \begin{equation}\label{eq:RadialCompSubordinatenessCondition}
     (\vrho_2)'(\xi) \leq C (\vrho_1)'(\xi),\quad \text{for all } \xi\in\RR^+,
  \end{equation}
  then $\Phi_{\vrho_1}, \Phi_{\vrho_2}$ satisfy the conditions of Lemma \ref{lem:AbstractWarpingSuffCond}. In particular, the covering $\CalQ^{(\delta_1,r_1)}_{\Phi_{\vrho_1}}$ is almost subordinate to $\CalQ^{(\delta_2,r_2)}_{\Phi_{\vrho_2}}$, for every $\delta_1,\delta_2>0$ and $r_1,r_2>\sqrt{d}/2$.
\end{corollary}
\begin{proof}
 By \cite[Lemma 8.4]{VoHoPreprint}, we have for $\vrho\in\{\vrho_1,\vrho_2,(\vrho_1)_\ast,(\vrho_2)_\ast\}$ and $\xi\in\RR^d\setminus\{0\}$ that 
 \[
  D\Phi_\vrho(\xi) = \widetilde{\vrho}(|\xi|)\cdot \pi_\xi^\perp + \vrho'(|\xi|)\cdot \pi_\xi, 
 \]
 where $\pi_\xi$ is the orthogonal projection onto the span of $\xi$, and $\pi_\xi^\perp = \operatorname{id}_{\RR^d} - \pi_\xi$. It is easy to see that
 \[
  (a \pi_\xi^\perp + b \pi_\xi)\cdot (c \pi_\xi^\perp + d \pi_\xi) = ac \pi_\xi^\perp + bd \pi_\xi\quad \text{and}\quad \|a \pi_\xi^\perp + b \pi_\xi\| = \max\{a,b\},
 \]
 for all $a,b,c,d\in\RR^+$. Using $\Phi_{\vrho_1}^{-1} = \Phi_{{\vrho_1}_\ast}$, we obtain for $\tau\in\RR^d\setminus\{0\}$
 \[
  \| D\Phi_{\vrho_2}(\Phi_{\vrho_1}^{-1} (\tau))\cdot D\Phi_{\vrho_1}^{-1}(\tau)\| \leq
  \max\{\ (\vrho_1)_\ast'(|\tau|)\cdot(\vrho_2)'((\vrho_1)_\ast(|\tau|)),\ \widetilde{(\vrho_1)_\ast}(|\tau|)\cdot\widetilde{\vrho_2}((\vrho_1)_\ast(|\tau|))\ \}.
 \]
 By assumption, $(\vrho_1)_\ast'(|\tau|)\cdot(\vrho_2)'((\vrho_1)_\ast(|\tau|)) \leq C (\vrho_1)_\ast'(|\tau|)\cdot(\vrho_1)'((\vrho_1)_\ast(|\tau|)) = C$. Further, since $\vrho_1(0) = \vrho_2(0) = 0$, Equation \eqref{eq:RadialCompSubordinatenessCondition} implies $\vrho_2 \leq C \vrho_1$ on $\RR^+$, and thus
 \[
  \widetilde{(\vrho_1)_\ast}(|\tau|)\cdot\widetilde{\vrho_2}((\vrho_1)_\ast(|\tau|)) = \widetilde{(\vrho_1)_\ast}(|\tau|)\cdot \frac{\vrho_2((\vrho_1)_\ast(|\tau|))}{(\vrho_1)_\ast(|\tau|)}
  \leq C \widetilde{(\vrho_1)_\ast}(|\tau|)\cdot \frac{\vrho_1((\vrho_1)_\ast(|\tau|))}{(\vrho_1)_\ast(|\tau|)} = C.
 \]
 By continuity, the estimate $\| D\Phi_{\vrho_2} (\Phi_{\vrho_1}^{-1} (\tau)) \cdot D\Phi_{\vrho_1}^{-1}(\tau) \| \leq C$ extends from $\tau\in\RR^d\setminus\{0\}$ to all $\tau\in\RR^d$. This finishes the proof.
\end{proof}

\begin{corollary}\label{cor:EqualityOfRadialWarpedCoorbitSpaces}
  Let $\vrho_1,\vrho_2 : \RR \to \RR$ and $\kappa : \RR^d \rightarrow \RR^+$ be such that $(\Phi_{\vrho_1},\kappa)$ and $(\Phi_{\vrho_2},\kappa)$ form compatible pairs as per Definition \ref{def:StandingDecompositionAssumptions}. In particular, 
  assume that $\vrho_1,\vrho_2$ are $(d+1)$-admissible radial components. If $\vrho_1' \asymp \vrho_2'$ and $p,q\in [1,\infty]$, then we have the equality
  \begin{equation}\label{eq:EqualityOfRadialWarpedCoorbitSpaces}
    \Co(\Phi_{\vrho_1},\lebesgue_{\kappa}^{p,q}) = \Co(\Phi_{\vrho_2},\lebesgue_\kappa^{p,q}).
  \end{equation}
\end{corollary}
\begin{proof}
  Use Corollary \ref{cor:SubordinatenessForBoundingDerivativeRadialComp} twice to verify  \eqref{eq:InducedCoverSufficientEquivalenceCond} in the assumptions of Theorem \ref{thm:EqualityOfWarpedCoorbitSpaces}, setting $\Phi_1 = \Phi_{\vrho_1}$ and $\Phi_2 = \Phi_{\vrho_2}$. Then said theorem yields \eqref{eq:EqualityOfRadialWarpedCoorbitSpaces}.
\end{proof}

\begin{remark}\label{rem:SlowStartEquivalence}
  As a consequence of Corollary \ref{cor:EqualityOfRadialWarpedCoorbitSpaces}, we immediately obtain that different slow start versions $\vrho_1,\ \vrho_2$ of the same weakly $k$-admissible radial component $\vsig$ generate the same coorbit spaces.
\end{remark}

If $\vrho_1'$ and $\vrho_2'$ are not proportional, i.e., \eqref{eq:RadialCompSubordinatenessCondition} \emph{does not hold} with the roles of $\vrho_1$ and $\vrho_2$ interchanged, then we require additional restrictions on $\vrho_1$ to prove embeddings between the coorbit spaces associated with the radial warping functions $\Phi_{\vrho_1}$ and $\Phi_{\vrho_2}$. The next result, concerned with relative moderateness of the induced $(\delta,r)$-frequency coverings, is the final step towards such embedding relations.


\begin{proposition}\label{prop:RadialCoveringSubordinateness}
  Let $\vrho_1,\vrho_2 : \RR \to \RR$ be $0$-admissible radial components and
  $\CalQ_{\Phi_{\vrho_1}}^{(\delta,r)}$, $\CalQ_{\Phi_{\vrho_2}}^{(\delta,r)}$ the respective induced $(\delta,r)$-fine frequency coverings, for arbitrary $\delta>0$ and $r > \sqrt{d} / 2$.
  Then
  \begin{equation}
    1 + |\xi|
    \asymp 1 + (\vrho_j)_\ast (\delta |k|)
    \quad \forall \, j\in\{1,2\},\, k \in \ZZ^d,\,
                     \text{ and } \xi \in Q_{\Phi_{\vrho_j},k}^{(\delta,r)} \, .
    \label{eq:RadialWarpingsCoveringNormControl}
  \end{equation}
  If, additionally, there exists a function $C_{\vrho_1} : [1,\infty)\rightarrow [1,\infty)$, such that
  \begin{equation}\label{eq:WarpingOneDerivativeProportionalRelation}
    C_{\vrho_1}(\alpha)^{-1}\vrho_1'(\xi)\leq \vrho_1'(\alpha\xi) \leq C_{\vrho_1}(\alpha)\vrho_1'(\xi),\quad \text{for all } \alpha\in [1,\infty),\ \xi\in\RR^+,
  \end{equation}
  then $\CalQ_{\Phi_{\vrho_1}}^{(\delta,r)}$ is relatively $\CalQ_{\Phi_{\vrho_2}}^{(\delta,r)}$-moderate.

  Specifically, if \eqref{eq:WarpingOneDerivativeProportionalRelation} holds, we have 
  \begin{equation}
    \widetilde{(\vrho_1)_\ast} (\delta |k|)
    \asymp [\widetilde{\vrho_1}((\vrho_2)_\ast(\delta|\ell|))]^{-1}
    \quad \forall \, k,\ell \in \ZZ^d
                  \text{ with }
                  Q_{\Phi_{\vrho_1}, k}^{(\delta,r)} \cap Q_{\Phi_{\vrho_2}, \ell}^{(\delta,r)} \neq \emptyset \, ,
    \label{eq:RadialWarpingsCoveringRhoTildeStarModerateness}
  \end{equation}
  and
  \begin{equation}
    w_1(\delta k) = \det(D\Phi_{\vrho_1}^{-1}(\delta k)) \asymp \frac{[\widetilde{\vrho_1} ((\vrho_2)_\ast(\delta|\ell|))]^{1-d}}{\vrho_1'((\vrho_2)_\ast(\delta|\ell|))}
    \quad \text{if} \quad
    Q_{\Phi_{\vrho_1}, k}^{(\delta,r)} \cap Q_{\Phi_{\vrho_2}, \ell}^{(\delta,r)} \neq \emptyset \, .
    \label{eq:RadialWarpingsCoveringRelativeModerateness}
  \end{equation}
  In particular, $w_1$ is relatively $\CalQ_{\Phi_{\vrho_2}}^{(\delta,r)}$-moderate.
\end{proposition}
\begin{proof}
  We first prove \eqref{eq:RadialWarpingsCoveringNormControl}.
  Fix $j\in\{1,2\}$ and let $v_j : \RR \to \RR^+$ be a control weight for $\vrho_j$.
  For $\xi \in Q_{\Phi_{\vrho_j}, k}^{(\delta,r)}
  = \Phi_{\vrho_j}^{-1} \big(\delta\cdot B_{ r} (k) \big)$, we have
  ${\vrho_j} (|\xi|)
   = |\Phi_{\vrho_j} (\xi)|
   \in \delta|k| + [- \delta r, \delta r]$.
  Since ${\vrho_j}$ is strictly increasing, this implies
  \begin{equation}
    ({\vrho_j})_\ast (\delta |k| - \delta r)
    \leq |\xi|
    \leq ({\vrho_j})_\ast (\delta |k| + \delta r) \, .
    \label{eq:RadialWarpingNormControlStep1}
  \end{equation}

  Now, let us first consider the case $|k| \geq 2r + 2$, so that
  $|k| \asymp |k| - r > 0$ and $|k| + r \asymp |k|$.
    Since by definition of a $0$-admissible radial component,
  $\widetilde{({\vrho_j})_\ast}$ is $v$-moderate, we see
  \[
    |\xi|
    \leq ({\vrho_j})_\ast (\delta |k| + \delta r)
    = \widetilde{({\vrho_j})_\ast} (\delta |k| + \delta r)
      \cdot (\delta |k| + \delta r)
    \lesssim v_j(\delta r)
             \cdot \widetilde{({\vrho_j})_\ast} (\delta |k|)
             \cdot \delta |k|
    = v_j (\delta r) \cdot (\vrho_j)_\ast (\delta |k|) \, ,
  \]
  for $\xi \in Q_{\Phi_{\vrho_j}, k}^{(\delta,r)}$. The lower bound $[v_j(\delta r)]^{-1} ({\vrho_j})_\ast(\delta |k|) \lesssim |\xi|$ is derived similarly, such that overall $|\xi| \asymp ({\vrho_j})_\ast (\delta |k|)$, for $\xi \in Q_{\Phi_{\vrho_j}, k}^{(\delta,r)}$, and thus, whenever $|k| \geq 2r + 2$, we have $1 + |\xi| \asymp 1 + ({\vrho_j})_\ast (\delta |k|)$.

  To prove \eqref{eq:RadialWarpingsCoveringNormControl}, it remains
  to consider the case $|k| \leq 2r + 2$, for which Equation \eqref{eq:RadialWarpingNormControlStep1} yields
  \[
    |\xi| \leq (\vrho_j)_\ast(\delta |k| + \delta r)
         \leq  (\vrho_j)_\ast(3\delta r + 2 \delta)
         \lesssim 1\quad \text{and}\quad (\vrho_j)_\ast(\delta |k|) \leq  (\vrho_j)_\ast(2\delta r + 2 \delta)\lesssim 1.
  \]
  Together, $1 + |\xi| \asymp 1 \asymp 1 +  (\vrho_j)_\ast (\delta |k|)$ if $\xi \in Q_{\Phi_\vrho, k}^{(\delta,r)}$ and $|k| \leq 2r + 2$, as desired.

  \medskip{}

  To prove Equation~\eqref{eq:RadialWarpingsCoveringRhoTildeStarModerateness}, first note that $\emptyset \neq Q_{\Phi_{\vrho_1}, k}^{(\delta,r)} \cap Q_{\Phi_{\vrho_2}, \ell}^{(\delta,r)} \ni \xi$ implies,
  by \eqref{eq:RadialWarpingsCoveringNormControl}, that $1+({\vrho_1})_\ast(\delta |k|)\asymp 1 + |\xi| \asymp 1+({\vrho_2})_\ast(\delta |\ell|)$, i.e., there exists $C_0\geq 1$, such that 
  \[
   C_0^{-1}\cdot (1+({\vrho_2})_\ast(\delta |\ell|)) \leq 1+({\vrho_1})_\ast(\delta |k|) \leq C_0\cdot  (1+({\vrho_2})_\ast(\delta |\ell|)).
  \]

  Moreover, since $\vrho_1(0)=0$, \eqref{eq:WarpingOneDerivativeProportionalRelation} implies
  \begin{equation}\label{eq:WarpingOneProportionalRelation}
   \alpha C_{\vrho}(\alpha)^{-1}\vrho_1(\xi)\leq \vrho_1(\alpha\xi) \leq \alpha C_{\vrho}(\alpha)\vrho_1(\xi),\quad \text{for all } \alpha\in [1,\infty),\ \xi\in\RR^+.
  \end{equation}

  If $k=0$ or $\ell=0$, then Equations \eqref{eq:RadialWarpingsCoveringRhoTildeStarModerateness} and \eqref{eq:RadialWarpingsCoveringRelativeModerateness} are trivial: Both sides are positive and there are only finitely many $k\in\ZZ^d$ with $Q_{\Phi_{\vrho_1}, k}^{(\delta,r)} \cap Q_{\Phi_{\vrho_2}, 0}^{(\delta,r)}\neq \emptyset$. Likewise, there are only finitely many $\ell\in\ZZ^d$ with $Q_{\Phi_{\vrho_1}, 0}^{(\delta,r)} \cap Q_{\Phi_{\vrho_2}, \ell}^{(\delta,r)}\neq \emptyset$. 
  Hence, we can assume that $k\neq 0 \neq \ell$. Then, \eqref{eq:RadialWarpingsCoveringNormControl} implies $({\vrho_1})_\ast(\delta |k|)\asymp ({\vrho_2})_\ast(\delta |\ell|)$, and we obtain
  \[
   [\widetilde{\vrho_1}((\vrho_2)_\ast(\delta|\ell|))]^{-1} = \frac{(\vrho_2)_\ast(\delta|\ell|)}{\vrho_1((\vrho_2)_\ast(\delta|\ell|))} \asymp \frac{(\vrho_1)_\ast(\delta|k|)}{\delta |k|} = \widetilde{(\vrho_1)_\ast} (\delta |k|),
  \]
  where we used \eqref{eq:WarpingOneProportionalRelation} and $(\vrho_1)_\ast(\delta|k|) \asymp (\vrho_2)_\ast(\delta|\ell|)$ to deduce $\vrho_1((\vrho_2)_\ast(\delta|\ell|)) \asymp \vrho_1((\vrho_1)_\ast(\delta|k|)) = \delta |k|$.

  \medskip{}

  To prove the last part of the proposition, recall from \eqref{eq:RadialWarpingWeightExplicit} that $\det(D\Phi_{\vrho_1}^{-1}(\delta k)) = (\vrho_1)_\ast' (\delta |k|) \cdot [\widetilde{(\vrho_1)_\ast} (\delta |k|)]^{d-1}$ and recall that it suffices to consider the case $k\neq 0\neq l$. Because of $(\vrho_1)_\ast(\delta|k|) \asymp (\vrho_2)_\ast(\delta|\ell|)$ we get
  \[
   (\vrho_1)_\ast' (\delta |k|) = [\vrho_1'((\vrho_1)_\ast (\delta |k|))]^{-1} \asymp [\vrho_1'((\vrho_2)_\ast (\delta |\ell|))]^{-1},
  \]
  by an application of \eqref{eq:WarpingOneDerivativeProportionalRelation}. Estimating the remaining term $[\widetilde{(\vrho_1)_\ast} (\delta |k|)]^{d-1}$ with \eqref{eq:RadialWarpingsCoveringRhoTildeStarModerateness} yields the desired estimate \eqref{eq:RadialWarpingsCoveringRelativeModerateness}.

  With \eqref{eq:RadialWarpingsCoveringRelativeModerateness}, we can derive that $\CalQ^{(\delta,r)}_{\Phi_{\vrho_1}}$ is relatively $\CalQ^{(\delta,r)}_{\Phi_{\vrho_2}}$-moderate. To see this, recall from Lemma \ref{lem:DecompInducedCoveringIsNice} that the elements $Q^{(\delta,r)}_{\Phi_{\vrho_1},k}$ of the semi-structured coverning $\CalQ^{(\delta,r)}_{\Phi_{\vrho_1}}$ can be written as $Q^{(\delta,r)}_{\Phi_{\vrho_1},k} = T_k Q_k ' + b_k$, with $T_k = D\Phi_{\vrho_1}^{-1}(\delta k)$, and hence $\det(T_k) = w_1(\delta k)$, such that \eqref{eq:RadialWarpingsCoveringRelativeModerateness} yields relative $\CalQ^{(\delta,r)}_{\Phi_{\vrho_2}}$-moderateness.
\end{proof}

We are now ready to prove embeddings $\Co(\Phi_{\vrho_1},\lebesgue_{\kappa_1}^{p,q})\hookrightarrow \Co(\Phi_{\vrho_2},\lebesgue_{\kappa_2}^{p,q})$, for appropriate weights $\kappa_{1},\kappa_2$.

\begin{theorem}\label{thm:EmbeddingForRadialComp}
  Let $\vrho_1,\vrho_2 : \RR \to \RR$ and $\kappa : \RR^d \rightarrow \RR^+$ be such that $(\Phi_{\vrho_1},\kappa)$ and $(\Phi_{\vrho_2},\kappa)$ form compatible pairs as per Definition \ref{def:StandingDecompositionAssumptions}. In particular, 
  assume that $\vrho_1,\vrho_2$ are $(d+1)$-admissible radial components. If there is a constant $C>0$, and a function $C_{\vrho_1} : [1,\infty)\rightarrow [1,\infty)$, such that
  \[
     (\vrho_2)'(\xi) \leq C (\vrho_1)'(\xi),\quad \text{and}\quad C_{\vrho_1}(\alpha)^{-1}\vrho_1'(\xi)\leq \vrho_1'(\alpha\xi) \leq C_{\vrho_1}(\alpha)\vrho_1'(\xi),\quad \text{for all } \alpha\in [1,\infty),\ \xi\in\RR^+,
  \]
  then, for all $p,q\in [1,\infty]$, we have the embeddings
   \begin{equation}\label{eq:RadialWarpedCoorbitEmbedding}
     \Co(\Phi_{\vrho_1},\lebesgue_{\kappa_{\vrho_1,\vrho_2,-\tilde{t}}}^{p,q})\hookrightarrow \Co(\Phi_{\vrho_2},\lebesgue_\kappa^{p,q}) \hookrightarrow \Co(\Phi_{\vrho_1},\lebesgue_{\kappa_{\vrho_1,\vrho_2,t}}^{p,q}),\
  \end{equation}
  with $\tilde{t} = \max(0,\max(p^{-1},1-p^{-1})-q^{-1})\leq 1$, and $t = \max(0,q^{-1}-\min(p^{-1},1-p^{-1}))\leq 1$. Here, for $t\in\RR$,
  \[
   \kappa_{\vrho_1,\vrho_2,t} := \kappa_{\vrho_1,\vrho_2,t}^{(q)} := \kappa \cdot \left[\frac{w_2\circ\Phi_{\vrho_2}}{w_1\circ \Phi_{\vrho_1}}\right]^{\frac{1}{q}-\frac{1}{2}-t} \quad\text{and}\quad w_j = \det(D\Phi^{-1}_{\vrho_j}),\quad j\in\{1,2\}.
  \]
\end{theorem}
\begin{proof}
   We begin by showing that $(\Phi_1,\kappa_{\vrho_1,\vrho_2,t})$ forms a compatible pair in the sense of Definition \ref{def:StandingDecompositionAssumptions}, for any $t\in\RR$. By assumption, $(\Phi_{\vrho_1},\kappa)$ form a compatible pair, and $w_1\circ \Phi_{\vrho_1}\circ \Phi_{\vrho_1}^{-1} = w_1$ is $v_1$-moderate, with $v_1$ a (submultiplicative) control weight for the $(d+1)$-admissible warping function $\Phi_{\vrho_1}$. Further, it is a direct consequence of the definition of admissible warping functions, Definition \ref{assume:DiffeomorphismAssumptions}, that $(\Phi_{\vrho_2},w_2\circ \Phi_{\vrho_2})$ form a compatible pair. By Corollary \ref{cor:SubordinatenessForBoundingDerivativeRadialComp}, we can apply Lemma \ref{lem:AbstractWarpingSuffCond}, which yields compatibility of $\Phi_{\vrho_1}$ and $w_2\circ \Phi_{\vrho_2}$. Altogether, we conclude that $\kappa_{\vrho_1,\vrho_2,t}\circ \Phi_{\vrho_1}^{-1}$ is moderate with respect to the weight
   \[
    \tilde{\kappa}_1 \cdot \left[v_2(C\bullet)\cdot v_1 \right]^{\left|\frac{1}{q}-\frac{1}{2}-t\right|},
   \]
   where $\tilde{\kappa}_1$ is the moderating weight for the compatible pair $(\Phi_{\vrho_1},\kappa)$, see Definition \ref{def:StandingDecompositionAssumptions}, $v_2$ is a submultiplicative, radially increasing control weight for $\Phi_{\vrho_2}$, see Definition \ref{assume:DiffeomorphismAssumptions}, and the proof of Lemma  \ref{lem:AbstractWarpingSuffCond} shows that $w_2\circ \Phi_{\vrho_2}\circ \Phi_{\vrho_1}^{-1}$ is moderate with respect to $v_2(C\bullet)$.

   Hence, by Theorem \ref{thm:CoorbitDecompositionIsomorphism} we can identify $\Co(\Phi_{\vrho_1},\lebesgue_{\kappa_{\vrho_1,\vrho_2,t}}^{p,q})$ and $\Co(\Phi_{\vrho_2},\lebesgue_\kappa^{p,q})$ with decomposition spaces. Specifically, we obtain, for all $\delta>0$, $r>\sqrt{d}/2$,
   \[
     \Co(\Phi_{\vrho_2},\lebesgue_\kappa^{p,q}) = \DecompSp(\CalQ^{(\delta,r)}_{\Phi_{\vrho_2}},\lebesgue^p,\ell^q_{u_2}),\quad \text{and}\quad \Co(\Phi_{\vrho_1},\lebesgue_{\kappa_{\vrho_1,\vrho_2,t}}^{p,q}) = \DecompSp(\CalQ^{(\delta,r)}_{\Phi_{\vrho_1}},\lebesgue^p,\ell^q_{u_1^{(t)}}),
  \]   
  with 
   \[
      (u_2)_\ell = \kappa(\Phi_{\vrho_2}^{-1}(\delta \ell)) \cdot [w_{2}(\delta \ell)]^{1/q-1/2}\ \ \text{and}\ \   (u_1^{(t)})_k = \kappa(\Phi_{\vrho_1}^{-1}(\delta k)) \cdot \left[w_{1}(\delta k)\right]^{t} \cdot \left[w_{2}(\Phi_{\vrho_2}(\Phi_{\vrho_1}^{-1}(\delta k)))\right]^{1/q-1/2-t},
   \]
   for all $\ell,k\in\ZZ^d$. For ease of notation, and without loss of generality, we will subsequently choose the same values of $\delta>0$ and $r>\sqrt{d}/2$ for both coverings.

   \medskip{}

   In order to apply Theorem \ref{thm:DecompEmbeddings} to prove the desired embeddings, it remains to verify that $\CalQ^{(\delta,r)}_{\Phi_{\vrho_1}}$ is almost subordinate to $\CalQ^{(\delta,r)}_{\Phi_{\vrho_2}}$, and that $\CalQ^{(\delta,r)}_{\Phi_{\vrho_1}}$ and $u_1^ {(t)}$ are relatively $\CalQ^{(\delta,r)}_{\Phi_{\vrho_2}}$-moderate. The first property follows directly from Corollary \ref{cor:SubordinatenessForBoundingDerivativeRadialComp}, and an application of Proposition \ref{prop:RadialCoveringSubordinateness} yields relative $\CalQ^{(\delta,r)}_{\Phi_{\vrho_2}}$-moderateness of $\CalQ^{(\delta,r)}_{\Phi_{\vrho_1}}$.

   To show that $u_1 := u_1^{(t)}$ is relatively $\CalQ^{(\delta,r)}_{\Phi_{\vrho_2}}$-moderate, 
   note as a consequence of Lemma \ref{lem:AbstractWarpingSuffCond} (applied with $\kappa$ and with $w_2\circ \Phi_{\vrho_2}$ instead of $\kappa$) that if $Q^{(\delta,r)}_{\Phi_{\vrho_1},k}\cap Q^{(\delta,r)}_{\Phi_{\vrho_2},\ell} \neq \emptyset$, then 
   \begin{equation}\label{eq:KappaAndWOneTwoPhiOneTwoEquivalence}
     \kappa(\Phi^{-1}_{\vrho_1}(\delta k)) \asymp \kappa(\Phi^{-1}_{\vrho_2}(\delta \ell))\quad \text{and}\quad w_2(\delta \ell) \asymp w_2((\Phi_{\vrho_2}\circ \Phi_{\vrho_1}^{-1}(\delta k)).
   \end{equation}
   Finally, Equation \eqref{eq:RadialWarpingsCoveringRelativeModerateness} in Proposition \ref{prop:RadialCoveringSubordinateness} shows that $w_1(\delta k) = \det(D\Phi_{\vrho_1}^{-1}(\delta k)) \asymp \frac{[\widetilde{\vrho_1} ((\vrho_2)_\ast(\delta|\ell|))]^{1-d}}{\vrho_1'((\vrho_2)_\ast(\delta|\ell|))}$. This clearly shows that $u_1^{(t)}$ is relatively $\CalQ^{(\delta,r)}_{\Phi_{\vrho_2}}$-moderate.

   \medskip{}

   We have now verified all the prerequisites for applying Theorem \ref{thm:DecompEmbeddings}, Items (1)
   and (2), with $\CalP = \CalQ^{(\delta,r)}_{\Phi_{\vrho_1}}$, $\CalQ = \CalQ^{(\delta,r)}_{\Phi_{\vrho_2}}$
   and the weights $v=u_1^{(t)}$, $u = u_2$. To apply Theorem \ref{thm:DecompEmbeddings}(1), we set $t = \max(0,q^{-1}-\min(p^{-1},1-p^{-1})$.
   Select, for each $\ell\in\ZZ^d$, some arbitrary $k_\ell\in\ZZ^d$ such that $Q^{(\delta,r)}_{\Phi_{\vrho_1},k_\ell}\cap Q^{(\delta,r)}_{\Phi_{\vrho_2},\ell}\neq \emptyset$. Then, recalling the form of $\CalQ^{(\delta,r)}_{\Phi_{\vrho_1}}$ and $\CalQ^{(\delta,r)}_{\Phi_{\vrho_2}}$ as semi-structured coverings, given in Lemma \ref{lem:DecompInducedCoveringIsNice},  $K_t$ in Theorem \ref{thm:DecompEmbeddings}(1) takes the form
   \[
     \begin{split}
         K_t & = \sup_{\ell\in\ZZ^d} \left| \frac{(u_1^{(t)})_{k_\ell}}{(u_2)_{\ell}}\cdot [w_2(\delta{\ell})]^{t} \cdot [w_1(\delta k_\ell)]^{-t}\right|\\
         & = \sup_{\ell\in\ZZ^d} \left|\frac{\kappa(\Phi_{\vrho_1}^{-1}(\delta k_\ell))}{\kappa(\Phi_{\vrho_2}^{-1}(\delta \ell))}\right| \cdot \left[\frac{w_{1}(\delta k_\ell)}{w_1(\delta k_\ell)}\right]^{t} \cdot \left[\frac{w_{2}(\Phi_{\vrho_2}(\Phi_{\vrho_1}^{-1}(\delta k_\ell)))}{w_2(\delta {\ell} )}\right]^{1/q+1/2-t}.
     \end{split}
   \]
   The second factor above trivially equals $1$. The other two factors are bounded by a constant, see Equation \eqref{eq:KappaAndWOneTwoPhiOneTwoEquivalence}. In total, we obtain $K_t < \infty$. Since $K_t$ is finite, we obtain the embedding $\Co(\Phi_{\vrho_2},\lebesgue_\kappa^{p,q}) \hookrightarrow \Co(\Phi_{\vrho_1},\lebesgue_{\kappa_{\vrho_1,\vrho_2,t}}^{p,q})$. The second embedding is proven analogously, with an application of Theorem \ref{thm:DecompEmbeddings}(2). Specifically, we set $t = -\tilde{t} = -\max(0,\max(p^{-1},1-p^{-1})-q^{-1})$ and obtain
   \[
    K_{\tilde{t}} = \sup_{\ell\in\ZZ^d} \left| \frac{(u_2)_{\ell}}{(u_1^{(-\tilde{t})})_{k_\ell}}\cdot [w_2(\delta {\ell})]^{\tilde{t}} \cdot [w_1(\delta k_\ell )]^{-\tilde{t}}\right| < \infty,
   \]
   such that Theorem \ref{thm:DecompEmbeddings}(2) yields $\Co(\Phi_{\vrho_1},\lebesgue_{\kappa_{\vrho_1,\vrho_2,-\tilde{t}}}^{p,q}) \hookrightarrow \Co(\Phi_{\vrho_2},\lebesgue_\kappa^{p,q})$. This finishes the proof.
\end{proof}

\begin{remark}
  Note that, in in the context of Theorem \ref{thm:EmbeddingForRadialComp}, the seemingly weaker assumption that $\Phi_{\vrho_1}$ be a $\mathcal C^\infty$-diffeomorphism and a $(d+1)$-admissible warping function is, in fact, equivalent to $(\Phi_{\vrho_1},\kappa)$ being a compatible pair: Under the conditions of the theorem, we can apply Lemma \ref{lem:AbstractWarpingSuffCond}, which then yields compatibility of $\Phi_{\vrho_1}$ and $\kappa$.
\end{remark}

  \section{Embeddings between Besov spaces and warped coorbit spaces}
\label{sec:WarpedCoorbitAndBesov}

The homogeneous and inhomogeneous Besov spaces~\cite{besov59,triebel2010theory}, although they have historically been defined in various ways, are prototypical decomposition spaces. Specifically, they can be described as decomposition spaces with respect to a radial, dyadic covering and exponential weights. In particular, for the \emph{inhomogeneous} Besov spaces $B^{p,q}_s(\RR^d)$, we have the equality $B_s^{p,q} (\RR^d) = \DecompSp (\CalB, \lebesgue^p, \ell_{v^s}^q)$, with the weight $v^s = (v^s_j)_{j\in\NN_0}$, $v^s_j = 2^{js}$ and the \emph{inhomogeneous Besov covering} 
$\CalB = (B_j)_{j \in \NN_0}$ of $\RR^d$, given by
  \begin{equation}
    B_0 := \{\xi \in \RR^d \,:\, |\xi| < 2 \}
    \quad \text{and} \quad
    B_j := \{\xi \in \RR^d \,:\, 2^{j-1} < |\xi| < 2^{j+1} \}
    \quad \text{for } j \in \NN \, , 
    \label{eq:BesovCoveringDefinition}
\end{equation}
see also \cite[Definition 2.2.1]{Grafakos2009Modern}. In particular, $v^s$ is $\CalB$-moderate and the covering $\CalB$ is semi-structured and tight~\cite[Lemma 9.10]{voigtlaender2016embeddings}.

We now proceed to investigate how warped coorbit spaces relate to inhomogeneous Besov spaces. In particular, we show that, for any admissible radial component $\vrho$, the $\Phi_\vrho$-induced $(\delta,r)$-fine frequency covering is almost subordinate to the standard inhomogeneous Besov covering $\CalB$ just defined. However, we also show that the converse can only be true in dimension $d=1$, even if $\Phi$ is not a radial warping function. Nonetheless, we obtain general conditions for embeddings between $\Co (\Phi_\vrho, \lebesgue_\kappa^{p_1,q_1})$ and $B_{s}^{p_2, q_2}(\RR^d)$, for specific weights $\kappa$ depending on $s$. Specifically, if $\vrho_1$ is any slow start version of the weakly admissible radial component $\vsig_{1} = \ln(1+\bullet)$, we show equality of $\Co (\Phi_\vrho, \lebesgue_\kappa^{p,q})$ and $B_{s}^{p, q}(\RR^d)$ in dimension $d=1$, for a certain weight $\kappa$. In dimension $d\geq 2$, we obtain a \emph{closeness} relation, in the sense that $B_{s}^{p, q}(\RR^d)$ can be nested between two warped coorbit spaces with respect to $\Phi_{\vrho_1}$ that only differ by a minor modification of the weight $\kappa$. 

The following result shows that the covering $\CalQ_{\Phi_\vrho}^{(\delta,r)}$ induced by $\Phi_\vrho$ is almost subordinate to, and even relatively moderate with respect to, the inhomogeneous Besov covering.

\begin{proposition}\label{prop:RadialCoveringBesovSubordinateness}
  Let $\vrho : \R \to \R$ be a $0$-admissible radial component and let $\CalB$ be the inhomogeneous Besov covering.
  Then the $(\delta,r)$-fine frequency covering $\CalQ_{\Phi_\vrho}^{(\delta,r)}$
  induced by $\Phi_\vrho$ is almost subordinate to $\CalB$,
  for arbitrary $\delta>0$, $r > \sqrt{d} / 2$.
  Specifically, we have the properties  
  \begin{equation}
    \Big|
      j - \log_2 \big(1 + \vrho_\ast (\delta |k|) \big)
    \Big|
    \lesssim 1
    \quad \text{if} \quad
    B_j \cap Q_{\Phi_\vrho,k}^{(\delta,r)} \neq \emptyset \, ,
    \label{eq:RadialWarpingCoveringBesovSubordinateness}
  \end{equation}
  as well as
  \begin{equation}
    \widetilde{\vrho_\ast} (\delta |k|)
    \asymp \frac{2^j}{\vrho (2^j)}
    \quad \forall \, j \in \NN_0 \text{ and } k \in \ZZ^d
                  \text{ with }
                  B_j \cap Q_{\Phi_\vrho, k}^{(\delta,r)} \neq \emptyset \, .
    \label{eq:RadialWarpingCoveringRhoTildeStarModerateness}
  \end{equation}

  \medskip{}

  Furthermore, if $\vrho$ is $C^2$, and there are $R, C > 0$ such that
  \begin{equation}
    \left|
      \frac{\vrho_\ast (\gamma) \cdot \vrho_\ast ''(\gamma)}
           {[\vrho_\ast '(\gamma)]^2}
    \right|
    \leq C < \infty
    \quad \forall \, \gamma \in [R, \infty) \, ,
    \label{eq:RadialComponentBesovModeratenessCondition}
  \end{equation}
  then $\CalQ_{\Phi_\vrho}^{(\delta,r)}$ is relatively $\CalB$-moderate, with
  \begin{equation}
    w(\delta k) \asymp \vrho_\ast ' \big( \vrho (2^j) \big)
                       \cdot \left(
                               2^j \big/ \vrho (2^j)
                             \right)^{d-1}
    \quad \text{whenever } j\in\NN_0 \text{ and } k\in\ZZ^d \text{ with }
    B_j \cap Q_{\Phi_\vrho,k}^{(\delta,r)} \neq \emptyset \, .
    \label{eq:RadialWarpingCoveringRelativeModerateness}
  \end{equation}
\end{proposition}

\begin{remark}\label{rem:RadialWarpingDecompositionSpaceWeightRelativeModerateness}
  If Condition \eqref{eq:RadialComponentBesovModeratenessCondition} is satisfied,
  and if the weight $\kappa : \RR^d \to \RR^+$ is given by
  $\kappa(\xi) = (1 + |\xi|)^s$ for some $s \in \RR$, then
  the associated weight $u^{(q)}$ defined in \eqref{eq:DecompositionSpaceWeight}
  is relatively $\CalB$-moderate. Indeed, if $\emptyset \neq B_j \cap Q_{\Phi_\vrho, k}^{(\delta,r)} \ni \xi$,
  then $1 + |\xi| \asymp 2^j$,
  so that Equation~\eqref{eq:RadialWarpingsCoveringNormControl} shows
  \[
    \kappa \big( \Phi_\vrho^{-1} (\delta k) \big)
    = \kappa \big( \Phi_{\vrho_\ast} (\delta k) \big)
    = (1 + |\Phi_{\vrho_\ast} (\delta k)|)^s
    = \big( 1 + \vrho_\ast (\delta |k|) \big)^s
    \asymp (1 + |\xi|)^s \asymp 2^{s j} \, .
  \]
  In combination with Equation~\eqref{eq:RadialWarpingCoveringRelativeModerateness},
  we thus see
  \begin{equation}
    u_k^{(q)}
    = \kappa \big( \Phi_\vrho^{-1} (\delta k) \big)
      \cdot [w(\delta k)]^{\frac{1}{q} - \frac{1}{2}}
    \asymp 2^{s j}
           \cdot \left[
                   \vrho_\ast ' \big( \vrho (2^j) \big)
                   \cdot \left(
                           2^j \big/ \vrho(2^j)
                         \right)^{d-1}
                 \right]^{\frac{1}{q} - \frac{1}{2}}
    \quad \text{if } B_j \cap Q_{\Phi_\vrho, k}^{(\delta,r)} \neq \emptyset \, ,
    \label{eq:RadialWarpingDecompositionSpaceWeightRelativeModerateness}
  \end{equation}
  which implies in particular that $u^{(q)}$ is relatively $\CalB$-moderate.
\end{remark}

\begin{proof} 
Let $v : \RR \to \RR^+$ be a control weight for $\vrho$. 
  To prove Equation~\eqref{eq:RadialWarpingCoveringBesovSubordinateness},
  fix $k \in \ZZ^d$, and let $j \in \NN_0$ be arbitrary. If $\emptyset \neq B_j \cap Q_{\Phi_\vrho, k}^{(\delta,r)} \ni \xi$, then $1 + |\xi| \asymp 2^j$ follows directly from the definition of $B_j$. In combination with \eqref{eq:RadialWarpingsCoveringNormControl}, we thus see
  that there are $0 < c_0 \leq 1 \leq c_1$ (independent of $j,k,\xi$) such
  that $c_0 \cdot \big(1 + \vrho_\ast (\delta |k|) \big)
        \leq 2^j
        \leq c_1 \cdot \big(1 + \vrho_\ast (\delta |k|) \big)$. Hence,
  $\big|j - \log_2 (1 + \vrho_\ast (\delta |k|)) \big|
   \leq \max \{ \log_2 (c_1), \log_2 (c_0^{-1}) \}$, such that
  \eqref{eq:RadialWarpingCoveringBesovSubordinateness} is satisfied.

  In particular, there are no more than $1 + 2 \max \{ \log_2 (c_1), \log_2 (c_0^{-1}) \}$
  nonnegative integers $j \in \NN_0$ which satisfy this estimate. We thus see that
  \[
    \sup_{k \in \ZZ^d}
      |
       \{
         j \in \NN_0
         \,:\,
         B_j \cap Q_{\Phi_\vrho, k}^{(\delta,r)} \neq \emptyset
       \}
      |
    \leq 1 + 2 \max \{ \log_2 (c_1), \log_2 (c_0^{-1}) \}
    < \infty \, ,
  \]
  and $\CalQ_{\Phi_\vrho}^{(\delta,r)}$ is weakly subordinate to $\CalB$.
  Since the elements of  $\CalQ_{\Phi_\vrho}^{(\delta,r)}$ are path-connected, Lemma~\ref{lem:almost_subordinateness_connected} shows that $\CalQ_{\Phi_\vrho}^{(\delta,r)}$ is almost subordinate to $\CalB$.

  \medskip{}

  Our proof of Equation~\eqref{eq:RadialWarpingCoveringRhoTildeStarModerateness} distinguishes two cases:  Since $\CalQ_{\Phi_\vrho}^{(\delta,r)}$ is weakly subordinate to $\CalB$, the set
  $\{
         (j,k) \in \NN_0\times \ZZ^d
         \,:\,
         |k|\leq 2r \text{ and } B_j \cap Q_{\Phi_\vrho, k}^{(\delta,r)} \neq \emptyset
  \}$ is finite. Since $\widetilde{\vrho_\ast}$ and $2^{\bullet} \big/ \vrho(2^{\bullet})$ are positive functions,
  $\widetilde{\vrho_\ast} (\delta |k|) \asymp 2^j \big/ \vrho(2^j)$ for all $(j,k)$ in the above set. Note that Equation \eqref{eq:RadialWarpingNormControlStep1} implies that the set $\{
         k \in \ZZ^d
         \,:\,
         B_0 \cap Q_{\Phi_\vrho, k}^{(\delta,r)} \neq \emptyset
  \}$ is likewise finite, yielding $\widetilde{\vrho_\ast} (\delta |k|) \asymp 2^0 \big/ \vrho(2^0)$ for all such $k$.

  Finally, we consider the case $|k|> 2r$ and $j\neq 0$. To this end, note that the map $h : \RR \to \RR, \tau \mapsto \ln (\vrho (e^\tau))$
  is well-defined and strictly increasing, since $\vrho$ is strictly increasing
  with $\vrho (0) = 0$. Setting $\gamma := \vrho (e^\tau) > 0$, we see that the derivative
  of $h$ satisfies
  \[
    0 \leq h'(\tau)
    = \frac{1}{\vrho (e^\tau)} \cdot \vrho' (e^\tau) \cdot e^\tau
    = \frac{1}{\gamma}
      \cdot \frac{1}{\vrho_\ast ' (\gamma)}
      \cdot \vrho_\ast (\gamma)
    = \frac{\widetilde{\vrho_\ast} (\gamma)}
           {\vrho_\ast ' (\gamma)}
    \leq C_0^{-1} \, ,
  \]
  where we used $\vrho' = [\vrho_\ast'\circ \vrho]^{-1}$ by the inverse function theorem and, since $\vrho$ is $0$-admissible, that $\vrho_\ast '(\gamma) \geq C_0 \cdot \widetilde{\vrho_\ast} (\gamma)$, cf. Equation~\eqref{eq:AdmissibilityFirstDerivative}. The above estimate implies in particular that
  $|h (\tau) - h(\sigma)| \leq C_0^{-1} \cdot |\tau - \sigma|$
  for all $\tau,\sigma \in \RR$.
  Let $\xi\in B_j\cap Q^{\delta,r}_{\Phi_\vrho,k}$ and set $\tau := \ln (|\xi|)$ and $\sigma := \ln(2^j)$, such that
  $|\tau - \sigma| \leq \ln(2)$, since $2^{-1} \leq |\xi| / 2^j \leq 2$.
  Therefore,
  \begin{equation}\label{eq:ModeratenessRhoOnBesovElements}
    C_0^{-1} \ln(2)
    \geq |h (\tau) - h(\sigma)|
    =    \big|
           \ln \big( \vrho(|\xi|) \big) - \ln \big( \vrho(2^j) \big)
         \big|
    = \ln \bigg(
            \max \bigg\{
                   \frac{\vrho(|\xi|)}{\vrho(2^j)},
                   \frac{\vrho(2^j)}{\vrho(|\xi|)}
                 \bigg\}
          \bigg) \, ,
  \end{equation}
  and hence $\vrho(2^j)\asymp \vrho (|\xi|) \asymp \delta |k|$, where the latter equivalence follows from Equation \eqref{eq:RadialWarpingNormControlStep1}, in the proof of Proposition \ref{prop:RadialCoveringSubordinateness}, with $|k|\geq 2r$.
 Since $k,j \neq 0$ and $|\xi| \geq 2^{j-1} \geq 1$, we have $1 + \vrho_\ast (\delta |k|) \asymp \vrho_\ast (\delta |k|)$ and
 $1 + |\xi| \asymp |\xi|$. Therefore, Equation~\eqref{eq:RadialWarpingsCoveringNormControl} implies $\vrho_\ast (\delta |k|) \asymp |\xi| \asymp 2^j$. Together with the prior estimate, we obtain \eqref{eq:RadialWarpingCoveringRhoTildeStarModerateness}.

  \medskip{}

  To prove the last part of the proposition, let us assume that
  \eqref{eq:RadialComponentBesovModeratenessCondition} is satisfied.
  Note that $\vrho_\ast '$ is a \emph{positive}, continuous function,
  since $\vrho$ is a strictly increasing diffeomorphism.
  Therefore, the left-hand side of
  \eqref{eq:RadialComponentBesovModeratenessCondition} is a continuous function,
  so that \eqref{eq:RadialComponentBesovModeratenessCondition} is in fact
  satisfied for all $\gamma \in [0,\infty)$,
  after possibly enlarging the constant $C$.
  In the following, we will assume that $C$ is chosen in this way.

  Now, define $f : \RR \to \RR^+, t \mapsto \vrho_\ast ' (\vrho (t))$
  and $g : \RR \to \RR, \tau \mapsto \ln(f(e^\tau))$.
  Setting $\gamma := \vrho(e^\tau) > 0$,
  we obtain, with \eqref{eq:RadialComponentBesovModeratenessCondition},
  \[
    |g'(\tau)|
    = \left|\frac{1}{f(e^\tau)} \cdot f'(e^\tau) \cdot e^\tau\right|
    = \left|\frac{e^\tau
            \cdot \vrho_\ast '' (\vrho (e^\tau))
            \cdot \vrho' (e^\tau)}
           {\vrho_\ast ' (\vrho (e^\tau))}\right|
    = \left|\frac{\vrho_\ast (\gamma) \cdot \vrho_\ast ''(\gamma)}
           {[\vrho_\ast ' (\gamma)]^2}\right| \leq C < \infty,\quad \text{for all } \tau\in\RR,
  \]
  where we used $\vrho' = [\vrho_\ast'\circ \vrho]^{-1}$. We conclude that
  \begin{equation}\label{eq:gIsLipschitz}
  |g(\tau) - g(\sigma)| \leq C \cdot |\tau - \sigma|,\quad \text{for all } \tau, \sigma \in \RR.
  \end{equation}

  Now, let $j \in \NN$ and $k \in \ZZ^d$ satisfy
  $\emptyset \neq B_j \cap Q_{\Phi_\vrho, k}^{(\delta,r)} \ni \xi$,
  and recall from \eqref{eq:RadialWarpingNormControlStep1} that
  $\big| \, \vrho(|\xi|) - \delta |k| \, \big| \leq \delta r$.
  Since $\vrho_\ast '$ is $v$-moderate, and since $v$ is radially increasing,
  this implies
  \[
    \vrho_\ast ' (\delta |k|)
    = \vrho_\ast ' \big( \delta |k| - \vrho(|\xi|) + \vrho(|\xi|) \big)
    \leq v(\delta r) \cdot \vrho_\ast ' (\vrho (|\xi|))
    =    v (\delta r) \cdot f (|\xi|) \, .
  \]
  Likewise, $f (|\xi|) \leq v(\delta r) \cdot \vrho_\ast ' (\delta |k|)$. Now,
  $f (|\xi|) \asymp f (2^j)$ and consequently $\vrho_\ast ' (\delta |k|) \asymp
  \vrho_\ast ' (\vrho (2^j))$, for all $j \in \NN$, $k \in \ZZ^d$ with $
  \emptyset \neq B_j \cap Q_{\Phi_\vrho, k}^{(\delta,r)}$
  is shown analogous to \eqref{eq:ModeratenessRhoOnBesovElements}, with an application of \eqref{eq:gIsLipschitz}. 
  Combining this with Equations~\eqref{eq:RadialWarpingWeightExplicit} and
  \eqref{eq:RadialWarpingCoveringRhoTildeStarModerateness}, we thus see
  \[
    w(\delta k)
    = \vrho_\ast ' (\delta |k|)
      \cdot [\widetilde{\vrho_\ast} (\delta |k|)]^{d-1}
    \asymp \vrho_\ast ' (\vrho (2^j))
           \cdot \left( 2^j \big/ \vrho (2^j) \right)^{d-1}
  \]
  for $j \in \NN$ and $k \in \ZZ^d$ with
  $B_j \cap Q_{\Phi_\vrho, k}^{(\delta,r)} \neq \emptyset$,
  proving \eqref{eq:RadialWarpingCoveringRelativeModerateness}
  for positive $j \in \NN$. However, if $j = 0$, then we recall that there are only finitely many $k\in\ZZ^d$ such that $B_0 \cap Q_{\Phi_\vrho, k}^{(\delta,r)} \neq \emptyset$. Since both sides of \eqref{eq:RadialWarpingCoveringRelativeModerateness} are positive, equivalence is trivial.
\end{proof}

We just saw that the covering $\CalQ_{\Phi_\vrho}^{(\delta,r)}$ is always
almost subordinate to the inhomogeneous Besov covering $\CalB$.
Our next result provides a straightforward, sufficient condition for the converse to be true, but only in the 
one-dimensional case $d=1$. In Theorem \ref{thm:LnWarpingBesovEmbeddingsExplicit},  we will construct an example of an admissible radial component which satisfies this condition. The situation is radically different in higher dimensions: In Theorem \ref{thm:WarpingNoBesovEquivalence} below, we show that $\CalB$ is never weakly subordinate to $\CalQ_{\Phi}^{(\delta,r)}$, if $d \geq 2$. This holds regardless of $\Phi$ being radial or not.

\begin{lemma}\label{lem:RadialWarpingBesovEquivalence}
  Let $d = 1$, let $r > 1/2$,
  let $\vrho : \RR \to \RR$ be a $0$-admissible radial component,
  and assume that $\vrho$ is $C^2$ and satisfies
  Equation~\eqref{eq:RadialComponentBesovModeratenessCondition}.
  If
  \begin{equation}
    2^j \lesssim \vrho_\ast ' (\vrho (2^j)) = [\vrho'(2^j)]^{-1}
    \qquad \forall \, j \in \NN_0 \, ,
    \label{eq:RadialWarpingBesovEquivalenceCondition}
  \end{equation}
  then the inhomogeneous Besov covering $\CalB = (B_j)_{j \in \NN_0}$ defined
  in \eqref{eq:BesovCoveringDefinition} is weakly subordinate to
  $\CalQ_{\Phi_\vrho}^{(\delta,r)}$.
\end{lemma}
\begin{proof}
  Proposition~\ref{prop:RadialCoveringBesovSubordinateness} shows that there is
  some $N \in \NN$ such that, for each $k \in \ZZ$,
  there is some $j_k \in \NN_0$ satisfying
  $Q_{\Phi_\vrho, k}^{(\delta,r)} \subset B_{j_k^{N\ast}}$.
  For $j \in \NN_0$, set
  \[
    I_j := \big\{
             k \in \ZZ
             \,:\,
             Q_{\Phi_\vrho, k}^{(\delta,r)} \cap B_j \neq \emptyset
           \big\} \, .
  \]
  Now, if $k \in I_j$, then
  $\emptyset \neq B_j \cap Q_{\Phi_\vrho, k}^{(\delta,r)}
             \subset B_j \cap B_{j_k^{N\ast}}$,
  and hence $j_k \in j^{(N+1)\ast}$,
  which implies $j_k^{N\ast} \subset j^{(2N+1)\ast}$, and finally
  \begin{equation}
    Q_{\Phi_\vrho, k}^{(\delta,r)}
    \subset B_{j_k^{N\ast}}
    \subset B_{j^{(2N+1)\ast}}
    \qquad \forall \, j \in \NN_0 \text{ and } k \in I_j \, .
    \label{eq:RadialWarpingBesovSubordinatenessStrengthened}
  \end{equation}
  Since $\CalQ_{\Phi_\vrho}^{(\delta,r)}$ is an admissible covering,
  there is some $K \in \NN$ such that
  $\sum_{k \in \ZZ}
     \Indicator_{Q_{\Phi_\vrho, k}^{(\delta,r)}} \leq K$.
  In view of Equation~\eqref{eq:RadialWarpingBesovSubordinatenessStrengthened},
  this shows $\sum_{k \in I_j} \Indicator_{Q_{\Phi_\vrho, k}^{(\delta,r)}}
              \leq K \cdot \Indicator_{B_{j^{(2N+1)\ast}}}$.
  Thus, noting that the factor $(2^j/\vrho(2^j))^{d-1}$ in Equation \eqref{eq:RadialWarpingCoveringRelativeModerateness} vanishes for $d=1$, we obtain 
  \[
   \begin{split}
    |I_j|
    \cdot \vrho_\ast' (\vrho(2^j)) & \overset{\eqref{eq:RadialWarpingCoveringRelativeModerateness}}{\asymp} \sum_{k \in I_j} w(\delta k) \overset{\text{ Lem. } \ref{lem:DecompInducedCoveringIsNice}}{\asymp} \sum_{k \in I_j}
             \mu \big( Q_{\Phi_\vrho, k}^{(\delta,r)} \big)\\
    & \overset{\text{(MCT)}}{=} \int_{\RR}
        \sum_{k \in I_j}
          \Indicator_{Q_{\Phi_\vrho, k}^{(\delta,r)}} (\xi)
      \, d \xi
    \leq K \cdot \mu \big( B_{j^{(2N+1)\ast}} \big)
    \lesssim 2^{j} \, ,
    \end{split}
  \]
  where (MCT) indicates an application of the monotone convergence theorem. Now,
  apply \eqref{eq:RadialWarpingBesovEquivalenceCondition} to show that $\CalB$ is weakly subordinate to $\CalQ_{\Phi_\vrho}^{(\delta,r)}$. Specifially, we obtain $|I_j| \lesssim 2^j \big/ \vrho_\ast ' (\vrho(2^j)) \lesssim 1$.
\end{proof}

The next result shows that Lemma \ref{lem:RadialWarpingBesovEquivalence} cannot be extended to dimensions $d\geq 2$.

 \begin{theorem}\label{thm:WarpingNoBesovEquivalence}
   Let $2\leq d \in \NN$, $D\subset \RR^d$ open, and $\Phi\colon D\rightarrow \RR^d$ a diffeomorphism. The $(\delta,r)$-fine frequency covering $\CalQ^{(\delta,r)}_{\Phi}$ induced by $\Phi$, see Definition \ref{def:frequency_space_covering}, is not weakly equivalent to $\CalB$, for any choice of $\delta>0$ and $r>\sqrt{d}/2$. 
 \end{theorem}
 
 Before we prove Theorem \ref{thm:WarpingNoBesovEquivalence}, we require some auxiliary results, estimating the number of $n$-th order neighbors for $\CalQ^{(\delta,r)}_{\Phi}$ and $\CalB$, see also Remark \ref{rem:higher_cluster_sets}. 
 
 \begin{lemma}\label{lem:EstimateOfNthOrderNeighbors}
    Let $2\leq d \in \NN$, $D\subset \RR^d$ open, and $\Phi\colon D\rightarrow \RR^d$ a diffeomorphism. Then, for any $\delta>0$ and $r>\sqrt{d}/2$, the $(\delta,r)$-fine frequency covering $\CalQ^{(\delta,r)}_{\Phi}= (Q^{(\delta,r)}_{\Phi,k})_{k\in\ZZ^d}$ induced by $\Phi$ satisfies 
    \begin{equation}\label{eq:lowerNeighborBoundForWarpedCover}
      k^{n\ast} \supset \{\ell\in\ZZ^d ~:~  \|k-\ell\|_{\ell^{\infty}} \leq n\},\quad \text{such that } |k^{n\ast}|\geq (1+2n)^d,\quad \text{for all } k\in\ZZ^d,\ n\in\NN_0.
    \end{equation}
    
    \medskip{}
    
    Moreover, the inhomogeneous Besov covering  $\CalB = (B_j)_{j\in\NN_0}$ satisfies  
    \begin{equation}\label{eq:upperNeighborBoundForBesovCover}
       j^{n\ast} \subset \{\ell\in\NN_0 ~:~ |j-\ell| \leq n\}, \quad \text{such that } |j^{n\ast}|\leq 1+2n,\quad \text{for all } j, n\in\NN_0.
    \end{equation}
 \end{lemma}
 \begin{proof}
   We first prove the upper bound for $|j^{n\ast}|$, i.e., Equation \eqref{eq:upperNeighborBoundForBesovCover}. If $\xi\in B_j\cap B_\ell\neq \emptyset$, with $j,\ell\in\NN$, then $2^{\ell-1}<|\xi|<2^{j+1}$, and thus $\ell-1<j+1$. implying $\ell\leq j+1$. Repeat the argument with interchanged roles of $j,\ell$ to obtain $|j-\ell|\leq 1$. Further, if $j=0$ instead, then $2^{\ell-1}<|\xi|< 2 = 2^1$, similarly yielding $|0-\ell|\leq 1$. Overall, 
   \[
    j^{\ast} \subset \{\ell\in\NN_0 ~:~ |j-\ell| \leq 1\}, \quad \text{such that } |j^{\ast}|\leq 3,\quad \text{for all } j\in\NN_0.
   \]
   Equation \eqref{eq:upperNeighborBoundForBesovCover} now follows by a straightforward induction. 
   
   \medskip{}
   
   The proof of Equation \eqref{eq:lowerNeighborBoundForWarpedCover} is slightly more involved. Note that $r>\sqrt{d}/2$ implies
   \[
    P_k\subset B_r(k),\quad \text{where}\quad P_k = k + [-1/2,1/2]^d.
   \]
   Therefore $P_k \cap P_\ell \neq \emptyset$ implies $Q^{(\delta,r)}_{\Phi,k}\cap Q^{(\delta,r)}_{\Phi,\ell} = \Phi^{-1}(\delta\cdot (B_r(k)\cap B_r(\ell))) \neq \emptyset$. Hence, it suffices to prove \eqref{eq:lowerNeighborBoundForWarpedCover} for $\CalP = (P_k)_{k\in\ZZ^d}$ in place of $\CalQ^{(\delta,r)}_{\Phi}$.
   
   Now define, for $k,\ell\in\ZZ^d$, $\eps:= k-\ell$ and note that $\|k-\ell\|_{\ell^{\infty}}\leq 1$ implies $\eps\in [-1,1]^d$. In that case, we clearly have $k-\eps/2 = \ell+\eps/2 \in P_k\cap P_\ell$. Hence, 
   \[
    \ell\in k^{\ast},\quad \text{whenever } k,\ell\in\ZZ^d \text{ with } \|k-\ell\|_{\ell^{\infty}}\leq 1.
   \]
      
   Finally, assume that the first part of Equation \eqref{eq:lowerNeighborBoundForWarpedCover} holds for some $n\in\NN_0$ and let $k,\ell\in\ZZ^d$ with $\|k-\ell\|_{\ell^{\infty}} = n+1$. There exists a $\tilde{k}\in\ZZ^d$ with $\|k-\tilde{k}\|_{\ell^{\infty}}\leq n$ and $\|\tilde{k}-\ell\|_{\ell^{\infty}}=1$. Hence, $\tilde{k}\in k^{n\ast}$ and $\ell\in\tilde{k}^\ast$, such that $\ell\in k^{(n+1)\ast}$ follows. This proves the first part of Equation \eqref{eq:lowerNeighborBoundForWarpedCover} by induction. To prove the second part, simply note that $|\{\ell\in\ZZ^d ~:~ \|k-\ell\|_{\ell^{\infty}} \leq n \}| = (1+2n)^d$.
 \end{proof}

Our second auxiliary result provides an estimate for neighbor sets of arbitrary order, for two equivalent coverings.

\begin{lemma}\label{lem:RelativeEstimateOfSubordinateCoverNeighbors}
  Let $\CalQ = (Q_i)_{i\in I}$ and $\CalP = (P_j)_{j\in J}$ be two equivalent, admissible coverings of a common set $\mathcal O$, such that $Q_i$, $i\in I$, are nonempty. Then there are constants $M,N\in\NN$ such that 
  \begin{equation}\label{eq:RelativeEstimateOfSubordinateCoverNeighbors}
    |i^{n\ast}| \leq M\cdot |j^{N(n+1)\ast}|,
  \end{equation}
  for all $i\in I, j\in J$ such that $Q_i\cap P_j\neq \emptyset$, and arbitrary $n\in\NN_0$.
\end{lemma}
\begin{proof}
  Recall that $Q_L$, for some $L\subset I$, denotes the set $Q_L:=\bigcup_{i\in L} Q_i$ and $P_L$ for some $L\subset J$ is defined analogously. Since $\CalQ$ is almost subordinate to $\CalP$, there exists, by \cite[{Lemma 2.11 (3)}]{voigtlaender2016embeddings}, an  $N\in\NN$, such that $Q_i\subset P_{j^{N\ast}}$, for all $i\in I, j\in J$ such that $Q_i\cap P_j\neq \emptyset$. A straightforward induction shows that, in fact, $Q_{i^{n\ast}} \subset P_{j^{N(n+1)\ast}}$ holds for all $n\in\NN_0$, whenever $Q_i\cap P_j\neq \emptyset$. Denote by 
  $I_j$, $j\in J$, the sets
  \[
    I_j = \{ i\in I ~:~ Q_i\cap P_j \neq \emptyset \}.
  \]
  Since $\CalP$ is almost subordinate to $\CalQ$, it is, in particular, weakly subordinate to $\CalQ$ (\cite[{Lemma 2.11 (1)}]{voigtlaender2016embeddings}, see also Remark \ref{rem:AlmostSubImpliesWeakSub}), such that $M := \sup_{j\in J} |I_j| < \infty$. 
  
  Now, let $i\in I$ and $j\in J$ be arbitrary, with $P_j\cap Q_i \neq \emptyset$. Since $Q_{i^{n\ast}} \subset P_{j^{N(n+1)\ast}}$ holds, we also have $Q_k\subset P_{j^{N(n+1)\ast}}$ for all $k \in i^{n\ast}$, i.e., there is some $\ell\in j^{N(n+1)\ast}$ with $Q_k\cap P_\ell\neq \emptyset$. We conclude that $i^{n\ast} \subset \bigcup_{\ell\in j^{N(n+1)\ast}} I_\ell$. Therefore,
  \[
   |i^{n\ast}| \leq \Big| \bigcup_{\ell\in j^{N(n+1)\ast}} I_\ell \Big| \leq \sum_{\ell\in j^{N(n+1)\ast}} |I_\ell| \leq M\cdot |j^{N(n+1)\ast}|.\hfill\qedhere
  \]
\end{proof}

\begin{proof}[Proof of Theorem \ref{thm:WarpingNoBesovEquivalence}]
  If $\CalB$ and $\CalQ^{(\delta,r)}_{\Phi}$ are weakly equivalent, then they are equivalent, since their elements are path-connected non-empty. Applying Lemma \ref{lem:RelativeEstimateOfSubordinateCoverNeighbors} with $\CalQ = \CalQ^{(\delta,r)}_{\Phi}$ and $\CalP = \CalB$ and inserting the estimates in Lemma \ref{lem:EstimateOfNthOrderNeighbors}, 
  we obtain, for all $n\in\NN_0$,  
  \[
   (1+2n)^d \leq |k^{n\ast}| \leq M \cdot |j^{N(n+1)\ast}| \leq M \cdot (1+2N(n+1)),
  \]
  for some constants $M,N\in \NN$, and all $k\in \ZZ^d$, $j\in\NN_0$ with $\CalQ^{(\delta,r)}_{\Phi,k}\cap B_j\neq \emptyset$. Since $d\geq 2$, this is a contradiction (let $n\rightarrow\infty$). 
\end{proof}

As a consequence of Theorem \ref{thm:WarpingNoBesovEquivalence}, we obtain the following negative result on equality between warped coorbit spaces and Besov spaces.

\begin{theorem}\label{thm:WarpedCoorbitNoBesovEquality}
  Let $d\geq 2$, $p_1,p_2,q_1,q_2\in [1,\infty]$, and let $\Phi\colon \RR^d \rightarrow \RR^d$, $\kappa\colon \RR^d \rightarrow \RR^+$ be compatible in the sense of Definition \ref{def:StandingDecompositionAssumptions}. If $\Co(\Phi,\lebesgue^{p_1,q_1}_\kappa)$ coincides, up to canonical identifications, with an (inhomogeneous) Besov space $B^{p_2,q_2}_s(\RR^d)$
  , then $p_1=p_2=q_1=q_2=2$. 
\end{theorem}
\begin{proof}
  By Theorem \ref{thm:CoorbitDecompositionIsomorphism}, we have 
  \[
    \Co(\Phi,\lebesgue^{p_1.q_1}_\kappa) = \DecompSp(\CalQ^{(\delta,r)}_{\Phi},\lebesgue^{p_1},\ell^{q_1}_u),
  \]
  where $\CalQ^{(\delta,r)}_{\Phi}= (Q^{(\delta,r)}_{\Phi,k})_{k\in\ZZ^d}$, with arbitrary $\delta>0$ and $r>\sqrt{d}/2$, is the $\Phi$-induced $(\delta,r)$-fine frequency covering and $u = (u_k)_{k\in\ZZ^d}$ is a suitable $\CalQ^{(\delta,r)}_{\Phi}$-moderate weight. 
  Hence, the assumptions of the theorem imply 
  \[
   \DecompSp(\CalQ^{(\delta,r)}_{\Phi},\lebesgue^{p_1},\ell^{q_1}_u) = \Co(\Phi,\lebesgue^{p_1.q_1}_\kappa) = B^{p_2,q_2}_s(\RR^d)  = \DecompSp(\CalB,\lebesgue^{p_2},\ell^{q_2}_v),
  \]
   where $v = (2^{js})_{j\in\NN_0}$. Since $\CalQ^{(\delta,r)}_{\Phi}$ and $\CalB$ are not weakly equivalent by Theorem \ref{thm:WarpingNoBesovEquivalence}, \cite[Theorem 6.9, Items (1) and (3)]{voigtlaender2016embeddings} shows that $p_1=p_2=2$. Finally, \cite[Theorem 6.9, Item (4)]{voigtlaender2016embeddings} shows that $q_1=q_2=2$ and the proof is finished. 
\end{proof}
 
Now that we have determined the limits of identifying warped coorbit spaces with Besov spaces, we return to what is possible in the setting of radial warping functions $\Phi_{\vrho}$. The next lemma is the final step for preparing our embedding result, stated in Proposition \ref{prop:RadialWarpingBesovEmbedding} below. We show that the weight $\kappa(\xi) = (1 + |\xi|)^s$, together with any radial warping function $\Phi_\vrho$ derived from a $0$-admissible radial component $\vrho$, form a compatible pair, see Definition~\ref{def:StandingDecompositionAssumptions}.

\begin{lemma}\label{lem:RadialWarpingStandardWeightAdmissible}
  Let $\vrho : \R \to \R$ be a $0$-admissible radial component with control weight $v\colon \RR\rightarrow \RR^+$, 
  and let $\kappa : \RR^d \to \RR^+, \xi \mapsto (1 + |\xi|)^s$
  for some $s \in \RR$. Then, there exists a constant $C>0$, such that $\kappa\circ \Phi_{\vrho}^{-1}$ is $\tilde{\kappa}$-moderate, with $\tilde{\kappa} = \left(1 + C \cdot (1 + |\bullet|) \cdot v(\bullet) \right)^{|s|}$. In particular, if $\vrho\in\CalC^{\infty}(\RR)$ is $(d+1)$-admissible, then $\Phi_\vrho$ and $\kappa$ are compatible in the sense of Definition \ref{def:StandingDecompositionAssumptions}.  
\end{lemma}
\begin{proof}
  It is sufficient to show the result for $\kappa,\tilde{\kappa}$ with $s=1$, since all relevant 
  properties are retained if real powers $s\in\RR$ of the moderate weight, respectively nonnegative powers $|s|$ 
  of the moderating weight are taken.
  
  Let $v : \RR \to \RR^+$ denote a control weight for the radial component
  $\vrho$. In particular, this implies that $\widetilde{\vrho_\ast}$ is
  $v$-moderate, see Definition \ref{def:AdmissibleRho}. 
  We will show that $\kappa_{\Phi_\vrho} = \kappa\circ \Phi_{\vrho}^{-1}$ is $\tilde{\kappa}$-moderate for the weight
  \[
    \tilde{\kappa}\colon \RR^d \to \RR^+,\
        \xi \mapsto 1 + C \cdot (1 + |\xi|) \cdot v(|\xi|) \, ,\quad \text{with}\quad C := \max \{2, \max_{\tau_0 \in [-1,1]} \widetilde{\vrho_\ast} (\tau_0) \}, 
  \]
  which is easily seen to be radially increasing, continuous,
  and submultiplicative.

  \medskip{}

  Observe that $
    \kappa_{\Phi_\vrho} (\xi)
    = 1 + |\Phi_{\vrho_\ast} (\xi)|
    = 1 + \vrho_\ast (|\xi|)$, for all $\xi \in \RR^d$.
  Since $\vrho_\ast$ is increasing, and since $\widetilde{\vrho_\ast}$ is
  $v$-moderate, we have
  \[
    \kappa_{\Phi_\vrho} (\xi + \tau)
    = 1 + \vrho_\ast (|\xi + \tau|)
    \leq 1 + \vrho_\ast (|\xi| + |\tau|)
    =    1 + \widetilde{\vrho_\ast} (|\xi| + |\tau|) \cdot (|\xi| + |\tau|)
    \leq 1 + \big(
               v(|\xi|) \cdot \widetilde{\vrho_\ast} (|\tau|)
             \big)
             \cdot (|\xi| + |\tau|)
    =: \circledast \, .
  \]
  We now distinguish three cases:

  \noindent
  \textbf{Case 1:} If $|\tau| \leq 1$, then 
  \[
    \circledast
    \leq 1 + \max_{\tau_0 \in [-1,1]} \widetilde{\vrho_\ast} (\tau_0) \cdot (1 + |\xi|) \cdot v(|\xi|)
    \leq    \tilde{\kappa}(\xi)
    \leq \tilde{\kappa}(\xi) \cdot (1 + \vrho_\ast (|\tau|))
    =    \tilde{\kappa}(\xi) \cdot \kappa_{\Phi_\vrho} (\tau) \, .
  \]

  \noindent
  \textbf{Case 2:} If $|\xi| \leq |\tau|$ and $|\tau| > 1$, then 
  \[
    \circledast
    \leq 1 + v (|\xi|) \cdot \widetilde{\vrho_\ast} (|\tau|) \cdot 2 |\tau|
    =    1 + 2 v(|\xi|) \cdot \vrho_\ast (|\tau|)
    \leq \tilde{\kappa} (\xi) \cdot (1 + \vrho_\ast (|\tau|))
    =    \tilde{\kappa} (\xi) \cdot \kappa_{\Phi_\vrho} (\tau) \, .
  \]

  \noindent
  \textbf{Case 3:} If $1 < |\tau| \leq |\xi|$, then 
  $\widetilde{\vrho_\ast} (|\tau|)
   = \vrho_\ast (|\tau|) / |\tau|
   \leq \vrho_\ast (|\tau|)$, and 
  \[
    \circledast
    \leq 1 + 2|\xi| \cdot v(|\xi|) \cdot \vrho_\ast (|\tau|)
    \leq \tilde{\kappa} (\xi) \cdot (1 + \vrho_\ast (|\tau|))
    =    \tilde{\kappa}(\xi) \cdot \kappa_{\Phi_\vrho} (\tau) \, .
    \qedhere
  \]
\end{proof}

Finally, we can characterize the existence of embeddings of the warped
coorbit spaces $\Co(\Phi_\vrho, \lebesgue_\kappa^{p,q})$, with the 
weight $\kappa = (1 + |\bullet|)^s$ into inhomogeneous Besov spaces, and vice versa.

\begin{proposition}\label{prop:RadialWarpingBesovEmbedding}
  Let $\vrho : \RR \to \RR$ be $\CalC^\infty$ and a $(d+1)$-admissible radial component, with 
  \[
     \left|
      \tfrac{\vrho_\ast (\gamma) \cdot \vrho_\ast ''(\gamma)}
           {[\vrho_\ast '(\gamma)]^2}
    \right|
    \leq C < \infty
    \quad \text{for some } R>0\ \text{and}\ \gamma \in [R, \infty) \, .
  \]
  Furthermore, let $\kappa : \RR^d \to \RR^+, \xi \mapsto (1 + |\xi|)^{s_1}$
  for some $s_1 \in \RR$.
  
  \medskip{}

  Then the warped coorbit spaces
  $\Co (\Phi_\vrho, \lebesgue_\kappa^{p_1, q_1})$
  are well-defined and independent of the choice
  of $0\neq \theta \in \TestFunctionSpace (\RR^d) \setminus \{0\}$,
  for arbitrary $p_1, q_1 \in [1,\infty]$. We further have the following equivalences, for $p_1,p_2\in [1,\infty]$:
  \begin{enumerate}
   \item $\Co (\Phi_\vrho, \lebesgue_\kappa^{p_1, q_1})
          \hookrightarrow B_{s_2}^{p_2, q_2} (\RR^d)$ holds if and only if
  \begin{equation}
    p_1 \leq p_2\quad
    \text{and }
    \left\|
      \left(
        2^{j (s_2 - s_1 + d \tilde{t})}
        \cdot \left[
                \vrho_\ast ' (\vrho(2^j))
                \cdot \big( 2^j \big/ \vrho(2^j) \big)^{d-1}
              \right]^{p_1^{-1} - p_2^{-1} - \tilde{t} + 2^{-1} - q_1^{-1}}
      \right)_{j \in \NN_0}
    \right\|_{\ell^{q_2 \cdot (q_1 / q_2)'}} < \infty \, ,
    \label{eq:RadialWarpedCoorbitIntoBesov}
  \end{equation}
  where $\tilde{t} = \max(0,\max(p_2^{-1},1-p_2^{-1})-q_1^{-1})$.
  \item $B_{s_2}^{p_2, q_2} (\RR^d)
   \hookrightarrow \Co (\Phi_\vrho, \lebesgue_\kappa^{p_1, q_1})$
  holds if and only if
  \begin{equation}
    p_2 \leq p_1\quad
    \text{and }
    \left\|
      \left(
        2^{j (s_1 - s_2 + d t)}
        \cdot \left[
                \vrho_\ast ' (\vrho(2^j))
                \cdot \big( 2^j \big/ \vrho(2^j) \big)^{d-1}
              \right]^{p_2^{-1} - p_1^{-1} - t + q_1^{-1} - 2^{-1}}
      \right)_{j \in \NN_0}
    \right\|_{\ell^{q_1 \cdot (q_2 / q_1)'}} < \infty \, ,
    \label{eq:BesovIntoRadialWarpedCoorbit}
  \end{equation}
  where $t := \max(0,q_1^{-1} - \min(p_2^{-1}, 1-p_2^{-1}))$.
  \end{enumerate}
\end{proposition}

\begin{proof}
  By \cite[Lemma 9.10]{voigtlaender2016embeddings}, $\CalB$ is a tight semi-structured covering of $\RR^d$. 
  Further, Theorem \ref{cor:RadialWarpingIsWarping} shows that $\Phi_\vrho$ is a $(d+1)$-admissible warping function, such that Theorem~\ref{cor:warped_disc_frames} shows that the warped coorbit spaces
  $\Co( \Phi_\vrho, \lebesgue_\kappa^{p_1, q_1})$ are well-defined and independent of the choice of $0\neq \theta\in\TestFunctionSpace(\RR^d)$. Lemmas  \ref{lem:DecompInducedCoveringIsNice} and \ref{lem:DecompInducedCoveringIsTight} show that $\CalQ_{\Phi_\vrho}^{(\delta,r)}$ is semi-structured and tight. Finally, by Lemma~\ref{lem:RadialWarpingStandardWeightAdmissible}, the tuple $(\Phi_\vrho, \kappa)$ is a compatible pair as per Definition~\ref{def:StandingDecompositionAssumptions}, such that the assumptions of Theorem~\ref{thm:CoorbitDecompositionIsomorphism} are satisfied, and 
  $\Co( \Phi_\vrho, \lebesgue_\kappa^{p_1, q_1})
   = \DecompSp (\CalQ_{\Phi_\vrho}^{(\delta,r)}, \lebesgue^{p_1}, \ell_{u^{(q_1)}}^{q_1})$, with $\delta > 0$ and $r > \sqrt{d}/2$ arbitrary, up to canonical identifications. The weight $u^{(q_1)}$ is defined as in \eqref{eq:DecompositionSpaceWeight}, i.e.,
   \[
    u_k^{(q_1)} = \kappa_{\Phi_\vrho}(\delta k) \cdot [w(\delta k)]^{\frac{1}{q_1}-\frac{1}{2}}.
   \]
 
   Since, by assumption, there exist $C, R>0$, such that $\left|\tfrac{\vrho_\ast (\gamma) \cdot \vrho_\ast ''(\gamma)}{[\vrho_\ast '(\gamma)]^2}\right| \leq C$, for all $\gamma \in [R, \infty)$, Proposition~\ref{prop:RadialCoveringBesovSubordinateness} shows that $\CalQ_{\Phi_\vrho}^{(\delta,r)}$
   is almost subordinate to, and relatively moderate with respect to, $\CalB$ and, as discussed in Remark~\ref{rem:RadialWarpingDecompositionSpaceWeightRelativeModerateness}, $u^{(q_1)}$ is relatively $\CalB$-moderate as well. Hence, the assumptions of Theorem \ref{thm:DecompEmbeddings}, Items (1) and (2), are satisfied.
   
   Finally, together with Lemma~\ref{lem:DecompInducedCoveringIsNice}, which specifies that the operator $T_k$ in the semi-structured representation of $Q_{\Phi_\vrho,k}^{(\delta,r)}$ is given by $T_k = D\Phi_\vrho^{-1}(\delta k)$, we obtain 
   \[
    |\det T_k|
    = w(\delta k)
    \asymp \vrho_\ast ' (\vrho(2^j))
           \cdot \big( 2^j \big/ \vrho(2^j) \big)^{d-1}
    =: \gamma_j
    \quad \text{and} \quad
    u_k^{(q_1)}
    \asymp 2^{s_1 j} \cdot \gamma_j^{q_1^{-1} - 2^{-1}}\, ,
  \]
  if $k \in \ZZ^d$ and $j \in \NN_0$ satisfy
  $B_j \cap Q_{\Phi_\vrho, k}^{(\delta,r)} \neq \emptyset$. The equivalences \eqref{eq:RadialWarpedCoorbitIntoBesov} and \eqref{eq:BesovIntoRadialWarpedCoorbit} are now direct
  consequences of Theorem \ref{thm:DecompEmbeddings}, Items (2) and (1), respectively.
\end{proof}

Closing the discussion of Besov embeddings, we show that for $d=1$, and choosing the radial warping function derived from a certain weakly $k$-admissible radial component, $B^{p,q}_s(\RR)$ is recovered as a warped coorbit space. For $d\geq 2$, $B^{p,q}_s(\RR^d)$ is nested between two warped coorbit spaces which only differ by a minimal change of the weight $\kappa$.

\begin{theorem}\label{thm:LnWarpingBesovEmbeddingsExplicit}
   The function $\vsig_{1} : [0,\infty) \to [0,\infty), \xi \mapsto \ln(1+\xi)$ is a weakly $k$-admissible radial component for any $k\in\NN_0$. Assume that $1\leq p,q\leq \infty$ and let $\vrho_{1}:\RR\rightarrow\RR$ be any \emph{slow-start} version of $\vsig_{1}$, see Theorem \ref{thm:SlowStartFullCriterion}. Further, define for $s\in\RR$ and $d\in\NN$,
   \[
    \kappa^{(s,q)} := \kappa^{(d,s,q)} \colon \RR^d \rightarrow \RR^+,\quad \xi \mapsto (1+|\xi|)^{s+d(2^{-1}-q^{-1})}.
   \]
   We have
   \begin{equation}\label{eq:WarpedBesovCloseness}
      \Co (\Phi_{\vrho_{1}}, \lebesgue^{p,q}_{\kappa^{(s+\eps,q)}})
    \hookrightarrow
    B_s^{p,q} (\RR^d)
    \hookrightarrow
    \Co (\Phi_{\vrho_{1}}, \lebesgue^{p,q}_{\kappa^{(s-\eps,q)}})
    \qquad \forall \, \eps > 0 \, .
  \end{equation}
  If and only if $d=1$ or $p=q=2$, the equality 
  \begin{equation}\label{eq:DImOneWarpedBesovEquality}
  B_s^{p,q} (\RR^d)
  = \Co (\Phi_{\vrho_{1}}, \lebesgue^{p,q}_{\kappa^{(s,q)}})\, 
  \end{equation}
  holds. If $d\geq 2$ and $p\neq 2$, then neither of the two embeddings implied by \eqref{eq:DImOneWarpedBesovEquality} is valid. If $d\geq 2$, and $p=2$, then $B_s^{p,q} (\RR^d) \hookrightarrow \Co (\Phi_{\vrho_{1}}, \lebesgue^{p,q}_{\kappa^{(s,q)}})$ if and only if $q\leq 2$ and, likewise, $\Co (\Phi_{\vrho_{1}}, \lebesgue^{p,q}_{\kappa^{(s,q)}}) \hookrightarrow B_s^{p,q} (\RR^d)$ if and only if $q\geq 2$.
\end{theorem}
\begin{proof}
   It was already shown in \cite[Example 8.16]{VoHoPreprint} that $\vsig_1$ is weakly $k$-admissible, for any $k\in\NN_0$.
   Hence, we conclude that $\Phi_\vrho$ is $(d+1)$-admissible for any $d\in\NN$, by Theorems \ref{thm:SlowStartFullCriterion} and \ref{cor:RadialWarpingIsWarping}. Before we can apply Proposition \ref{prop:RadialWarpingBesovEmbedding}, it remains to show that there exist $C, R>0$, such that $\left|\tfrac{\vrho_\ast (\gamma) \cdot \vrho_\ast ''(\gamma)}{[\vrho_\ast '(\gamma)]^2}\right| \leq C$, for all $\gamma \in [R, \infty)$. However, by construction, see Theorem \ref{thm:SlowStartFullCriterion}, there is some $R > 0$ such that $\vrho_1(\xi) = \vsig_1(\xi)$ for all $\xi \in [R,\infty)$ and, possibly after increasing $R$, $(\vrho_1)_\ast (\xi) = (\vsig_1)_\ast(\xi) = e^{\xi} - 1$ as well.
   Consequently, we obtain the desired bound  
   \[
  \Big|
    \frac{(\vrho_1)_\ast (\gamma) \cdot (\vrho_1)_\ast ''(\gamma) }
         {[(\vrho_1)_\ast ' (\gamma)]^2}
  \Big|
= \frac{(e^\gamma - 1) e^\gamma}{e^{2 \gamma}}
  \leq 1
  \qquad \forall \, \gamma \in [R,\infty) \, .
\]

Given that $(\vrho_1)_\ast '(\xi) = e^{|\xi|}$ for all $\xi\in\RR\setminus (-R,R)$, a short calculation show that 
\begin{equation}
  2^j \asymp (\vrho_1)_\ast ' (\vrho_1(2^j))
  \qquad \forall \, j \in \NN_0 \, .
  \label{eq:BesovWarpingEquivalenceCondition}
\end{equation}

We are now ready to apply Proposition \ref{prop:RadialWarpingBesovEmbedding}, where due to the chosen weight $\kappa^{(s,q)}$ we have $s_1 + d(2^{-1}-q^{-1})$ in place of $s_1$. By Equation \eqref{eq:BesovWarpingEquivalenceCondition}, and additionally noting that $\vrho_1(2^j) \asymp \vsig_1 (2^j) \asymp 1 + j$,
for all $j\in\NN_0$, we can rewrite the norm expression in Equation \eqref{eq:RadialWarpedCoorbitIntoBesov}, with $p_1=p_2=p$ and $q_1=q_2=q$, as
\begin{equation}\label{eq:CoorbitIntoBesovSimple}
  \begin{split}
    K_{\tilde{t}} & =\left\|
      \left(
        2^{j (s_2 - (s_1 + d (2^{-1} - q^{-1})) + d \tilde{t})}
        \cdot \left[
                (\vrho_1)_\ast ' (\vrho_1(2^j))
                \cdot \big( 2^j \big/ \vrho_1(2^j) \big)^{d-1}
              \right]^{2^{-1} - q^{-1} - \tilde{t}}
      \right)_{j \in \NN_0}
    \right\|_{\ell^{\infty}}\\
    &  \asymp  \left\|
      \left(
        2^{j (s_2 - s_1)}
        \big/ (1+j)^{(d-1) ( - \tilde{t} + 2^{-1} -q^{-1})}
      \right)_{j \in \NN_0}
    \right\|_{\ell^{\infty}} =: (*),
  \end{split}
\end{equation}
with $\tilde{t} = \max(0,\max(p^{-1},1-p^{-1})-q^{-1})$. Similarly, 
the norm expression in Equation \eqref{eq:BesovIntoRadialWarpedCoorbit} simplifies to 
\begin{equation}\label{eq:BesovIntoCoorbitSimple}\begin{split}
      K_t & = \left\|
      \left(
        2^{j [ (s_1 + d (2^{-1} - q^{-1})) - s_2 + d t]}
        \cdot \left[
                (\vrho_1)_\ast ' (\vrho_1(2^j))
                \cdot \big( 2^j \big/ \vrho_1(2^j) \big)^{d-1}
                  \right]^{q^{-1} - 2^{-1} - t}
      \right)_{j \in \NN_0}
    \right\|_{\ell^{\infty}} \\
    & \asymp \left\|
      \left(
        2^{j (s_1 - s_2)}
        \big/ (1 + j)^{(d-1) (q^{-1} - 2^{-1} - t)}
      \right)_{j \in \NN_0}
    \right\|_{\ell^{\infty}} =: (\star),
  \end{split}
 \end{equation}
 with $t := \max(0,q^{-1} - \min(p^{-1}, 1-p^{-1}))$. In the setting of this Theorem, we consider 
 $s_2 = s$ and $s_1\in \{s-\eps,s,s+\eps\}$. It is easily seen that $(\ast) < \infty$ if $s = s_2 < s_1 = s+\eps$ and likewise $(\star) < \infty$ if $s = s_2 > s_1 = s-\eps$, such that the closeness relation \eqref{eq:WarpedBesovCloseness} follows from Proposition \ref{prop:RadialWarpingBesovEmbedding}.
 
 If $s_1=s_2 = s$, then 
 \[
   (\ast) = \left\|
      \left(
         (1+j)^{(d-1) (\tilde{t} - 2^{-1} + q^{-1})}
      \right)_{j \in \NN_0}
    \right\|_{\ell^{\infty}},\quad \text{and}\quad (\star) =  \left\|
      \left(
        (1 + j)^{(d-1) (t - q^{-1} + 2^{-1})}
      \right)_{j \in \NN_0}
    \right\|_{\ell^{\infty}}
 \]
 which are finite if and only if $0\geq (d-1) (\tilde{t} - 2^{-1} + q^{-1})$, respectively $0\geq (d-1) (t - q^{-1} + 2^{-1})$. Both inequalities are trivially satisfied if $d=1$, such that the equality \eqref{eq:DImOneWarpedBesovEquality} holds. If $d>1$, then $d-1>0$, but a simple check of all possible configurations reveals that $\tilde{t} - 2^{-1} + q^{-1}$ is positive, unless $p=2$ and $q\geq 2$, when it equals $0$. Likewise, $t - q^{-1} + 2^{-1}>0$, unless $p=2$ and $q\leq2$, when it equals $0$. Hence,  Proposition \ref{prop:RadialWarpingBesovEmbedding} yields the equality \eqref{eq:DImOneWarpedBesovEquality} for $d=1$ and $p=q=2$. The above discussion also settles the remaining cases for $d\geq 2$, as stated in the theorem. This completes the proof.
\end{proof}

\begin{remark}\label{rem:WarpingBesovEmbeddingsModifiedWeight}
  It is easy to see that the statement of Theorem \ref{thm:LnWarpingBesovEmbeddingsExplicit} remains true, if we substitute $\kappa = \kappa^{s,q}$ given in the theorem by 
  \[
    \kappa^{(1,s,q)} := \kappa^{(1,s,d,q)} \colon \RR^d \rightarrow \RR^+,\ \xi \mapsto (1+|\xi|)^{s+d(2^{-1}-q^{-1})}\cdot (1+|\ln(1+|\xi|)|)^{(1-d)(2^{-1}-q^{-1})},
  \]
  a weight that we will encounter again in our study of embeddings between warped coorbit spaces and $\alpha$-modulation spaces in Section \ref{sec:alpha}.  
  
  To see that $\Phi_{\vrho_1}$ and $\kappa^{(1,s,q)}$ are compatible, recall that $\Phi_{\vrho_1}$ and $(1+|\bullet|)^s$ are compatible by Lemma \ref{lem:RadialWarpingStandardWeightAdmissible}, and that $(1+|\bullet|)^{d(2^{-1}-q^{-1})}\cdot (1+|\ln(1+|\bullet|)|)^{(1-d)(2^{-1}-q^{-1})}$ is just $[w \circ \Phi_{\vrho_1}]^{2^{-1}-q^{-1}}$, which is clearly compatible with $\Phi_{\vrho_1}$, in disguise:
  By construction of $\vrho_1$, we have, for all $\xi\in \RR^d\setminus B_{2\eps}(0)$, $|\Phi_{\vrho_1}(\xi)|=\ln(1+|\xi|)$. Furthermore, with $\delta = \Phi_{\vrho_1}(2\eps)>0$, we have, for all $\tau\in \RR^d\setminus B_{\delta}(0)$,that $w(\tau) = e^{d|\tau|}\cdot |\tau|^{1-d} \asymp e^{d|\tau|}\cdot (1+|\tau|)^{1-d}$. Hence, $w(\Phi_{\vrho_1}(\xi)) \asymp (1+|\bullet|)^{d}\cdot (1+|\ln(1+|\bullet|)|)^{(1-d)}$ for all $\xi\in \RR^d\setminus B_{2\eps}(0)$. On the other hand, if $\xi\in \overline{B_{2\eps}(0)}$, then $w(\Phi_{\vrho_1}(\xi)) \asymp 1 \asymp (1+|\bullet|)^{d}\cdot (1+|\ln(1+|\bullet|)|)^{(1-d)}$, such that both weights are equivalent. In particular, $(1+|\bullet|)^{d}\cdot (1+|\ln(1+|\bullet|)|)^{(1-d)}$ is $Cv_0^d$-moderate, where $C\geq 1$ and $v_0$ is a control weight for the $k$-admissible warping function $\Phi_{\vrho_1}$.
  
  Now, by exactly the same steps as in the proof of Theorem \ref{thm:LnWarpingBesovEmbeddingsExplicit}, the expressions \eqref{eq:RadialWarpedCoorbitIntoBesov} and \eqref{eq:BesovIntoRadialWarpedCoorbit} are equivalent to the even simpler 
 \[
   \left\|
      \left(
        2^{j(s_2-s_1)}\cdot (1+j)^{(d-1)\tilde{t}}
      \right)_{j \in \NN_0}
    \right\|_{\ell^{\infty}},\quad \text{and}\quad \left\|
      \left(
       2^{j(s_1-s_2)} \cdot (1 + j)^{(d-1) t}
      \right)_{j \in \NN_0}
    \right\|_{\ell^{\infty}},\quad \text{respectively.}
 \]
 These conditions are, in turn, satisfied for the exact same values of $s_1,\ s_2$ and $\tilde{t},\ t$ as the analogous conditions given in \eqref{eq:CoorbitIntoBesovSimple} and \eqref{eq:BesovIntoCoorbitSimple}.
\end{remark}

\begin{remark}
Note that it can be shown that shearlet-type coverings as discussed in
\cite{labate2013shearlet,StructuredBanachFrames2}
are weakly subordinate to the covering $\CalQ^{(\delta,r)}_{\Phi_{\vrho}}$.
In fact, the \emph{connected} shearlet covering
(see \cite[Definition 3.1]{StructuredBanachFrames2}) 
is almost subordinate to $\CalQ^{(\delta,r)}_{\Phi_{\vrho}}$ and also relatively
moderate with respect to $\CalQ^{(\delta,r)}_{\Phi_{\vrho}}$.
Hence, one can also completely characterize the embedding relations between
$\Co(\Phi_{\rho},L^{p,q}_{\kappa_{\gamma}})$
and the shearlet smoothness spaces discussed in
\cite{labate2013shearlet,StructuredBanachFrames2}. We leave the detailed derivation to the interested reader.
\end{remark}

We will now turn our attention to \emph{$\alpha$-modulation spaces} and (Besov-type) \emph{spaces of dominating mixed smoothness}~\cite{nikol1962boundary,nikol1963boundary,vybiral2006function}, showing that these spaces coincide with certain warped coorbit spaces. In either case, we note that the discretization results in \cite{VoHoPreprint} imply that warped time-frequency systems provide atomic decompositions for the respective spaces that differ from previous constructions, see, e.g., \cite{BorupNielsenAlphaModulationSpaces,speckbacher2016alpha,fefo06} for $\alpha$-modulation spaces and \cite[Section 2.4]{vybiral2006function} for spaces of dominating mixed smoothness.

\section{Warped coorbit spaces as ($\alpha$-)modulation spaces}
\label{sec:alpha}

The $\alpha$-modulation spaces~\cite{gr92-2,BorupNielsenAlphaModulationSpaces} are a family of decomposition spaces defined with respect to coverings that are in some sense geometrically intermediate between the uniform covering and the dyadic, (inhomogeneous) Besov covering. We will show that the radial warping obtained from $\vsig_\beta (t) = (1+t)^{1/\beta}-1$ for $\beta\geq 1$ recovers
$\alpha$-modulation spaces with $\alpha = 1-\beta^{-1} \in [0,1)$. The following definition is a slightly extended version of ~\cite[Definitions 2.1 and 2.4]{BorupNielsenAlphaModulationSpaces}.

\begin{definition}\label{def:AlphaCoveringAndModulationSpaces}
   Let $\CalQ = (Q_i)_{i \in I}$ be an admissible covering of $\RR^d$, 
   and let $\alpha \in (-\infty,1]$.
  The family $\CalQ$ is called an $\alpha$-covering of $\RR^d$, if the
  following hold:
  \begin{enumerate}
     \item\label{enu:AlphaCoveringInnerOuter} With 
            \begin{equation}
            \quad\quad\quad
            r_Q
            := \sup \{
                      r > 0 \,:\, \exists \ \xi \in \RR^d : B_r (\xi) \subset Q
                    \}
            \quad \!\! \text{and} \quad \!\!
            R_Q
            := \inf \{
                      R > 0 \,:\, \exists \ \xi \in \RR^d : Q \subset B_R (\xi)
                    \}
            \, ,
            \label{eq:InnerOuterRadiusDefintion}
          \end{equation}
          we have
          $\sup_{i \in I} R_{Q_i} / r_{Q_i} < \infty$.

    \item\label{enu:AlphaCoveringMeasure} We have $\mu(Q_i) \asymp (1 + |\xi|)^{\alpha \cdot d}$
          for all $i \in I$ and $\xi \in Q_i$,
          where the implied constant does not depend on $i,\xi$,
          and where $\mu$ denotes the Lebesgue measure as usual.
  \end{enumerate}
  
  If $1\leq p,q\leq \infty$, $\CalQ^{(\alpha)}$ is an $\alpha$-covering, and $u^s = (u^s_i)_{i\in I}$, with $u^s_i = [\mu(Q_i^{(\alpha)})]^{s(d\alpha)^{-1}}$ for some $s\in\RR$, then the \emph{$\alpha$-modulation space} $M^{s,\alpha}_{p,q}(\RR^d)$ is defined by 
  \begin{equation}\label{eq:AlphaModulationDefinition}
         M_{p,q}^{s,\alpha}(\RR^d)
    := \DecompSp (\CalQ^{(\alpha)}, \lebesgue^p, \ell_{u^s}^q).
  \end{equation}
\end{definition}

\begin{remark}
   In Appendix \ref{sec:AlphaCoverings}, we show that the resulting function space is independent of the choice of the $\alpha$-covering $\CalQ^{(\alpha)}$, see also \cite[Lemma B.2]{BorupNielsenAlphaModulationSpaces}.
\end{remark}

Note that prior work~\cite{gr92-2,BorupNielsenAlphaModulationSpaces} was concerned specifically with the case $\alpha \in [0,1]$. We show in Appendix \ref{sec:AlphaCoverings} that $\alpha\in (-\infty,1]$ is indeed the most general setting, and the natural generalization of Defintion \ref{def:AlphaCoveringAndModulationSpaces} to $\alpha>1$ yields an empty set of $\alpha$-coverings for any such $\alpha$. To justfiy the notion of $\alpha$-coverings, $\alpha \in [0,1]$, as \emph{intermediate} between the uniform and dyadic coverings, consider Item (\ref{enu:AlphaCoveringMeasure}) in Definition \ref{def:AlphaCoveringAndModulationSpaces} and observe that the uniform covering $\CalQ = ( Q + k )_{k \in \ZZ^d}$ with a
  suitable base set $Q$ satisfies
  $\mu(Q + k) \asymp 1 = (1+|\xi|)^{0 \cdot d}$
  for $\xi \in Q+k$, whereas the dyadic covering $\CalB$ consisting of the annuli
  $B_j = B_{2^{j+1}} (0) \setminus \overline{B_{2^{j-1}} (0)}$ for $j \in \N$,
  and a ball $B_0 = B_2 (0)$, satisfies
  $\mu(B_j) \asymp 2^{dj} \asymp (1 + |\xi|)^{1 \cdot d}$ for
  $\xi \in B_j$.

\medskip{}

In the following, we will consider warping functions $\Phi_{\vrho_\alpha}$, $\alpha\in(-\infty,1)$, derived from a slow start version of
\begin{equation}\label{eq:WeaklyAdmissibleAlphaComponent}
  \vsig_{\alpha}\colon [0,\infty)\rightarrow [0,\infty),\quad \xi \mapsto (1+\xi)^{1-\alpha}-1.
\end{equation}
A straightforward derivation shows that $\vsig_{\alpha}$ is indeed a weakly $k$-admissible radial component, for any $k\in\NN_0$. Hence, any slow start version $\vrho_\alpha$ of $\vsig_{\alpha}$, generates a $k$-admissible warping function $\Phi_{\vrho_\alpha}$, see Theorem \ref{thm:SlowStartFullCriterion} and the ensuing remark. We now proceed to show that $\CalQ^{(\delta,r)}_{\Phi_{\vrho_{\alpha}}}$ is an $\alpha$-covering, such that $\Co(\Phi_{\vrho_{\alpha}},\lebesgue^{p,q}_{\kappa}) = \DecompSp (\CalQ^{(\delta,r)}_{\Phi_{\vrho_{\alpha}}}, \lebesgue^p, \ell_{u^s}^q) = M_{p,q}^{s,\alpha}(\RR^d)$, for an appropriate choice of $\kappa = \kappa^{(\alpha,s,d,q)}$.

\begin{lemma}\label{lem:InducedCoverIsAlphaCover}
  Let $\alpha\in (-\infty,1)$ and $\vrho_{\alpha}$ be an arbitrary slow start version of $\vsig_\alpha$, as in Theorem \ref{thm:SlowStartFullCriterion}. Then, for any $\delta>0$ and $r\geq \sqrt{d}/2$, $\CalQ^{(\delta,r)}_{\Phi_{\vrho_{\alpha}}}$ is an $\alpha$-covering.
\end{lemma}
\begin{proof}
  Since $\vrho_{\alpha}$ is a slow start version of $\vsig_{\alpha}$, there is a $R>0$, such that with $\beta = (1-\alpha)^{-1}\in\RR^+$, we have
  \[
    (\vrho_{\alpha})_\ast(\xi) = (\vsig_{\alpha})_\ast(\xi) = (1+\xi)^{\beta}-1\quad \text{and}\quad (\vrho_{\alpha})_\ast'(\xi) = \beta\cdot (1+\xi)^{\beta-1},\quad \text{for all } \xi\geq R,
  \]
  as well as $\widetilde{(\vrho_{\alpha})_\ast} \asymp 1 \asymp (\vrho_{\alpha})_\ast'$ in $[0,R)$. Using $\xi \asymp 1+\xi$ for $\xi>R$, we see that in fact
  \[\widetilde{(\vrho_{\alpha})_\ast}(\xi) \asymp (1+\xi) ^{\beta-1}\asymp (\vrho_{\alpha})_\ast'(\xi)\] on all of $[0,\infty)$ and with the trivial equality $\frac{x^p}{y^p} = \frac{y^{-p}}{x^{-p}}$ for $x,y,p\in\RR$, we find that
  \[
    v\, \colon\, \RR\rightarrow \RR^+,\ \xi\mapsto C\cdot (1+|\xi|)^{|\beta-1|},
  \]
  is a control weight for $\vrho_{\alpha}$ as per Definition \ref{def:AdmissibleRho}, for an appropriate constant $C>0$ depending on $d$ and $\beta$.
  
  In order to show that $\CalQ^{(\delta,r)}_{\Phi_{\vrho_{\alpha}}}$ satisfies Definition \ref{def:AlphaCoveringAndModulationSpaces}(\ref{enu:AlphaCoveringInnerOuter}), note that
  \[
   \begin{split}
   \left|  \Phi_{\vrho_{\alpha}}^{-1}(\xi) -  \Phi_{\vrho_{\alpha}}^{-1}(\omega)\right| & = \left|\int_0^1 D\Phi_{\vrho_{\alpha}}^{-1}(\omega + s(\xi-\omega))\langle \xi-\omega\rangle~ds\right|\\
   & \leq \int_0^1 \left| D\Phi_{\vrho_{\alpha}}^{-1}(\omega + s(\xi-\omega))\langle \xi-\omega\rangle\right|~ds.
   \end{split}
  \]
  By \cite[Lemma 8.4]{VoHoPreprint}, we have 
 \[
    D\Phi_{\vrho_{\alpha}}^{-1}( \omega+s(\xi-\omega)) = \widetilde{(\vrho_{\alpha})_\ast}(| \omega+s(\xi-\omega)|)\cdot \pi^\bot_{\omega+s(\xi-\omega)} + (\vrho_{\alpha})_\ast'(| \omega+s(\xi-\omega)|)\cdot \pi_{\omega+s(\xi-\omega)}
 \]
  In particular, an application of Parseval's identity for orthonormal bases yields 
   \[
   \begin{split}
     \lefteqn{\left| D\Phi_{\vrho_{\alpha}}^{-1}(\omega + s(\xi-\omega))\langle \xi-\omega\rangle\right|^2}\\ 
     & =  |\widetilde{(\vrho_{\alpha})_\ast}(| \omega+s(\xi-\omega)|)\cdot \pi^\bot_{\omega+s(\xi-\omega)}(\xi-\omega)|^2 + |(\vrho_{\alpha})_\ast'(| \omega+s(\xi-\omega)|)\cdot \pi_{\omega+s(\xi-\omega)}(\xi-\omega)|^2\\
     & \asymp (\vrho_{\alpha})_\ast'(| \omega+s(\xi-\omega)|)^2 \left(|\pi^\bot_{\omega+s(\xi-\omega)}(\xi-\omega)|^2 + |\pi_{\omega+s(\xi-\omega)}(\xi-\omega)|^2\right)\\
     & = |(\vrho_{\alpha})_\ast'(| \omega+s(\xi-\omega)|)|^2 \cdot |\xi-\omega|^2, 
   \end{split}
  \]
  where we used $(\vrho_{\alpha})_\ast' \asymp \widetilde{(\vrho_{\alpha})_\ast}$ and $(\vrho_{\alpha})_\ast'>0$. 
  Altogether, we obtain 
  \[
   \left|  \Phi_{\vrho_{\alpha}}^{-1}(\xi) -  \Phi_{\vrho_{\alpha}}^{-1}(\omega)\right| \asymp
   \int_0^1  |(\vrho_{\alpha})_\ast'(| \omega+s(\xi-\omega)|)| \cdot |\xi-\omega|~ds.
  \]
  
  \medskip{}
  
  Now set $\omega = \delta k$ and let $\xi\in \delta\cdot (\overline{B_r(k)}\setminus B_r(k))$, i.e., $|\delta k-\xi|=\delta r$. Since $s(\omega-\xi)\in \overline{B_{\delta r}(0)}$, and $(\vrho_{\alpha})_\ast'$ is $v$-moderate with radially increasing weight $v$, we further have
  \[
   [v(\delta r)]^{-1}\cdot (\vrho_{\alpha})_\ast'(|\delta k|) \leq (\vrho_{\alpha})_\ast'(| \delta k+s(\xi-\delta k)|) \leq v(\delta r)\cdot(\vrho_{\alpha})_\ast'(| \delta k|),
  \]
  such that 
  \[
    \left|  \Phi_{\vrho_{\alpha}}^{-1}(\xi) -  \Phi_{\vrho_{\alpha}}^{-1}(\delta k)\right| \asymp (\vrho_{\alpha})_\ast'(|\delta k|) \asymp (1+|\delta k|)^{\beta-1}.
  \]
  We conclude that $R_{Q^{(\delta,r)}_{\Phi_{\vrho_{\alpha}},k}}$ and $r_{Q^{(\delta,r)}_{\Phi_{\vrho_{\alpha}},k}}$,
  defined as in Equation \eqref{eq:InnerOuterRadiusDefintion} satisfy
  \begin{equation}\label{eq:RQrQProportionality} 
    R_{Q^{(\delta,r)}_{\Phi_{\vrho_{\alpha}},k}} \asymp r_{Q^{(\delta,r)}_{\Phi_{\vrho_{\alpha}},k}} \asymp (1+|\delta k|)^{\beta-1} \quad \text{and}\quad \sup_{k\in\ZZ^d} R_{Q^{(\delta,r)}_{\Phi_{\vrho_{\alpha}},k}}/ r_{Q^{(\delta,r)}_{\Phi_{\vrho_{\alpha}},k}} < \infty.
  \end{equation}
    
  \medskip{}
  
  It remains to prove that $\mu(Q^{(\delta,r)}_{\Phi_{\vrho_{\alpha}},k}) \asymp (1+|\tau|)^{d\alpha}$, for all $\tau\in Q^{(\delta,r)}_{\Phi_{\vrho_{\alpha}},k}$. For $\tau = \Phi_{\vrho_{\alpha}}^{-1}(\delta k)$, we have $|\delta k| = |\Phi_{\vrho_{\alpha}}(\tau)| \asymp (1+|\tau|)^{1/\beta}-1$, such that
  \[(1+|\delta k|)^{\beta-1} \asymp (1+|\tau|)^{(\beta-1)/\beta} \asymp (1+|\tau|)^{\alpha}.
  \]
  With Equation \eqref{eq:RQrQProportionality}, we conclude that $\mu(Q^{(\delta,r)}_{\Phi_{\vrho_{\alpha}},k}) \asymp (1+|\tau|)^{d\alpha}$ for this value of $\tau$. On the other hand, if $\tau = \Phi_{\vrho_{\alpha}}^{-1}(\xi)$, with $\xi\in \delta\cdot B_r(\delta k)$, then by the same argument,
  \[
   (1+|\tau|)^{(\beta-1)/\beta} \asymp (1+|\xi|)^{\beta-1} \overset{(\ast)}{\asymp}(1+|\delta k|)^{\beta-1},
  \]
  where the proportionality constants for $(\ast)$ can be chosen as $(1+\delta r)^{-|\beta-1|}$ and $(1+\delta r)^{|\beta-1|}$, using that $(1+|\cdot|)^{\beta-1}$ is $(1+|\cdot|)^{|\beta-1|}$-moderate. Overall, 
  $\mu(Q^{(\delta,r)}_{\Phi_{\vrho_{\alpha}},k}) \asymp (1+|\tau|)^{d\alpha}$ for all $\tau\in Q^{(\delta,r)}_{\Phi_{\vrho_{\alpha}},k}$, completing the proof.
\end{proof}

We are now ready to show equality between $\alpha$-modulation spaces and certain coorbit spaces associated to $\Phi_{\vrho_{\alpha}}$.

\begin{proposition}\label{cor:AlphaModulationAsCoorbit}
  Let $\alpha \in (-\infty,1)$, $1\leq p,q \leq \infty$ and choose any slow start version $\vrho_{\alpha}$ of $\vsig_{\alpha}$. Denote $w_\alpha = \det(D\Phi_{\vrho_\alpha}^{-1}(\bullet))$ and define the weight $\kappa := \kappa^{(\alpha,s,d,q)}\,\colon\, \RR^d\rightarrow [0,\infty)$ by
  \begin{equation}\label{eq:CoorbitSpaceModulationWeight}
    \kappa := \kappa^{(\alpha,s,d,q)}\,\colon\, \RR^d\rightarrow [0,\infty),\ \xi \mapsto (1 + |\xi|)^{\gamma},\quad \text{with } \gamma = s-d\alpha(q^{-1}-2^{-1}).
  \end{equation}
  Then, up to canonical identifications,
  \begin{equation}\label{eq:CoorbitModulationEquality}
    \Co(\Phi_{\vrho_{\alpha}},\lebesgue^{p,q}_{\kappa}) = \DecompSp (\CalQ^{(\delta,r)}_{\Phi_{\vrho_{\alpha}}}, \lebesgue^p, \ell_u^q) = M_{p,q}^{s,\alpha}(\RR^d) \, ,
  \end{equation}
  where $u = (u_k)_{k\in\ZZ^d}$,  with $u_k = \kappa(\Phi^{-1}_{\vrho_{\alpha}}(\delta k))\cdot [w_\alpha(\delta k)]^{q^{-1}-2^{-1}}$.
\end{proposition}
\begin{proof}
  In Lemma \ref{lem:InducedCoverIsAlphaCover}, we have shown that $\CalQ^{(\delta,r)}_{\Phi_{\vrho_{\alpha}}}$ is an $\alpha$-cover, such that the second equality in \eqref{eq:CoorbitModulationEquality} follows from the definition of the $\alpha$-modulation space $M_{p,q}^{s,\alpha}(\RR^d)$ as decomposition space if we show that
  \[
   u_k = \kappa(\Phi^{-1}_{\vrho_{\alpha}}(\delta k))\cdot [w_\alpha(\delta k)]^{q^{-1}-2^{-1}} \asymp [\mu(Q^{(\delta,r)}_{\Phi_{\vrho_{\alpha}},k})]^{s(d\alpha)^{-1}}
  \]
  as a function in $k\in\ZZ^d$. In the proof of Lemma \ref{lem:InducedCoverIsAlphaCover}, we have shown that $\mu(Q^{(\delta,r)}_{\Phi_{\vrho_{\alpha}},k}) \asymp (1+|\delta k|)^{d(\beta-1)} \asymp (1+|\delta k|)^{\frac{d\alpha}{1-\alpha}}$, with $\beta = (1-\alpha)^{-1}$. Furthermore, $\kappa(\Phi^{-1}_\alpha(\delta k)) = (1 + |\Phi_{\vrho_\alpha}^{-1}(\delta k)|)^{\gamma}$
  and, by Theorem \ref{cor:RadialWarpingIsWarping}, $w(\delta k) = (\vrho_{\alpha})_\ast'(|\delta k|)\cdot [\widetilde{(\vrho_{\alpha})_\ast}(|\delta k|)]^{d-1} \asymp (1 + |\delta k|)^{d(\beta-1)} = (1 + |\delta k|)^{\frac{d\alpha}{1-\alpha}}$. Finally, since there exists some $R>0$, such that $(1 + |\Phi_{\vrho_\alpha}^{-1}(\xi)|)^{\gamma} = (1 + |\xi|)^{\gamma/(1-\alpha)}$ for all $\xi\in\RR^d\setminus B_R(0)$, we have, in fact $(1 + |\Phi_{\vrho_\alpha}^{-1}(\xi)|)^{\gamma} \asymp (1 + |\xi|)^{\gamma/(1-\alpha)}$ for all $\xi\in\RR^d$.
  Putting all these estimates together, we obtain
  \[
   u_k \asymp (1 + |\delta k|)^{\frac{s}{1-\alpha}} = \left((1 + |\delta k|)^{\frac{d\alpha}{1-\alpha}}\right)^{\frac{s}{d\alpha}} \asymp [\mu(Q^{(\delta,r)}_{\Phi_{\vrho_{\alpha}},k})]^{s(d\alpha)^{-1}} .
  \]
  Hence, we have shown that $\DecompSp (\CalQ^{(\delta,r)}_{\Phi_{\vrho_{\alpha}}}, \lebesgue^p, \ell_u^q) = M_{p,q}^{s,\alpha}(\RR^d)$.
  
  To show the first equality in \eqref{eq:CoorbitModulationEquality}, we want to apply Theorem \ref{thm:CoorbitDecompositionIsomorphism}. To do so, we only need to show that $(\Phi_{\vrho_{\alpha}},\kappa_{\Phi_{\vrho_{\alpha}}})$ form a compatible pair as per Defintion \ref{def:StandingDecompositionAssumptions}. $\Phi_{\alpha}$ is a $\mathcal C^\infty$-diffeomorphism and a $(d+1)$-admissible warping function by construction, see also Theorem \ref{cor:RadialWarpingIsWarping}. As discussed above, the transplanted weight $\kappa_{\Phi_{\vrho_{\alpha}}}$ satisfies $\kappa_{\Phi_{\vrho_{\alpha}}}(\xi) \asymp (1 + |\xi|)^{\gamma/(1-\alpha)}$, for all $\xi\in\RR^d$, and is clearly moderate with respect to the continuous, radial and submultiplicative weight $\xi \mapsto C(1+|\xi|)^{|\gamma/(1-\alpha)|}$, for any sufficiently large $C\geq 1$. Hence, $\Phi_{\vrho_{\alpha}}$ and $\kappa_{\Phi_{\vrho_{\alpha}}}$ are compatible and Theorem \ref{thm:CoorbitDecompositionIsomorphism} yields the desired equality $\Co(\Phi_{\vrho_{\alpha}},\lebesgue^{p,q}_{\kappa}) = \DecompSp (\CalQ^{(\delta,r)}_{\Phi_{\vrho_{\alpha}}}, \lebesgue^p, \ell_u^q) = M_{p,q}^{s,\alpha}(\RR^d)$.
\end{proof}

After Definition \ref{def:AlphaCoveringAndModulationSpaces}, we remarked the the inhomogeneous, dyadic Besov covering $\mathcal B$ is an $\alpha$-covering for $\alpha = 1$. However, the annulus-based geometry of $\mathcal B$ is fundamentally different from the the coverings $\CalQ^{(\delta,r)}_{\Phi_{\vrho_{\alpha}}}$, $\alpha\in [0,1)$, which consist of convex sets. Warped coorbit spaces provide an interesting alternative with radial warping based on the weakly $k$-admissible radial component $\vsig_{1} = \ln(1+\bullet)$, which we have related to Besov spaces in the previous subsection. We obtain the following result.

\begin{proposition}\label{prop:ModulationIntoLnWarpingEmbedding}
  Let $\alpha \in (-\infty,1)$, $1\leq p,q \leq \infty$ and let $\vrho_{1}$ be any slow start version of $\vsig_{1} = \ln(1+\bullet)$. Define $\Phi_1 := \Phi_{\vrho_{1}}$ and let $w_1 := \det(\mathrm{D}\Phi_1^{-1})$ be the associated weight.
  Fix $\eps>0$ and let $\tilde{T} := \tilde{T}_{d,\alpha} := \tilde{t}d(1-\alpha)$, with $\tilde{t} = -\max(0,\max(p^{-1},1-p^{-1})-q^{-1})$ and $T := T_{d,\alpha} := td(1-\alpha)\geq 0$, with $t=\max(0,q^{-1}-\min(p^{-1},1-p^{-1}))$. Then we have the embedding
  \begin{equation}\label{eq:ModulationIntoLnWarpingEmbedding}
   M_{p,q}^{s+\tilde{T},\alpha}(\RR^d) \hookrightarrow \Co(\Phi_{\vrho_{1}},\lebesgue^{p,q}_{\kappa}) \hookrightarrow M_{p,q}^{s-T,\alpha}(\RR^d)
  \end{equation}
  with the weight $\kappa := \kappa^{(1,s,d,q)}\,\colon\, \RR^d\rightarrow [0,\infty)$ given by
  \begin{equation}\label{eq:CoorbitSpaceBesovWeight}
    \kappa^{(1,s,d,q)}\,\colon\, \RR^d\rightarrow [0,\infty),\ \xi \mapsto (1 + |\xi|)^{s+d(2^{-1}-q^{-1})} \cdot (1+\ln(1+|\xi|))^{(1-d)(2^{-1}-q^{-1})}. 
  \end{equation}
\end{proposition}
\begin{proof}
  Let, as before, $\vrho_{\alpha}$ be a slow start version of $\vsig_{\alpha}$, see Equation \eqref{eq:WeaklyAdmissibleAlphaComponent}. Let further, 
  $w_\alpha = \det(D\Phi_{\vrho_\alpha}(\bullet))$ denote the weight associated with 
  $\Phi_{\vrho_\alpha}$. 
  Proposition \ref{cor:AlphaModulationAsCoorbit} shows that $
  M_{p,q}^{s,\alpha}(\RR^d) = \Co(\Phi_{\vrho_{\alpha}},\lebesgue^{p,q}_{\kappa})$, where $\kappa(\xi) = \kappa^{(\alpha,s,d,q)}(\xi) = (1 + |\xi|)^{s+d\alpha(2^{-1}-q^{-1})}$ is as in \eqref{eq:CoorbitSpaceModulationWeight}.
  In particular, $\Phi_{\vrho_{\alpha}}$ and $\kappa^{(\alpha,s,d,q)}$ are compatible. Compatibility of $\Phi_{\vrho_1}$ and $\kappa^{(1,s,d,q)}$ was already proven in Remark \ref{rem:WarpingBesovEmbeddingsModifiedWeight}, where we also saw that $w_1 \circ \Phi_{\vrho_1}\asymp (1 + |\bullet|)^{d}(1 + |1+|\ln(1+|\bullet|)|)^{1-d}$, such that $\kappa^{(1,s,d,q)} \asymp (1+|\bullet|)^s \cdot [w_1 \circ \Phi_{\vrho_1}]^{2^{-1}-q^{-1}}$. Similarly we can show that $w_\alpha \circ \Phi_{\vrho_\alpha}\asymp (1 + |\bullet|)^{d\alpha}$ and thus $\kappa^{(\alpha,s,d,q)} \asymp (1+|\bullet|)^s \cdot [w_\alpha \circ \Phi_{\vrho_\alpha}]^{2^{-1}-q^{-1}}$.
  
  Moreover, with the above derivations, and setting $\kappa = \kappa^{(1,s,d,q)}$, we can already derive the following equivalence, which holds for any $t\in\RR$ and will be used later in the proof. 
  \begin{equation}\label{eq:KappaEquivalence}
   \begin{split}
   \kappa \cdot \left[\frac{w_1\circ\Phi_{\vrho_1}}{w_\alpha\circ \Phi_{\vrho_\alpha}}\right]^{q^{-1}-2^{-1}-t} & \asymp (1+|\bullet|)^s \cdot [w_1\circ\Phi_{\vrho_1}]^{-t} \cdot [w_\alpha\circ \Phi_{\vrho_\alpha}]^{t+2^{-1}-{q}^{-1}}\\
   & \asymp (1+|\xi|)^{s+td(\alpha-1)} \cdot [w_\alpha\circ \Phi_{\vrho_\alpha}]^{{2}^{-1}-{q}^{-1}} \cdot (1+|\ln(1+|\xi||)^{t(d-1)}\\
   & \asymp \kappa^{(\alpha,s-td(1-\alpha),d,q)}(\xi) \cdot (1+|\ln(1+|\xi|)|)^{t(d-1)}.
   \end{split}
  \end{equation}

  \medskip{}

  We will prove \eqref{eq:ModulationIntoLnWarpingEmbedding} by means of Theorem \ref{thm:EmbeddingForRadialComp}. Hence, we must first show that its assumptions are met. By Definition of $\vrho_\alpha,\ \vrho_1$, there exists some $R>0$, such that $(\vrho_\alpha)'(\xi) = (1-\alpha)(1+|\xi|)^{-\alpha}$ and $(\vrho_1)'(\xi) = (1+|\xi|)^{-1}$ holds for all $\xi\in\RR\setminus [-R,R]$. The existence of $C_\alpha>0$, such that $(\vrho_1)'(\xi) \leq C_\alpha (\vrho_\alpha)'(\xi)$ for all such $\xi$ is clear. After possibly increasing $C_\alpha$, this inequality extends to all $\xi\in\RR$, since $[-R,R]$ is compact and $(\vrho_1)',(\vrho_\alpha)'>0$ are continuous. In particular, we can use Corollary \ref{cor:SubordinatenessForBoundingDerivativeRadialComp} to see that the conditions of Lemma \ref{lem:AbstractWarpingSuffCond} are satisfied, such that $\Phi_{\vrho_{\alpha}}$ and $\kappa^{(1,s,d,q)}$ are compatible. Finally, we have, for all $a\in[1,\infty)$, $\xi\in\RR$, and $\alpha\in [0,1)$, that $(1+|a\xi|)^{-\alpha} \leq (1+|\xi|)^{-\alpha} \leq a^\alpha \cdot (a(1+|\xi|))^{-\alpha} \leq C(a) (1+|a\xi|))^{-\alpha}$, with $C(a)<\infty$. For $\alpha\in (-\infty,0)$, analogous steps yield a similar proportionality relation. Hence, all assumptions of Theorem \ref{thm:EmbeddingForRadialComp} are met, with $\kappa = \kappa^{(1,s,d,q)}$.

 Note that $t=\max(0,q^{-1}-\min(p^{-1},1-p^{-1}))\in [0,1]$, such that Equation \eqref{eq:KappaEquivalence} yields, with $T := T_{d,\alpha} := td(1-\alpha) \geq 0$,
  \[
   \kappa_{\vrho_{\alpha},\vrho_1,t} = \kappa \cdot \left[\frac{w_1\circ\Phi_{\vrho_1}}{w_\alpha\circ \Phi_{\vrho_\alpha}}\right]^{\frac{1}{q}-\frac{1}{2}-t} \gtrsim \kappa^{(\alpha,s-T,d,q)}.
  \]
  Likewise, if we set $t = -\tilde{t} = \max(0,\max(p^{-1},1-p^{-1})-q^{-1})\in [-1,0]$ and $\tilde{T} := \tilde{T}_{d,\alpha} := \tilde{t}d(1-\alpha) \geq 0$, then Equation \eqref{eq:KappaEquivalence} yields
  \[
   \kappa_{\vrho_{\alpha},\vrho_1,-\tilde{t}} = \kappa \cdot \left[\frac{w_1\circ\Phi_{\vrho_1}}{w_\alpha\circ \Phi_{\vrho_\alpha}}\right]^{\frac{1}{q}-\frac{1}{2}+\tilde{t}} \lesssim \kappa^{(\alpha,s+\tilde{T},d,q)}.
  \]
  The discussion in Section \ref{sub:WarpedEmbeddingsSame}, in particular Equation \eqref{eq:SimplifiedWarpedEmbeddingCond}, shows that the estimates above imply
  \[
   \Co(\Phi_{\vrho_{\alpha}},\lebesgue^{p,q}_{\kappa^{(\alpha,s+\tilde{T},d,q)}}) \hookrightarrow \Co(\Phi_{\vrho_{\alpha}},\lebesgue^{p,q}_{\kappa_{\vrho_{\alpha},\vrho_1,-\tilde{t}}})\quad \text{and}\quad \Co(\Phi_{\vrho_{\alpha}},\lebesgue^{p,q}_{\kappa_{\vrho_{\alpha},\vrho_1,t}}) \hookrightarrow \Co(\Phi_{\vrho_{\alpha}},\lebesgue^{p,q}_{\kappa^{(\alpha,s-T,d,q)}}).
  \]
  Moreover, by Theorem \ref{thm:EmbeddingForRadialComp}, we obtain 
  \[
   \Co(\Phi_{\vrho_{\alpha}},\lebesgue^{p,q}_{\kappa_{\vrho_{\alpha},\vrho_1,-\tilde{t}}}) \hookrightarrow \Co(\Phi_{\vrho_{1}},\lebesgue^{p,q}_{\kappa^{(1,s,d,q)}}) \hookrightarrow \Co(\Phi_{\vrho_{\alpha}},\lebesgue^{p,q}_{\kappa_{\vrho_{\alpha},\vrho_1,t}}),
  \]
  and Proposition \ref{cor:AlphaModulationAsCoorbit} yields $M_{p,q}^{s+\tilde{T},\alpha}(\RR^d) = \Co(\Phi_{\vrho_{\alpha}},\lebesgue^{p,q}_{\kappa^{(\alpha,s+\tilde{T},d,q)}})$ and $\Co(\Phi_{\vrho_{\alpha}},\lebesgue^{p,q}_{\kappa^{(\alpha,s-T,d,q)}}) = M_{p,q}^{s-T,\alpha}(\RR^d)$.
  Overall, we have shown that 
  \[
   M_{p,q}^{s+\tilde{T},\alpha}(\RR^d) \hookrightarrow \Co(\Phi_{\vrho_{1}},\lebesgue^{p,q}_{\kappa^{(1,s,d,q)}}) \hookrightarrow M_{p,q}^{s-T,\alpha}(\RR^d),
  \]
  as desired.  

\end{proof}

\begin{remark}\label{rem:WarpingVsAlphaVsBesov}
  Combining our results from this section and Section \ref{sec:WarpedCoorbitAndBesov}, we have shown that any $\alpha$-covering with $\alpha\in (-\infty,1)$ is almost subordinate to $\CalQ^{(\delta,r)}_{\Phi_{\vrho_1}}$, which is in turn almost subordinate to the inhomogeneous Besov covering $\mathcal B$, which also is an $\alpha$-covering for $\alpha=1$. In dimension $d=1$, $\CalQ^{(\delta,r)}_{\Phi_{\vrho_1}}$ and $\mathcal B$ are weakly equivalent, but for $d>1$, Theorem  \ref{thm:WarpingNoBesovEquivalence} shows that $\mathcal B$ is never almost subordinate to $\CalQ^{(\delta,r)}_{\Phi_{\vrho_1}}$, i.e., it is not an $\alpha$-covering for $\alpha=1$. Likewise, it can be shown that $\CalQ^{(\delta,r)}_{\Phi_{\vrho_1}}$ is not almost subordinate to any $\alpha$-covering with $\alpha\in (-\infty,1)$, in particular, it is not an $\alpha$-covering at all.

  To see this, fix some $\delta>0$ and $r>\sqrt{d}/2$ and denote by $K_\ell = \{k\in \ZZ^d ~:~ Q^{(\delta,r)}_{\Phi_{\vrho_\alpha},k}\cap Q^{(\delta,r)}_{\Phi_{\vrho_1},\ell}\neq \emptyset\}$, $\ell\in\ZZ^d$. Now, consider $Q^{(\delta,r)}_{\Phi_{\vrho_1},ne_1}$, for $n\in\NN$ and the first standard unit vector $e_1\in\RR^d$. For any sufficiently large $n$, we have
  \[
   Q^{(\delta,r)}_{\Phi_{\vrho_1},ne_1} = \Phi_{\vrho_1}^{-1}(\delta\cdot B_r(ne_1)) \supset \{ ce_1 ~:~ c+1 \in e^{|\delta\cdot B_r(ne_1)|}\} = \{ ce_1 ~:~ c+1 \in e^{(\delta (n-r),\delta (n+r))}\}.
  \]
  Furthermore, for any sufficiently large $m \in\NN$, we have
  \[
    Q^{(\delta,r)}_{\Phi_{\vrho_\alpha},me_1} \ni \Phi_{\vrho_\alpha}^{-1}(\delta me_1) = ((1+\delta m)^{\frac{1}{1-\alpha}}-1)e_1.
  \]
  Hence, if $(1+\delta m)^{\frac{1}{1-\alpha}} \in e^{(\delta (n-r),\delta (n+r))}$, or equivalently $\ln(1+\delta m) \in (1-\alpha)\cdot (\delta (n-r),\delta (n+r))$, then $me_1\in K_{ne_1}$. In other words, for any sufficiently large $n\in\NN$, $\{ m\in\NN~:~ \ln(1+\delta m) \in (1-\alpha)\cdot (\delta (n-r),\delta (n+r)) \} \subset K_{ne_1}$, but clearly
  \[
   |\{ m\in\NN~:~ \ln(1+\delta m) \in (1-\alpha)\cdot (\delta (n-r),\delta (n+r)) \}| \overset{n\rightarrow \infty}{\longrightarrow} \infty.
  \]
  Hence, $\CalQ^{(\delta,r)}_{\Phi_{\vrho_1}}$ is not weakly subordinate to $\CalQ^{(\delta,r)}_{\Phi_{\vrho_\alpha}}$. Since the elements of both coverings are path-connected, it is thus not almost subordinate to $\CalQ^{(\delta,r)}_{\Phi_{\vrho_\alpha}}$. Since all $\alpha$-coverings are equivalent and $\alpha\in (-\infty,1)$ was arbitrary, we conclude that $\CalQ^{(\delta,r)}_{\Phi_{\vrho_1}}$ is not almost subordinate to any $\alpha$-covering.

  Overall, it seems natural to consider $\CalQ^{(\delta,r)}_{\Phi_1}$ as an intermediate between $\alpha$-coverings with $\alpha\rightarrow 1$ and the inhomogeneous Besov covering $\mathcal B$. Interestingly, the embedding relations in Proposition \ref{prop:ModulationIntoLnWarpingEmbedding} exactly mirror those obtained between $\alpha$-modulation and Besov spaces: Specifically, with $\tilde{T} = \tilde{T}_{d,\alpha}$ and $T := T_{d,\alpha}$ as in the proposition, we obtain by \cite[Theorem 9.13, Lemma 9.15, Corollary 9.16]{voigtlaender2016embeddings},
  \[M_{p,q}^{s+\tilde{T},\alpha}(\RR^d) \hookrightarrow B_s^{p,q}(\RR^d) \hookrightarrow M_{p,q}^{s-T,\alpha}(\RR^d).\]
\end{remark}

\medskip{}

\section{Warped coorbit spaces as spaces of dominating mixed smoothness}\label{sec:DominatingMixedSmoothness}

In Section \ref{sec:WarpedCoorbitAndBesov}, we have seen that---at least in dimension $d=1$---%
the radial warping generated from $\vsig = \log(1+|\cdot|)$ induces a
frequency covering equivalent to the inhomogeneous Besov covering, so that the
associated coorbit spaces are inhomogeneous Besov spaces.
In the following, we will show, for arbitrary dimension $d \in \NN$, that certain warped coorbit spaces obtained by $d$-dimensional \emph{tensorization} of a $1$-dimensional warping function, i.e., $\Psi = \Psi_\vrho \otimes \ldots \otimes \Psi_\vrho$ (where $\Phi_\vrho$ is symmetric and thus radial), are
identical to (Besov-type) \emph{spaces of dominating mixed smoothness}~\cite{nikol1962boundary,nikol1963boundary,vybiral2006function}. The latter are defined as follows.

\begin{definition}[Definition 1.2 \cite{vybiral2006function}]\label{def:SpacesOfDomMixedSmoothness}
   Let $1\leq p,q\leq \infty$ and $s\in\RR^d$ and choose a $\mathcal B$-BAPU $\{\varphi_j\}_{j\in\NN_0}\subset \TestFunctionSpace(\RR)$ for the $1$-dimensional, inhomogenous Besov covering $\mathcal B = \{ B_{i}\}_{i\in\NN_0}$, such that additionally, there are $c_k>0$, $k\in\NN_0$, with
   \begin{equation}\label{eq:MixedSmoothnessBAPURegular}
    2^jk \cdot |\varphi_j^{(k)}(t)| \leq c_k, \quad \text{for all } j,k\in\NN_0,\ t\in\RR. 
   \end{equation}
   Define $\mathcal B_d = \{ B_{d,\ell}\}_{\ell\in\NN_0^d}$, where $B_{d,\ell} = B_{\ell_1}\times\ldots\times B_{\ell_d}$, and $v^{(s)} = \{v^{(s)}_\ell\}_{\ell\in\NN_0^d}$, where $v^{(s)}_\ell = 2^{\langle \ell,s\rangle}$. Then the space $S^s_{p,q}B(\RR^d)$ is given by the collection
   \begin{equation}\label{eq:MixedSmoothnessDef}
     S^s_{p,q}B(\RR^d) := \left\{ f\in\CalS'(\RR^d)~\colon~ \|f\|_{\DecompSp(\mathcal B_d, \lebesgue^p,\ell_{v^{(s)}}^q)} < \infty\right\},    \end{equation}
   equipped with the usual decomposition space norm $f\mapsto \|f\|_{\DecompSp(\mathcal B_d, \lebesgue^p,\ell_{v^{(s)}}^q)}$, defined with respect to the $\mathcal B_d$-BAPU $\{{\psi_\ell}\}_{\ell\in\NN_0^d}\subset \TestFunctionSpace(\RR^d)$, where $\psi_\ell = \varphi_{\ell_1}(\bullet_1)\cdot\ldots\cdot \varphi_{\ell_d}(\bullet_d)$.
\end{definition}

\begin{remark}\label{rem:BdProperties}
  The decomposition space norm in Definition \ref{def:SpacesOfDomMixedSmoothness} was used only for notational convenience and it remains to be shown that $S^s_{p,q}B(\RR^d)$ does in fact coincide with $\DecompSp(\mathcal B_d, \lebesgue^p,\ell_{v^{(s)}}^q)$. Strictly speaking, we also have to prove the claim in Definition \ref{def:SpacesOfDomMixedSmoothness}, that $\{{\psi_\ell}\}_{\ell\in\NN_0^d}$ is a $\mathcal B_d$-BAPU. The required support and summation properties are, however, obvious. Observe further that \eqref{eq:MixedSmoothnessBAPURegular} is equivalent to requiring that $\{\varphi_j\}_{j\in\NN_0}$ is a regular partition of unity subordinate to $\mathcal B$ in the sense of \cite[Definition 8.1]{voigtlaender2016embeddings}, and as an immediate consequence, $\{{\psi_\ell}\}_{\ell\in\NN_0^d}$ is a regular partition of unity subordinate to $\mathcal B_d$. Hence, $\{{\psi_\ell}\}_{\ell\in\NN_0^d}$ is in fact a $\mathcal B_d$-BAPU (by \cite[Theorem 8.2]{voigtlaender2016embeddings}) and $\mathcal B_d$ a decomposition covering. Likewise, it is easy to see that $\mathcal B_d$-moderateness of $v^{(s)}$ is implied by $\mathcal B$-moderateness of the $1$-dimensional Besov weight $v^s$ with $v^s_j = 2^{js}$. 
\end{remark}

It can be seen for $d=1$ that the space $S^s_{p,q}B(\RR)$ of dominating mixed smoothness coincides with the inhomogeneous Besov space $B^s_{p,q}(\RR)$. For larger values of $d$,  $S^s_{p,q}B(\RR^d)$ is, in some sense, an anisotropic alternative to Besov spaces, penalizing large values in the Fourier-domain with a weight that does not depend only on the modulus of the frequency position $\xi\in\RR^d$, but the modulus of its individual components $\xi_1,\ldots,\xi_d\in\RR$. 

We now proceed to prove that $S^s_{p,q}B(\RR^d)$ coincides with the decomposition space $\DecompSp(\mathcal B_d, \lebesgue^p,\ell_{v^{(s)}}^q)$. For this purpose, we need an auxilliary result on the covering $\CalB_d$.

\begin{proposition}\label{prop:DominatingMixedSmoothnessCoveringWeightIsSummable}
  Let $\mathcal B = \{ B_{j}\}_{j\in\NN_0}$ denote the $1$-dimensional, inhomogenous Besov covering, see \eqref{eq:BesovCoveringDefinition}. The covering $\mathcal B_d = \{ B_{d,\ell}\}_{\ell\in\NN_0^d}$, where $B_{d,\ell} = B_{\ell_1}\times\ldots\times B_{\ell_d}$, of $\RR^d$ is semi-structured with
  \[
    B_{d,\ell} = S_l B_{d,\ell}'\, ,\quad \text{with}\quad S_l = \mathrm{diag}(2^{\ell_1},\ldots,2^{\ell_d})\quad \text{and}\quad B_{d,\ell}' = B_{\ell_1}'\times\cdots\times B_{\ell_d}'.
  \]
  Here, $B_j ' = (-2,-1/2) \cup (1/2, 2)$ for $j \in \NN$, and $B_0' = (-2,2)$. Furthermore, for all $1\leq p\leq \infty$, there exists some $N\in\NN_0$, such that the weight $w^{(N)} := w^{(N,p)}:= (w_\ell^{(N,p)})_{\ell\in\NN_0^d}$, with
  \[
    w_\ell^{(N,p)}
    = |\det S_\ell|^{1/p}
      \cdot \Big[
              \inf_{\xi \in B_{d,\ell^\ast}} (1 + |\xi|)
            \Big]^{-N}
    \quad \text{for} \quad
    \ell \in \NN_0^d \, ,
  \]
  satisfies $w^{(N)} \in \ell^1(\NN_0^d) \subset \ell^q(\NN_0^d)$.
\end{proposition}
\begin{proof}
   It is well-known that the inhomogeneous Besov covering $\CalB = (B_j)_{j \in \NN_0} = (T_j B_j ')_{j \in \NN_0}$
  is a semi-structured covering of $\RR$, with $T_j = 2^j$ for $j \in \NN_0$, $B_0 ' = (-2,2)$ and $B_j ' = (-2,-1/2) \cup (1/2, 2)$ for $j \in \NN$. Hence, the elements of $\CalB_d$ can be represented as $B_{d,\ell} = S_\ell B_{d,\ell}'$, with $S_\ell = \mathrm{diag} (2^{\ell_1}, \dots, 2^{\ell_d})$ and $B_{d,\ell} ' = B_{\ell_1}' \times \cdots \times B_{\ell_d}'$ for $\ell \in \NN_0^d$. In particular, $|\det S_\ell| = 2^{\|\ell\|_1}$ and $\|S_\ell^{-1}\| \leq 1$. That $\mathcal B_d$ is semi-structured now follows easily.  
  
  \medskip{}
  
  In Section \ref{sec:WarpedCoorbitAndBesov} we already noted that $1 + |\xi| \asymp 2^j$ for all $\xi \in B_j$, $j\in\NN_0$,  and it is easily seen that even $1 + |\xi| \asymp 2^j$ for all $\xi \in B_{j^\ast}$. Since $B_{d,\ell^\ast} = B_{\ell_1^\ast} \times \cdots \times B_{\ell_d^\ast}$, we obtain
  \[
    1 + |\xi|
    \asymp 1 + \|\xi\|_{\infty}
    \asymp 2^{\|\ell\|_{\ell^\infty}}
    \quad \text{for} \quad \xi \in B_{d,\ell^\ast} \, .
  \]
  Therefore, if $N > d/p$
  \[
   w_\ell^{(N,p)}
    \lesssim 2^{\|\ell\|_1/p} \cdot 2^{-N\|\ell\|_{\ell^\infty}} 
      \lesssim 2^{(\frac{d}{p}-N)\cdot\|\ell\|_{\ell^\infty}}
      \lesssim (1 + \|\ell\|_{\ell^\infty})^{-(d+1)}
    \quad \text{for} \quad
    \ell \in \NN_0^d \, ,
  \]
  such that $w^{(N)}\in\ell^1 (\NN_0^d)$. In the last step, we used the elementary estimate
  $2^{x} = e^{x \cdot \ln 2} = \sum_{k=0}^\infty (x \cdot \ln 2)^k / k! \gtrsim (1 + x)^{d+1}$, valid for $x \geq 0$. 
\end{proof}

\begin{corollary}\label{cor:DominatingMixedSmoothnessIsDecomposition}
  For $1\leq p,q\leq \infty$ and $s\in\RR^d$ and with $\mathcal B_d$ and $\{{\psi_\ell}\}_{\ell\in\NN_0^d}$ the $\mathcal B_d$-BAPU as in Definition \ref{def:SpacesOfDomMixedSmoothness}, then every $g\in \DecompSp(\mathcal B_d, \lebesgue^p,\ell_{v^{(s)}}^q)$ extends to a tempered distribution $\tilde{g}\in\Schwartz'(\RR^d)$ given by 
  \[
      \tilde{g}: \Schwartz(\RR^d) \rightarrow \CC,\ f \mapsto \sum_{\ell\in\ZZ^d} \langle \Fourier^{-1}(\hat{g}\cdot \psi_\ell ),f\rangle_{\Schwartz',\Schwartz}.                                                                                                                                                                                                                                                                                                          
  \]
  In particular, $\langle \tilde{g},f\rangle_{\Schwartz',\Schwartz} = \langle g,f\rangle_{Z',Z}$, for all $f\in Z(D)$ and $g\in \DecompSp(\mathcal B_d, \lebesgue^p,\ell_{v^{(s)}}^q)$.  
\end{corollary}
\begin{proof}
  This is a direct application of \cite[Theorem 8.3]{voigtlaender2016embeddings}, provided that all its assumptions are satisfied. As discussed in Remark \ref{rem:BdProperties}, $\{{\psi_\ell}\}_{\ell\in\NN_0^d}$ is a regular partition of unity subordinate to $\mathcal B_d$, such that $\mathcal B_d$ itself is a regular covering in the sense of \cite[Definition 8.1]{voigtlaender2016embeddings}.  Note that $\max\{1,\|S^{-1}_\ell\|\} = 1$ for all $\ell\in\NN^d_0$, as observed in the proof of Proposition \ref{prop:DominatingMixedSmoothnessCoveringWeightIsSummable}, such that the weight $w^{(N,p)}$ in said proposition coincides with the weight $w^{(N)}\in\CC^{\NN_0^d}$ defined in \cite[Theorem 8.3]{voigtlaender2016embeddings}. Therefore Proposition \ref{prop:DominatingMixedSmoothnessCoveringWeightIsSummable} shows the existence of $N\in\NN_0$, such that $w^{(N)}$ is summable and all prerequisites of \cite[Theorem 8.3]{voigtlaender2016embeddings} are satisfied. This completes the proof. 
\end{proof}

  In the above proof, and in Remark \ref{rem:BdProperties}, we did not derive the desired properties of the product weight $v^{(s)}$ in full detail. Although these derivations are rather straightforward, we provide, for the convenience of the reader, a full proof in Appendix \ref{sec:semistructuredproducts}, together with general results on products of semi-structured coverings. There, we also show that the inheritance of almost subordinateness, which will be used in Theorem \ref{prop:DominatingMixedSmoothnessAsDecomposition}, is a simple consequence of the product construction.
  
  The warping function that generates spaces of dominating mixed smoothness will turn out to be a $d$-dimensional \emph{tensor product}, the factors of which are given by the $1$-dimensional warping function $\Phi_{\vrho_{1}}$ considered in our treatment of Besov spaces in Section \ref{sec:WarpedCoorbitAndBesov}. Therefore, the following lemma, which considers general \emph{separable} warping functions $\Phi$, as they were referred to previously~\cite{VoHoPreprint}, will be helpful.

\begin{lemma}\label{lem:SeperableWarping}
  For $i = 1,\dots,N$ let $\Phi_i : D_i \to \RR^{d_i}$ be a
  $k$-admissible warping function with control weight $v_i$.
  Let $d := d_1 + \dots + d_N$ and $D := D_1 \times \cdots \times D_N$.
  Then
  \[
    \Phi := \Phi_1 \otimes \cdots \otimes \Phi_N
         \colon D = D_1 \times \cdots \times D_N \to \RR^d,\
         (\xi_1, \dots, \xi_N) \mapsto \big(
                                            \Phi_1 (\xi_1),
                                            \dots,
                                            \Phi_N (\xi_N)
                                          \big)
  \]
  is a $k$-admissible warping function with control weight
  $v : \RR^d \to \RR^+,
       \tau \mapsto \max \{ v_1 (|\tau|e_1^{(d_1)}), \dots, v_N (|\tau|e_1^{(d_N)}) \}$, where the superscripts indicate the dimensionality of the standard unit vectors. Moreover, the $\Phi$-induced $(\delta,r)$-fine product covering
  $\CalQ^{(\delta,r)}_{\Phi,\otimes} =  \{ Q^{(\delta,r)}_{\Phi,\otimes,\ell}\}_{\ell\in\ZZ^d}$, where $Q^{(\delta,r)}_{\Phi,\otimes,\ell} = Q^{(\delta,r)}_{\Phi_1,\ell_1} \times \ldots \times Q^{(\delta,r)}_{\Phi_N,\ell_N}$, with $\ell = (\ell_1,\ldots,\ell_N)\in\ZZ^d$ and  $\ell_i\in\ZZ^{d_i}$, $i = 1,\ldots,N$, is equivalent to the usual $\Phi$-induced $(\delta,r)$-fine frequency covering $\CalQ^{(\delta,r)}_{\Phi}$, for all $\delta>0$, $r>\sqrt{d}/2$.

  If furthermore for each $i = 1,\dots,N$ we are given a weight
  $\kappa_i : D_i \to \RR^+$ such that $\kappa_i \circ \Phi_i^{-1}$
  is $\tilde{\kappa}_i$-moderate, for some continuous, radially increasing,
  submultiplicative weight $\tilde{\kappa}_i$, then setting
  \[
    \kappa : D \to \RR^+,
            (\xi_1, \dots, \xi_N) \mapsto \prod_{i=1}^N \kappa_i (\xi_i)
    \quad \text{and} \quad
    \tilde{\kappa} : \RR^d \to \RR^+, \xi \mapsto \prod_{i=1}^N \tilde{\kappa}_i (|\xi|) \, ,
  \]
  the weight $\kappa_ \Phi = \kappa \circ \Phi^{-1}$ is $\tilde{\kappa}$-moderate,
  and $\tilde{\kappa}$ is continuous, radially increasing, and submultiplicative. Hence, if $\Phi_i$ is a $\mathcal C^\infty$-diffeomorphism, for all $i = 1,\dots,N$, then $(\Phi,\kappa)$ forms a compatible pair as per Definition \ref{def:StandingDecompositionAssumptions}.
\end{lemma}

\begin{proof}
  Across the proof, we will identify $\RR^d$ with
  $\RR^{d_1} \times \cdots \times \RR^{d_N}$, and likewise
  $\NN_0^d$ with $\NN_0^{d_1} \times \cdots \times \NN_0^{d_N}$.
  
  Since each $\Phi_i$ is a $C^{k+1}$-diffeomorphism, so is $\Phi$.
  The inverse of $\Phi$ is given by
  \[
    \Phi^{-1} : \RR^d \to D,
                (\tau_1, \dots, \tau_N) \mapsto \big(
                                                     \Phi_1^{-1} (\tau_1),
                                                     \dots,
                                                     \Phi_N^{-1} (\tau_N)
                                                   \big) \, .
  \]
  Therefore, if we set $A_i := D \Phi_i^{-1}$, then
  \[
    A(\tau_1,\dots,\tau_N)
    := D \Phi^{-1} (\tau_1, \dots, \tau_N)
    = \mathrm{diag} \big( A_1 (\tau_1), \dots, A_N (\tau_N) \big)
  \]
  is a block-diagonal matrix, so that
  \begin{equation}\label{eq:AssociatedProductWeight}
  w(\tau) = \det (A (\tau_1,\dots,\tau_N)) = \prod_{i=1}^N \det (A_i (\tau_i)) = \prod_{i=1}^N w_i(\tau_i) > 0 \, ,
  \end{equation}
  where $w$ and $w_i$ are the weights associated with the warping functions $\Phi$ and $\Phi_i$, respectively.

  Furthermore, if we set
  $\phi_\tau^{(i)} (\upsilon) := A_i^T (\upsilon + \tau) \cdot A_i^{-T} (\tau)$
  for $\tau, \upsilon \in \RR^{d_i}$, then for
  $\tau = (\tau_1,\dots,\tau_N) \in \RR^d$ and
  $\upsilon = (\upsilon_1, \dots, \upsilon_N) \in \RR^d$, we have
  \[
    \phi_\tau (\upsilon)
    := A^T (\upsilon + \tau) \cdot A^{-T} (\tau)
    = \mathrm{diag} \big(
                      \phi_{\tau_1}^{(1)} (\upsilon_1),
                      \dots,
                      \phi_{\tau_N}^{(N)} (\upsilon_N)
                    \big) \, .
  \]
  Therefore, for any $\alpha = (\alpha_1,\dots,\alpha_N) \in \NN_0^{d}$
  with $|\alpha| \leq k$, we have
  \begin{align*}
    \|\partial^\alpha \phi_\tau (\upsilon)\|
    & = \big\|
          \mathrm{diag}
          \big(
            \partial^{\alpha_1} \phi_{\tau_1}^{(1)} (\upsilon_1),
            \dots,
            \partial^{\alpha_N} \phi_{\tau_N}^{(N)} (\upsilon_N)
          \big)
        \big\|
      = \max \big\{
               \| \partial^{\alpha_i} \phi_{\tau_i}^{(i)} (\upsilon_i) \|
               \,:\,
               i = 1,\dots,N
             \big\} \\
    & \overset{(\lozenge)}{\leq}
        \max \big\{
               v_i (|\upsilon_i|e_1^{(d_i)}) \,:\, i = 1, \dots, N
             \big\}
    \overset{(\ast)}{\leq}
      \max \big\{ v_i (|\upsilon|e_1^{(d_i)}) \,:\, i = 1,\dots,N \big\}
    =: v (\upsilon) \, .
  \end{align*}
  Here, the step marked with $(\lozenge)$ used that each $\Phi_i$ is
  $k$-admissible with control weight $v_i$, and the step marked with
  $(\ast)$ used that $v_i$ is radially increasing
  and that $|\upsilon_i| \leq |\upsilon|$.

  It is clear that $v$ is continuous and radially increasing.
  Furthermore,
  \[
    v_i (|\tau + \upsilon|e_1^{(d_i)})
    \leq v_i (|\tau|e_1^{(d_i)} + |\upsilon|e_1^{(d_i)})
    \leq v_i (|\tau|e_1^{(d_i)}) \cdot v_i (|\upsilon|e_1^{(d_i)})
    \leq v(\tau) \cdot v(\upsilon)
  \]
  for all $i=1,\dots,N$, which implies that $v$ is submultiplicative.
  In a similar manner, one can show that $\tilde{\kappa}$ is continuous, radially increasing,
  and submultiplicative.

  Finally, to see that $\kappa_ \Phi$ is $\tilde{\kappa}$-moderate,
  note for $\tau = (\tau_1,\dots,\tau_N) \in \RR^d$ and
  $\upsilon = (\upsilon_1, \dots, \upsilon_N) \in \RR^d$ that
  \begin{align*}
    (\kappa_ \Phi) (\tau + \upsilon)
    & = \prod_{i = 1}^N
          (\kappa_i \circ \Phi_i^{-1}) (\tau_i + \upsilon_i)
      \leq \prod_{i=1}^N
             \left[
               (\kappa_i \circ \Phi_i^{-1}) (\tau_i)
               \cdot \tilde{\kappa}_i (\upsilon_i)
             \right] \\
    & \leq \Big(
             \prod_{i=1}^N
               \tilde{\kappa}_i (|\upsilon|)
           \Big)
           \cdot
           (\kappa_ \Phi) (\tau)
      = \tilde{\kappa} (\upsilon) \cdot (\kappa_ \Phi) (\tau) \, .
  \end{align*}
  
  To finish the proof, we will now show that $\CalQ^{(\delta,r)}_{\Phi,\otimes}$ is equivalent to $\CalQ^{(\delta,r)}_{\Phi}$. For clarity, we will indicate the dimensionality of an open ball by a superscript in the remainder of the proof, e.g., $B_r^{(d)}(\xi)$ denotes the $d$-dimensional open ball of radius $r>0$ around $\xi\in\RR^d$. It is easy to see that
  \[\Phi(Q^{(\delta,r)}_{\Phi,\ell}) = B_{\delta r}^{(d)}(\delta \ell) \subset B_{\delta r}^{(d_1)}(\delta \ell_1) \times \ldots \times B_{\delta r}^{(d_N)}(\delta \ell_N) = \Phi(Q^{(\delta,r)}_{\Phi,\otimes,\ell}),\ \text{for all } \ell\in\ZZ^d,\]
  such that $\CalQ^{(\delta,r)}_{\Phi}$ is subordinate to $\CalQ^{(\delta,r)}_{\Phi,\otimes}$, in particular, it is almost subordinate. On the other hand, 
  \[
   \Phi(Q^{(\delta,r)}_{\Phi,\otimes,\ell}) = B_{\delta r}^{(d_1)}(\delta \ell_1) \times \ldots \times B_{\delta r}^{(d_N)}(\delta \ell_N) \subset B_{\delta \cdot \sqrt{d}r}^{(d)}(\delta \ell) = \Phi(Q^{(\delta,\sqrt{d}r)}_{\Phi,\ell}),\ \text{for all } \ell\in\ZZ^d,
   \]
  i.e., $\CalQ^{(\delta,r)}_{\Phi,\otimes}$ is subordinate to $\CalQ^{(\delta,\sqrt{d}r)}_{\Phi}$ and therefore almost subordinate to $\CalQ^{(\delta,r)}_{\Phi}$, which is equivalent to $\CalQ^{(\delta,\sqrt{d}r)}_{\Phi}$ by Lemma \ref{lem:DecompInducedCoveringIsNice}.
\end{proof}

We are now ready to show that the spaces of dominating smoothness of Besov type
can be realized as certain warped coorbit spaces.

\begin{theorem}\label{prop:DominatingMixedSmoothnessAsDecomposition}
  Let $1\leq p,q \leq\infty$, $s \in \RR^d$, and $\vrho_{1}$ any slow start version of 
  the weakly $k$-admissible radial component $\vsig_{1} = \ln(1+\bullet)$. 
  If $\Phi_{\vrho_{1}}: \RR \rightarrow \RR$ is the $1$-dimensional radial warping function 
  associated with $\vrho_{1}$ and 
  \[\Phi: \RR^d \rightarrow \RR^d,\ \xi\mapsto (\Phi_{\vrho_{1}}(\xi_1),\ldots,\Phi_{\vrho_{1}}(\xi_d)),\]
  then with 
  \[
    \kappa :
    \RR^d \to \RR^+,
    \xi = (\xi_1,\dots,\xi_d)
    \mapsto \prod_{i=1}^d
              (1 + |\xi_i|)^{s_i + 2^{-1} - q^{-1}} \, ,
  \]
  we have
  \[
    S^s_{p,q} B (\RR^d)
    = \DecompSp(\CalB_d, \lebesgue^p, \ell_{v^{(s)}}^q)
    = \Co (\Phi, \lebesgue^{p,q}_\kappa)\, ,
  \]
  up to canonical identifications. Here, $S^s_{p,q} B (\RR^d)$ and $v^{(s)}$ are as in Definition \ref{def:SpacesOfDomMixedSmoothness}, 
\end{theorem}

\begin{proof}
  Combining Remark \ref{rem:BdProperties} and Proposition \ref{cor:DominatingMixedSmoothnessIsDecomposition}, we have already shown that $S^s_{p,q} B (\RR^d) = \DecompSp(\CalB_d, \lebesgue^p, \ell_{v^{(s)}}^q)$ up to canonical identifications.
  
  It remains to prove that the decomposition space $\DecompSp(\CalB_d, \lebesgue^p, \ell_{v^{(s)}}^q)$ coincides with the
  warped coorbit space $\Co (\Phi, \lebesgue^{p,q}_\kappa)$. We have already discussed in Section~\ref{sec:WarpedCoorbitAndBesov}, that the warping function $\Phi_{\vrho_{1}} : \RR \to \RR$ is smooth and $d+1$-admissible and that $Q_{\Phi_{\vrho_{1}},k}^{(\delta,r)}$ is equivalent to the $1$-dimensional, inhomogeneous Besov covering $\CalB$, for any $\delta>0$ and $r > \sqrt{d} / 2$. 
  
  Furthermore, Lemma~\ref{lem:RadialWarpingStandardWeightAdmissible}
  shows that if we set
  $\kappa_i : \RR^d \to \RR^+,
              \xi \mapsto (1 + |\xi|)^{s_i + 2^{-1} - q^{-1}}$,
  then $\kappa_i \circ \Phi_{\vrho_{1}}$ is $\tilde{\kappa}_i$-moderate for some
  continuous, submultiplicative, radially increasing weight
  $\tilde{\kappa}_i : \RR^d \to \RR^+$. Noting that $\kappa(\xi) = \prod_{i=1}^d \kappa_i(\xi_i)$, we conclude that Lemma~\ref{lem:SeperableWarping} shows
  that $(\Phi,\kappa)$ is a compatible pair, see Definition~\ref{def:StandingDecompositionAssumptions}.
  Thus, Theorem~\ref{thm:CoorbitDecompositionIsomorphism} shows that
  $\Co(\Phi, \lebesgue_\kappa^{p,q})
   = \DecompSp (\CalQ_\Phi^{(\delta,r)}, \lebesgue^p, \ell_{u}^q)$
  up to canonical identifications, with
  $u_k = \kappa_\Phi(\delta k) \cdot [w(\delta k)]^{q^{-1} - 2^{-1}}$. Here, by Equation \eqref{eq:AssociatedProductWeight} in the proof of Lemma \ref{lem:SeperableWarping}, with $w_1 = \det(\Phi^{-1}_{\vrho_1}(\bullet))$, the weight $w$ is given by $w(\tau) = \prod_{i=1}^d w_1(\tau_i)$.
  
  As a direct consequence of Equation \eqref{eq:BesovWarpingEquivalenceCondition} in the proof of Theorem \ref{thm:LnWarpingBesovEmbeddingsExplicit}, we can use Proposition \ref{prop:RadialCoveringSubordinateness} and Lemma \ref{lem:RadialWarpingBesovEquivalence} to obtain weak equivalence of $\CalB$ and $\CalQ_{\Phi_{\vrho_1}}^{(\delta,r)}$ and by extension of $\CalB_d$ and the product covering $\CalQ_{\Phi,\otimes}^{(\delta,r)}$ defined in Lemma \ref{lem:SeperableWarping}. By that lemma, $\CalQ_{\Phi,\otimes}^{(\delta,r)}$ is further equivalent to the usual $\Phi$-induced $(\delta,r)$-fine frequency covering $\CalQ_{\Phi}^{(\delta,r)}$, such that in total, $\CalB_d$ is weakly equivalent to $\CalQ_{\Phi}^{(\delta,r)}$. 

  Hence, Theorem \ref{thm:decomposition_space_coincidence} shows that
  $\DecompSp(\CalQ_\Phi^{(\delta,r)}, \lebesgue^p, \ell_u^q)
   = \DecompSp(\CalB_d, \lebesgue^p, \ell_{v^{(s)}}^q)$, provided that
  \begin{equation}
    u_k \asymp v_\ell^{(s)}
    \quad \text{for all} \quad \ell \in \NN_0^d \text{ and } k \in \ZZ^d
          \text{ satisfying } Q_{\Phi,k}^{(\delta,r)} \cap B_{d,\ell} \neq \emptyset
    \, .
    \label{eq:DominatingMixedSmoothnessDesiredWeightEquivalence}
  \end{equation}
  Using $Q_{\Phi,k}^{(\delta,r)} \subset Q_{\Phi,\otimes,k}^{(\delta,r)}$, we can instead show that \eqref{eq:DominatingMixedSmoothnessDesiredWeightEquivalence} holds for $\ell,k$ such that $Q_{\Phi,\otimes,k}^{(\delta,r)} \cap B_{d,\ell} \neq \emptyset$. Since $\xi \in Q_{\Phi,\otimes,k}^{(\delta,r)} \cap B_{d,\ell}$ is equivalent to $\xi_i \in Q_{\Phi_{\vrho_{1}},k_i}\cap B_{\ell_i}$ and both $u$ and $v^{(s)}$ have product structure, we can verify $u_k \asymp v_\ell^{(s)}$ componentwise. However, 
  \[
   (1+|\xi_i|)^{s_i+2^{-1}-q^{-1}} \cdot [w(\delta k)]^{q^{-1}-2^{-1}} \asymp (1+|\xi_i|)^{s_i} \asymp 2^{s_i \ell_i}
  \]
  is a direct consequence of Equations \eqref{eq:RadialWarpingsCoveringNormControl} and \eqref{eq:RadialWarpingCoveringRelativeModerateness}, after noting that $1 + (\vrho_{1})_\ast(|\tau|) = (\vrho_{1})_\ast'(|\tau|) = e^{|\tau|}$. Hence, we can apply Theorem \ref{thm:decomposition_space_coincidence} to obtain
  \[
    \Co(\Phi, \lebesgue_\kappa^{p,q})
   = \DecompSp(\CalQ_\Phi^{(\delta,r)}, \lebesgue^p, \ell_u^q)
   = \DecompSp(\CalB_d, \lebesgue^p, \ell_{v^{(s)}}^q) = S^s_{p,q} B (\RR^d).
  \]
  as desired. 
\end{proof}

\begin{remark}
Using Lemma  \ref{lem:SeperableWarping} and the equality $\Co(\Phi, \lebesgue_\kappa^{p,q}) = S^s_{p,q} B (\RR^d)$, it is straightforward to extend 
Definition \ref{def:SpacesOfDomMixedSmoothness} to obtain a larger family of anisotropic 
smoothness spaces. In particular, choosing $\alpha\in (-\infty,1]^d$ and with $\vsig_{{\alpha_i}} = (1+\bullet)^{(1-\alpha_i)}$ if $\alpha_i\in(-\infty,1)$, and $\vsig_{1} = \ln(1+\bullet)$ otherwise, we can define \emph{anisotropic $\alpha$-modulation spaces} as $\Co(\Phi_{\alpha}, \lebesgue_\kappa^{p,q})$ with an appropriate choice of $\kappa$ and $\Phi_{\vrho_{\alpha}} = \Phi_{\vrho_{{\alpha_1}}}\otimes\ldots\otimes \Phi_{\vrho_{{\alpha_d}}}$, where each $\vrho_{{\alpha_i}}$ is a slow start version of $\vsig_{{\alpha_i}}$, and the radial warping function $\Phi_{\vrho_{{\alpha_i}}}: \RR\rightarrow \RR$ is defined as usual.

A rather straightforward adaptation of Proposition \ref{prop:ModulationIntoLnWarpingEmbedding} using the considerations in the proof of Theorem \ref{prop:DominatingMixedSmoothnessAsDecomposition} yields embeddings between these \emph{anisotropic $\alpha$-modulation spaces} and spaces of dominating mixed smoothness.
\end{remark}

\section{Conclusion}
 
In this paper we have shown that warped coorbit spaces $\Co(\Phi,\bd L^{p,q}_\kappa)$ can be identified with certain decomposition spaces. This has allowed us to derive embedding relations between warped coorbit spaces and study the relations between these spaces and semi-classical smoothness spaces. In particular, we have shown that certain warped coorbit spaces coincide with $\alpha$-modulation spaces and (Besov-type) spaces of dominating mixed smoothness. For Besov spaces spaces, we have shown that such equality is only possible when considering functions on the real line. For Besov spaces of functions on $\RR^d$ with $d>1$, we have proposed a warping function that achieves a certain \emph{closeness} relation between the associated warped coorbit spaces and Besov spaces. 

\appendix

  \section{\texorpdfstring{$\alpha$-coverings for general $\alpha \in \RR$}
                        {α-coverings for general α∈ℝ}}
\label{sec:AlphaCoverings}

In this appendix, we collect two secondary facts about $\alpha$-coverings.
As our first result, we show that any two $\alpha$-coverings, with respect to the same $\alpha\in\RR$, of $\RR^d$ are weakly equivalent.
This generalizes \cite[Lemma B.2]{BorupNielsenAlphaModulationSpaces}, where
the same observation is made for $\alpha \in [0,1]$.
As our second result, we show that there are no $\alpha$-coverings for
$\alpha > 1$, such that $\alpha\in (-\infty,1]$, as considered in Section \ref{sec:alpha}, is
indeed the most general setting. For ease of reference, we recall the definition of an $\alpha$-covering, where, for now, we only assume $\alpha\in\RR$.

\begin{definition}\label{def:AlphaCovering}
   Let $\CalQ = (Q_i)_{i \in I}$ be an admissible covering of $\RR^d$,
   and let $\alpha \in \RR$.
  The family $\CalQ$ is called an $\alpha$-covering of $\RR^d$, if the
  following hold:
  \begin{enumerate}
     \item\label{enu:AlphaCoveringInnerOuterA} With
            \begin{equation}
            \quad\quad\quad
            r_Q
            := \sup \{
                      r > 0 \,:\, \exists \, \xi \in \RR^d : B_r (\xi) \subset Q
                    \}
            \quad \!\! \text{and} \quad \!\!
            R_Q
            := \inf \{
                      R > 0 \,:\, \exists \, \xi \in \RR^d : Q \subset B_R (\xi)
                    \}
            \, ,
            \label{eq:InnerOuterRadiusDefintionA}
          \end{equation}
          we have
          $\sup_{i \in I} R_{Q_i} / r_{Q_i} < \infty$.

    \item\label{enu:AlphaCoveringMeasureA} We have $\mu(Q_i) \asymp (1 + |\xi|)^{\alpha \cdot d}$
          for all $i \in I$ and $\xi \in Q_i$,
          where the implied constant does not depend on $i,\xi$.
  \end{enumerate}
\end{definition}

Below, we will repeatedly use the following consequence of the definition of
$\alpha$-coverings: If $\CalQ = (Q_i)_{i \in I}$ is an $\alpha$-covering,
then $r_{Q_i}^d \lesssim \mu(Q_i) \lesssim R_{Q_i}^d \lesssim r_{Q_i}^d$,
and hence
\begin{equation}
    (1 + |\xi|)^{\alpha d}
    \asymp \mu(Q_i)
    \asymp R_{Q_i}^d
    \asymp r_{Q_i}^d
    \qquad \forall \, i \in I \text{ and } \xi \in Q_i \, .
    \label{eq:AlphaCoveringAsymptotic}
  \end{equation}

The next result shows that the specific $\alpha$-covering chosen to define $M^{s,\alpha}_{p,q}(\RR^d)$ is inconsequential, by Theorem \ref{thm:decomposition_space_coincidence}.

\begin{lemma}\label{lem:AlphaCoveringEquivalence}
  Let $\alpha \in \R$, and let $\CalQ$ and $\CalP$ be two $\alpha$-coverings
  of $\RR^d$. Then $\CalQ$ and $\CalP$ are weakly equivalent.
\end{lemma}

\begin{proof}
  Before we properly start the proof, we observe for any bounded set
  $Q \subset \RR^d$ that $|\xi - \eta| \leq 2 R_Q$ for all $\xi, \eta \in Q$.
  Indeed, for arbitrary $R > R_Q$, there is some
  $\zeta \in \RR^d$ with $Q \subset \zeta + B_R (0)$, and hence
  $|\xi - \eta| \leq |\xi - \zeta| + |\zeta - \eta| \leq 2R$.

  \medskip{}

  Now, write $\CalQ = (Q_i)_{i \in I}$ and $\CalP = (P_j)_{j \in J}$.
  By symmetry, it suffices to show that there is some $N \in \NN$ with
  $|I_j| \leq N$ for all $j \in J$, with
  $I_j = \{i \in I \,:\, Q_i \cap P_j \neq \emptyset \}$.

  To see this, fix $j \in J$ and $\eta \in P_j$.
  For any $i \in I_j$, there is some $\xi_i \in Q_i \cap P_j$.
  For arbitrary $\xi \in Q_i$, this implies\vspace{-0.2cm}
  \begin{align*}
    |\xi - \eta|
    & \leq |\xi - \xi_i| + |\xi_i - \eta|
      \leq 2 (R_{Q_i} + R_{P_j}) \\
    ({\scriptstyle{\text{Eq.~} \eqref{eq:AlphaCoveringAsymptotic}
                   \text{ and } \xi_i \in Q_i \cap P_j}})
    & \lesssim (1 + |\xi_i|)^{\alpha} \\
    ({\scriptstyle{\xi_i \in P_j \text{ and } \eta \in P_j}})
    & \lesssim (1 + |\eta|)^{\alpha} \, .
  \end{align*}
  Therefore, $Q_i \subset B_{C (1 + |\eta|)^{\alpha}}(\eta)$ for a constant
  $C = C(d,\alpha,\CalQ,\CalP)$ independent of $i,j$.
  Since we also have $\sum_{i \in I_j} \Indicator_{Q_i} \lesssim 1$ by the
  admissibility of $\CalQ$, we see $\sum_{i \in I_j} \Indicator_{Q_i}
  \lesssim \Indicator_{B_{C (1 + |\eta|)^\alpha} (\eta)}$, and hence
  \begin{align*}
    (1 + |\eta|)^{\alpha d}
    & \gtrsim \mu (B_{C (1+|\eta|)^{\alpha}} (\eta))
      \gtrsim \int \sum_{i \in I_j} \Indicator_{Q_i} \, dx
      = \sum_{i \in I_j} \mu (Q_i) \\
    ({\scriptstyle{\text{Eq. } \eqref{eq:AlphaCoveringAsymptotic}}})
    & \gtrsim \sum_{i \in I_j}  (1 + |\xi_i|)^{\alpha d} \\
    ({\scriptstyle{\text{Eq. } \eqref{eq:AlphaCoveringAsymptotic}
                   \text{ and } \xi_i, \eta \in P_j}})
    & \gtrsim \sum_{i \in I_j}  (1 + |\eta|)^{\alpha d}
    = |I_j| \cdot (1 + |\eta|)^{\alpha d} \, ,
  \end{align*}
  which implies $|I_j| \lesssim 1$, as desired.
\end{proof}

\begin{lemma}\label{lem:AlphaCoveringAlphaLargerOne}
  Let $\alpha > 1$ and $d \in \NN$ be arbitrary.
  Then there does not exist an $\alpha$-covering of $\RR^d$.
\end{lemma}

\begin{proof}
  Assume towards a contradiction that $\CalQ = (Q_i)_{i \in I}$ is such a
  covering. As seen in Equation~\eqref{eq:AlphaCoveringAsymptotic}, there is
  a constant $c > 0$ with $r_{Q_i} \geq 2 c \cdot (1 + |\eta|)^{\alpha}$
  for all $i \in I$ and $\eta \in Q_i$.

  By definition of $r_Q$, there is for each $i \in I$ some $\xi_i \in \RR^d$
  such that
  \[
    \xi_i + B_{c \, (1 + |\eta|)^{\alpha}} (0)
    \subset \xi_i + B_{r_{Q_i} / 2} \, (0)
    \subset Q_i
    \qquad \forall \, i \in I \text{ and } \eta \in Q_i \, .
  \]
  Next, since $\CalQ$ covers all of $\RR^d$, there is for each $n \in \NN$ some
  $i_n \in I$ satisfying $\eta_n := (n,0,\dots,0) \in Q_{i_n}$.
  Therefore,
  \[
    \xi_{i_n} + B_{c \, \cdot n^\alpha} (0)
    \subset \xi_{i_n} + B_{c \, (1 + |\eta_n|)^{\alpha}} \, (0)
    \subset Q_{i_n} \, .
  \]
  It is not hard to see that there is some
  $\zeta_n \in B_{c \cdot n^\alpha} (0)$ satisfying $|\xi_{i_n} + \zeta_n|
  = |\xi_{i_n}| + \frac{c}{2} \cdot n^\alpha \gtrsim n^\alpha$.
  However, since $\xi_{i_n} + \zeta_n \in Q_{i_n}$, $\eta_n \in Q_{i_n}$,
  and $\alpha d \neq 0$, Equation~\eqref{eq:AlphaCoveringAsymptotic}
  implies
  \[
    n \asymp 1 + |\eta_n|
      \asymp 1 + |\xi_{i_n} + \zeta_n|
      \gtrsim n^\alpha \, .
  \]
  Since this holds for all $n \in \NN$, with the implicit constant being
  independent of $n$, and since $\alpha > 1$, this yields the desired
  contradiction.
\end{proof}

\section{Products of semi-structured coverings}
\label{sec:semistructuredproducts}

For the sake of completeness, we prove that the product of several
(tight) semi-structured coverings is again (tight and) semi-structured, that subordinateness relations are inherited by the product covering and moderateness relations between coverings and weights are inherited by the product covering and a separable product of the weights.

\begin{lemma}\label{lem:ProductCovering}
  For each $\ell = 1,\dots,N$, let
  $\CalQ_\ell = (Q_i^{(\ell)})_{i \in I_\ell}
           = (T_i^{(\ell)} P_i^{(\ell)} + b_i^{(\ell)})_{i \in I_\ell}$
  be a semi-structured covering of an open set $\CalO_\ell \subset \RR^{d_\ell}$.
  Let $I := I_1 \times \cdots \times I_N$ and
  $d := d_1 + \dots + d_N$, and for $i = (i_1,\dots,i_N) \in I$ set
  \[
    T_i := \mathrm{diag} \big(
                           T_{i_1}^{(1)},
                           \dots,
                           T_{i_N}^{(N)}
                         \big) \in \GL(\RR^d) \, ,
    \quad
    b_i := \big(b_{i_1}^{(1)}, \dots, b_{i_N}^{(N)} \big) \in \RR^d \, ,
    \quad
    Q_i ' := P_{i_1}^{(1)}
             \times \cdots \times
             P_{i_N}^{(N)} \subset \RR^d \, .
  \]
  Then the following hold:
  \begin{enumerate}
  \item The covering
   $ \CalQ_1 \otimes \cdots \otimes \CalQ_N
    := \big(
         Q_{i_1}^{(1)} \times \cdots \times Q_{i_N}^{(N)}
       \big)_{i \in I}
     = \big(
         T_i \, Q_i ' + b_i
    \big)_{i \in I}
   $
  is a semi-structured covering of
  $\CalO := \CalO_1 \times \cdots \times \CalO_N \subset \RR^d$. If the $\CalQ_\ell$, $\ell = 1,\dots,N$,
  are tight, then so is $\CalQ_1 \otimes \cdots \otimes \CalQ_N$.

  \item If $v^{(\ell)}$ is a $\CalQ_\ell$-moderate
  weight, for $\ell = 1,\dots,N$, then the weight $v = (v_i)_{i \in I}$, with
  $v_i = \prod_{\ell=1}^N v_{i_\ell}^{(\ell)}$ is
  $\CalQ_1 \otimes \cdots \otimes \CalQ_N$-moderate.

  \item If $\CalQ_\ell = (Q_i^{(\ell)})_{i \in I_\ell}$
  and $\CalP_\ell = (P_j^{(\ell)})_{j \in J_\ell}$, for $\ell = 1,\dots,N$, are admissible coverings of
  $\CalO_\ell \subset \RR^{d_\ell}$, and if, for all $\ell$, $\CalQ_\ell$ is weakly subordinate to
  $\CalP_\ell$, then $\CalQ_1 \otimes \cdots \otimes \CalQ_N$ is
  weakly subordinate to $\CalP_1 \otimes \cdots \otimes \CalP_N$. If the $\CalQ_\ell$ are even almost subordinate to $\CalP_\ell$, then $\CalQ_1 \otimes \cdots \otimes \CalQ_N$ is
  almost subordinate to $\CalP_1 \otimes \cdots \otimes \CalP_N$.
  \end{enumerate}
\end{lemma}
\begin{proof}
  For brevity, let $\CalQ := \CalQ_1 \otimes \cdots \otimes \CalQ_N$.
  Clearly, $\CalQ$ is indeed a covering of $\CalO$.
  To prove the semi-structuredness of $\CalQ$, first note that for each
  $\ell = 1,\dots,N$, there is some $R_\ell > 0$ such that
  $Q_i^{(\ell)} \subset \overline{B_{R_\ell} (0)}$ for all $i \in I_\ell$.
  If we set $R := \sqrt{N} \cdot \max \{R_\ell \,:\, \ell=1,\dots,N\}$,
  then this implies $Q_i ' \subset \overline{B_R (0)}$ for all $i \in I$.

  Next, note for $i = (i_1, \dots, i_N) \in I$ and
  $j = (j_1, \dots, j_N) \in I$ that
  \[
    i \in j^\ast
    \quad \Longleftrightarrow \quad
    Q_i \cap Q_j \neq \emptyset
    \quad \Longleftrightarrow \quad
    \forall \, \ell = 1,\dots,N \, : Q_{i_\ell} \cap Q_{j_\ell} \neq \emptyset
    \quad \Longleftrightarrow \quad
    i \in \prod_{\ell=1}^N j_\ell^\ast \, .
  \]
  On the one hand, this shows
  $M_{\CalQ}
   = \sup_{i \in I} |i^\ast|
   = \sup_{i \in I} \prod_{\ell=1}^N |i_\ell^\ast|
   \leq \prod_{\ell=1}^N M_{\CalQ_\ell} < \infty$,
  so that $\CalQ$ is admissible.
  On the other hand, we see that if $i \in j^\ast$, then
  \begin{align*}
    \|T_i^{-1} T_j\|
    & = \Big\|
          \mathrm{diag} \big(
                          (T_{i_1}^{(1)})^{-1} \, T_{j_1}^{(1)},
                          \dots,
                          (T_{i_N}^{(N)})^{-1} \, T_{j_N}^{(N)}
                        \big)
        \Big\| \\
    & \leq \max \left\{
                  \sup_{\substack{m, n \in I_\ell
                        \text{ and }\\
                        Q_m^{(\ell)} \cap Q_n^{(\ell)} \neq \emptyset}}
                    \|(T_m^{(\ell)})^{-1} T_n^{(\ell)}\|
                  \,:\,
                  \ell = 1,\dots,N
                \right\}
    \leq \max \{ C_{\CalQ_\ell} \,:\, \ell = 1,\dots,N \} \, ,
  \end{align*}
  where $C_{\CalQ_\ell}$ is the constant $C>0$ in Definition \ref{def:SemiStructuredCoverings}(2), for $\CalQ = \CalQ_\ell$. Hence, $\CalQ$ is a semi-structured covering of $\CalO$.

  We further see that $v$ is $\CalQ$-moderate,
  since if $Q_i \cap Q_j \neq \emptyset$, then
  \[
    \frac{v_i}{v_j}
    = \prod_{\ell=1}^N
        \frac{v_{i_\ell}^{(\ell)}}{v_{j_\ell}^{(\ell)}}
    \leq \prod_{\ell=1}^N C_{v^{(\ell)}, \CalQ_\ell} < \infty \, .
  \]

  If additionally, all the coverings $\CalQ_\ell$ are tight, then there is for each
  $\ell = 1,\dots,N$ some $\eps_\ell > 0$ and for each $i \in I_\ell$ some
  $c_i^{(\ell)} \in \RR^{d_\ell}$ satisfying $B_{\eps_\ell} (c_i^{(\ell)}) \subset Q_i^{(\ell)}$.
  Set $c_i := (c_{i_1}^{(1)},\dots,c_{i_N}^{(N)}) \in \RR^d$
  for $i = (i_1,\dots,i_N) \in I$ and furthermore
  $\eps := \min \{\eps_\ell \,:\, \ell =1,\dots,N\}$. It follows that
  $B_{\eps} (c_i) \subset Q_i '$ for all $i \in I$, i.e., $\CalQ$ is tight.

  \medskip{}

  Finally, we prove the last part of the lemma. If we define
  $J_i^{(\ell)} := \{ j \in J_\ell \,:\, P_j^{(\ell)} \cap Q_i^{(\ell)} \neq \emptyset \}$
  for $i \in I_\ell$, then there is some $M \in \NN$ such that
  $|J_i^{(\ell)}| \leq M$ for all $i \in I_\ell$ and $\ell = 1,\dots,N$,
  since $\CalQ_\ell$ is weakly subordinate to $\CalP_\ell$.

  Now, set $I := I_1 \times \cdots \times I_N$ and
  $J := J_1 \times \cdots \times J_N$.
  Furthermore, let $Q_i := Q_{i_1}^{(1)} \times\cdots\times Q_{i_N}^{(N)}$
  and $P_j := P_{j_1}^{(1)} \times \cdots \times P_{j_N}^{(N)}$ for
  $i = (i_1, \dots, i_N) \in I$ and $j = (j_1, \dots, j_N) \in J$.
  Then
  \[
\begin{split}
    J_i
    := \{ j \in J \,:\, P_j \cap Q_i \neq \emptyset \}
    & = \big\{
        j = (j_1, \dots, j_N) \in J
        \,:\,
        \text{for all } \ell=1,\dots,N \,:\,
          P_{j_\ell}^{(\ell)} \cap Q_{i_\ell}^{(\ell)} \neq \emptyset
      \big\}\\
    & = J_{i_1}^{(1)} \times \cdots \times J_{i_N}^{(N)} \, ,
    \end{split}
  \]
  and hence $|J_i| = |J_{i_1}^{(1)}| \cdots |J_{i_N}^{(N)}| \leq M^N$
  for arbitrary $i = (i_1,\dots,i_N) \in I$.
  Hence, $\CalQ_1 \otimes \cdots \otimes \CalQ_N$
  is weakly subordinate to $\CalP_1 \otimes \cdots \otimes \CalP_N$.

  If $\CalQ_\ell$ is almost subordinate to $\CalP_\ell$, $\ell = 1,\dots,N$, i.e., there is $M_\ell\in \NN$ such that for all $i_\ell\in I_\ell$ there is a $j_\ell=j_{i,\ell}\in J_\ell$,
  such that $Q^{(\ell)}_i\subset (P^{(\ell)}_j)^{N_\ell\ast}$, then with $M = \max_\ell M_\ell$, $Q^{(\ell)}_{(i_1,\cdot,i_N)}\subset (P^{(\ell)}_{(j_1,\ldots,j_N)})^{M\ast}$ follows.
  \end{proof}

  \section{The structured $\Phi$-induced $(\delta,r)$-fine frequency covering}
\label{appendix:StructuredCovering}

In this final appendix, we demonstrate that every warping function $\Phi$ induces
a \emph{structured} $(\delta,r)$-fine frequency covering, for $r\in (0,\vartheta_0)$ and $\delta\in (0,\Delta_r)$, with $\vartheta_0>0$ as in Lemma \ref{lem:CoveringLinearization} and a certain $\Delta_r>0$ .

\begin{definition}\label{def:StructuredCovering}
    A semi-structured covering $\CalQ = (Q_i)_{i \in I}$ of $\CalO$, with
    $Q_i = T_iQ_i' + b_i$, $i\in I$, as in Definition \ref{def:SemiStructuredCoverings}(2)
    is called \emph{structured}, if the operators $T_i$ can be chosen
     such that $Q_i' = Q$, for all $i\in I$, and there is an open set
    $P\subset \RR^d$ with $\overline{P}\subset Q$ and
    \[
      \CalO = \bigcup_{i \in I} (T_i P + b_i).
    \]
\end{definition}

   It should be noted that every structured admissible covering $\CalQ$ is a decomposition covering,
   that is, there is a $\CalQ$-BAPU, see \cite[Proposition 1]{boni07}. While the covering $\CalQ_{\Phi}^{(\delta, r)}$ is in general not structured, we can define an equivalent family of structured coverings using Lemma \ref{lem:CoveringLinearization}. Thus, for the purpose of investigating the associated decomposition spaces,
   it does not matter which of the two families of coverings is considered, see Theorem \ref{thm:decomposition_space_coincidence}. However, the structured covering may simplify the derivation of certain results in future work.

 \begin{definition}\label{def:StructuredFrequencyCovering}
   Let $\Phi : D \to \RR^d$ be a diffeomorphism.
   For $\delta, r > 0$, we define
   \[
     \CalS_{\Phi}^{(\delta,r)}
     := \left( S_{\Phi,k}^{(\delta, r)} \right)_{k \in \ZZ^d}
     := \Big(
          \Phi^{-1}(\delta k) + D\Phi^{-1}(\delta k) \langle B_{r} (0) \rangle
        \Big)_{k \in \ZZ^d} \, .
   \]
   The family $\CalS_{\Phi}^{(\delta,r)}$ is called the
   \emph{structured $(\delta, r)$-fine frequency covering induced by $\Phi$}.
 \end{definition}

 We proceed to show that
 $\CalQ_{\Phi}^{(\eps, \rho)}$ and $\CalS_{\Phi}^{(\delta,r)}$
 are equivalent, for suitable choices of the parameters.

 \begin{lemma}\label{lem:StructuredFrequencyCoveringEquivalence}
   Let $\Phi : D \to \RR^d$ be a $1$-admissible warping function
   with control weight $v_0$ and set $\vartheta_0 := (2d v_0(1))^{-1}$.
   If $r\in (0, \vartheta_0/4)$ and $\delta\in \big( 0,\min\{1,4r\}(2\sqrt{d} \cdot v_0(1/2))^{-1} \big)$,
   then $\CalS_{\Phi}^{(\delta,r)}$ is a structured admissible covering of $D$.

   \medskip{}

   Moreover, $\CalS_{\Phi}^{(\delta,r)}$ is equivalent to
   $\CalS_{\Phi}^{(\delta',r')}$ for any $r' \in (0, \vartheta_0/4)$ and $\delta'\in \big( 0,\min\{1,4r'\}(2\sqrt{d} \cdot v_0(1/2))^{-1} \big)$, and to the coverings
   $\CalQ_{\Phi}^{(\eps, \rho)}$, for any $\eps > 0$, $\rho > \sqrt{d}/2$.
 \end{lemma}
 \begin{proof}
   Let $r \in (0, \vartheta_0/4)$ be arbitrary, such that $\vartheta := 4r \in (0,\vartheta_0)$ and note that
   $\vartheta_0$ is as in Lemma \ref{lem:CoveringLinearization}. Hence,
   Equation \eqref{eq:covering_linearization} of said lemma, with $\xi = \Phi^{-1}(\delta k)$,
   yields
   \begin{equation}
     \begin{split}
           S_{\Phi,k}^{(\delta, r)}
       &=      \xi + D\Phi^{-1} (\Phi(\xi)) \langle B_{\vartheta/4} (0) \rangle
       \subset \Phi^{-1} \big( \Phi(\xi) + B_{\vartheta} (0) \big) \\
       &=       \Phi^{-1} \big(\delta k + B_{4 r} (0)\big)
       =      \Phi^{-1} \big(\delta \cdot B_{4 r / \delta} (k)\big) \\
       &\subset \Phi^{-1} \big(\delta \cdot B_{\rho^\ast} (k)\big)
       =       Q_{\Phi,k}^{(\delta, \rho^\ast)} \subset D \, ,
       \qquad \text{ whenever } \qquad \rho^\ast \geq 4 r / \delta.
     \end{split}
     \label{eq:StructuredFrequencyCoveringSubordinateToFrequencyCovering}
   \end{equation}
   In addition to showing that the family $\CalS_{\Phi}^{(\delta, r)}$
   is subordinate (and thus almost subordinate) to the covering
   $\CalQ_{\Phi}^{(\delta, \rho^\ast)}$ with arbitrary
   $\rho^\ast > \max\{\frac{\sqrt{d}}{2}, \frac{4r}{\delta} \}$,
   this also implies for all $k \in \Z^d$ that
   $S_{\Phi,k}^{(\delta, r)} \subset D$.
   Furthermore, Equation \eqref{eq:StructuredFrequencyCoveringSubordinateToFrequencyCovering}
   also shows that $\CalS_{\Phi}^{(\delta, r)}$ is admissible, since
    \[
     \left|
       \left\{
         \ell \in \ZZ^d
         \,:\,
         S_{\Phi,\ell}^{(\delta, r)} \cap S_{\Phi,k}^{(\delta, r)}
         \neq \emptyset
       \right\}
     \right|
     \leq \left|
            \left\{
              \ell \in \ZZ^d
              \,:\,
              Q_{\Phi,\ell}^{(\delta, \rho^\ast)} \cap Q_{\Phi,k}^{(\delta, \rho^\ast)}
              \neq \emptyset
            \right\}
          \right|
     \leq N_{\CalQ_{\Phi}^{(\delta, \rho^\ast)}} < \infty,
   \]
   for all $k \in \Z^d$, by admissibility of $\CalQ_{\Phi}^{(\delta, \rho^\ast)}$ (see Lemma \ref{lem:DecompInducedCoveringIsNice}).

   \medskip{}

   If $\rho > \sqrt{d} / 2$ is arbitrary,
   then, by performing the same derivations as in the proof of
   Lemma~\ref{lem:DecompInducedCoveringIsNice}, we can show that
   \begin{equation}
             Q_{\Phi,k}^{(\delta, \rho)}
             = \Phi^{-1}(\delta\cdot B_\rho(k))
     \subset \Phi^{-1}(\delta k)
             + A(\delta k)
                 \langle
                   B_{\delta \rho \cdot v_0(\delta\rho)} (0)
                 \rangle
     =       S_{\Phi,k}^{(\delta,\delta\rho \cdot v_0(\delta\rho))}
     \, .
     \label{eq:StructuredCoveringIsBig}
   \end{equation}
   Now, for fixed $r \in (0, \vartheta_0 / 4)$ and
   $0 < \delta < \frac{2}{\sqrt{d}} \frac{\min \{1,4r\}}{4\cdot v_0(1/2)}$
   as per our assumption, we can choose $\rho_0 > \sqrt{d}/2$, such that
   \[
     \delta\rho_0 < \frac{\min \{1,4r\}}{4\cdot v_0(1/2)} \leq 1/4.
   \]
   Hence, using that $v_0$ is radially increasing, we can estimate
   $\delta\rho_0$ by either $v_0(1/2)^{-1}r$ or simply $1/4$,
   such that $\delta \rho_0 \cdot v_0 (2\delta \rho_0)
    < \frac{r v_0 (1/2)}{v_0 (1/2)} \leq r$.
   Therefore, Equation \eqref{eq:StructuredCoveringIsBig} can be applied with
   $\rho = \rho_0$, yielding
   \[
     Q_{\Phi, k}^{(\delta, \rho_0)}
     \subset S_{\Phi,k}^{(\delta,\delta\rho_0 \cdot v_0(2\delta \rho_0))}
     \subset S_{\Phi,k}^{(\delta,r)}
     \overset{\text{Equation ~} \eqref{eq:StructuredFrequencyCoveringSubordinateToFrequencyCovering}}
             {\subset}
     D
     \qquad \forall \, k \in \Z^d \, .
   \]
   On the one hand, this proves that $\CalQ_\Phi^{(\delta,\rho_0)}$ is
   subordinate to $\CalS_\Phi^{(\delta,r)}$.
   On the other hand, since $\rho_0 > \sqrt{d}/2$,
   Lemma \ref{lem:WhenIsFrequencyCoveringCovering}
   shows $D = \bigcup_{k \in \Z^d} Q_{\Phi,k}^{(\delta,\rho_0)}
   \subset \bigcup_{k\in\ZZ^d} S_{\Phi,k}^{(\delta,r)} \subset D$, so that
   $\CalS_{\Phi}^{(\delta,r)}$ covers all of $D$.

   Moreover, since all coverings $\CalQ_{\Phi}^{(\eps,\rho)}$ with
   $\rho>\sqrt{d}/2$ and $\eps > 0$ are equivalent,
   since $\CalQ_\Phi^{(\delta,\rho_0)}$ is subordinate to
   $\CalS_\Phi^{(\delta,r)}$, and since we saw above that
   $\CalS_\Phi^{(\delta,r)}$ is subordinate to $\CalQ_\Phi^{(\delta,\rho^\ast)}$
   for $\rho^\ast > \max \{\frac{\sqrt{d}}{2}, \frac{4r}{\delta} \}$,
   we see that the coverings $\CalS_\Phi^{(\delta,r)}$ and
   $\CalQ_\Phi^{(\eps,\rho)}$ are equivalent, as long as
   $\rho > \sqrt{d} / 2$, $r \in (0, \vartheta_0 / 4)$, and
   $0<\delta < \frac{\min\{1,4r\}}{2\sqrt{d} \cdot v_0(1/2)}$.

   It remains to show that $\CalS_{\Phi}^{(\delta,r)}$ is indeed
   a structured covering.
   But setting $T_k = A(\delta k)$ and $b_k = \Phi^{-1}(\delta k)$, as well as
   $Q := B_r (0)$, we clearly have $S_{\Phi,k}^{(\delta,r)} = T_k Q + b_k$
   for all $k \in \Z^d$.
   Finally, if we fix $\rho^\ast > \max \{\frac{\sqrt{d}}{2}, \frac{4r}{\delta} \}$,
   and if $S_{\Phi,k}^{(\delta,r)} \cap S_{\Phi,\ell}^{(\delta,r)} \neq \emptyset$,
   then Equation  \eqref{eq:StructuredFrequencyCoveringSubordinateToFrequencyCovering}
   shows $\emptyset \neq Q_{\Phi,k}^{(\delta,\rho^\ast)} \cap Q_{\Phi,\ell}^{(\delta,\rho^\ast)}$,
   and hence $\|T_k^{-1} T_\ell\| \leq C$ for an absolute constant
   $C = C(\Phi,\delta,r) > 0$, see Lemma \ref{lem:DecompInducedCoveringIsNice}.

   Finally, if we choose
   $\tilde{r} \in (0, r) \subset (0, \vartheta_0/4)$ such that we still have
   $\delta < \frac{\min\{1,4\tilde{r}\}}{2\sqrt{d} \cdot v_0(1/2)}$,
   then it is clear that $P := B_{\tilde{r}}(0)$ satisfies $\overline{P} \subset Q$,
   and the preceding part of the proof shows that
   $(T_k P + b_k)_{k \in \Z^d}$ still covers $D$.
 \end{proof}

 Using the equivalence between the coverings $\CalQ_\Phi^{(\delta,\vrho)}$
 and $\CalS_\Phi^{(\delta,r)}$, we can now show that the associated
 decomposition spaces also coincide.

 \begin{corollary}\label{cor:StructuredDecompositionSpacesCoincide}
   With notation as in Lemma \ref{lem:StructuredFrequencyCoveringEquivalence},
   let $r \in (0, \vartheta_0 / 4)$, $s > \sqrt{d} / 2$, and
   $0 < \delta < \frac{\min\{1,4r\}}{2 \sqrt{d} \cdot v_0 (1/2)}$.

   If $u = (u_k)_{k \in \ZZ^d}$ is $\CalQ_{\Phi}^{(\delta,s)}$-moderate,
   then $u$ is also $\CalS_\Phi^{(\delta,r)}$-moderate, and we have
   \[
     \DecompSp(\CalQ_\Phi^{(\delta,s)}, \lebesgue^p, \ell_u^q)
     = \DecompSp(\CalS_\Phi^{(\delta,r)}, \lebesgue^p, \ell_u^q)
     \qquad \forall \, p,q \in [1,\infty] \, .
   \]
 \end{corollary}

 \begin{proof}
   Directly from the definition of the coverings $\CalQ_\Phi^{(\delta,s)}$
   and $\CalS_\Phi^{(\delta,r)}$, we see
   \[
     \Phi^{-1} (\delta k)
     \in S_{\Phi,k}^{(\delta,r)} \cap Q_{\Phi,k}^{(\delta,s)}
     \neq \emptyset
     \qquad \forall \, k \in \ZZ^d \, .
   \]
   Note that this derivation crucially uses that the \emph{sampling density}
   $\delta > 0$ of both coverings coincides.

   In Lemma~\ref{lem:StructuredFrequencyCoveringEquivalence}, we saw that
   the coverings $\CalQ_\Phi^{(\delta,s)}$ and
   $\CalS_\Phi^{(\delta,s)}$ are equivalent.
   In particular, there is a fixed $N \in \NN$ such that for each $k \in \ZZ^d$,
   there is some $j_k \in \ZZ^d$ satisfying $S_{\Phi,k}^{(\delta,r)}
   \subset Q_{\Phi,j_k^{N\ast}}^{(\delta,s)}$.
   But then
   \[
     \Phi^{-1}(\delta k)
     \in S_{\Phi,k}^{(\delta,r)} \cap Q_{\Phi,k}^{(\delta,s)}
     \subset  Q_{\Phi,j_k^{N\ast}}^{(\delta,s)}
             \cap Q_{\Phi,k}^{(\delta,s)} \, ,
   \]
   so that $k \in j_k^{(N+1)\ast}$, and hence $j_k \in k^{(N+1)\ast}$.
   But this implies $j_k^{N\ast} \subset k^{(2N+1)\ast}$, and hence
   \begin{equation}
     S_{\Phi,k}^{(\delta,r)}
     \subset Q_{\Phi,j_k^{N\ast}}^{(\delta,s)}
     \subset Q_{\Phi,k^{(2N+1)\ast}}^{(\delta,s)}
     \qquad \forall \, k \in \ZZ^d \, .
     \label{eq:StructuredCoveringSameSpacesMainInclusion}
   \end{equation}

   The inclusion
   \eqref{eq:StructuredCoveringSameSpacesMainInclusion} yields the
   $\CalS_\Phi^{(\delta,r)}$-moderateness of $u$, since if
   $\emptyset \neq S_{\Phi,k}^{(\delta,r)} \cap S_{\Phi,\ell}^{(\delta,r)}$,
   then $\emptyset \neq Q_{\Phi,k^{(2N + 1) \ast}}^{(\delta,s)}
                        \cap Q_{\Phi,\ell^{(2N + 1) \ast}}^{(\delta,s)}$,
   and thus $\ell \in k^{(4N+3)\ast}$, so that $u_\ell \asymp u_k$. Here, we used
   that $v_i \leq C^\ell \cdot v_j$ for $j \in i^{\ell\ast}$, whenever $\ell\in\NN$ and $v = (v_i)_{i \in I}$ is $\CalQ$-moderate, i.e.,
   $v_i \leq C \cdot v_j$ if $Q_i \cap Q_j \neq \emptyset$. This inequality can be verified
   by a straightforward induction.

   Likewise, the inclusion
   \eqref{eq:StructuredCoveringSameSpacesMainInclusion} also implies that
   $u_k \asymp u_\ell$ if $\emptyset
   \neq Q_{\Phi,k}^{(\delta,s)} \cap S_{\Phi,\ell}^{(\delta,r)}$, since
   then $\emptyset \neq Q_{\Phi,k}^{(\delta,s)}
   \cap Q_{\Phi,\ell^{(2N+1)\ast}}^{(\delta,s)}$,
   which implies $k \in \ell^{(2N+2)\ast}$, and hence
   $u_k \asymp u_\ell$.

   Overall, since the coverings $\CalQ_\Phi^{(\delta,s)}$ and
   $\CalS_\Phi^{(\delta,r)}$ are equivalent, and since $u_k \asymp u_\ell$
   if $Q_{\Phi,k}^{(\delta,s)} \cap S_{\Phi,\ell}^{(\delta,r)} \neq \emptyset$,
   Theorem~\ref{thm:decomposition_space_coincidence} shows for arbitrary
   $p,q \in [1,\infty]$ that
   $\DecompSp(\CalQ_{\Phi}^{(\delta,s)}, \lebesgue^p, \ell_u^q)
    = \DecompSp(\CalS_{\Phi}^{(\delta,r)}, \lebesgue^p, \ell_u^q)$, with
   equivalent norms.
 \end{proof}

\let\section\origsection
\markleft{References}
\markright{}


\begin{thebibliography}{10}

\bibitem{besov59}
O.~V. Besov.
\newblock On a certain family of functional spaces. {E}mbedding and extension
  theorems.
\newblock {\em Dokl. Akad. Nauk SSSR}, (126):1163--1165, 1959.

\bibitem{BorupNielsenAlphaModulationSpaces}
L.~Borup and M.~Nielsen.
\newblock Banach frames for multivariate $\alpha$-modulation spaces.
\newblock {\em J. Math. Anal. Appl.}, 321(2):880--895, 2006.

\bibitem{boni07}
L.~{B}orup and M.~{N}ielsen.
\newblock {F}rame decomposition of decomposition spaces.
\newblock {\em J. Fourier Anal. Appl.}, 13(1):39--70, 2007.

\bibitem{daforastte08}
S.~{D}ahlke, M.~{F}ornasier, H.~{R}auhut, G.~{S}teidl, and G.~{T}eschke.
\newblock {G}eneralized coorbit theory, {B}anach frames, and the relation to
  $\alpha$-modulation spaces.
\newblock {\em Proc. London Math. Soc.}, 96(2):464--506, 2008.

\bibitem{fe82}
H.~G. {F}eichtinger.
\newblock {B}anach spaces of distributions defined by decomposition methods and
  some of their applications.
\newblock In {\em {R}ecent {T}rends in {M}athematics}, volume~50 of {\em
  {T}eubner {T}exte zur {M}athematik}, pages 123--132, {R}einhardsbrunn, 1982,
  1982.

\bibitem{fefo06}
H.~G. {F}eichtinger and M.~{F}ornasier.
\newblock {F}lexible {G}abor-wavelet atomic decompositions for ${{L}_2}$
  {S}obolev spaces.
\newblock {\em {A}nn. {M}at. {P}ura {A}ppl.}, 185(1):105--131, 2006.

\bibitem{fegr85}
H.~G. {F}eichtinger and P.~{G}r{\"o}bner.
\newblock {B}anach spaces of distributions defined by decomposition methods.
  {I}.
\newblock {\em Math. Nachr.}, 123:97--120, 1985.

\bibitem{fegr89}
H.~G. {F}eichtinger and K.~{G}r{\"o}chenig.
\newblock {B}anach spaces related to integrable group representations and their
  atomic decompositions, {I}.
\newblock {\em J. Funct. Anal.}, 86(2):307--340, 1989.

\bibitem{fegr89-1}
H.~G. {F}eichtinger and K.~{G}r{\"o}chenig.
\newblock {B}anach spaces related to integrable group representations and their
  atomic decompositions, {I}{I}.
\newblock {\em Monatsh. Math.}, 108(2-3):129--148, 1989.

\bibitem{FollandRA}
G.~Folland.
\newblock {\em {R}eal {A}nalysis: {M}odern {T}echniques and {T}heir
  {A}pplications}.
\newblock Pure and applied mathematics. Wiley, second edition, 1999.

\bibitem{fora05}
M.~{F}ornasier and H.~{R}auhut.
\newblock {C}ontinuous frames, function spaces, and the discretization problem.
\newblock {\em J. Fourier Anal. Appl.}, 11(3):245--287, 2005.

\bibitem{fuhr2021coorbit}
H.~F{\"u}hr and J.~T. van Velthoven.
\newblock Coorbit spaces associated to integrably admissible dilation groups.
\newblock {\em Journal d'Analyse Math{\'e}matique}, pages 1--45, 2021.

\bibitem{fuhr2015wavelet}
H.~F{\"u}hr and F.~Voigtlaender.
\newblock Wavelet coorbit spaces viewed as decomposition spaces.
\newblock {\em Journal of Functional Analysis}, 269(1):80--154, 2015.

\bibitem{Grafakos2009Modern}
L.~Grafakos.
\newblock {\em Modern Fourier Analysis}.
\newblock Springer New York, New York, NY, 2009.

\bibitem{gr92-2}
P.~{G}r{\"o}bner.
\newblock {\em {B}anachr{\"a}ume glatter {F}unktionen und
  {Z}erlegungsmethoden}.
\newblock PhD thesis, {U}niversity of {V}ienna, 1992.

\bibitem{gr91}
K.~{G}r{\"o}chenig.
\newblock {D}escribing functions: atomic decompositions versus frames.
\newblock {\em Monatsh. Math.}, 112(3):1--41, 1991.

\bibitem{helawe02}
E.~{H}ern{\'a}ndez, D.~{L}abate, and G.~{W}eiss.
\newblock {A} unified characterization of reproducing systems generated by a
  finite family. {I}{I}.
\newblock {\em J. Geom. Anal.}, 12(4):615--662, 2002.

\bibitem{VoHoModules}
N.~Holighaus and F.~Voigtlaender.
\newblock {Schur-type Banach modules of integral kernels acting on mixed-norm
  Lebesgue spaces}.
\newblock {\em Journal of Functional Analysis}, 281(9):109197, 2021.

\bibitem{VoHoPreprint}
N.~Holighaus and F.~Voigtlaender.
\newblock {Coorbit Theory of Warped Time-Frequency Systems in $\mathbb{R}^d$}.
\newblock {\em Journal of Fourier Analysis and Applications}, to appear, 2024.
\newblock {available at: } \url{https://arxiv.org/abs/2208.01342}.

\bibitem{bahowi15}
N.~Holighaus, C.~Wiesmeyr, and P.~Balazs.
\newblock Continuous warped time-frequency representations—coorbit spaces and
  discretization.
\newblock {\em Applied and Computational Harmonic Analysis}, 47(3):975--1013,
  2018.

\bibitem{howi14}
N.~{Holighaus}, C.~{Wiesmeyr}, and Z.~{Pr\r{u}\v{s}a}.
\newblock {A Class of Warped Filter Bank Frames Tailored to Non-linear
  Frequency Scales}.
\newblock {\em Journal of Fourier Analysis and Applications}, 26(22), 2020.

\bibitem{kempka2015general}
H.~Kempka, M.~Sch{\"a}fer, and T.~Ullrich.
\newblock General coorbit space theory for quasi-{B}anach spaces and
  inhomogeneous function spaces with variable smoothness and integrability.
\newblock {\em Journal of Fourier Analysis and Applications}, pages 1--60,
  2015.

\bibitem{labate2013shearlet}
D.~Labate, L.~Mantovani, and P.~Negi.
\newblock Shearlet smoothness spaces.
\newblock {\em Journal of Fourier Analysis and Applications}, 19(3):577--611,
  2013.

\bibitem{LangRealFunctional}
S.~Lang.
\newblock {\em Real and {F}unctional {A}nalysis}, volume 142 of {\em Graduate
  Texts in Mathematics}.
\newblock Springer-Verlag, New York, third edition, 1993.

\bibitem{ni14}
M.~{N}ielsen.
\newblock {F}rames for decomposition spaces generated by a single function.
\newblock {\em Collect. Math.}, 65(2):183--201, 2014.

\bibitem{CompactlySupportedFramesForDecompositionSpaces}
M.~Nielsen and K.~N. Rasmussen.
\newblock Compactly supported frames for decomposition spaces.
\newblock {\em Journal of Fourier Analysis and Applications}, 18(1):87--117,
  2012.

\bibitem{nikol1962boundary}
S.~M. Nikol'skij.
\newblock On boundary properties of differentiable functions of several
  variables.
\newblock In {\em Dokl. Akad. Nauk SSSR}, volume 146, pages 542--545, 1962.

\bibitem{nikol1963boundary}
S.~M. Nikol'skij.
\newblock On stable boundary values of differentiable functions of several
  variables.
\newblock In {\em Mat. Sb.}, volume~61, pages 224--252, 1963.

\bibitem{rauhut2011generalized}
H.~Rauhut and T.~Ullrich.
\newblock Generalized coorbit space theory and inhomogeneous function spaces of
  {B}esov--{L}izorkin--{T}riebel type.
\newblock {\em Journal of Functional Analysis}, 260(11):3299--3362, 2011.

\bibitem{rosh04}
A.~{R}on and Z.~{S}hen.
\newblock {G}eneralized shift-invariant systems.
\newblock {\em Constr. Approx.}, pages {O}{F}1--{O}{F}45, 2004.

\bibitem{rudin1987real}
W.~Rudin.
\newblock {\em Real and complex analysis}.
\newblock McGraw-Hill, 1987.

\bibitem{RudinFA}
W.~Rudin.
\newblock {\em Functional analysis}.
\newblock International series in pure and applied mathematics. McGraw-Hill,
  1991.

\bibitem{speckbacher2016alpha}
M.~Speckbacher, D.~Bayer, S.~Dahlke, and P.~Balazs.
\newblock The $\alpha$-modulation transform: Admissibility, coorbit theory and
  frames of compactly supported functions.
\newblock {\em Monatshefte für Mathematik}, 184:133--169, 2017.

\bibitem{triebel2010theory}
H.~Triebel.
\newblock {\em Theory of Function Spaces II}.
\newblock Modern Birkh{\"a}user Classics. Springer Basel, 2010.

\bibitem{voigtlaender2015embedding}
F.~Voigtlaender.
\newblock {\em Embedding theorems for decomposition spaces with applications to
  wavelet coorbit spaces}.
\newblock PhD thesis, PhD thesis, RWTH Aachen University, 2015.
  http://publications. rwth-aachen. de/record/564979, 2015.

\bibitem{voigtlaender2016embeddings}
F.~Voigtlaender.
\newblock Embeddings of decomposition spaces.
\newblock {\em arXiv preprints}, 2016.
\newblock \url{http://arxiv.org/abs/1605.09705}.

\bibitem{DecompositionIntoSobolev}
F.~Voigtlaender.
\newblock {E}mbeddings of {D}ecomposition {S}paces into {S}obolev and {BV}
  {S}paces.
\newblock {\em arXiv preprints}, 2016.
\newblock \url{http://arxiv.org/abs/1601.02201}.

\bibitem{StructuredBanachFrames}
F.~Voigtlaender.
\newblock Structured, compactly supported {B}anach frame decompositions of
  decomposition spaces.
\newblock {\em arXiv preprints}, 2016.
\newblock \url{http://arxiv.org/abs/1612.08772}.

\bibitem{StructuredBanachFrames2}
F.~Voigtlaender and A.~Pein.
\newblock Analysis vs.\@ synthesis sparsity for $\alpha$-shearlets.
\newblock {\em arXiv preprints}, 2017.
\newblock \url{http://arxiv.org/abs/1702.03559}.

\bibitem{vybiral2006function}
J.~Vybiral.
\newblock Function spaces with dominating mixed smoothness.
\newblock {\em Dissertationes Math. (Rozprawy Mat.)}, 436, 01 2006.

\end{thebibliography}
\end{document}